\documentclass[11pt,a4paper,reqno]{amsart}

\usepackage{etoolbox}
\patchcmd{\subsubsection}{\itshape}{\bfseries}{}{}

\usepackage{tikz}
\usetikzlibrary{intersections,calc,arrows.meta}

\usepackage{physics}
\usepackage{color}
\usepackage{indentfirst}
\usepackage[top=30truemm,bottom=30truemm,left=25truemm,right=25truemm]{geometry}
\usepackage{bm}
\usepackage{amsmath,amssymb}
\usepackage{amsthm}
\usepackage[numbers]{natbib}
\usepackage{subcaption}
\usepackage{here}  
\usepackage{mathtools}
\usepackage{url}
\usepackage{bbm}
\usepackage{breqn}

\makeatletter\newcommand\exchange[2]{\let\@tempa#1\let#1#2\let#2\@tempa}\makeatother
\exchange\epsilon\varepsilon 

\makeatletter\newcommand\exchangee[2]{\let\@tempa#1\let#1#2\let#2\@tempa}\makeatother
\exchangee\phi\varphi 


\renewcommand{\labelenumi}{\arabic{enumi}.}

\numberwithin{equation}{section}

\usepackage[hidelinks]{hyperref}

\newtheorem{thm}{Theorem}[section]
\newtheorem{lem}[thm]{Lemma}
\newtheorem{cor}[thm]{Corollary}
\newtheorem{prp}[thm]{Proposition}

\theoremstyle{definition}

\newtheorem{cond}[thm]{Condition}
\newtheorem{rem}[thm]{Remark}
\newtheorem{exm}[thm]{Example}

\newcommand{\al}{\alpha}
\newcommand{\be}{\beta}
\newcommand{\ga}{\gamma}

\newcommand{\la}{\lambda}

\newcommand{\si}{\sigma}
\newcommand{\Ga}{\Gamma}
\newcommand{\De}{\Delta}

\newcommand{\La}{\Lambda}
\newcommand{\Si}{\Sigma}

\newcommand{\mA}{\mathcal{A}}

\newcommand{\mC}{\mathcal{C}}

\newcommand{\mF}{\mathcal{F}}
\newcommand{\mG}{\mathcal{G}}
\newcommand{\mH}{\mathcal{H}}

\newcommand{\mL}{\mathcal{L}}

\newcommand{\mN}{\mathcal{N}}

\newcommand{\mT}{\mathcal{T}}

\newcommand{\R}{\mathbb{R}}

\newcommand{\Z}{\mathbb{Z}}

\DeclareMathOperator{\p}{\mathbb{P}}
\DeclareMathOperator{\e}{\mathbb{E}}
\DeclareMathOperator{\Var}{Var}
\DeclareMathOperator{\cov}{Cov}

\newcommand{\indi}{1}

\newcommand{\flr}[1]{\left\lfloor#1\right\rfloor}

\newcommand{\Po}{\mathrm{Po}}

\DeclareMathOperator{\gen}{\mA}

\def\edge{{\begin{tikzpicture}[scale=0.12]%
			\node at (0,0) [draw,circle,fill=black,inner sep=0.45pt] (A) {};%
			\node at (0.8,1) [draw,circle,fill=black,inner sep=0.45pt] (B) {};%
			\draw (A) -- (B); %
\end{tikzpicture}}}

\def\twostar{{\begin{tikzpicture}[scale=0.15]%
			\node at (0,0.79) [draw,circle,fill=black,inner sep=0.45pt] (A) {};%
			\node at (0.5,0) [draw,circle,fill=black,inner sep=0.45pt] (B) {};%
			\node at (1,0.79) [draw,circle,fill=black,inner sep=0.45pt] (C) {};%
			\draw (A) -- (B); %
			\draw (B) -- (C); %
\end{tikzpicture}}}

\begin{document}
	\title[Normal approximation for isolated edges and 2-stars in uniform simple graphs]{Normal approximation of the numbers of isolated edges and isolated 2-stars in uniform simple graphs with given vertex degrees}
	\author[R.~Imai]{Ryo Imai}
		\address[R.~Imai]{Department of Statistics and Data Science, National University of Singapore.}
		\email{rimai@u.nus.edu}
	\date{First version: May 9, 2026. This version: May 28, 2026}
	\subjclass[2010]{}
	
	\begin{abstract}
We consider the configuration model and the uniform simple graph with given degree sequence $\bm{d}=\left(d_i\right)_{i=1}^n$. We derive quantitative bounds for the errors in (i) joint normal-Poisson approximation to the numbers of isolated edges, isolated 2-stars, self-loops and double edges in the configuration model, and (ii) normal approximation to the numbers of isolated edges and isolated 2-stars conditioned on that the configuration model is simple. The latter provides the first finite sample normal approximation results for the uniform simple graph with given vertex degrees. To achieve this, we develop a new Stein's method for joint normal-Poisson approximation and a new coupling approach to sums of indicators, which may be of independent interest.
	\end{abstract}
	
	\maketitle
	
	\section{Introduction}
	
The \emph{configuration model}, which is a random multigraph with prescribed degrees, is defined by taking a set of labeled $d_i$ \emph{half-edges} at each vertex $i$ and then joining the half-edges into edges by taking a uniform partition of the set of all half-edges into pairs. Since we match the half-edges uniformly, we may create self-loops and multiple edges. It was originally introduced by \cite{Bollo80} to obtain the asymptotics of the number of regular graphs and has gained attention from both theoretical and applied research; see, for instance, \cite{J09,J14,AHH19} on the Poisson approximation to the numbers of self-loops and multiple edges (pairs of parallel edges) and its variants, \cite{MolloyReed1995,MolloyReed1998,JL09,BolloRiordan2015,BallNeal2017,JPRR18} on the existence of a giant component in the supercritical case, and  \cite{Hofstad17,FLNU18} on the use of it as the null model in network analysis.
	
	\medskip
	A more precise description of the model is as follows. Let $\bm{d}=\qty(d_i)_{i=1}^n$ be a degree sequence (a sequence of $n$ nonnegative integers) and let $N=\sum_{i=1}^{n}d_i$ be the total number of half-edges. Assume that $N$ is even. We will denote the vertices by $v_1,\dots,v_n$; thus $v_i$ has by definition degree $d_i$. The half-edges are numbered in an arbitrary order from $1$ to $N$, and we write $s\in v_i$ if the half-edge $s$ is incident to $v_i$. Let $\mathfrak{g}$ denote the uniform random partition of the half-edges $\{1,\dots,N\}$ into pairs (e.g., $\qty{\{1,2\},\dots,\{N-1,N\}}$ is one of possible realizations of the random partition $\mathfrak{g}$).\footnote{$\mathfrak{g}$ is often referred to as the \emph{configuration}. In this paper, a realization of $\mathfrak{g}$ is also often called a ``configuration'', slightly abusing the term.} Then the configuration model $\mathrm{CM}_n(\bm{d})$ is defined through $X_{ii}$, the number of self-loops in $v_i$, and $X_{ij}$, $i\neq j$, the number of edges between vertices $v_i$ and $v_j$, as 
	\begin{equation}\label{configuration model definition}
	X_{ii}\coloneqq\sum_{s<t\in v_i}I_{st},\quad\quad X_{ij}\coloneqq\sum_{s\in v_i,t\in v_j}I_{st},
	\end{equation}
	where $I_{st}$ denotes the random indicator that the half-edges $s$ and $t$ are paired in $\mathfrak{g}$. (Thus, $I_{st}=I_{ts}$.) As is well-known, when conditioned on the event that it is simple, the configuration model is uniform over all simple graphs with degree sequence $\bm{d}$ (assuming tacitly that $\bm{d}$ is such that some such a simple graph exists, \cite[Proposition 7.15]{Hofstad17}). In particular,
	\begin{equation}\label{configuration model uniform over simple}
\p\qty(\mathrm{CM}_n(\bm{d})=G)=\frac{\prod_{i=1}^nd_i!}{(N-1)!!}\quad\text{for any simple graph $G$ with degree sequence $\bm{d}$},
	\end{equation}
where $(N-1)!!$ denotes the double factorial.\footnote{We understand that $1/(N-1)!!=1$ if $N=0$, i.e., $d_i=0$ for all $i$.} We include a formal proof of \eqref{configuration model uniform over simple} in Appendix \ref{configuration model uniform over simple proof} for the sake of completeness. The random graph chosen uniformly at random among all such simple graphs is referred to as the \emph{uniform simple graph} with degree sequence $\bm{d}$ in this paper.

	\medskip
	In contrast to Poisson and other approximation to the number of self-loops and multiple edges and the probability of simplicity, asymptotic normality results (central limit theorems) for the configuration model have not been obtained in the literature until recently. \cite{AY18} showed asymptotic normality of certain statistics in a subcritical case using the martingale CLT. \cite{BR19} used Stein's method and proved a theorem on normal approximation of a class of ``local'' statistics that include, for example, the number of small components of a given type. They further used this to show asymptotic normality of the size of the giant component in the supercritical case. However, they left the problem of transferring the distributional limit results to the uniform simple graph being open (\cite[Remark 2.5]{BR19}, see also \cite[Remark 1.4, the second paragraph]{J10}).
	
	\medskip
	Finally, \cite{J20} gave the asymptotic distributions (under the common assumption on the degree sequence $(d_i^{(n)})_{i=1}^n$ in asymptotics as $n\rightarrow\infty$) for the numbers of small components of different types in the uniform simple graph, specifically, asymptotic normality of the numbers of isolated trees and Poisson convergence of the numbers of isolated connected unicyclic multigraphs. He used the method of moments to show the joint convergence of them and the numbers of self-loops and double edges in the configuration model, where the normal and Poisson parts are asymptotically independent and all the Poisson parts are asymptotically independent (Lemma 7.10 ibidem).  Then by conditioning on the numbers of self-loops and double edges being $0$, he obtained the same joint convergence of the numbers of isolated trees and isolated connected unicyclic multigraphs in the uniform simple graph (Theorem 3.9 ibidem). He conjectured that a combination of Stein's method for normal approximation and the Chen--Stein method for Poisson approximation might be used to prove the same results (see Remark 1.2 ibidem).
	
	\medskip
In this paper, we study distributional approximation in the configuration model and the uniform simple graph with given vertex degrees via Stein's method. In particular, we consider normal approximation to the numbers of two kinds of the simplest trees, isolated edges and isolated 2-stars, along with Poisson approximation to the numbers of self-loops and double edges (pairs of parallel edges), in the configuration model. Building on a new development of Stein's method for sums of indicator random variables, we derive quantitative bounds for the errors in (i) joint normal-Poisson approximation to the numbers of isolated edges, isolated 2-stars, self-loops and double edges, and (ii) normal approximation to the numbers of isolated edges and isolated 2-stars conditioned on the event that the configuration model is simple. The latter provides the first finite sample normal approximation results for the uniform simple graph with given vertex degrees.
	
	\medskip
	More precisely, we consider the following four quantities in the configuration model:
	\begin{align*}
		(\text{\# of isolated edges})\:\: \overline{Z}_{\edge}=&\sum_{\substack{1\le i<j\le n\\ \mathrm{s.t.}\ d_i=d_j=1}}X_{ij},\\
		(\text{\# of isolated 2-stars})\:\: \overline{Z}_{\twostar}=&\sum_{\substack{1\le i<j\le n,\ 1\le k\le n\\ \mathrm{s.t.}\ d_i=d_j=1,\ d_k=2}}X_{ik}X_{jk},\\
		(\text{\# of self-loops})\:\: S_n=&\sum_{1\le i\le n}X_{ii},\\
		(\text{\# of double edges})\:\: M_n=&\sum_{1\le i<j\le n}\binom{X_{ij}}{2}.
		\end{align*}
		We define the normalized versions of $\overline{Z}_{\edge}$ and $\overline{Z}_{\twostar}$ as follows, in accordance with \cite[(3.2)]{J20}:
		\begin{equation}\label{noramalized sums edge twostar}
		W_\edge\coloneqq\frac{\overline{Z}_{\edge}-\e\overline{Z}_{\edge}}{\sqrt{n}},\qquad W_\twostar\coloneqq\frac{\overline{Z}_{\twostar}-\e\overline{Z}_{\twostar}}{\sqrt{n}}.
	\end{equation}
	Then the conditional distribution $\mL\qty(W_\edge,W_\twostar\:\middle|\:S_n+M_n=0)$ equals the corresponding distribution of $\qty(W_\edge, W_\twostar)$ in the uniform simple graph with the degree sequence $\bm{d}$. (Note that we continue to normalize $\overline{Z}_{\edge}$ and $\overline{Z}_{\twostar}$ using the expectations with respect to the unconditional $\mathrm{CM}_n(\bm{d})$, cf.\ \cite[Theorem 3.9]{J20}.) We will obtain the quantitative error bounds for joint normal-Poisson approximation to $\mL\qty(W_\edge,W_\twostar, S_n,M_n)$ (Theorem \ref{configuration model clt}(a)), Poisson approximation to the probability of simplicity $\p\qty(S_n+M_n=0)$ (Theorem \ref{configuration model clt}(b)), and normal approximation to the conditional distribution $\mL\qty(W_\edge,W_\twostar\:\middle|\:S_n+M_n=0)$ (Theorem \ref{configuration model clt}(c) and Theorem \ref{another conditional clt}). Furthermore, we deduce the weak convergence of the joint distribution $\mL\qty(W_\edge,W_\twostar, S_n,M_n)$ and the conditional distribution $\mL\qty(W_\edge,W_\twostar\:\middle|\:S_n+M_n=0)$ from these bounds under the same set of assumptions on the asymptotics of the degree sequence as in \cite{J20}; see Corollary \ref{weak convergence result} below. This partially answers the conjecture in \cite[Remark 1.2]{J20} mentioned above.
	
	\medskip
	Toward these goals, we prepare an abstract framework to approximate the distributions of sums of indicator random variables and apply it to the problem of the motif counts in the configuration model. In Section \ref{new method}, we develop a new	Stein's method for joint normal-Poisson approximation by introducing a new Stein equation, constructing a solution based on a combination of the Slepian interpolation for normal approximation and a (new) interpolation for multivariate Poisson approximation, and show the smoothness estimates of the solution. In Section \ref{indicator framework}, we propose a new coupling concept for indicator summands, which is a kind of rendition of the procedure described in \cite[Section 3.4]{BR19} and provides a powerful means to grasp the (even global) dependence structure among the summands. We prove an abstract theorem about joint normal-Poisson approximation to the sums of indicators based on the coupling, along with the new Stein equation and	solution from the previous section. These developments would be of
	independent interest.
	
\medskip
As for our new method developed in Sections \ref{new method} and \ref{indicator framework}, we mention the differences from \cite{BP14}, who also derive joint normal-Poisson approximation results using Stein's method. The differences are threefold: First, their approximation error bounds are expressed in terms of Malliavin-type operators defined on the space of functionals
of a Poisson measure and it seems difficult to apply their results to our problem of the configuration model. Second, their proof for the multivariate Poisson approximation part is based on the argument by \cite{AGG89}, which decomposes the discrepancy as a telescoping sum of the coordinate-wise discrepancies and controls each term appearing in the sum separately by the one-dimensional Chen--Stein method. In contrast, our approach is based on an interpolation version of the generator method by \cite{B88}, which is known to get close to the best known rate in the case of multivariate Poisson-binomial distribution (see \cite{B05} for details).

\medskip
Third, the most importantly, their proof strategy is to use a separate parametrized Stein equation in each case of normal approximation and Poisson approximation and combine the estimates using the triangle inequality. On the other hand, we consider a single Stein equation for joint normal-Poisson approximation and construct a new solution to it, which produces the magic factor even for the `cross' term. However,  the recent literature on the Stein--Malliavin calculus and quantitative asymptotic independence \cite{Pim24,TudorZurcher2023,Tudor2025,TudorZurcher2025,TudorZurcher2024} focuses on distributional approximation to the first coordinate random varianble and its asymptotic independence of the second coordinate random variable and uses the parametrized Stein equation for the first coordinate variable only, which is attractive if one is interested only in the conditional CLT and not in the joint approximation. We also employ this idea for normal approximation in uniform simple graphs with given vertex degrees (Theorem \ref{another conditional clt} below).
	
	\subsection{Organization} The rest of this paper is organized as follows. In Section \ref{new method}, we develop a new version of Stein's method for joint normal-Poisson approximation based on the interpolation approach. In Section \ref{indicator framework}, we propose a new coupling approach to sums of indicator random variables, and prove an abstract approximation bound for the sums of indicators.  In Section \ref{configuration model}, we apply the framework developed in the previous sections to our problem in the configuration model, and obtain the main results (Theorems \ref{configuration model clt} and \ref{another conditional clt}). Section \ref{concluding remarks} provides concluding remarks and discusses future research directions. Finally, Appendices \ref{configuration model uniform over simple proof}--\ref{proof of auxiliary lemmas} contain all the proofs deferred in the main text.
	
	\subsection{Notation and conventions}
	For a positive integer $n$, $[n]\coloneqq\{1,\dots, n \}$. For $a,b\in\R$, ($a\vee b$) and ($a\wedge b$) denote the maximum and minimum of $a$ and $b$, respectively. For $a\in\R$, $a_+\coloneqq(a\vee0)$.  
	
	\smallskip
	If $S$ is a statement, we use $\indi_S$ to denote its indicator, thus $1_S=1$ when $S$ is true and $1_S=0$ when $S$ is false. If $E$ is a set, we write $\indi_E$ for the indicator function $\indi_E(\omega)\coloneqq \indi_{\omega\in E}$.
	
	\smallskip
	We use $\log^+$ to denote the positive part of the natural logarithm function $x\mapsto\log x$, that is, $\log^+x=(\log x)\vee 0=(\log x)\indi_{x>1}$. 
	
	\smallskip
	We write $\qty|E|$ for the cardinality of a finite set $E$. We will often write $\qty|\sum_{\al\in E}|$ for $\sum_{\al\in E}1 =\qty|E|$.

	\smallskip
	In $\R^d$, $|\:\cdot\:|$ denotes the Euclidean norm. We use $\Z_+\coloneqq\{0,1,2,\dots\}$ to denote the non-negative integers. We regard $\Z_+^r$ as a subspace of $\R^r$, and the induced relative topology on $\Z_+^r$ coincides with the discrete topology. For functions, $\|\:\cdot\:\|$ denotes the supremum norm. Precisely, for every function $f:\R^d\times\Z^r_+\rightarrow\R^m$, $g:\R^d\rightarrow\R^m$ and $h:\Z_+^r\rightarrow\R^m$, let
	\begin{align*}
		\|f\|=\sup_{(x,y)\in\R^d\times\Z^r_+}|f(x,y)|,\quad\quad\|g\|=\sup_{x\in\R^d}|g(x)|,\quad\quad\|h\|=\sup_{y\in\Z_+^r}|h(y)|.
	\end{align*}
	Also, for functions on $\R^d$, $|\:\cdot\:|_{\mathrm{Lip}}$ denotes the Lipschitz seminorm
	\begin{align*}
		|g|_{\mathrm{Lip}}=\sup_{x\neq x'}\frac{|g(x)-g(x')|}{|x-x'|}.
	\end{align*}
	When $g:\R^d\rightarrow\R$ is of class $C^k$, the partial derivatives of order $\ell\le k$ are abbreviated as $\partial_{j_1\cdots j_\ell} g =\frac{\partial^\ell}{\partial x_{j_1}\cdots\partial x_{j_\ell}}g$. Let $C^k_b(\R^d)$ be the class of real-valued $C^k$-functions $g$ on $\R^d$ such that $g$ and all its partial derivatives up to the order $k$ are bounded. If $g\in C^k_b(\R^d)$, for $\ell\in\{1,\dots,k\}$, let
	\begin{align*}
		|g|_\ell=\max_{1\le j_1\le\cdots\le j_\ell\le d}\|\partial_{j_1\cdots j_\ell} g \|.
	\end{align*}
	If $g:\R^d\rightarrow\R$ is of class $C^1$, we write $\grad{g}(x)$ for the gradient vector of $g$, i.e., $\grad{g}(x)=\qty(\partial_1 g(x),\dots,\partial_dg(x))^\top$. If $g\in C^1_b(\R^d)$, we have $|g|_1\le|g|_{\mathrm{Lip}}=\|\grad{g}\|\le\sqrt{d}|g|_1$.
	
	\smallskip
	In this paper, we reserve the notation $\{e_j\}_{1\le j\le r}$ for a vector $e_j=(\indi_{k=j})_{1\le k\le r}\in\Z_+^r$, $1\le j\le r$. To avoid confusion, we use $\{e_{j,d}\}_{1\le j\le d}$ to denote the standard basis for $\R^d$ if necessary, though it is rare. Given a function $h:\Z_+^r\rightarrow\R$, we define the first-order partial differences $\De_jh:\Z_+^r\rightarrow\R$ by $\Delta_jh(y)=h(y+e_j)-h(y)$. Then the second-order partial differences $\De_{jk}h:\Z_+^r\rightarrow\R$ are defined by $\Delta_{jk}h=\Delta_j(\Delta_kh)$. For a function $f(x,y):\R^d\times\Z^r_+\rightarrow\R$ which is sufficiently smooth in $x$, we often need to take partial derivatives about $x$ of $f$ keeping $y$ fixed, and partial differences about $y$ of $f$ keeping $x$ fixed. Define the $x$-section $f_x$ and the $y$-section $f^y$ of $f$ by $f_x(y)=f^y(x)=f(x,y)$. Then we abide by the convention that $\partial_{j}f(x,y)=\partial_jf^y(x)$ and $\De_jf(x,y)=\De_jf_x(y)$.

	\smallskip
	$U(0,1)$ denotes the uniform distribution over $(0,1)$. For a positive integer $n$, $\mathrm{Unif}[n]$ denotes the discrete uniform distribution over $[n]=\{1,\dots,n\}$. For a positive semidefinite matrix $\Si$, $\mN_d(0,\Si)$ denotes the $d$-dimensional normal distribution with mean 0 and covariance matrix $\Si$. For $\la>0$, $\Po(\la)$ denotes a Poisson measure on $\Z_+$ with mean $\la$. For an $r$-tuple $\la=(\la_1,\dots,\la_r)$ of positive real numbers, $\Po(\la)$ then denotes the product measure $\prod_{j=1}^r\Po(\la_j)$ on $\Z_+^r$. For two random elements $X$ and $Y$, $X\stackrel{d}{=}Y$ means that the distributions of $X$ and $Y$ are the same.
	
	\section{Stein's method for joint normal-Poisson approximation}
	\label{new method}
	
	Throughout this section, we assume that the following independent random elements are defined on a probability space $(\Omega,\mF,\p)$:
	
	{\setlength{\leftmargini}{29.5pt}  	
		\begin{itemize}
			\setlength{\labelsep}{7pt}     
			\setlength{\itemsep}{3pt}      
			
			\item $Z\sim\mN_d(0,\Si)$, a $d$-dimensional normal random vector with a possibly singular covariance matrix $\Si=\mqty(\si_{jk})$.
			
			\item $N=(N_1,\dots,N_r)\sim\Po(\la)=\prod_{j=1}^r\Po(\la_j)$, a random vector in $Z_+^r$ whose coordinates are independent Poisson random variables $N_j\sim\Po(\la_j)$ with an $r$-tuple $\la=(\la_1,\dots,\la_r)$ of positive means.

			\item $\eta_{j,1},\eta_{j,2},\dots$, $1\le j\le r$, independent exponentially distributed random variables with mean $1$.
			
			\item $\zeta_{j,1},\zeta_{j,2},\dots$, $1\le j\le r$, independent exponentially distributed random variables with mean $1$.

	\end{itemize} }
	
	Then:
	
	{\setlength{\leftmargini}{29.5pt}  	
		\begin{itemize}
			\setlength{\labelsep}{7pt}     
			\setlength{\itemsep}{3pt}      
			
			\item  Define a stochastic process $\{Z_0(s),s\in[0,1] \}$ on $\Z_+^r$ by setting its $j$-th coordinate to be a compound Poisson process  $Z_{0,j}(s)=\sum_{\ell=1}^{N_j}\indi_{\{\eta_{j,\ell}\ge-\log s \}}$ for $s\in(0,1]$ (recall that $N_j\sim\Po(\la_j)$) and $Z_{0,j}(0)\coloneqq 0$. Note that $Z_0(0)=0$ and
			\begin{align*}
				Z_0(s)\sim\Po(s\la)=\prod_{j=1}^r\Po(s\la_j),\quad\quad s\in(0,1].
			\end{align*}
			Indeed, for $s\in(0,1]$ the probability generating function of $Z_{0,j}(s)$ is
			\begin{align*}
				\e\qty[z^{ Z_{0,j}(s) }]=\sum_{n=0}^\infty e^{-\la_j}\frac{{\la_j}^n}{n!}\qty(1-s+sz)^n=e^{-\la_j}e^{\la_j(1-s+sz)}=\exp{-s\la_j(1-z)},\quad|z|\le1,
			\end{align*}
			the probability generating function of $\Po(s\la_j)$.

			\item For each $y\in\Z_+^r$, let $\{D_y(s),s\in[0,1] \}$ be a stochastic process on $\Z_+^r$ whose $j$-th coordinate is $D_{y,j}(s)=\sum_{\ell=1}^{y_j}\indi_{\{\zeta_{j,\ell}>-\log (1-s) \}}$ for $s\in[0,1)$ and $D_{y,j}(1)=0$. Note that $D_y(0)=y$ a.s.\ and 
			\begin{align*}
				D_y(s)\sim\mathrm{Bin}(y,1-s)=\prod_{j=1}^{r}\mathrm{Bin}(y_j,1-s), \quad\quad s\in[0,1].
			\end{align*}

	\end{itemize} }

	\medskip
	
	Let
	\begin{equation}\label{C3}
		\mC_3=\qty{g\in C_b^3(\R^d):|g|_{\mathrm{Lip}}\vee|g|_2\vee|g|_3\le1}.
	\end{equation}
	Consider a test function class
	\begin{equation}\label{H3}
		\mH_3=\qty{h:\R^d\times\Z_+^r\rightarrow\R:\text{$\|h\|\le1 $, and $h(\cdot,y)\in \mC_3$ for each $y\in\Z^r_+$} }.
	\end{equation}
	These are \cite[Definitions 3.1 and 3.3]{BP14}'s $\mC_3$ and $\mH_3$. Note that, for $h\in\mH_3$, $k\in\{1,2,3\}$ and $j_1,\dots,j_k\in\{1,\dots,d\}$, $\|\partial_{j_1\cdots j_k} h\|\le1$. The test function class $\mH_3$ is large enough to characterize convergence in distribution in $\R^d\times\Z_+^r$; we can show that a sequence of random elements $(X_n,Y_n),n=1,2,\dots$ in $\R^d\times\Z_+^r$ converges in distribution to a random element $(X,Y)$ in $\R^d\times\Z_+^r$ if and only if $\lim_n\e[h(X_n,Y_n)]=\e[h(X,Y)]$ for all $h\in\mH_3$ (Proposition \ref{convergence in distribution in RZ}). Thus, taking the limit random element as $(X,Y)=(Z,N)$, it is natural to aim to bound the quantity 
	\begin{equation}\label{goal quantity}
		\sup_{h\in\mH_3}\Big|\e\qty[h\qty(X_n,Y_n)]-\e\qty[h\qty(Z,N)]  \Big|
	\end{equation}
	to quantify the error in joint normal-Poisson approximation to the distribution of $(X_n,Y_n)$. Define an operator $\gen_{\Si,\la}$ that acts on a function $f(x,y):\R^d\times\Z_+^r\rightarrow\R$ whose any $y$-section $f^y$ is of class $C^2$ by the formula
		\begin{equation}\label{new generator}
			\begin{split}
			(\gen_{\Si,\la} f)(x,y)=&\sum_{j,k=1}^d\si_{jk}(\partial_{jk}f )(x,y)-\sum_{j=1}^dx_j (\partial_jf)(x,y)\\
			&+2\sum_{j=1}^{r}\qty[\la_j\Big(f(x,y+e_j)-f(x,y)\Big)+y_j\Big(f(x,y-e_j)-f(x,y)\Big) ].
			\end{split}
		\end{equation}
		If $y_j=0$, $f(x,y-e_j)$ can be defined arbitrarily and the term 
		$y_j\qty(f(x,y-e_j)-f(x,y))$ should be understood as $0$. Now for $h\in\mH_3$, consider the following Stein equation
		\begin{align}\label{eq Stein equation original}
			(\gen_{\Si,\la} f)(x,y)=h(x,y)-\e[h(Z,N)],\quad\quad (x,y)\in\R^d\times\Z_+^r.
		\end{align}
		From now on, we will show that the equation \eqref{eq Stein equation original} admits a solution $f_h$ whose any $y$-section has continuous bounded first, second and third partial derivatives. With this solution, we have
		\begin{align*}
			\Big|\e\qty[h\qty(X_n,Y_n)]-\e\qty[h\qty(Z,N)]  \Big|\le\Big|\e\qty[(\gen_{\Si,\la} f_h)(X_n,Y_n)]\Big|,
		\end{align*}
	and if we can evaluate the right-hand side (using the suitable smoothness estimates about the solution $f_h$) in such a way that is uniform in $h\in\mH_3$, we obtain an upper bound for the quantity \eqref{goal quantity}. In order to construct such a solution, we take an \emph{interpolation} argument, which will be detailed below.
		
		\medskip
		 For $s\in[0,1]$, define an interpolation operator $T_s$ that acts on a function $h:\R^d\times\Z_{+}^r\rightarrow\R$ in $\mH_3$ and is defined by $T_sh: \R^d\times\Z_+^r\rightarrow\R$ given as
	\begin{equation}\label{new interpolation}
		T_sh(x,y)=\e\qty[h\qty(\sqrt{1-s}x+\sqrt{s}Z,D_y(s)+Z_0(s))  ]\quad\quad((x,y)\in\R^d\times\Z_+^r).
	\end{equation}
	Then we have
	\begin{align*}
		T_0h(x,y)=h(x,y),\quad\quad T_1h(x,y)=\e[h(Z,N)].
	\end{align*}
	For each $y\in\Z_+^r$, let $S_s h(\cdot,y):\R^d\rightarrow\R$ be the Slepian interpolation:
	\begin{align*}
		S_s h(x,y)=\e\qty[h\qty(\sqrt{1-s}x+\sqrt{s}Z,y)  ]\quad\quad(x\in\R^d).
	\end{align*}
	Then by the independence of $Z$ and $D_y(s)+Z_0(s)$,
	\begin{equation}\label{T independence}
		T_sh(x,y)=\sum_{z\in\Z_+^r}S_s h(x,z)p_s(y,z),
	\end{equation}
	where $p_s(y,z)=\p(D_y(s)+Z_0(s)=z)$, whose concrete expression is given by the formula
	\begin{equation}\label{transition probability}
		p_s(y,z)=\prod_{j=1}^{r}	\sum_{k=0}^{y_j\wedge z_j}\binom{y_j}{k}(1-s)^{k}s^{y_j-k}[\la_js]^{z_j-k}\frac{e^{-\la_js}}{(z_j-k)!}.
	\end{equation}
Conversely, we may \emph{define} the interpolation operator $T_s$ through \eqref{T independence} with $p_s(y,z)$ defined by the right-hand side of \eqref{transition probability}. One can check that this is well-defined by $|S_sh(x,z)|\le 1 $ and \eqref{M bound} and \eqref{M finite} below. We will adopt this definition of $T_sh$ without introducing the processes $Z_0(s)$ and $D_y(s)$ in the subsequent sections. However, in order to prove various properties of $T_s$ and the Stein solution $f_h$ which will be defined in \eqref{eq new solution} below, it is convenient to introduce these processes as in the beginning of this section and express $T_sh$ in the form \eqref{new interpolation}, and we may do it without loss of generality by extending the probability space.

\begin{rem}
The use of the process $D_y(s)+Z_0(s)$ and its marginal distributions \eqref{transition probability} here is inspired by the so-called generator approach to Stein's method for multivariate Poisson approximation first introduced by \cite{B88}. Indeed, if one makes the
	substitution $s=1-e^{-t}$, $t\ge0$, then \eqref{transition probability} becomes the transition probabilities of the Markov pure jump process on $\Z_+^r$ whose coordinates are independent and each of them is an immigration-death process with immigration rate $\la_j$ and unit \emph{per capita} death rate starting at $y_j$ (cf.\ \cite[equation (3.2.40)]{And91} with $a=\la_j$ and $\mu=1$). This is analogous to the relationship between the Slepian interpolation and the Ornstein--Uhlenbeck semigroup. An advantage of the interpolation approach is that we can eschew the theory of semigroups or Markov processes in justifying Stein's method.
\end{rem}
	
	\medskip
	For $h\in\mH_3$, $s\mapsto S_sh(x,y)$ is continuous on $[0,1]$ and continuously differentiable on $(0,1)$ with the derivative satisfying
	\begin{equation}\label{h Slepian backward}
		\pdv{s}S_s h(x,y)=\frac{1}{2(1-s)}\qty(\sum_{j,k=1}^d\si_{jk}(\partial_{jk}S_s h)(x,y)-\sum_{j=1}^dx_j (\partial_jS_s h)(x,y)),\quad\quad s\in(0,1)
	\end{equation}
	(see Lemma \ref{Slepian backward equation}). Similarly, it can be directly verified that  $s\mapsto p_s(y,z)$ is continuous on $[0,1]$ and continuously differentiable on $(0,1)$ with the derivative satisfying
	\begin{equation}\label{h immigration death backward}
		\pdv{s}p_s(y,z)=\frac{1}{1-s}\sum_{j=1}^{r}\qty[\la_j\Big(p_s(y+e_j,z)-p_s(y,z)\Big)+y_j\Big(p_s(y-e_j,z)-p_s(y,z)\Big) ],\quad s\in(0,1)
	\end{equation}
	(see Lemma \ref{immigration death backward equation}). Again, if $y_j=0$, $p_s(y-e_j,z)$ can be defined arbitrarily and the second term of the $j$-th summand in \eqref{h immigration death backward} is understood as $0$.

	\medskip
	We summarize the useful properties of the interpolation operator $T_s$.
	\begin{prp}\label{T interpolation backward}
		For $h\in\mH_3$, $x\mapsto T_sh(x,y)$ is of class $C^2$, and $s\mapsto T_sh(x,y) $ is continuous on $[0,1]$ and continuously differentiable on $(0,1)$ with the derivative satisfying
		\begin{equation}\label{eq T interpolation backward}
			\pdv{s}T_sh(x,y)=\frac{1}{2(1-s)}\qty(\gen_{\Si,\la} T_sh)(x,y),\quad\quad s\in(0,1),
		\end{equation}
		where $\gen_{\Si,\la}$ is the operator defined in \eqref{new generator}.
	\end{prp}
	\begin{proof}
		Set
		\begin{align*}
			M_{y_j,z_j}\coloneqq\begin{cases}
				1	& (z_j\le y_j)\\
				\frac{1}{(z_j-y_j)!}\qty(1\vee[\la_j]^{z_j})& (z_j> y_j)
			\end{cases},\quad\quad M_{y,z}\coloneqq\prod_{j=1}^{r}M_{y_j,z_j}.
		\end{align*}
		Then
		\begin{equation}\label{M bound}
			\forall s\in[0,1],\quad p_s(y,z)\le M_{y,z},
		\end{equation}
		and 
		\begin{equation}\label{M finite}
			\sum_{z\in\Z_+^r}M_{y,z}=\prod_{j=1}^{r}\qty(\sum_{z_j=0}^{\infty}M_{y_j,z_j})<\infty.
		\end{equation}
		Since $h(\cdot,y)\in\mC_3$ for each $y\in\Z_+^r$, $x\mapsto S_sh(x,y)$ is of class $C^2$ with the derivative
		\begin{equation}\label{x derivative of slepian y fixed}
			\begin{split}
				(\partial_j S_s h)(x,y)&=\sqrt{1-s}\e\qty[(\partial_j h)(\sqrt{1-s}x+\sqrt{s}Z,y)],\\
				(\partial_{jk} S_s h)(x,y) &=(1-s)\e\qty[(\partial_{jk} h)(\sqrt{1-s}x+\sqrt{s}Z,y)],
			\end{split}
		\end{equation}
		and $s\mapsto S_sh(x,y)$ is continuous on $[0,1]$ and continuously differentiable on $(0,1)$ by Lemma \ref{Slepian backward equation}. By \eqref{T independence}, $|S_sh(x,z)|\le 1 $, \eqref{M bound} and \eqref{M finite}, and taking a sequence $s_n\rightarrow s$ and dominated convergence, one can see that $s\mapsto T_sh(x,y)$ is continuous on $[0,1]$. (Or one can resort to the Weierstrass M-test to conclude that the series in \eqref{T independence} is uniformly convergent on $[0,1]$.) 
		
		\medskip
		Next, note that
		\begin{equation}\label{Slepian x derivative bound}
			\qty|(\partial_j S_s h)(x,y)|\le\sqrt{1-s}\le1,\quad\quad
			\qty|(\partial_{jk} S_s h)(x,y)|\le 1-s\le1
		\end{equation}
		by \eqref{x derivative of slepian y fixed}. Thus for each compact interval $[s_1,s_2]\subset(0,1)$, we have
		\begin{align}\label{Slepian time derivative bound}
			\forall s\in[s_1,s_2],\quad\;\qty|\pdv{s}S_s h(x,z)|\le\frac{1}{2(1-s_2)}\qty(\sum_{j,k=1}^d|\si_{jk}|+\sum_{j=1}^d|x_j|)
		\end{align}
		by \eqref{h Slepian backward}. Similarly, by \eqref{h immigration death backward} and \eqref{M bound}, for each compact interval $[s_1,s_2]\subset(0,1)$ we have
		\begin{align}\label{immigration death time derivative bound}
			\forall s\in[s_1,s_2],\quad\;\qty|\pdv{s}p_s(y,z)|\le\frac{1}{1-s_2} \sum_{j=1}^{r}\qty[\la_j\qty(M_{y+e_j,z}+M_{y,z})+y_j\qty(M_{y-e_j,z}+M_{y,z})].
		\end{align}
		Hence, by \eqref{M bound}, \eqref{Slepian time derivative bound}, $|S_sh(x,z)|\le1$ and \eqref{immigration death time derivative bound}, for each compact interval $[s_1,s_2]\subset(0,1)$, we have
		\begin{equation*}\label{Ts s derivative bound}
			\begin{split}
				\forall s\in[s_1,s_2],\quad&\qty|\qty(\pdv{s}S_s h(x,z))p_s(y,z)|+\qty|S_s h(x,z)\pdv{s}p_s(y,z)|\\
				\le \frac{1}{2(1-s_2)}&\qty[\qty(\sum_{j,k=1}^d|\si_{jk}|+\sum_{j=1}^d|x_j|)M_{y,z}+2\sum_{j=1}^{r}\qty[\la_j\qty(M_{y+e_j,z}+M_{y,z})+y_j\qty(M_{y-e_j,z}+M_{y,z})]].
			\end{split}
		\end{equation*}
		Note that
		\begin{equation*}\label{Ts s derivative dominating function}
			\sum_{z\in\Z_+^r}	\qty[\qty(\sum_{j,k=1}^d|\si_{jk}|+\sum_{j=1}^d|x_j|)M_{y,z}+2\sum_{j=1}^{r}\qty[\la_j\qty(M_{y+e_j,z}+M_{y,z})+y_j\qty(M_{y-e_j,z}+M_{y,z})]]<\infty
		\end{equation*}
		by \eqref{M finite}. Therefore term by term differentiation in \eqref{T independence} is justified\footnote{To see this, one may regard the series in \eqref{T independence} as an integral with respect to counting measure and differentiate under the integral sign, or may invoke the Weierstrass M-test to obtain the uniform convergence of the series expressing the derivative on each compact set and use the theorem on term by term differentiation (e.g., \cite[Corollary 7.3 of Chapter IX]{Lang97}).} and we have
		\begin{equation}\label{term by term differentiation}
			\begin{split}
				\forall s\in(0,1),\quad	\pdv{s}T_sh(x,y)=&\sum_{z\in\Z_+^r}\qty(\qty(\pdv{s}S_s h(x,z))p_s(y,z)+S_s h(x,z)\pdv{s}p_s(y,z))\\
				=&\sum_{z\in\Z_+^r}\qty(\pdv{s}S_s h(x,z))p_s(y,z)+\sum_{z\in\Z_+^r}S_s h(x,z)\pdv{s}p_s(y,z).
			\end{split}
		\end{equation}
		Continuity of $s\mapsto \pdv{s}T_sh(x,y)$ on $(0,1)$ also follows.
		
		\medskip
		By differentiating with respect to $x$ under the integral sign, we have
		\begin{equation}\label{Ts x derivative}
			\begin{split}
				\qty(\partial_jT_sh)(x,y)&=\sqrt{1-s}\e\qty[(\partial_jh)\qty(\sqrt{1-s}x+\sqrt{s}Z,D_y(s)+Z_0(s))], \\
				\qty(\partial_{jk}T_sh)(x,y)&=(1-s)\e\qty[(\partial_{jk}h)\qty(\sqrt{1-s}x+\sqrt{s}Z,D_y(s)+Z_0(s))].
			\end{split}
		\end{equation}
		$\qty(\partial_{j}T_sh)(x,y) $ and $\qty(\partial_{jk}T_sh)(x,y)  $ are continuous in $x$ on $\R^d$ by bounded convergence. This shows that $x\mapsto T_sh(x,y)$ is of class $C^2$. By \eqref{x derivative of slepian y fixed} and the independence of $Z$ and $D_y(s)+Z_0(s)$, we have
		\begin{equation}\label{Ts x derivative iterated integral}
			\begin{split}
				\qty(\partial_jT_sh)(x,y)&=\sum_{z\in\Z_+^r}\qty(\partial_jS_s h)(x,z)p_s(y,z),   \\
				\qty(\partial_{jk}T_sh)(x,y)&=\sum_{z\in\Z_+^r}\qty(\partial_{jk}S_s h)(x,z)p_s(y,z). 
			\end{split}
		\end{equation}
		Since $\qty(\partial_{j}S_sh)(x,y) $ and $\qty(\partial_{jk}S_sh)(x,y)  $ are continuous in $s$ on $[0,1]$ by bounded convergence (see \eqref{x derivative of slepian y fixed}), $\qty(\partial_{j}T_sh)(x,y) $ and $\qty(\partial_{jk}T_sh)(x,y)  $ are also continuous in $s$ on $[0,1]$ by \eqref{M bound}, \eqref{M finite}, \eqref{Slepian x derivative bound}, and \eqref{Ts x derivative iterated integral}. By \eqref{h Slepian backward} and \eqref{Ts x derivative iterated integral}, we have
		\begin{equation}\label{backward Ornstein Uhlenbeck part}
			\sum_{z\in\Z_+^r}\qty(\pdv{s}S_s h(x,z))p_s(y,z)=\frac{1}{2(1-s)}\qty(\sum_{j,k=1}^d\si_{jk}(\partial_{jk}T_s h)(x,y)-\sum_{j=1}^dx_j (\partial_jT_s h)(x,y)).
		\end{equation}
		On the other hand, by \eqref{T independence} and \eqref{h immigration death backward}, we have
		\begin{align}\label{backward immigration death part}
			&\sum_{z\in\Z_+^r}S_s h(x,z)\pdv{s}p_s(y,z)  \notag \\ 
			=&\frac{1}{1-s}\sum_{z\in\Z_+^r}S_s h(x,z)\qty(\sum_{j=1}^{r}\qty[\la_j\Big(p_s(y+e_j,z)-p_s(y,z)\Big)+y_j\Big(p_s(y-e_j,z)-p_s(y,z)\Big) ]) \notag  \\
			=&\frac{1}{1-s}\sum_{j=1}^{r}\qty[\la_j\Big(T_sh(x,y+e_j)-T_sh(x,y)\Big)+y_j\Big(T_sh(x,y-e_j)-T_s(x,y)\Big) ].
		\end{align}
		Combining \eqref{term by term differentiation}, \eqref{backward Ornstein Uhlenbeck part} and \eqref{backward immigration death part} gives \eqref{eq T interpolation backward}. 
	\end{proof}
	\begin{rem}
		That $\qty(\partial_{j}T_sh)(x,y) $ and $\qty(\partial_{jk}T_sh)(x,y)  $ are continuous in $s$ on $[0,1]$ cannot be deduced from \eqref{Ts x derivative} and bounded convergence because for a fixed $\omega$, the path $s\mapsto D_y(s)(\omega)+Z_0(s)(\omega)$ is right-continuous but not left-continuous.
	\end{rem}
	\begin{prp}\label{new solution}
		For any $h\in\mH_3$, the function $f_h:\R^d\times\Z_+^r\rightarrow\R$ given by
		\begin{equation}\label{eq new solution}
			f_h(x,y)=-\int_{0}^{1}\frac{1}{2(1-s)}\qty[T_sh(x,y)-\e h(Z,N)]\dd{s}
		\end{equation}
		is well-defined. In particular, the integral in \eqref{eq new solution} is absolutely convergent.
	\end{prp}
	\begin{proof}
		By extending the probability space, we may assume that additional independent random elements
		
		{\setlength{\leftmargini}{29.5pt}  	
			\begin{itemize}
				\setlength{\labelsep}{7pt}    
				\setlength{\itemsep}{3pt}

				\item $\tilde{N}=(\tilde{N}_1,\dots,\tilde{N}_r)\sim\Po(\la)=\prod_{j=1}^r\Po(\la_j)$,

				\item $\tilde{\zeta}_{j,1},\tilde{\zeta}_{j,2},\dots$, $1\le j\le r$, independent exponentially distributed random variables with mean $1$,

		\end{itemize} }
		
		are defined on the same probability space $(\Omega,\mF,\p)$, independently of everything else. Let $\{\tilde{D}(s),s\in[0,1] \}$ be a stochastic process on $\Z_+^r$ whose $j$-th coordinate is a compound Poisson process  $\tilde{D}_{j}(s)=\sum_{\ell=1}^{\tilde{N}_j}\indi_{\{\tilde{\zeta}_{j,\ell}>-\log (1-s) \}}$ for $s\in[0,1)$ and $\tilde{D}_{j}(1)=0$. Then $\tilde{D}(s)\sim\Po((1-s)\la)$ for $s\in[0,1)$ and $\tilde{D}(s)+Z_0(s)\sim\Po(\la)$ for all $s\in[0,1]$. This decomposition of the measure $\Po(\la)$ is reminiscent of the discrete stability proposed in \cite{SH79} (the Poisson distribution is discrete stable with exponent one). Define
		\begin{align*}
			\tau&\coloneqq\inf\{s\in[0,1]:D_y(s)=\tilde{D}(s)=0\}.
		\end{align*}
		Then
		\begin{align*}
			\tau&=1-\exp\qty{-\max\{\zeta_{j,1},\dots,\zeta_{j,y_j},\tilde{\zeta}_{j,1},\dots,\tilde{\zeta}_{j,N_j},1\le j\le r \}}\\
			&=1-\exp\qty{-\sum_{n\in\Z_+^r}\indi_{\{N=n\}}\max\{\zeta_{j,1},\dots,\zeta_{j,y_j},\tilde{\zeta}_{j,1},\dots,\tilde{\zeta}_{j,n_j},1\le j\le r \}},
		\end{align*}
		which is a random variable. Moreover, $D_y(s)=\tilde{D}(s)=0$ if and only if $s\ge\tau$. Hence
		\begin{align}\label{well defined bound}
			&\int_{0}^{1}\qty|\frac{1}{2(1-s)}\qty[T_sh(x,y)-\e h(Z,N)]|\dd{s}  \notag \\
			=&\int_{0}^{1}\frac{1}{2(1-s)}\qty|\e\qty[h\qty(\sqrt{1-s}x+\sqrt{s}Z,D_y(s)+Z_0(s))-h(Z,\tilde{D}(s)+Z_0(s))]|\dd{s} \notag \\
			\le&\int_{0}^{1}\frac{1}{2(1-s)}\e\qty|h\qty(\sqrt{1-s}x+\sqrt{s}Z,Z_0(s))-h(Z,Z_0(s))|\dd{s}+\int_{0}^{1}\frac{1}{1-s}\p(\tau>s)\dd{s}.
		\end{align}
		For the first term in \eqref{well defined bound}, since $|h(\cdot,Z_0(s))|_{\mathrm{Lip}}\le1$,
		\begin{align*}
			&\int_{0}^{1}\frac{1}{2(1-s)}\e\qty|h\qty(\sqrt{1-s}x+\sqrt{s}Z,Z_0(s))-h(Z,Z_0(s))|\dd{s}\\
			\le&\int_{0}^{1}\frac{1}{2(1-s)}\qty(\sqrt{1-s}|x|+(1-\sqrt{s})\e|Z|)\dd{s}\\
			=&|x|\int_{0}^{1}\frac{1}{2\sqrt{1-s}}\dd{s}+\e|Z|\int_{0}^{1}\frac{1}{2(1+\sqrt{s})}\dd{s}\le|x|+\frac{1}{2}\e|Z|<\infty.
		\end{align*}
		The second term in \eqref{well defined bound} can be written as
		\begin{align*}
			\int_{0}^{1}\frac{1}{1-s}\p(\tau>s)\dd{s}=&\e\qty[\int_{0}^{\tau}\frac{1}{1-s}\dd{s}]\\
			=& \e\qty[-\log(1-\tau)] \\
			=&\e\qty[\max\{\zeta_{j,1},\dots,\zeta_{j,y_j},\tilde{\zeta}_{j,1},\dots,\tilde{\zeta}_{j,N_j},1\le j\le r \}].
		\end{align*}
		For $n\in\Z_+^r$, we have\footnote{$\e[\max\{\zeta_{j,1},\dots,\zeta_{j,y_j},\tilde{\zeta}_{j,1},\dots,\tilde{\zeta}_{j,n_j},1\le j\le r \}]=\psi(\sum_{j=1}^{r}y_j+\sum_{j=1}^{r}n_j)$, where $\psi(m)=\sum_{\ell=1}^{m}1/\ell$, by the R\'{e}nyi Representation of the order statistics of i.i.d.\ exponentials.}
		\begin{align*}
			\e\qty[\max\{\zeta_{j,1},\dots,\zeta_{j,y_j},\tilde{\zeta}_{j,1},\dots,\tilde{\zeta}_{j,n_j},1\le j\le r \}]\le\sum_{j=1}^{r}y_j+\sum_{j=1}^{r}n_j.
		\end{align*}
		Since $N$ is independent of the family $ \{\zeta_{j,1},\zeta_{j,2},\dots,\tilde{\zeta}_{j,1},\tilde{\zeta}_{j,2},\dots,1\le j\le r \}$, we have
		\begin{align*}
			\int_{0}^{1}\frac{1}{1-s}\p(\tau>s)\dd{s}=\e\qty[\max\{\zeta_{j,1},\dots,\zeta_{j,y_j},\tilde{\zeta}_{j,1},\dots,\tilde{\zeta}_{j,N_j},1\le j\le r \}]\le \sum_{j=1}^{r}y_j+\sum_{j=1}^{r}\la_j.
		\end{align*}
		Therefore, the sum of two integrals in \eqref{well defined bound} is finite and the integral in \eqref{eq new solution} is absolutely convergent.
	\end{proof}

	\begin{prp}\label{Stein equation}
		For all $h\in\mH_3$, the function $f_h$ defined in Proposition \ref{new solution} is of class $ C^2$ in $x$ and satisfies the equations
		\begin{align}\label{eq Stein equation}
			(\gen_{\Si,\la} f_h)(x,y)=h(x,y)-\e[h(Z,N)],\quad\quad (x,y)\in\R^d\times\Z_+^r.
		\end{align}
	\end{prp}
	\begin{proof}
		For $0\le a<b\le1$, let $f_h^{a,b}$ be the truncated version of $f_h$
		\begin{equation}\label{truncated solution}
			f_h^{a,b}(x,y)=-\int_{a}^{b}\frac{1}{2(1-s)}\qty[T_sh(x,y)-\e h(Z,N)]\dd{s}.
		\end{equation}
		By \eqref{Ts x derivative},
		\begin{align*}
			&\int_{a}^{b}\frac{1}{2(1-s)}\qty|(\partial_jT_sh)(x,y)|\dd{s}\le\int_{a}^{b}\frac{1}{2\sqrt{1-s}}\dd{s}\le1,\\
			&\int_{a}^{b}\frac{1}{2(1-s)}\qty|(\partial_{jk}T_sh)(x,y)|\dd{s}\le\int_{a}^{b}\frac{1}{2}\dd{s}\le\frac{1}{2}.
		\end{align*}
		Thus differentiating under the integral sign is justified and we have
		\begin{equation}\label{truncated solution derivative}
			(\partial_jf_h^{a,b})(x,y)=-\int_{a}^{b}\frac{1}{2(1-s)}(\partial_jT_sh)(x,y)\dd{s},\quad(\partial_{jk}f_h^{a,b})(x,y)=-\int_{a}^{b}\frac{1}{2(1-s)}(\partial_{jk}T_sh)(x,y)\dd{s}.
		\end{equation}
		By taking a sequence $x_n\rightarrow x$ and dominated convergence, we see that $(\partial_jf_h^{a,b})(x,y)$ and $(\partial_{jk}f_h^{a,b})(x,y)$ are continuous in $x$. For $0< a<b<1$,  
		\begin{align}\label{truncated Stein equation}
			(\gen_{\Si,\la} f_h^{a,b})(x,y)=&-\int_{a}^{b}\frac{1}{2(1-s)}(\gen_{\Si,\la} T_sh)(x,y)\dd{s}=-\int_{a}^{b}\pdv{s}T_sh(x,y)\dd{s}  \notag \\
			=&	T_ah(x,y)- T_bh(x,y),
		\end{align}
		by Proposition \ref{T interpolation backward} and the fundamental theorem of calculus. Since the integrals in \eqref{truncated solution} and \eqref{truncated solution derivative} with $a=0$ and $b=1$ are absolutely convergent, letting $a\rightarrow0$ and $b\rightarrow1$ yields
		\begin{align*}
			f_h^{a,b}(x,y)\rightarrow f_h(x,y),\quad (\partial_jf_h^{a,b})(x,y)\rightarrow (\partial_jf_h)(x,y),\quad (\partial_{jk}f_h^{a,b})(x,y)\rightarrow (\partial_{jk}f_h)(x,y).
		\end{align*}
		Therefore, letting $a\rightarrow0$ and $b\rightarrow1$ in both sides of \eqref{truncated Stein equation} yields
		\begin{align*}
			(\gen_{\Si,\la} f_h)(x,y)=T_0h(x,y)- T_1h(x,y)=h(x,y)-\e[h(Z,N)].
		\end{align*}
	\end{proof}
	We can further show that, for $h\in\mH_3$ and $y\in\Z_+^r$, $f_h(\cdot,y)$ has continuous bounded first, second and third partial derivatives. We give the proof in Appendix \ref{sec: smoothness of the solution}.
	\begin{lem}\label{smoothness x}
		For fixed $h\in\mH_3$ and $y\in\Z_+^r$, $f_h(\cdot,y):\R^d\rightarrow\R$ is of class $ C^3$ with estimates
		\begin{align*}
			&|f_h(\cdot,y)|_{\mathrm{Lip}}\le1;\\
			&|f_h(\cdot,y)|_k\le\frac{1}{k},\quad k=2,3.
		\end{align*}
		If furthermore $h$ is such that $h(\cdot,y')$ is a constant function on $\R^d$ for any $y'\in\Z_+^r$, we have $|f_h(\cdot,y)|_{\mathrm{Lip}}=|f_h(\cdot,y)|_2=|f_h(\cdot,y)|_3=0$.
	\end{lem}
	The next lemma collects the smoothness estimates for the solution with respect to $y$:
		\begin{lem}\label{smoothness y}
		For each $h\in\mH_3$, $x\in\R^d$ and $y\in\Z_+^r$, we have
		\begin{align}
			&\qty|\De_{j}f_h(x,y)|\le\min\qty{1,\sqrt{\frac{2}{e\la_j}}}; \label{smoothness first order difference} \\
			&\qty|\De_{jj}f_h(x,y)|\le\min\qty{1,\frac{8}{3}\sqrt{\frac{1}{e\la_j}},\frac{2+2\log^+(e\la_j)}{e\la_j}}; \label{smoothness second order difference diagonal}  \\
			&\qty|\De_{jk}f_h(x,y)|\le\min\qty{1,\frac{4}{3}\sqrt{\frac{2}{e\la_j}},\frac{4}{3}\sqrt{\frac{2}{e\la_k}},\frac{1+\log^+(2e(\la_j\wedge\la_k))}{e(\la_j\wedge\la_k)}},\quad j\neq k; \label{smoothness second order difference off diagonal}  \\
			&|\De_k(\partial_jf_h)(x,y)|\le \min\qty{\frac{2}{3},\frac{\pi}{2\sqrt{2}}\sqrt{\frac{1}{e\la_k}}};  \label{smoothness first order difference of first derivative} \\
			&|\De_\ell(\partial_{jk}f_h)(x,y)|\le \min\qty{\frac{1}{2},\frac{4}{3\sqrt{2}}\sqrt{\frac{1}{e\la_\ell}}}.  \label{smoothness first order difference of second derivative}
		\end{align}
		If furthermore $h$ is such that $h(x',y')=0$ for $y'\neq0$, these bounds improve
		\footnote{That $\frac{1-e^{-x}}{x}<\frac{1}{\sqrt{x}}$ for $x>0$ may be checked as follows: For $x>0$, by the AM--GM inequality, $\sqrt{x}\le\frac{1+x}{2}<\frac{e^x}{2}$ and $(\sqrt{x}+e^{-x}-1)'=(2\sqrt{x})^{-1}-e^{-x}>0$. Thus $x\mapsto\sqrt{x}+e^{-x}-1$ is strictly increasing on $x\ge0$ and we have $\sqrt{x}+e^{-x}-1>0$ for $x>0$. Thus, since $\frac{1}{2}<\sqrt{\frac{2}{e}}$, \eqref{smoothness first order difference improved}, \eqref{smoothness second order difference diagonal improved} and \eqref{smoothness second order difference off diagonal improved} are indeed better than \eqref{smoothness first order difference}, \eqref{smoothness second order difference diagonal} and \eqref{smoothness second order difference off diagonal}, respectively.
			
			That \eqref{smoothness first order difference of first derivative improved} is better than \eqref{smoothness first order difference of first derivative} can be seen as follows: Since $\frac{\pi}{2\sqrt{2}}\sqrt{\frac{1}{ex}}>\frac{1}{3}$ for $0<x<3$, it suffices to show that $\frac{1-e^{-x}}{2x}<\frac{\pi}{2\sqrt{2}}\sqrt{\frac{1}{ex}}$ only for $x\ge3$. Then it is straightforward to check that $x\mapsto\frac{1-e^{-x}}{\sqrt{x}}$ is strictly decreasing on $x\ge3$, implying that $\frac{1-e^{-x}}{2x}<\frac{1}{2\sqrt{3}\sqrt{x}}<\frac{\pi}{2\sqrt{2}}\sqrt{\frac{1}{ex}}$ on $x\ge3$. That \eqref{smoothness first order difference of second derivative improved} is better than \eqref{smoothness first order difference of second derivative} can be seen similarly.} as
		\begin{align}
			&\qty|\De_jf_h(x,y)|\le\min\qty{\frac{1}{2},\frac{1-e^{-\la_j}}{2\la_j}}; \label{smoothness first order difference improved} \\
			&\qty|\De_{jj}f_h(x,y)|\le\min\qty{\frac{1}{4},\frac{1-e^{-\la_j}}{2\la_j}};  \label{smoothness second order difference diagonal improved} \\
			&\qty|\De_{jk}f_h(x,y)|\le\min\qty{\frac{1}{4},\frac{1-e^{-(\la_j+\la_k)}}{2(\la_j+\la_k)}},\quad j\neq k; \label{smoothness second order difference off diagonal improved} \\
			&\qty|\De_k(\partial_jf_h)(x,y)|\le\min\qty{\frac{1}{3},\frac{1-e^{-\la_k}}{2\la_k}};  \label{smoothness first order difference of first derivative improved} \\
			&\qty|\De_\ell(\partial_{jk}f_h)(x,y)|\le\min\qty{\frac{1}{4},\frac{1-e^{-\la_\ell}}{2\la_\ell}}.  \label{smoothness first order difference of second derivative improved}
		\end{align}
		Alternatively, if $h$ is such that $h(x',\cdot)$ is a constant function on $\Z_+^r$ for any $x'\in\R^d$, then all the bounds can be taken to be $0$.
	\end{lem}
	\begin{rem}\label{smoothness y uniform}
		Since $x\mapsto(1+\log^+x)/x$, $x>0$ is strictly decreasing,\footnote{$x\mapsto(1+\log x)/x$, $x\ge1$ is strictly decreasing.} we obtain
		\begin{align*}
			&\qty|\De_{jj}f_h(x,y)|\le\min\qty{1,\frac{8}{3}\sqrt{\frac{1}{e\min_j\la_j}},\frac{2+2\log^+(e\min_j\la_j)}{e\min_j\la_j}}, \\
			&\qty|\De_{jk}f_h(x,y)|\le\min\qty{1,\frac{4}{3}\sqrt{\frac{2}{e\min_j\la_j}},\frac{1+\log^+(2e\min_j\la_j)}{e\min_j\la_j}},\quad j\neq k
		\end{align*}
		from \eqref{smoothness second order difference diagonal} and \eqref{smoothness second order difference off diagonal}. Using $\log^+(2x)\le\log2+\log^+x$ for $x>0$, we obtain a uniform bound
		\begin{align}
			\qty|\De_{jk}f_h(x,y)|\le&\min\qty{1,\frac{8}{3}\sqrt{\frac{1}{e\min_j\la_j}},\frac{2+2\log^+(e\min_j\la_j)}{e\min_j\la_j}}\notag\\
			=&\min\qty{1,\frac{2+2\log^+(e\min_j\la_j)}{e\min_j\la_j}}\quad\text{for any pair $\{j,k\}$}.\label{smoothness second order difference uniform}
		\end{align}
		The last equality can be seen as follows: It suffices to show that $1+\log x<\frac{4}{3}\sqrt{x}$ for $x\ge 64/9$. We can check that $x\mapsto\frac{4}{3}\sqrt{x}-\log x-1$ is strictly increasing on $x\ge64/9$ and $\frac{23}{9}>\log\qty(\frac{64}{9})$.  
	\end{rem}
	\begin{rem}
		The improvement of \eqref{smoothness first order difference} to \eqref{smoothness first order difference improved} for $h$ such that $h(x',y')=0$ for $y'\neq0$ is analogous to the last part of \cite[Lemma 1]{AGG89}, which is about the one-dimensional case.
	\end{rem}
Since the `cross' smoothness estimates \eqref{smoothness first order difference of first derivative} and \eqref{smoothness first order difference of first derivative improved} are new and will be used later, we only prove them here. We give the proofs for the remaining estimates in Appendix \ref{sec: smoothness of the solution}.
	\begin{proof}[Proof of \eqref{smoothness first order difference of first derivative} and \eqref{smoothness first order difference of first derivative improved}]
				The proof is done by a suitable adaptation of the proof of \cite[Lemmas 10.2.9]{BHJ92}. Note that
		\begin{align*}
			(\partial_j f_h)(x,y)=-\int_{0}^{1}\frac{1}{2\sqrt{1-s}}\e\qty[(\partial_jh)\qty(\sqrt{1-s}x+\sqrt{s}Z,D_y(s)+Z_0(s))]\dd{s}
		\end{align*}
		(cf.\ \eqref{solution kth derivative} in Appendix \ref{sec: smoothness of the solution}) and $\|\partial_jh\|\le 1$ holds.  For $s\in[0,1)$, since
		\begin{align*}
			D_{y+e_k}(s)=D_y(s)+e_k\indi_{\{\zeta_{k,y_k+1}>-\log(1-s) \}},
		\end{align*}
		we have
		\begin{align*}
			&\left.(\partial_jh)\qty(\sqrt{1-s}x+\sqrt{s}Z,D_{y+e_k}(s)+Z_0(s))-(\partial_jh)\qty(\sqrt{1-s}x+\sqrt{s}Z,D_{y}(s)+Z_0(s))   \right.\\
			=&\indi_{\{\zeta_{k,y_k+1}>-\log(1-s)\}}\De_{k}(\partial_jh)\qty(\sqrt{1-s}x+\sqrt{s}Z,D_y(s)+Z_0(s))
		\end{align*}
		and 
		\begin{align*}
			&\e\qty[(\partial_jh)\qty(\sqrt{1-s}x+\sqrt{s}Z,D_{y+e_k}(s)+Z_0(s))-(\partial_jh)\qty(\sqrt{1-s}x+\sqrt{s}Z,D_{y}(s)+Z_0(s)) ]\\
			=&(1-s)\e\qty[\De_{k}(\partial_jh)\qty(\sqrt{1-s}x+\sqrt{s}Z,D_y(s)+Z_0(s))].
		\end{align*}
		This identity extends to hold for $s=1$. It thus follows that
		\begin{align*}
			\De_{k}(\partial_j f_h)(x,y)
			=-\int_{0}^{1}\frac{\sqrt{1-s}}{2}\e\qty[\De_{k}(\partial_jh)\qty(\sqrt{1-s}x+\sqrt{s}Z,D_y(s)+Z_0(s))]\dd{s}.
		\end{align*}
		Let $Z_{0,-k}(s)\coloneqq Z_0(s)-Z_{0,k}(s)e_k$. Defining $\widetilde{\partial_jh}:\R^d\times\Z^r_+\times\Z_+\rightarrow[-1,1]$ by
		\begin{align*}
			\widetilde{\partial_jh}(x,y,n)=\partial_jh(x,y+ne_k),
		\end{align*}
		we have 
		\begin{align*}
			&\De_{k}(\partial_j f_h)(x,y)\\
			=&-\int_{0}^{1}\frac{\sqrt{1-s}}{2}\e\bigg[\widetilde{\partial_jh}\qty(\sqrt{1-s}x+\sqrt{s}Z,D_y(s)+Z_{0,-k}(s),Z_{0,k}(s)+1) \bigg.  \\
			&\quad\quad\quad\quad\quad\quad\quad\quad\quad\quad	\bigg.	-\widetilde{\partial_jh}\qty(\sqrt{1-s}x+\sqrt{s}Z,D_y(s)+Z_{0,-k}(s),Z_{0,k}(s)) \bigg]\dd{s}.
		\end{align*}
		For $s\in(0,1]$, given that $(Z,D_y(s)+Z_{0,-k}(s))=(z,y')$, the expectation is just
		\begin{align}\label{first order difference of first derivative conditional expectation}
			-\widetilde{\partial_jh}(\sqrt{1-s}x+\sqrt{s}z,y',0)\Po\qty(s\la_k)\{0\}	-\sum_{n\ge1}		\widetilde{\partial_jh}(\sqrt{1-s}x+\sqrt{s}z,y',n)\Big[\Po\qty(s\la_k)\{n\}-\Po\qty(s\la_k)\{n-1\}\Big],
		\end{align}
		since $Z_{0,k}(s)\sim\Po(s\la_k)$ for $s\in(0,1]$ and $Z_{0,k}(s)$ is independent of $(Z,D_y(s)+Z_{0,-k}(s))$. It is hence bounded in magnitude by
		\begin{align*}
			2\max_{n\ge0}\Po\qty(s\la_k)\{n\}\le2\min\{1,(2es\la_k)^{-1/2}\},
		\end{align*}
	where the first bound comes from the unimodality of a Poisson distribution (see Lemma \ref{Poisson smoothness} in Appendix \ref{sec: smoothness of the solution} for a precise statement) and the last inequality is \cite[Proposition A.2.7]{BHJ92}. Thus 
		\begin{align*}
			\qty|\De_k(\partial_jf_h)(x,y)|&\le \int_{0}^{1}\sqrt{1-s}\min\{ 1, (2es\la_k)^{-1/2} \}\dd{s}\le\min\qty{\frac{2}{3},\frac{\pi}{2\sqrt{2}}\sqrt{\frac{1}{e\la_k}}}.
		\end{align*}
		If furthermore $h$ is such that $h(x',y')=0$ for $y'\neq0$, $\partial_jh(x',y')=0$ for $y'\neq0$ and $\widetilde{\partial_jh}(x',y',n)=0$ unless $y'=n=0$. \eqref{first order difference of first derivative conditional expectation} is then bounded in magnitude by $\Po\qty(s\la_k)\{0\}=e^{-\la_ks}$. This gives the improved bound in \eqref{smoothness first order difference of first derivative improved}:
		\begin{align*}
			\qty|\De_k(\partial_jf_h)(x,y)|\le \int_{0}^{1}\frac{\sqrt{1-s}}{2}e^{-\la_ks}\dd{s}\le\min\qty{\frac{1}{3},\frac{1-e^{-\la_k}}{2\la_k}}.
		\end{align*}
	\end{proof}
		As a preliminary result, we derive a comparison bound for two joint normal-Poisson random vectors. Suppose furthermore that the following independent random vectors are defined on the probability space $(\Omega,\mF,\p)$:
		
				{\setlength{\leftmargini}{29.5pt}  	
			\begin{itemize}
				\setlength{\labelsep}{7pt}     
				\setlength{\itemsep}{3pt}      
				
			\item $Z'\sim\mN_d(0,\Si')$, a $d$-dimensional normal random vector with a possibly singular covariance matrix $\Si'=\mqty(\si'_{jk})$.
				
				\item $N'=(N'_1,\dots,N'_r)\sim\Po(\la')=\prod_{j=1}^r\Po(\la'_j)$, a random vector in $Z_+^r$ whose coordinates are independent Poisson random variables $N'_j\sim\Po(\la'_j)$ with an $r$-tuple $\la'=(\la'_1,\dots,\la'_r)$ of positive means.
				
		\end{itemize} }

	\begin{thm}\label{comparison bound joint normal-Poisson}
	For $Z,N,Z',N'$ defined above, we have
	\begin{align*}
		\sup_{h\in\mH_3}\Big|\e\qty[h(Z,N)]-\e\qty[h(Z',N')]\Big|
		\le\frac{1}{2}\sum_{j,k=1}^d\qty|\si_{jk}-\si'_{jk}|+2\sum_{j=1}^{r}\qty|\la_j-\la'_j|.
	\end{align*}
	\end{thm}
	\begin{proof}
Fix an arbitrary $h\in\mH_3$, and let $f=f_h$ be given by \eqref{eq new solution} with this $h$. By Proposition \ref{Stein equation}, it solves the Stein equation \eqref{eq Stein equation} and satisfies
\begin{equation}\label{comparison first expression}
	\begin{split}
	&\e[h(Z',N')]-\e\qty[h(Z,N)]	=\e[\gen_{\Si,\la} f(Z',N')]\\
	=&\e\qty[\sum_{j,k=1}^d\si_{jk}(\partial_{jk}f )(Z',N')-\sum_{j=1}^dZ'_{j} (\partial_jf)(Z',N')]\\
	&\quad\quad+2\sum_{j=1}^{r}\e\Big[\la_j\De_jf(Z',N')-N'_{j}\De_jf(Z',N'-e_j) \Big].
	\end{split}
\end{equation}
The integrability of each term is ensured by Lemma \ref{smoothness x} and Lemma \ref{smoothness y}. By the independence of $Z'$ and $N'$, we have
\begin{align}
	\sum_{j=1}^d\e\qty[\sum_{k=1}^d\si'_{jk}(\partial_{jk}f )(Z',N')-Z'_{j} (\partial_jf)(Z',N')]&=0,\label{OU part zero}\\
	2\sum_{j=1}^{r}\e\Big[\la'_j\De_jf(Z',N')-N'_{j}\De_jf(Z',N'-e_j) \Big]&=0.\notag
\end{align}
(As for \eqref{OU part zero}, see the multivariate Stein identity in the form of Lemma \ref{multivariate Stein identity}.) By subtracting these two terms from \eqref{comparison first expression}, we have
\begin{align*}
\e[h(Z',N')]-\e\qty[h(Z,N)]	=&\e\qty[\sum_{j,k=1}^d\qty(\si_{jk}-\si'_{jk})(\partial_{jk}f )(Z',N')]+2\sum_{j=1}^{r}\e\Big[\qty(\la_j-\la'_j)\De_jf(Z',N') \Big].
\end{align*}
 Taking absolute values, we obtain the upper bound
 \begin{align*}
	\Big|\e\qty[h(Z,N)]-\e\qty[h(Z',N')]\Big|
\le\frac{1}{2}\sum_{j,k=1}^d\qty|\si_{jk}-\si'_{jk}|+2\sum_{j=1}^{r}\qty|\la_j-\la'_j|
 \end{align*}
 by Lemma \ref{smoothness x} and Lemma \ref{smoothness y}, \eqref{smoothness first order difference}. Taking the supremum over $h\in\mH_3$ completes the proof.
	\end{proof}
	
	\section{A new coupling approach to sums of indicators}
	\label{indicator framework}
	
	Let $(I_\al,\al\in\Ga)$ be a finite collection of 0--1 valued random variables. Suppose that for each $\al$ there exist 0--1 valued random variables $(J_{\beta\al},\be\in\Ga)$ defined on the same probability space as $(I_\beta,\beta\in\Ga)$ with\footnote{We understand $\mL(J_{\be\al},\be\in\Ga|I_\al=1)$ as a regular conditional distribution $\mu(s,B)$ of $(J_{\be\al},\be\in\Ga)$ given $I_\al$ with $s=1$. This definition of conditional distribution is valid even when $\p(I_\al=1)=p_\al=0$.}
	\begin{equation}\label{Stein coupling conditional dist}
		\mL(J_{\be\al},\be\in\Ga|I_\al=1)=\mL(I_\be,\be\in\Ga).
	\end{equation}
	 Now, assume that there is a partition $\Ga=\Ga_1\uplus\Ga_2$, $\Ga_{1}=\biguplus_{j=1}^d\Ga_{1j}$ and $\Ga_2=\biguplus_{j=1}^r\Ga_{2j}$. In our application in the configuration model, $d=2$ and $\Ga_{11}$ is to be taken as the set of possible isolated edges and $\Ga_{12}$ as the set of possible isolated 2-stars, whereas $r=2$ and $\Ga_{21}$ is to be taken as the set of possible self-loops and $\Ga_{22}$ as the set of possible double edges (we will make this more precise in Section \ref{configuration model}). Set $p_\al=\e I_\al$,
	\begin{gather*}
		W_{1j}=\frac{1}{s_{j}}\sum_{\al\in\Ga_{1j}}(I_\al-p_\al),\quad j=1,\dots,d,\quad\quad W_1=(W_{11},\dots,W_{1d})^\top,\\
		W_{2j}=\sum_{\al\in\Ga_{2j}}I_\al,\quad j=1,\dots,r,\quad\quad W_2=(W_{21},\dots,W_{2r})^\top,
	\end{gather*}
	where $s_j$ is a positive scaling constant. If each $\e I_\al$ and the cardinality $|\Ga|$ depend on a diverging positive integer $n$, we usually take $s_j=s_{n,j}$ such that $\si^2_j=\si^2_{n,j}=\Var(W_{1j})$ converges to a finite constant as $n\rightarrow\infty$, and if $\Var(\sum_{\al\in\Ga_{1j}}I_\al)$ and its limit are all positive we may want to take $s_{j}$ to be $\sqrt{\Var(\sum_{\al\in\Ga_{1j}}I_\al)}$. As in \eqref{noramalized sums edge twostar}, we take $s_{j}=\sqrt{n}$ for both $\Ga_{11}$ and $\Ga_{12}$ in our application in the configuration model, where $n$ denotes the number of vertices. Set
	\begin{align*}
		\Si=\cov(W_1),\quad\quad	\la_j = \e W_{2j},\quad j=1,\dots,r.
	\end{align*}
Assume that $\la_j>0$ for all $1\le j\le r$.	We regard $(W_1,W_2)$ as a random element in $\R^d\times\Z_+^r$. Let $Z$ be a $\mN_d(0,\Si)$-vector and let $N$ be a $\prod_{j=1}^{r}\Po(\la_j)$-vector. Suppose that $Z$ and $N$ are independent.
	\begin{thm}\label{Stein coupling theorem}
	Let $\mH_3$ be as in \eqref{H3} and let $f_h$ be given by \eqref{eq new solution} with $h\in\mH_3$ (where $T_sh$ is defined through \eqref{T independence} and \eqref{transition probability}). Under the above circumstances, for each $h\in\mH_3$, we have
		\begin{equation}\label{Stein coupling bound}
			\begin{split}
				\Big|\e\qty[h(W_1,W_2)]-\e\qty[h(Z,N)]\Big|
				\le&\sum_{j,k=1}^d\frac{\|\partial_{jk}f_h\|}{s_{j}s_{k}}\sqrt{\Var\qty(\e\qty[\sum_{\al\in\Ga_{1j}}I_\al\sum_{\be\in\Ga_{1k}}(I_{\be}-J_{\be\al})\Big|W_1,W_2])}\\
				&+\frac{1}{2}\sum_{j,k,\ell=1}^d\frac{\|\partial_{jk\ell}f_h\|}{s_{j}s_{k}s_{\ell}}\sum_{\al\in\Ga_{1j}}\e\qty|I_\al\sum_{\be\in\Ga_{1k}}\qty(I_{\be}-J_{\be\al})\sum_{\ga\in\Ga_{1\ell}}\qty(I_{\ga}-J_{\ga\al})|\\
				&+\sum_{j=1}^{d}\sum_{k=1}^r\frac{\|\De_k\qty(\partial_jf_h)\|}{s_{j}}\sum_{\al\in\Ga_{1j}}\e\qty|I_\al\sum_{\be\in\Ga_{2k}}\qty(I_\be-J_{\be\al})|\\
				&+2\sum_{k=1}^r\sum_{j=1}^d\frac{\|\De_k\qty(\partial_jf_h)\|}{s_{j}}\sum_{\al\in\Ga_{2k}}\e\qty|I_\al\sum_{\be\in\Ga_{1j}}\qty(I_\be-J_{\be\al})|\\
				&+2\max_{1\le j,k\le r}\|\De_{jk}f_h\|\qty(\sum_{\al\in\Ga_2} p_{\al}^2+\sum_{\al\in\Ga_2}\sum_{\be\in\Ga_2\setminus\{\al\}}\e\qty[I_\al|I_\beta-J_{\beta\al}|] ).
			\end{split}
		\end{equation}
	\end{thm}
	\begin{proof}
Fix an arbitrary $h\in\mH_3$, and let $f=f_h$ be given by \eqref{eq new solution} with this $h$. Then it solves the Stein equation \eqref{eq Stein equation} and satisfies
		\begin{align}
			&\e[h(W_1,W_2)]-\e\qty[h(Z,N)]	=\e[\gen_{\Si,\la} f(W_1,W_2)]\notag\\
			=&\e\qty[\sum_{j,k=1}^d\si_{jk}(\partial_{jk}f )(W_1,W_2)-\sum_{j=1}^dW_{1j} (\partial_jf)(W_1,W_2)]\label{OU part}\\
			&\quad\quad+2\sum_{j=1}^{r}\e\Big[\la_j\De_jf(W_1,W_2)-W_{2j}\De_jf(W_1,W_2-e_j) \Big].\label{ID part}
		\end{align}
		The integrability of each term is ensured by Lemma \ref{smoothness x} and Lemma \ref{smoothness y}. The proof below is a suitable adaptation of the proofs in the size-bias coupling approach (\cite[Theorem 1.2 and Lemma 2.1]{GR96} on multivariate normal approximation and \cite[Theorem 10.I]{BHJ92} on multivariate Poisson approximation) for our coupling \eqref{Stein coupling conditional dist}. Let us introduce a ``mark'' $Y_\be$ with each $\be\in\Ga$ such that
		\begin{align*}
			Y_\be=\sum_{j=1}^dj\indi_{\be\in\Ga_{1j} }+\sum_{j=1}^rj\indi_{\be\in\Ga_{2j} }.
		\end{align*}
		Using these marks, $W_1=(W_{11},\dots,W_{1d})^\top$ and $W_2=(W_{21},\dots,W_{2r})^\top$ can be written as
		\begin{align*}
			W_{1}=\sum_{\be\in\Ga_1}\frac{1}{s_{Y_\be}}(I_{\be}-p_\be)e_{Y_\be,d}	,\quad\quad W_2=\sum_{\be\in\Ga_2}I_\be e_{Y_\be},
		\end{align*}
		where $\{e_{j,d}\}_{1\le j\le d}$ denotes the standard basis for $\R^d$. 	For each $\al\in\Ga$, let
		\begin{align*}
			W_{1\al}=\sum_{\be\in\Ga_1}\frac{1}{s_{Y_\be}}(J_{\be\al}-p_\be)e_{Y_\be,d},\quad\quad W_{2\al}=\sum_{\be\in\Ga_2}J_{\be\al}e_{Y_\be}.
		\end{align*}
		Then $\mL(W_{1\al},W_{2\al}|I_\al=1)=\mL(W_1,W_2)$ by \eqref{Stein coupling conditional dist}.
		
		\medskip
		\underline{Bounding \eqref{OU part}}.  Note that, for $\al\in\Ga_{1j}$,
		\begin{align*}
			&\e\qty[I_\al(\partial_jf)(W_{1\al},W_{2\al}) ]\\
			=&p_\al\e\qty[I_\al(\partial_jf)(W_{1\al},W_{2\al})|I_\al=1]+(1-p_\al)\e\qty[I_\al(\partial_jf)(W_{1\al},W_{2\al})|I_\al=0 ]\\
			=&p_\al\e\qty[(\partial_jf)(W_{1\al},W_{2\al})|I_\al=1]\\
			=&p_\al\e\qty[(\partial_jf)(W_{1},W_{2})],
		\end{align*}
		by a standard disintegration argument using a regular conditional distribution. Thus,
		\begin{align*}
			&\e\qty[W_{1j}(\partial_jf)(W_1,W_2) ]\\
			=&\frac{1}{s_{j}}\sum_{\al\in\Ga_{1j}}\e\qty[(I_\al-p_\al)(\partial_jf)(W_1,W_2)]  \\
			=&\frac{1}{s_{j}}\sum_{\al\in\Ga_{1j}}\e\qty[I_\al\Big((\partial_jf)(W_{1},W_{2})-(\partial_jf)(W_{1\al},W_{2\al})\Big)]\\
			=&\frac{1}{s_{j}}\sum_{\al\in\Ga_{1j}}\e\qty[I_\al\Big((\partial_jf)(W_{1},W_{2})-(\partial_jf)(W_{1\al},W_{2})\Big)]\\
			&+\frac{1}{s_{j}}\sum_{\al\in\Ga_{1j}}\e\qty[I_\al\Big((\partial_jf)(W_{1\al},W_{2})-(\partial_jf)(W_{1\al},W_{2\al})\Big)].
		\end{align*}
		On the other hand, a similar argument shows that
		\begin{align*}
			\e\qty[I_{\al}W_{1\al k}]=p_\al\e\qty[W_{1 k}]
		\end{align*}
		for $\al\in\Ga_{1j}$ and $1\le k\le d$, where $W_{1\al k}$ is the $k$-th coordinate of $W_{1\al}$, that is,
		\begin{gather*}
			W_{1\al k}=\frac{1}{s_{k}}\sum_{\be\in\Ga_{1k}}(J_{\be\al}-p_\be),
		\end{gather*}
		and thus we have
		\begin{align*}
			\si_{jk}=\e\qty[W_{1j}W_{1k} ]=&\frac{1}{s_{j}}\sum_{\al\in\Ga_{1j}}\e\qty[\qty(I_\al-p_\al)W_{1 k}]\\
			=&\frac{1}{s_{j}}\sum_{\al\in\Ga_{1j}}\e\qty[I_\al\qty(W_{1 k}-W_{1\al k})].
		\end{align*}
		Thus,
		\begin{align}
			&\e\qty[\sum_{j,k=1}^d\si_{jk}(\partial_{jk}f )(W_1,W_2)-\sum_{j=1}^dW_{1j} (\partial_jf)(W_1,W_2)]\notag\\
			=&\e\qty[\sum_{j,k=1}^{d}\frac{1}{s_{j}}\sum_{\al\in\Ga_{1j}}\Big(\e\qty[I_\al\qty(W_{1k}-W_{1\al k})]-I_\al(W_{1k}-W_{1\al k})\Big)(\partial_{jk}f)(W_1,W_2)  ]\notag\\
			&-\e\qty[\sum_{j=1}^d\frac{1}{s_{j}}\sum_{\al\in\Ga_{1j}}I_\al\qty((\partial_jf)(W_{1},W_{2})-(\partial_jf)(W_{1\al},W_{2})-\sum_{k=1}^d(\partial_{jk}f)(W_1,W_2)(W_{1k}-W_{1\al k}))]\notag\\
			&-\sum_{j=1}^d\frac{1}{s_{j}}\sum_{\al\in\Ga_{1j}}\e\qty[I_\al\Big((\partial_jf)(W_{1\al},W_{2})-(\partial_jf)(W_{1\al},W_{2\al})\Big)]\notag\\
			\eqqcolon&R_1+R_2+R_3.\label{OU part decomposition}
		\end{align}
		First,
		\begin{align*}
			R_1=\sum_{j,k=1}^d\frac{1}{s_{j}}\e\qty[\qty(\e\qty[\sum_{\al\in\Ga_{1j}}I_\al\qty(W_{1\al k}-W_{1k})]-\e\qty[\sum_{\al\in\Ga_{1j}}I_\al(W_{1\al k}-W_{1k})\Big|W_1,W_2])(\partial_{jk}f)(W_1,W_2)  ].
		\end{align*}
		By the Cauchy--Schwarz inequality,
		\begin{align*}
			|R_1|\le&\sum_{j,k=1}^d\frac{\|\partial_{jk}f\|}{s_{j}}\e\left|\e\qty[\sum_{\al\in\Ga_{1j}}I_\al\qty(W_{1\al k}-W_{1k})]-\e\qty[\sum_{\al\in\Ga_{1j}}I_\al(W_{1\al k}-W_{1k})\Big|W_1,W_2]\right|\\
			\le&\sum_{j,k=1}^d\frac{\|\partial_{jk}f\|}{s_{j}}\sqrt{\Var\qty(\e\qty[\sum_{\al\in\Ga_{1j}}I_\al(W_{1\al k}-W_{1k})\Big|W_1,W_2])}\\
			=&\sum_{j,k=1}^d\frac{\|\partial_{jk}f\|}{s_{j}s_{k}}\sqrt{\Var\qty(\e\qty[\sum_{\al\in\Ga_{1j}}I_\al\sum_{\be\in\Ga_{1k}}(I_{\be}-J_{\be\al})\Big|W_1,W_2])}.
		\end{align*}
		Next, by the multivariate version of Taylor's theorem,
		\begin{align*}
			&\e\qty[I_\al\qty((\partial_jf)(W_{1\al},W_{2})-(\partial_jf)(W_{1},W_{2})-\sum_{k=1}^d(\partial_{jk}f)(W_1,W_2)(W_{1\al k}-W_{1k}))]\\
			=&\e\qty[I_\al(1-U_1)\sum_{k,\ell=1}^d\qty(\partial_{jk\ell}f)\qty(W_{1}+U_1(W_{1\al}-W_1),W_2)\qty(W_{1\al k}-W_{1k})\qty(W_{1\al \ell}-W_{1\ell})],
		\end{align*}
		where $U_1$ is a $U(0,1)$ random variable independent of everything else. Thus we have
		\begin{align*}
			|R_2|\le&\sum_{j=1}^d\frac{1}{s_{j}}\sum_{\al\in\Ga_{1j}}\sum_{k,\ell=1}^d\|\partial_{jk\ell}f\|\e\qty[ I_\al(1-U_1)\Big|\qty(W_{1k}-W_{1\al k})\qty(W_{1\ell}-W_{1\al \ell})\Big|]\\
			=&\frac{1}{2}\sum_{j,k,\ell=1}^d\frac{\|\partial_{jk\ell}f\|}{s_{j}s_{k}s_{\ell}}\sum_{\al\in\Ga_{1j}}\e\qty|I_\al\sum_{\be\in\Ga_{1k}}\qty(I_{\be}-J_{\be\al})\sum_{\ga\in\Ga_{1\ell}}\qty(I_{\ga}-J_{\ga\al})|.
		\end{align*}
		Finally, writing $W_{2\al}=(W_{2\al1},\dots,W_{2\al r})^\top$, where $W_{2\al k}=\sum_{\be\in\Ga_{2k}}J_{\be\al}$, we have
		\begin{align*}
			&(\partial_jf)(W_{1\al},W_{2})-(\partial_jf)(W_{1\al},W_{2\al})\\
			=&\sum_{k=1}^r\Big[(\partial_jf)(W_{1\al},W_{21},\dots,W_{2(k-1)},W_{2 k},W_{2\al(k+1)},\dots,W_{2\al r})\Big.\\
			&\Big.\quad\quad\quad-(\partial_jf)(W_{1\al},W_{21},\dots,W_{2 (k-1)},W_{2\al k},W_{2\al(k+1)},\dots,W_{2\al r})\Big].
		\end{align*}
		Further writing 
		\begin{align*}
			&(\partial_jf)(W_{1\al},W_{21},\dots,W_{2(k-1)},W_{2 k},W_{2\al(k+1)},\dots,W_{2\al r})\\
			&\quad-(\partial_jf)(W_{1\al},W_{21},\dots,W_{2 (k-1)},W_{2\al k},W_{2\al(k+1)},\dots,W_{2\al r})
		\end{align*}
		as a sum of $\qty|W_{2k}(\omega)-W_{2\al k}(\omega)|$ telescoping terms, we have
		\begin{align*}
			&\Big|(\partial_jf)(W_{1\al},W_{21},\dots,W_{2(k-1)},W_{2 k},W_{2\al(k+1)},\dots,W_{2\al r})\Big.\\
			&\Big.\quad-(\partial_jf)(W_{1\al},W_{21},\dots,W_{2 (k-1)},W_{2\al k},W_{2\al(k+1)},\dots,W_{2\al r})\Big|\\
			\le&\|\De_k\qty(\partial_jf)\|\qty|W_{2k}-W_{2\al k}|
		\end{align*}
		and
		\begin{align*}
			\Big|	(\partial_jf)(W_{1\al},W_{2})-(\partial_jf)(W_{1\al},W_{2\al})\Big|	\le&\sum_{k=1}^{r}\|\De_k\qty(\partial_jf)\|\qty|W_{2k}-W_{2\al k}|\\
			=&\sum_{k=1}^{r}\|\De_k\qty(\partial_jf)\|\qty|\sum_{\be\in\Ga_{2k}}\qty(I_{\be}-J_{\be\al})|.
		\end{align*}
		Thus
		\begin{align*}
			|R_3|&\le\sum_{j=1}^{d}\frac{1}{s_{j}}\sum_{\al\in\Ga_{1j}}\e\qty[I_\al\sum_{k=1}^r\|\De_k\qty(\partial_jf)\|\qty|\sum_{\be\in\Ga_{2k}}\qty(I_\be-J_{\be\al})|]\\
			&=\sum_{j=1}^{d}\sum_{k=1}^r\frac{\|\De_k\qty(\partial_jf)\|}{s_{j}}\sum_{\al\in\Ga_{1j}}\e\qty|I_\al\sum_{\be\in\Ga_{2k}}\qty(I_\be-J_{\be\al})|.
		\end{align*}
		
		\medskip
		\underline{Bounding \eqref{ID part}}.	For each $\al\in\Ga_2$, define
		\begin{align*}
			V_{2\al}=\sum_{\substack{\be\in\Ga_2\\\be\neq\al}}I_{\be} e_{Y_\be}=W_2-I_\al e_{Y_\al}.  
		\end{align*}
		Then, if $\al\in\Ga_{2j}$, since $e_{Y_\al}=e_j$, 
		\begin{align*}
			I_{\al}\De_j f(W_1,W_2-e_j)=I_\al\De_j f(W_1,W_2-I_\al e_j)=I_{\al}\De_j f(W_1,V_{2\al}).
		\end{align*}
		Also it follows that, if $\al\in\Ga_{2j}$,
		\begin{align*}
			&\e\qty[I_\al\De_j f(W_{1\al},W_{2\al})]\\
			=&p_\al\e\qty[I_\al\De_j f(W_{1\al},W_{2\al})|I_\al=1]+(1-p_\al)\e[I_\al\De_j f(W_{1\al},W_{2\al})|I_\al=0]\\
			=&p_\al\e\qty[\De_j f(W_{1\al},W_{2\al}) |I_\al=1]\\
			=&p_\al\e\qty[\De_j f(W_{1},W_{2}) ].
		\end{align*}
		Thus
		\begin{align*}
			&\sum_{j=1}^{r}\e\Big[\la_j\De_j f(W_1,W_2)-W_{2j}\De_j f(W_1,W_2-e_j) \Big]\\
			=&\sum_{j=1}^r\sum_{\al\in\Ga_{2j}}\Big[p_\al\e\qty[\De_j f(W_1,W_2)]-\e\qty[I_\al \De_j f(W_1,W_2-e_j)]\Big]\\
			=&\sum_{j=1}^r\sum_{\al\in\Ga_{2j}}\e\qty[I_\al\Big(\De_j f(W_{1\al},W_{2\al})-\De_j f(W_{1},V_{2\al}) \Big)]\\
			=&\sum_{j=1}^r\sum_{\al\in\Ga_{2j}}\e\qty[I_\al\Big(\De_j f(W_{1\al},W_{2\al})-\De_j f(W_{1},W_{2\al}) \Big)]\\
			&\quad\quad\quad+\sum_{j=1}^r\sum_{\al\in\Ga_{2j}}\e\qty[I_\al\Big(\De_j f(W_{1},W_{2\al})-\De_j f(W_{1},V_{2\al}) \Big)]\\
			\eqqcolon &R_4+R_5.
		\end{align*}
		By the fundamental theorem of calculus,
		\begin{align*}
			R_4=&-\sum_{j=1}^r\sum_{\al\in\Ga_{2j}}\e\qty[I_\al\De_j \Big(f(W_{1},W_{2\al})-f(W_{1\al},W_{2\al})\Big) ]\\
			=&-\sum_{j=1}^r\sum_{\al\in\Ga_{2j}}\e\qty[I_\al\De_j\qty(\sum_{k=1}^d(\partial_kf)\qty(W_{1\al}+U_2(W_1-W_{1\al}),W_{2\al})(W_{1k}-W_{1\al k})) ]\\
			=&-\sum_{j=1}^r\sum_{\al\in\Ga_{2j}}\sum_{k=1}^d\e\qty[I_\al(W_{1k}-W_{1\al k})(\De_j \partial_kf)\qty(W_{1\al}+U_2(W_1-W_{1\al}),W_{2\al}) ]\\
			=&-\sum_{j=1}^r\sum_{\al\in\Ga_{2j}}\sum_{k=1}^d\frac{1}{s_{k}}\e\qty[I_\al\qty(\sum_{\be\in\Ga_{1k}}\qty(I_{\be}-J_{\be\al }))(\De_j \partial_kf)\qty(W_{1\al}+U_2(W_1-W_{1\al}),W_{2\al}) ],
		\end{align*}
		where $U_2$ is a $U(0,1)$ random variable independent of everything else. Thus we have
		\begin{align*}
			|R_4|\le\sum_{j=1}^r\sum_{k=1}^d\frac{\|\De_j\qty(\partial_kf)\|}{s_{k}}\sum_{\al\in\Ga_{2j}}\e\qty|I_\al\sum_{\be\in\Ga_{1k}}\qty(I_\be-J_{\be\al})|.
		\end{align*}
		Next, note that
		\begin{align*}
			R_5=&\sum_{j=1}^r\sum_{\al\in\Ga_{2j}}\e\qty[I_\al\Big(\De_j f(W_{1},W_{2\al})-\De_j f(W_{1},V_{2\al}) \Big)]\\
			=&\sum_{\al\in\Ga_2}\sum_{j=1}^r\indi_{\al\in\Ga_{2j}}\e\qty[I_\al\Big(\De_j f(W_{1},W_{2\al})-\De_j f(W_{1},V_{2\al})\Big)  ].
		\end{align*}
		Enumerate the elements of $\Ga_2$ as $\be(1),\dots,\be(|\Ga_2|)$. Let, for $\al\in\Ga_2$ and $0\le i\le|\Ga_2|$, 
		\begin{align*}
			U^i_\al\coloneqq\begin{cases}
				V_{2\al}+\sum_{j=1}^{i}\{J_{\be(j)\al}-I_{\be(j)} \}e_{Y_{\be(j)}}  &  \text{if $\al\notin\{ \be(1),\dots,\be(i)\}$} \\
				V_{2\al}+I_{\al}e_{Y_\al}+\sum_{j=1}^{i}\{J_{\be(j)\al}-I_{\be(j)} \}e_{Y_{\be(j)}}   &  \text{if $\al\in\{ \be(1),\dots,\be(i)\}$}
			\end{cases},
		\end{align*}
		so that $U^0_\al=V_{2\al}$ and $U^{|\Ga_2|}_\al=W_{2\al}$.	Then
		\begin{align*}
			&\sum_{j=1}^r\indi_{\al\in\Ga_{2j}}\e\qty[I_\al\Big(\De_j f(W_{1},W_{2\al})-\De_j f(W_{1},V_{2\al})\Big) ]\\
			=&\sum_{i=1}^{|\Ga_2|}\e\sum_{j=1}^r\indi_{\al\in\Ga_{2j}}\qty[I_\al\Big(\De_j f\qty(W_{1},U^i_\al)-\De_j f\qty(W_{1},U^{i-1}_\al)\Big) ].
		\end{align*}
		Note that
		\begin{align*}
			U^i_\al=\begin{cases}
				U^{i-1}_\al+J_{\al\al} e_{Y_\al}&  \text{if $\al=\be(i)$}   \\
				U^{i-1}_\al+(J_{\be(i)\al}-I_{\be(i)})e_{Y_{\be(i)}} &\text{if $\al\neq\be(i)$}
			\end{cases}.
		\end{align*}
		So if $\al=\be(i)$,
		\begin{align*}
			&\De_j f\qty(W_{1},U^i_\al)-\De_j f\qty(W_{1},U^{i-1}_\al)  \\
			=&\sum_{k=1}^{r}\indi_{\al\in\Ga_{2k}}\qty[\De_j f\qty(W_{1},U^{i-1}_\al+J_{\al\al} e_k)-\De_j f\qty(W_{1},U^{i-1}_\al)  ]\\
			=&\sum_{k=1}^{r}\indi_{\al\in\Ga_{2k}}J_{\al\al}\qty[\De_j f\qty(W_{1},U^{i-1}_\al+e_k)-\De_j f\qty(W_{1},U^{i-1}_\al)   ]\\
			=&\sum_{k=1}^{r}\indi_{\al\in\Ga_{2k}}J_{\al\al}\De_{kj} f\qty(W_{1},U^{i-1}_\al)  ,
		\end{align*}
		and if $\al\neq\be(i)$,
		\begin{align*}
			&\De_j f\qty(W_{1},U^i_\al)-\De_j f\qty(W_{1},U^{i-1}_\al)  \\
			=&\sum_{k=1}^{r}\indi_{\be(i)\in\Ga_{2k}}\bigg[\De_j f\qty(W_{1},U^{i-1}_\al+(J_{\be(i)\al}-I_{\be(i)}) e_k)-\De_j f\qty(W_{1},U^{i-1}_\al)  \bigg] \\
			=&\sum_{k=1}^{r}\indi_{\be(i)\in\Ga_{2k}}|I_{\beta(i)}-J_{\beta(i)\al}|\\
			&\quad\quad\times\bigg\{ \indi_{\{I_{\be(i)}-J_{\be(i)\al}=-1 \}} \qty[\De_j f\qty(W_{1},U^{i-1}_\al+e_k)-\De_j f\qty(W_{1},U^{i-1}_\al)    ]    \bigg.\\
			&\quad\quad\quad\quad\quad\bigg. +\indi_{\{I_{\be(i)}-J_{\be(i)\al}=1 \}} \qty[\De_j f\qty(W_{1},U^{i-1}_\al-e_k)-\De_j f\qty(W_{1},U^{i-1}_\al)  ]    \bigg\}\\
			=&\sum_{k=1}^{r}\indi_{\be(i)\in\Ga_{2k}}|I_{\beta(i)}-J_{\beta(i)\al}|\\
			&\times\bigg\{ \indi_{\{I_{\be(i)}-J_{\be(i)\al}=-1 \}} \De_{kj} f\qty(W_{1},U^{i-1}_\al)      -\indi_{\{I_{\be(i)}-J_{\be(i)\al}=1 \}} \De_{kj} f\qty(W_{1},U^{i-1}_\al-e_k) \bigg\}.
		\end{align*}
		Thus, writing $\al\in\Ga_2$ as $\al=\be(i_0)$,
		\begin{align*}
			R_5=&\sum_{\al\in\Ga_2}\sum_{i=1}^{|\Ga_2|}\e\sum_{j=1}^r\indi_{\al\in\Ga_{2j}}\qty[I_\al\Big(\De_j f\qty(W_{1},U^i_\al)-\De_j f\qty(W_{1},U^{i-1}_\al)\Big) ]	\\
			=&\sum_{\al\in\Ga_2}\e\Biggl\{ I_\al J_{\al\al}\sum_{j=1}^{r}\indi_{\al\in\Ga_{2j}}\De_{jj}f\qty(W_{1},U^{i_0-1}_\al) \Biggr\}\\
			&+\sum_{\al\in\Ga_2}\sum_{\substack{1\le i\le|\Ga_2|\\\be(i)\neq\al}}\e\Biggl\{I_\al|I_{\beta(i)}-J_{\beta(i)\al}|\sum_{j=1}^{r}\indi_{\al\in\Ga_{2j}}\sum_{k=1}^{r}\indi_{\be(i)\in\Ga_{2k}}\Biggr.\\
			&\quad\quad\times\Bigl\{ \indi_{\{I_{\be(i)}-J_{\be(i)\al}=-1 \}} \De_{kj} f\qty(W_{1},U^{i-1}_\al)   -\indi_{\{I_{\be(i)}-J_{\be(i)\al}=1 \}} \De_{kj} f\qty(W_{1},U^{i-1}_\al-e_k)\Bigr\}\Biggr\}.
		\end{align*}
		Therefore, we obtain
		\begin{align*}
			|R_5|&\le \sum_{\al\in\Ga_2}\e\qty[I_\al J_{\al\al}]\sum_{j=1}^{r}\indi_{\al\in\Ga_{2j}}\|\Delta_{jj}f\|+\sum_{\al\in\Ga_2}\sum_{\substack{\be\in\Ga_2\\\be\neq\al}}\e\qty[I_\al|I_{\beta}-J_{\beta\al}|]\sum_{j=1}^{r}\indi_{\al\in\Ga_{2j}}\sum_{k=1}^{r}\indi_{\be\in\Ga_{2k}}\|\Delta_{kj}f\|\\
			&=\sum_{j=1}^{r}\|\Delta_{jj}f\|\sum_{\al\in\Ga_{2j}}\e\qty[I_\al J_{\al\al}]+\sum_{j=1}^{r}\sum_{k=1}^{r}\|\Delta_{kj}f\|\sum_{\al\in\Ga_{2j}}\sum_{\substack{\be\in\Ga_{2k}\\\be\neq\al}}\e\qty[I_\al|I_\beta-J_{\beta\al}|]\\
			&	\le\max_{1\le j,k\le r}\|\Delta_{jk}f\|\qty(\sum_{\al\in\Ga_2}\e\qty[I_\al J_{\al\al}]+\sum_{\al\in\Ga_2}\sum_{\be\in\Ga_2\setminus\{\al\}}\e\qty[I_\al|I_\beta-J_{\beta\al}|] ).
		\end{align*}
		Finally, the condition \eqref{Stein coupling conditional dist} gives $\e[I_\al J_{\al\al}]=p_\al\e[J_{\al\al}|I_\al=1]=p_\al^2$.
	\end{proof}
	
	\medskip
Taking the supremum over $h\in\mH_3$, we can derive a bound for the error in joint normal-Poisson approximation to the distribution of $(W_1,W_2)$ from Theorem \ref{Stein coupling theorem}. This, together with the smoothness estimates about the Stein solution $f_h$ which are uniform in $h$ (Lemmas \ref{smoothness x} and \ref{smoothness y}), tells us that if we can construct the coupled indicators $(J_{\be\al},\be\in\Ga)$ so that the moment terms about $(I_\be)$ and $(J_{\be\al})$ appearing in the bound \eqref{Stein coupling bound} are small, then the distributional approximation works. In order to gain some intuition as to why this is the case, we informally interpret the condition \eqref{Stein coupling conditional dist} on our coupling and the bound \eqref{Stein coupling bound} in terms of dependence among the indicators $I_\al$, $\al\in\Ga$: The condition \eqref{Stein coupling conditional dist} implies that for each $\al\in\Ga$,
\begin{equation}\label{coupling condition independence interpretation}
	\p\Big((J_{\be\al},\be\neq\al)\in B,\ I_\al=1\Big)=\p\Big((I_{\be},\be\neq\al)\in B\Big)\p\qty(I_\al=1).
\end{equation}
If the moment terms in the bound \eqref{Stein coupling bound} are small, $(J_{\be\al})$ are close to $(I_\be)$ in some sense, and interpreting $J_{\be\al}\approx I_\be$ in \eqref{coupling condition independence interpretation} tells us that $I_\be$, $\be\neq\al$ are approximately independent of $I_\al$. This is so for each $\al\in\Ga$, thus the dependence among $I_\al$, $\al\in\Ga$ is weak, and in such a case normal or Poisson approximation to the sum of indicators typically works well.
	\begin{exm}[Independent indicators]
		If $I_\al$, $\al\in\Ga$ are independent, then we may take $J_{\be\al}\equiv I_\be$ for $\be\neq\al$ and $J_{\al\al}=I'_{\al}$, an independent copy of $I_\al$ independent of everything else. For the first term in \eqref{Stein coupling bound}, we use 
		\begin{align*}
		\Var\qty(\e\qty[\sum_{\al\in\Ga_{1j}}I_\al(I_{\al}-I'_{\al})\Big|W_1,W_2])\le&\Var\qty(\e\qty[\sum_{\al\in\Ga_{1j}}I_\al(I_{\al}-I'_{\al})\Big|(I_\al,\al\in\Ga)])\\
		=&\Var\qty(\sum_{\al\in\Ga_{1j}}I_\al(1-p_\al)).
		\end{align*}
	A straightforward computation shows that, for each $h\in\mH_3$, we have
		\begin{equation*}
			\begin{split}
				\Big|\e\qty[h(W_1,W_2)]-\e\qty[h(Z,N)]\Big|
				\le&\sum_{j=1}^d\frac{\|\partial_{jj}f_h\|}{{s_{j}}^2}\sqrt{\sum_{\al\in\Ga_{1j}}p_\al(1-p_\al)^3}+\frac{1}{2}\sum_{j=1}^d\frac{\|\partial_{jjj}f_h\|}{{s_{j}}^3}\sum_{\al\in\Ga_{1j}}p_\al(1-p_\al)\\
				&+2\max_{1\le j,k\le r}\|\De_{jk}f_h\|\sum_{\al\in\Ga_2} p_{\al}^2.
			\end{split}
		\end{equation*}
		The first two terms correspond to the normal approximation to $W_1$ and the last term corresponds to the Poisson approximation to $W_2$, each of which being of the correct order of magnitude.
	\end{exm}
	\begin{rem}[Relation to Stein coupling]
		The argument appearing in bounding \eqref{OU part} can be generalized as follows: Let $\mathfrak{a}$ be a random index uniformly distributed on $\Ga_1$ and independent of all else. Let
		\begin{align*}
G_\al=-\frac{|\Ga_1|}{s_{Y_\al}}I_\al e_{Y_\al,d}\quad(\al\in\Ga_1),\quad G=G_{\mathfrak{a}},\quad \qty(W_1',W_2')=\qty(W_{1\mathfrak{a}},W_{2\mathfrak{a}}).
		\end{align*}
		Then
		\begin{align*}
			\e\qty[G^\top\qty(F(W_1',W_2')-F(W_1,W_2))]=\e\qty[W_1^\top F(W_1,W_2)]
		\end{align*}
		for all $F:\R^d\times\Z_+^r\rightarrow\R^d$ for which the expectations are finite. In particular, the triplet $\qty(W_1,W_1',G)$ is a $d$-dimensional Stein coupling in the sense of \cite[Definition 2.1]{FR15}. This is analogous to \cite[Lemma 3.2]{BR19}, which only considers one-dimensional normal approximation.
	\end{rem}
	\begin{rem}[Comparison with size-bias coupling]\label{size bias remark}
		The coupling \eqref{Stein coupling conditional dist} and the bound \eqref{Stein coupling bound} are similar to, but different from those of the \emph{size-bias coupling} approach (cf.\ \cite{BHJ92, GR96, J94, Ross11}). Indeed, suppose alternatively that for each $\al$, the 0--1 valued random variables $(K_{\beta\al},\be\in\Ga)$ are defined on the same probability space as $(I_\beta,\beta\in\Ga)$ with
		\begin{equation}\label{size bias coupling conditional dist}
			\mL(K_{\be\al},\be\in\Ga)=\mL(I_\be,\be\in\Ga|I_\al=1).
		\end{equation}
		Then we can show that, for each $h\in\mH_3$,
		\begin{equation}\label{size bias bound}
			\begin{split}
				\Big|\e\qty[h(W_1,W_2)]-\e\qty[h(Z,N)]\Big|\le&\sum_{j,k=1}^d\frac{\|\partial_{jk}f_h\|}{s_{j}s_{k}}\sqrt{\Var\qty(\e\qty[\sum_{\al\in\Ga_{1j}}p_\al\sum_{\be\in\Ga_{1k}}(I_{\be}-K_{\be\al})\Big|W_1,W_2])}\\
				&+\frac{1}{2}\sum_{j,k,\ell=1}^d\frac{\|\partial_{jk\ell}f_h\|}{s_{j}s_{k}s_{\ell}}\sum_{\al\in\Ga_{1j}}p_\al\e\qty|\sum_{\be\in\Ga_{1k}}\qty(I_{\be}-K_{\be\al})\sum_{\ga\in\Ga_{1\ell}}\qty(I_{\ga}-K_{\ga\al})|\\
				&+\sum_{j=1}^{d}\sum_{k=1}^r\frac{\|\De_k\qty(\partial_jf_h)\|}{s_{j}}\sum_{\al\in\Ga_{1j}}p_\al\e\qty|\sum_{\be\in\Ga_{2k}}\qty(I_\be-K_{\be\al})|\\
				&+2\sum_{k=1}^r\sum_{j=1}^d\frac{\|\De_k\qty(\partial_jf_h)\|}{s_{j}}\sum_{\al\in\Ga_{2k}}p_\al\e\qty|\sum_{\be\in\Ga_{1j}}\qty(I_\be-K_{\be\al})|\\
				&+2\max_{1\le j,k\le r}\|\De_{jk}f_h\|\qty(\sum_{\al\in\Ga_2}p_\al^2+\sum_{\al\in\Ga_2}\sum_{\be\in\Ga_2\setminus\{\al\}}p_\al\e\qty[|I_\beta-K_{\beta\al}|] ).
			\end{split}
		\end{equation}	
		The proof of the bound \eqref{size bias bound} under the assumption \eqref{size bias coupling conditional dist} is parallel to that of \eqref{Stein coupling bound} under the assumption \eqref{Stein coupling conditional dist} and thus omitted (although the bound \eqref{size bias bound} in this joint normal-Poisson approximation fashion is also new even for the size-bias coupling \eqref{size bias coupling conditional dist}). 
	\end{rem}
	By considering a subset of $\mH_3$ and combining Theorem \ref{Stein coupling theorem} with the smoothness estimates about the solution (Lemmas \ref{smoothness x} and \ref{smoothness y}), we can also obtain a bound for the error in normal or Poisson approximation alone.
	\begin{exm}[One-dimensional Poisson approximation]\label{one dimensional Poisson approximation}
			Set $r=1$ and $\la=\e[W_2]$. Suppose that $\la>0$. Combined with the last part of Lemma \ref{smoothness x} and Lemma \ref{smoothness y}, \eqref{smoothness second order difference diagonal improved}, Theorem \ref{Stein coupling theorem} gives
		\begin{equation}\label{one dim Poisson approximation zero probability bound}
			\qty|\p\qty(W_2=0)-e^{-\la}|\le\min\qty{\frac{1}{2},\frac{1-e^{-\la}}{\la}}\qty(\sum_{\al\in\Ga_2} p_{\al}^2+\sum_{\al\in\Ga_2}\sum_{\be\in\Ga_2\setminus\{\al\}}\e\qty[I_\al|I_\beta-J_{\beta\al}|] )
		\end{equation}
		in the case of one-dimensional Poisson approximation. Having the upper bound $1/2$ in the constant (when we restrict the test function to be $\indi_{y=0}$) may look new to the literature. In Appendix \ref{sec: smoothness of the solution}, we prove the smoothness estimate \eqref{smoothness second order difference diagonal improved} for our solution as a by-product of a probabilistic argument. Actually, one can reproduce the estimate \eqref{one dim Poisson approximation zero probability bound} by exploiting the explicit formula for the solution to the Chen--Stein equation for one-dimensional Poisson approximation
		\begin{equation}\label{Chen Stein equation}
			\la g(j+1)-jg(j)=\indi_{j=0}-\Po(\la)\{0\},\quad j\ge0
		\end{equation}
		(cf.\ \cite[(1.10)]{BHJ92}). The solution to \eqref{Chen Stein equation} is given by
		\begin{align*}
			g_{\la,\{0\}}(0)=0,\quad g_{\la,\{0\}}(j+1)=\frac{j!}{\la^{j+1}}\sum_{k=j+1}^{\infty}\frac{\la^ke^{-\la}}{k!},\quad j\ge0
		\end{align*}
		(cf.\ \cite[(1.18)]{BHJ92}). Then it is easy to verify that
		\begin{align*}
			g_{\la,\{0\}}(j+2)-g_{\la,\{0\}}(j+1)=&-\sum_{\ell=0}^{\infty}\frac{\la^\ell e^{-\la}}{\ell!}\frac{(\ell+1)!}{(\ell+j+1)\cdots(j+1)}\frac{1}{\ell+j+2},\quad j\ge0,
		\end{align*}
		whence $\sup_{j\ge1}\qty|g_{\la,\{0\}}(j+1)-g_{\la,\{0\}}(j)|\le1/2$. Now deriving the estimate \eqref{one dim Poisson approximation zero probability bound} under the assumption \eqref{Stein coupling conditional dist} is a standard exercise of the Chen--Stein method (see also the corresponding part of the proof of Theorem \ref{Stein coupling theorem} for bounding \eqref{ID part}).
	\end{exm}

	\section{The configuration model}
	\label{configuration model}

	\textbf{Additional notation.} Let $n_k=n_k(\bm{d})\coloneqq\qty|\qty{i\in[n]:d_i=k}|$, the number of vertices of degree $k$ in $\mathrm{CM}_n(\bm{d})$. For an integer $m$ and a positive integer $r$, let $(m)_r\coloneqq m!/(m-r)!=m(m-1)\cdots(m-r+1)$ be the descending factorial. We adopt the convention that $(m)_r=0$ if $r>m$, in particular, whenever $m\le0$. Similarly, $((m))_r\coloneqq m!!/(m-2r)!!=m(m-2)\cdots(m-2(r-1))$. We write $X=O(Y)$ to denote an inequality of the form $|X|\le CY$ for some constant $C$.

	\medskip
	 We rephrase our problem of the configuration model in terms of the notation introduced in Section \ref{indicator framework}. By making the substitution \eqref{configuration model definition} (the definition of $\mathrm{CM}_n(\bm{d})$), we can rewrite the four quantities of our interest in terms of the matching of half-edges:
		\begin{align*}
		(\text{\# of isolated edges})\:\: \overline{Z}_{\edge}=&\sum_{\substack{1\le i<j\le n\\ \mathrm{s.t.}\ d_i=d_j=1}}X_{ij}=\sum_{\substack{1\le i<j\le n\\ \mathrm{s.t.}\ d_i=d_j=1}}\sum_{s\in v_i,t\in v_j}I_{st},\\
		(\text{\# of isolated 2-stars})\:\: \overline{Z}_{\twostar}=&\sum_{\substack{1\le i<j\le n,\ 1\le k\le n\\ \mathrm{s.t.}\ d_i=d_j=1,\ d_k=2}}X_{ik}X_{jk}\\
		=&\sum_{\substack{1\le i<j\le n,\ 1\le k\le n\\ \mathrm{s.t.}\ d_i=d_j=1,\ d_k=2}}\sum_{\substack{s\in v_i,t\in v_j\\u<v\in v_k}}\qty(I_{su}I_{tv}+I_{sv}I_{tu})\\
		=&\sum_{\substack{1\le i<j\le n,\ 1\le k\le n\\ \mathrm{s.t.}\ d_i=d_j=1,\ d_k=2}}\sum_{\substack{s\in v_i,t\in v_j\\u\neq v\in v_k}}I_{su}I_{tv},\\
		(\text{\# of self-loops})\:\: S_n=&\sum_{1\le i\le n}X_{ii}=\sum_{1\le i\le n}\sum_{s<t\in v_i}I_{st},\\
		(\text{\# of double edges})\:\: M_n=&\sum_{1\le i<j\le n}\binom{X_{ij}}{2}\\
		=&\sum_{1\le i<j\le n}\sum_{\substack{s<t\in v_i\\ u<v\in v_j}}\qty(I_{su}I_{tv}+I_{sv}I_{tu})\\
		=&\sum_{1\le i<j\le n}\sum_{\substack{s<t\in v_i\\ u\neq v\in v_j}}I_{su}I_{tv}.
	\end{align*}
	Then we write these sums of indicators introducing index sets $\Ga_{11},\Ga_{12},\Ga_{21},\Ga_{22}$ as
	\begin{equation}\label{possible unlabeled copies indicator representation}
		\overline{Z}_{\edge}=\sum_{\al\in\Ga_{11}}I_\al,\quad\overline{Z}_{\twostar}=\sum_{\al\in\Ga_{12}}I_\al,\quad S_n=\sum_{\al\in\Ga_{21}}I_\al,\quad M_n=\sum_{\al\in\Ga_{22}}I_\al,
	\end{equation}
	with the following one-to-one correspondences (and with $I_\al=I_{su}I_{tv}$ for $\al$ in $\Ga_{12}$ or $\Ga_{22}$):
	{\setlength{\leftmargini}{29.5pt}  	
		\begin{itemize}
			\setlength{\labelsep}{7pt}     
			\setlength{\itemsep}{3pt}     
			
			\item  $\al\in\Ga_{11}$ $\leftrightarrow$ $\{s,t\}$, $s\in v_i$, $t\in v_j$, $d_i=d_j=1$, $1\le i<j\le n$;
			
			\item  $\al\in\Ga_{12}$ $\leftrightarrow$ $\Big\{\{s,u\},\{t,v\}\Big\}$, $s\in v_i$, $t\in v_j$, $u\neq v\in v_k$, $d_i=d_j=1$, $d_k=2$, $1\le i<j\le n$, $1\le k\le n$;

			\item  $\al\in\Ga_{21}$ $\leftrightarrow$ $\{s,t\}$, $s<t\in v_i$, $1\le i\le n$;
			
			\item  $\al\in\Ga_{22}$ $\leftrightarrow$ $\Big\{\{s,u\},\{t,v\}\Big\}$, $s<t\in v_i$, $u\neq v\in v_j$, $1\le i<j\le n$.
			
	\end{itemize} }
	Namely, $\Ga_{11}$, $\Ga_{12}$, $\Ga_{21}$ and $\Ga_{22}$ correspond to the sets of `possible' isolated edges, isolated 2-stars, self-loops and double edges, respectively. We call a member $\al$ in $\Ga_{11}$, $\Ga_{12}$, $\Ga_{21}$ and $\Ga_{22}$ as a \emph{possible isolated edge}, a \emph{possible isolated 2-star}, a \emph{possible self-loop} and a \emph{possible double edge}, respectively.
			{\setlength{\leftmargini}{29.5pt}  	
		\begin{itemize}
			\setlength{\labelsep}{7pt}    
			\setlength{\itemsep}{3pt}      
			
			\item  $\al\in\Ga_{11}$ is specified by a pair of two half-edges incident to two distinct degree 1 vertices. Thus $	\qty|\Ga_{11}|=(n_1)_2/2$.
			
			\item  $\al\in\Ga_{12}$ is specified by two pairs $\{s,u\}$ and $\{t,v\}$, where $s$ and $t$ are incident to two distinct degree 1 vertices, whereas $u$ and $v$ are incident to a degree 2 vertex. Note that, given $s\in v_i$ and $t\in v_j$ with $d_i=d_j=1$, $i\neq j$ and $u<v\in v_k$ with $d_k=2$, we distinguish two possible 2-stars, the one formed by $\{s,u\}$ and $\{t,v\}$, and the one formed by $\{s,v\}$ and $\{t,u\}$. Thus $\qty|\Ga_{12}|=(n_1)_2n_2$. 
					{\setlength{\leftmargini}{29.5pt}  	
				\begin{itemize}
					\setlength{\labelsep}{7pt}     
					\setlength{\itemsep}{3pt}      
					
					\item 	For future reference we write a possible isolated 2-star as $(st)\cup (uv)$, where $st$ and $uv$ correspond to its edges and $t$ and $u$ are incident to its degree 2 vertex in this expression (this will be specified every time).

			\end{itemize} }

			\item  $\al\in\Ga_{21}$ is specified by a pair of two distinct half-edges incident to the same vertex. Thus $\qty|\Ga_{21}|=\sum_{i\in[n]}(d_i)_2/2$.
			
			\item  $\al\in\Ga_{22}$ is specified by two pairs $\{s,u\}$ and $\{t,v\}$, where $s$ and $t$ are incident to the same vertex and $u$ and $v$ are incident to another vertex. Note that, given $s<t\in v_i$ and $u<v\in v_j$ with $i\neq j$, we distinguish two possible double edges, the one formed by $\{s,u\}$ and $\{t,v\}$, and the one formed by $\{s,v\}$ and $\{t,u\}$. Thus $\qty|\Ga_{22}|=\sum_{1\le i<j\le n}(d_i)_2(d_j)_2/2$.
			
	\end{itemize} }

	\medskip
	Also we introduce the \emph{type} of a member of each of $\Ga_{11}$, $\Ga_{12}$, $\Ga_{21}$ and $\Ga_{22}$ as $\mathsf{edge}$, $\mathsf{2\mathchar`-star}$, $\mathsf{self\mathchar`-loop}$ and $\mathsf{double\ edge}$, respectively. We also say that the types of the sets $\Ga_{11}$, $\Ga_{12}$, $\Ga_{21}$ and $\Ga_{22}$ are $\mathsf{edge}$, $\mathsf{2\mathchar`-star}$, $\mathsf{self\mathchar`-loop}$ and $\mathsf{double\ edge}$, respectively. If $H=\mathsf{edge}$ (resp.\ $\mathsf{2\mathchar`-star}$, $\mathsf{self\mathchar`-loop}$, $\mathsf{double\ edge}$), then the number of vertices of $H$ is defined as $v(H)=2$ (resp.\ $3$, $1$, $2$) and the number of edges of $H$ is defined as $e(H)=1$ (resp.\ $2$, $1$, $2$).
	
	\medskip
	Finally, let $\Ga_1=\Ga_{11}\uplus\Ga_{12}$, $\Ga_2=\Ga_{21}\uplus\Ga_{22}$ and $\Ga=\Ga_{1}\uplus\Ga_{2}$. For $\al\in\Ga$ of type $H$, $I_\al$ in \eqref{possible unlabeled copies indicator representation} then denotes the random indicator that $\al$ is indeed realized in $\mathfrak{g}$, i.e., all the $e(H)$ pairs of half-edges which specify $\al$ are realized in $\mathfrak{g}$. (For example, if $\al$ is a possible isolated 2-star $(st)\cup (uv)$, then $I_\al=I_{st}I_{uv}$.) Since each configuration (a matching of all the half-edges in pairs) has equal probability $1/(N-1)!!$, while there are in total $(N-2e(H)-1)!!$ possible configurations which have the $e(H)$ pairs specifying $\al$, we have
	\begin{equation}\label{success probability}
		p_\al=\e I_\al=\frac{(N-2e(H)-1)!!}{(N-1)!!}=\frac{1}{((N-1))_{e(H)}}.
	\end{equation}
	
	\subsection{Coupling construction via switching}
	
Inspired by \cite[Section 3.4]{BR19}, we construct such a coupling \eqref{Stein coupling conditional dist} for the indicators $(I_\be,\be\in\Ga)$ in \eqref{possible unlabeled copies indicator representation} via the \emph{switching} procedure. We introduce some notation. Let $\mG$ denote the set of perfect matchings of the $N$ half-edges. An element $g\in\mG$ is a partition of the set of the $N$ half-edges into $N/2$ pairs, e.g., $g=\qty{\{1,2\},\dots,\{N-1,N\}}$. For two distinct half-edges $s$ and $t$, $st$ denotes a possible edge that can be formed by joining $s$ and $t$. (Thus, $st=ts$.) Here, we are abusing the term ``edge'' to include a possible self-loop, i.e., the case when $s$ and $t$ are incident to the same vertex. Also note that it is not necessarily isolated. It can be identified with a pair $\{s,t\}$, but for notational reasons it is still convenient to make a distinction between a possible edge $st$ and a mere set of the half-edges $\{s,t\}$.\footnote{For example, for two possible edges $st$ and $uv$, we write $(st)\cap (uv)\neq\varnothing$ in later sections if those two possible edges share a vertex. Since they are not necessarily isolated ones, it is possible that $\{s,t\}\cap\{u,v\}=\varnothing$ as subsets of half-edges even if they share a vertex.} We define the type of a possible edge $st$ as $\mathsf{edge}$ (resp.\ $\mathsf{self\mathchar`-loop}$) if $s$ and $t$ are incident to two distinct vertices (resp.\  the same vertex). For $g\in\mG$ and $\al \in\Ga$ or $\al$ being a possible edge, we write $g\supset \al$ to mean that $\al$ is actually formed in $g$. If we identify $\al$ of type $H$ with a partition of the set of its $2e(H)$ half-edges into $e(H)$ pairs, this notation becomes consistent (e.g., for $g\in\mG$ and distinct half-edges $s$ and $t$, $g\supset st$ is equivalent to $g\supset\qty{\{s,t\}}$). 
	
	\medskip
	Let $\al\in\Ga$ or be a possible edge and let $H$ denote its type. In our case for $H=\mathsf{edge}$, $\mathsf{2\mathchar`-star}$, $\mathsf{self\mathchar`-loop}$ or $\mathsf{double\ edge}$, the number of edges $e(H)$ is at most $2$, but we describe our construction of the coupling in a somewhat general fashion; see also Remark \ref{unlabeled copy generalization} below. Let $s_{\ell,1}s_{\ell,2}$, $1\le\ell\le e(H)$ denote the $e(H)$ edges of $\al$ arranged in ascending order such that $s_{1,1}\wedge s_{1,2}<\cdots<s_{e(H),1}\wedge s_{e(H),2}$.  Then for $g\in\mG$, $g\supset s_{m,1}s_{m,2}$, $\ell\le m\le e(H)$ means that the possible edges $s_{m,1}s_{m,2}$, $\ell\le m\le e(H)$ are actually formed in $g$. Set 
	\begin{align*}
		\mG_\al\coloneqq &\{g\in\mG:g\supset\al\},\\
		\mG_{\al,\ell}\coloneqq&\qty{g\in\mG: g\supset s_{m,1}s_{m,2},\ \ell+1\le m\le e(H)}=\bigcap_{m=\ell+1}^{e(H)}\mG_{s_{m,1}s_{m,2}}, \quad0\le\ell\le e(H)-1.
	\end{align*}
	Since $g\supset \al$ and $g\supset s_{m,1}s_{m,2}$, $1\le m\le e(H)$ are equivalent, $\mG_{\al,0}=\mG_{\al}$. Also set $\mG_{\al,e(H)}\coloneqq\mG$.  This defines increasing sets of configurations $\mG_\al=\mG_{\al,0}\subset\cdots\subset\mG_{\al,e(H)}=\mG$.
	
	\medskip
	Now let $\al\in\Ga$ with type $H$. We will construct from $\mathfrak{g}$ a new configuration (a random perfect matching of the $N$ half-edges) $\mathfrak{g}_\al$ having the distribution such that
	\begin{align}\label{Stein coupling configuration conditional dist}
		\mL(\mathfrak{g}_\al|I_\al=1)=\mL(\mathfrak{g}).
	\end{align}
	Then we define the indicators $(J_{\be\al},\be\in\Ga)$ coupled with $(I_\be,\be\in\Ga)$ in \eqref{possible unlabeled copies indicator representation} satisfying \eqref{Stein coupling conditional dist} in the obvious manner; we define $J_{\be\al}$ to be the indicator that $\be$ is realized in $\mathfrak{g}_\al$, i.e., $J_{\be\al}\coloneqq\indi_{\mathfrak{g}_\al\supset\be}$. If $I_\al=0$, i.e., if even one of the edges of $\al$ is not realized in $\mathfrak{g}$, we do nothing and set $\mathfrak{g}_\al=\mathfrak{g}$. If $I_\al=1$, i.e., if all the $e(H)$ edges of $\al$ are realized in $\mathfrak{g}$, we have $\mathfrak{g}\in\mG_\al$, and we resample a perfect matching of the $N$ half-edges using a map
	\begin{equation}\label{coupling construction f al}
		f_\al:\mG_\al\times\prod_{\ell=1}^{e(H)}[N-2(e(H)-\ell)-1]\rightarrow\mG
	\end{equation}
	to be defined below and uniform indices $B_{\al,1},\dots,B_{\al,e(H)}$ as $\mathfrak{g}_\al=f_\al(\mathfrak{g},B_{\al,1},\dots,B_{\al,e(H)})$. 
	
	\medskip
	For each $1\le \ell\le e(H)$, we define a map
	\begin{equation}\label{coupling construction f al ell}
		f_{\al,\ell}:\mG_{\al,\ell-1}\times[N-2(e(H)-\ell)-1]\rightarrow\mG_{\al,\ell},
	\end{equation}
	as follows. Let $g\in\mG_{\al,\ell-1}$ and $b\in[N-2(e(H)-\ell)-1]$. Among the $N-2(e(H)-\ell)-1$ half-edges other than $s_{\ell,1}\wedge s_{\ell,2}$ and $s_{m,1},s_{m,2}$, $\ell+1\le m\le e(H)$, namely,
	\begin{equation}\label{f al ell choice set}
		\{1,\dots,N \}\setminus\qty(\{s_{\ell,1}\wedge s_{\ell,2} \}\uplus\biguplus_{m=\ell+1}^{e(H)}\{s_{m,1},s_{m,2} \}),
	\end{equation}
	choose the $b$-th smallest half-edge $t_{\ell,1}$ and call $t_{\ell,1}$'s partner in $g$ as $t_{\ell,2}$. Since $s_{\ell,1}$ and $s_{\ell,2}$ are paired to each other in $g$, $t_{\ell,2}=s_{\ell,1}\wedge s_{\ell,2}$ if $t_{\ell,1}= s_{\ell,1}\vee s_{\ell,2}$, and $t_{\ell,2}\neq s_{\ell,1}\wedge s_{\ell,2}$ if $t_{\ell,1}\neq s_{\ell,1}\vee s_{\ell,2}$. Then we define
	\begin{align*}
		f_{\al,\ell}(g,b)\coloneqq\qty(g\setminus\qty{ \{s_{\ell,1},s_{\ell,2}\},\{t_{\ell,1},t_{\ell,2}\} })\uplus\qty{ \{s_{\ell,1}\wedge s_{\ell,2},t_{\ell,1} \},\{s_{\ell,1}\vee s_{\ell,2},t_{\ell,2}\} },
	\end{align*}
	and this is well-defined. Namely, if $t_{\ell,1}=s_{\ell,1}\vee s_{\ell,2}$,  we keep the edge $s_{\ell,1}s_{\ell,2}$ and do nothing: $f_{\al,\ell}(g,b)=g$. If $t_{\ell,1}\neq s_{\ell,1}\vee s_{\ell,2}$, then, in $g$, we break up the bond between $s_{\ell,1}$ and $s_{\ell,2}$ and the bond between $t_{\ell,1}$ and $t_{\ell,2}$, form new pairs $(s_{\ell,1}\wedge s_{\ell,2})t_{\ell,1}$ and $(s_{\ell,1}\vee s_{\ell2})t_{\ell2}$ and return the result. Since $g\in\mG_{\al,\ell-1}$ and $t_{\ell,1}$ is in (\ref{f al ell choice set}), $t_{\ell,2}\notin\biguplus_{m=\ell+1}^{e(H)}\{s_{m,1},s_{m,2} \}$ and $	f_{\al,\ell}(g,b)\in\mG_{\al,\ell}$.
	
	\medskip
	Starting from $g\in\mG_\al$ and using indices $(b_\ell)_{\ell=1}^{e(H)}\in\prod_{\ell=1}^{e(H)}[N-2(e(H)-\ell)-1]$, we recursively define a sequence $g_0,\dots,g_{e(H)}$ of fixed configurations (perfect matchings of the $N$ half-edges) by
	\begin{align*}
		&g_0\coloneqq g\in\mG_\al, \\
		&g_{\ell}=g_{\ell}(g,b_1,\cdots,b_{\ell})\coloneqq f_{\al,\ell}\qty(g_{\ell-1},b_{\ell})\in\mG_{\al,\ell},\quad\ell=1,\dots,e(H).
	\end{align*}
	Now $f_{\al}$ in \eqref{coupling construction f al} is defined as $f_{\al}\qty(g,b_{1},\dots,b_{e(H)})\coloneqq g_{e(H)}\in\mG$. Symbolically, $f_\al$ has the form
	\begin{align*}
		f_\al\qty(g,b_{1},\dots,b_{e(H)})=f_{\al,e(H)}\qty(f_{\al,e(H)-1}\qty(\cdots f_{\al,2}\qty(f_{\al,1}\qty(g,b_{1}),b_2)\cdots,b_{e(H)-1}),b_{e(H)}).
	\end{align*}
	
	\medskip
	Generate uniform indices $B_{\al,\ell}\sim\mathrm{Unif}[N-2(e(H)-\ell)-1]$, $1\le \ell\le e(H)$.  We assume that $\mathfrak{g},B_{\al,1},\dots,B_{\al,e(H)}$ are all defined on the same probability space and independent. Finally, $\mathfrak{g}_\al$ is defined as
	\begin{align}\label{coupling construction}
		\mathfrak{g}_\al\coloneqq\begin{cases}
			\mathfrak{g} & (I_\al=0) \\
			f_\al\qty(\mathfrak{g},B_{\al,1},\dots,B_{\al,e(H)}) & (I_\al=1)
		\end{cases},
	\end{align}
	which is a random perfect matching of the $N$ half-edges coupled with $\mathfrak{g}$ on the same probability space.
	\begin{lem}\label{f al ell f al inverse}\ 
		{\setlength{\leftmargini}{29.5pt}  	
			\begin{itemize}
				\setlength{\labelsep}{7pt}     
				\setlength{\itemsep}{3pt}      
				
				\item[\emph{a}.] For $1\le\ell\le e(H)$,	$f_{\al,\ell}:\mG_{\al,\ell-1}\times[N-2(e(H)-\ell)-1]\rightarrow\mG_{\al,\ell}$ in \eqref{coupling construction f al ell} has the inverse $f_{\al,\ell}^{-1}:\mG_{\al,\ell}\rightarrow\mG_{\al,\ell-1}\times[N-2(e(H)-\ell)-1]$. 
				
				\item[\emph{b}.] $f_{\al}:\mG_\al\times\prod_{\ell=1}^{e(H)}[N-2(e(H)-\ell)-1]\rightarrow\mG$ in \eqref{coupling construction f al} has the inverse $f_\al^{-1}:\mG\rightarrow\mG_\al\times\prod_{\ell=1}^{e(H)}[N-2(e(H)-\ell)-1]$.
				
		\end{itemize} }
	\end{lem}
	\begin{proof}
		(a) Define $h_{\al,\ell}:\mG_{\al,\ell}\rightarrow\mG_{\al,\ell-1}\times[N-2(e(H)-\ell)-1]$ as follows. Let $g'\in\mG_{\al,\ell}$ and let $t'_{\ell,1}$ and $t'_{\ell,2}$ be the partners of $s_{\ell,1}\wedge s_{\ell,2}$ and $s_{\ell,1}\vee s_{\ell,2}$ in $g'$, respectively. Since $g'\in\mG_{\al,\ell}$, $t'_{\ell,1},t'_{\ell,2}\notin\biguplus_{m=\ell+1}^{e(H)}\{s_{m,1},s_{m,2} \}$. Define the first coordinate of $h_{\al,\ell}(g')$, $\qty(h_{\al,\ell}(g'))_1$, by
		\begin{align*}
			\qty(h_{\al,\ell}(g'))_1\coloneqq
			\qty(g'\setminus\qty{\{ s_{\ell,1}\wedge s_{\ell,2},t'_{\ell,1}\},\{s_{\ell,1}\vee s_{\ell,2},t'_{\ell,2}\}})\uplus\qty{\{s_{\ell,1},s_{\ell,2}\},\{t'_{\ell,1},t'_{\ell,2} \} }\in\mG_{\al,\ell-1}.
		\end{align*}
		Define the second coordinate of $h_{\al,\ell}(g')$, $\qty(h_{\al,\ell}(g'))_2$, as the rank (in ascending order) of $t'_{\ell,1}$ among \eqref{f al ell choice set}. Then it is easy to check that $h_{\al,\ell}\qty(f_{\al,\ell}(g,b))=(g,b)$ for $(g,b)\in\mG_{\al,\ell-1}\times[N-2(e(H)-\ell)-1]$ and $f_{\al,\ell}\qty(h_{\al,\ell}(g'))=g'$ for $g'\in\mG_{\al,\ell}$. Thus $h_{\al,\ell}$ is the inverse of $f_{\al,\ell}$.
		
		\medskip
		(b) By (a), for each $1\le\ell\le e(H)$, $f_{\al,\ell}$ has the inverse $f_{\al,\ell}^{-1}$. Define $h_\al:\mG\rightarrow\mG_\al\times\prod_{\ell=1}^{e(H)}[N-2(e(H)-\ell)-1]$ as follows. Let $g'\in\mG$ be arbitrary. Recursively define sequences of fixed configurations $g'_{e(H)},\dots,g'_0$ and indices $b'_{e(H)},\dots,b'_{1}$ by
		\begin{align*}
			g'_{e(H)}&\coloneqq g'\in\mG_{e(H)},\\
			\qty(g'_{\ell-1},b'_\ell)&\coloneqq f_{\al,\ell}^{-1}\qty(g'_\ell)\in\mG_{\al,\ell-1}\times[N-2(e(H)-\ell)-1],\quad\ell=e(H),\dots,1.
		\end{align*}
		Then define $h_\al(g')\coloneqq\qty(g'_0,\qty(b'_\ell)_{\ell=1}^{e(H)})$. It is straightforward to check that $h_{\al}\qty(f_{\al}\qty(g,\qty(b_\ell)_{\ell=1}^{e(H)}))=\qty(g,\qty(b)_{\ell=1}^{e(H)})$ for $\qty(g,\qty(b)_{\ell=1}^{e(H)})\in\mG_{\al}\times\prod_{\ell=1}^{e(H)}[N-2(e(H)-\ell)-1]$ and $f_{\al}\qty(h_{\al}(g'))=g'$ for $g'\in\mG$. Therefore, $h_{\al}$ is the inverse of $f_{\al}$.
	\end{proof}
	\begin{prp}\label{BRI coupling lemma}
		The distribution of $\mathfrak{g}_\al$ defined in \eqref{coupling construction} satisfies \eqref{Stein coupling configuration conditional dist}.
	\end{prp}
	\begin{proof}
		Let $g'$ be an arbitrary (fixed) configuration (perfect matching of the $N$ half-edges). It suffices to show that
		\begin{align*}
			\p(\mathfrak{g}_\al=g',I_\al=1)=\p\qty(\mathfrak{g}=g')\cdot p_\al=\frac{1}{(N-1)!!}\cdot\frac{1}{((N-1))_{e(H)}}.
		\end{align*}
		By Lemma \ref{f al ell f al inverse}b, $f_{\al}$ has the inverse $f^{-1}_{\al}$.  Let $\qty(g,\qty(b_\ell)_{\ell=1}^{e(H)})$ denote $f^{-1}_{\al}(g')\in\mG_\al\times\prod_{\ell=1}^{e(H)}[N-2(e(H)-\ell)-1]$. Since $\mathfrak{g},B_{\al,1},\dots,B_{\al,e(H)}$ are independent, we have
		\begin{align*}
			\p\qty(\mathfrak{g}_\al=g',I_\al=1)&=\p\qty(\mathfrak{g}=g,B_{\al,1}=b_{1},\dots,B_{\al,e(H)}=b_{e(H)})\\
			&=\p\qty(\mathfrak{g}=g)\prod_{\ell=1}^{e(H)}\p\qty(B_{\al,\ell}=b_{\ell})\\
			&=\frac{1}{(N-1)!!}\prod_{\ell=1}^{e(H)}\frac{1}{N-2(e(H)-\ell)-1}\\
			&=\frac{1}{(N-1)!!}\cdot\frac{1}{((N-1))_{e(H)}}.
		\end{align*}
	\end{proof}
	We do this construction of $\mathfrak{g}_\al$ (and $(J_{\be\al},\be\in\Ga)$) for different $\al$'s and generate the random numbers $B_\al=\qty(B_{\al,\ell})_{\ell=1}^{e(H)}$ for each $\al\in\Ga$, so we have made the dependence of the random numbers on $\al$ explicit. We generate those random numbers across different $\al$'s in such a way that $B_\al,\al\in\Ga$ are independent.
	
	\begin{rem}
	The idea of switching edges to analyze the configuration model and the uniform simple graph is old and has been used in the literature. The closest to the procedure described here are those in \cite{J09, BR19, J20b}; however, \cite{J09, J20b} switch only the edges in self-loops and double edges and \cite{BR19} switches only the edges in a rooted component $\mT(v)$ (which is isolated). In contrast, we are handling both (i) possible isolated edges and isolated 2-stars and (ii) possible self-loops and double edges which are not necessarily isolated in a unified manner.
	\end{rem}
	\begin{rem}[Comparison with the rewiring coupling]
		\cite{AHH19} considers Poisson approximation to the numbers of possible self-loops and double edges $(S_n,M_n)$ relying on the size-bias coupling approach (cf.\ Remark \ref{size bias remark}). They construct a new set of indicator variables satisfying the condition \eqref{size bias coupling conditional dist} coupled with $(I_\be)$ in \eqref{possible unlabeled copies indicator representation} by adapting the \emph{rewiring} procedure, as follows. Let $\al$ be a possible self-loop or a possible double edge. If $I_\al=1$, i.e., $\al$ is realized in $\mathfrak{g}$, we do nothing. If $I_\al=0$, then we break open all the pairs in $\mathfrak{g}$ containing the half-edges of $\al$, leaving the other pairs in $\mathfrak{g}$ untouched. We then \emph{rewire to create} $\al$, i.e., we pair the half-edges of $\al$ that are now unmatched
		forcing $\al$ to be realized, and pair the additional unmatched half-edges among them, uniformly at random if there are four in the case of double edge.
		
		\medskip
		The rewiring procedure may be generalized to possible isolated edges and isolated 2-stars also. However, the rewiring procedure starts if even one of the edges of $\al$ is not realized in $\mathfrak{g}$, and all the relevant half-edges are rewired at a time. This can make the analysis on how possible $\be$ would be created, destroyed or unchanged by the procedure harder. In contrast, our switching procedure starts at the simple event that all the edges of $\al$ are realized in $\mathfrak{g}$ and proceeds along the edges of $\al$ one by one. This feature makes computation of probabilities of creating or destructing possible $\be$ and the strategy to evaluate the moment terms in the bound \eqref{Stein coupling bound} more transparent. This will be clearer in the analysis of the Poisson approximation to $(S_n,M_n)$ (Proposition \ref{fifth term in stein coupling bound} below and its proof in Appendix \ref{stein coupling bound fifth term}), compared to that of \cite[Section 3]{AHH19}.
	\end{rem}
	
	\begin{rem}\label{unlabeled copy generalization}
The above construction of $\mathfrak{g}_\al$ and $(J_{\be\al})$ may be generalized to possible copies of arbitrary multigraphs once we define such a possible \emph{unlabeled} copy of a given multigraph as a possible matching of half-edges in a similar fashion to \cite[Remark 8.1]{J20}. Let $H$ be a multigraph with the number of vertices $v(H)$ and the number of edges $e(H)$. (We want to count the number of possible copies of $H$ which are actually formed in a given perfect matching $g\in\mG$ of the $N=\sum_{i=1}^{n}d_i$ half-edges, which are labeled as $1,\dots, N$.\footnote{We caution that $g\in\mG$ is a ``configuration'', a perfect matching of the half-edges which are associated with the vertices $v_1,\dots,v_n$, rather than mere numbers ``$1,\dots,N$''.}) We introduce half-edges also in $H$, which are labeled as $h_1,\dots,h_{2e(H)}$, so that each of the pairings $\{h_1,h_2\},\{h_3,h_4\},\dots,\{h_{2e(H)-1},h_{2e(H)}\}$ corresponds to each of the $e(H)$ edges in $H$ without loss of generality. (Here, we are abusing the term ``edge'' to include a self-loop in $H$ similarly.) Note that the half-edges $h_1,\dots,h_{2e(H)}$ thus introduced are associated with the vertices $v^{H}_1,\dots,v^{H}_{v(H)}$ of $H$.  We define a possible \emph{labeled} copy of $H$ as an injective map $\varphi$ which maps vertices to vertices and half-edges to half-edges such that the relations between them are preserved. Then each of the pairs $\varphi(h_1)\varphi(h_2),\varphi(h_3)\varphi(h_4),\dots, \varphi(h_{2e(H)-1)}\varphi(h_{2e(H)})$ induces a possible edge.\footnote{We need $\varphi$ to map vertices to vertices also. In particular, the condition ``half-edges incident to the same vertex in $H$ are mapped to half-edges incident to some one vertex $v_i$'' is not sufficient. For example, with only this condition, $h_{s_1}$ incident to a vertex and $h_{s_2}$ incident to another vertex which form a non-self-loop edge in $H$ can be mapped to $t_1$ and $t_2$ that are incident to only one vertex $v_i$, forming a possible self-loop, which is unreasonable as a possible copy of $H$.} The set of possible labeled copies of $H$ is denoted $\Phi(H)$. We then identify two such injections $\varphi$ and $\varphi'$ in $\Phi(H)$ if they are the same as a possible matching of half-edges in the range. Precisely, for possible labeled copies $\varphi$ and $\varphi'$ of $H$, we say that $\varphi$ and $\varphi'$ are equivalent, writing $\varphi\equiv\varphi'$, if and only if the ranges of $\varphi$ and $\varphi'$ are the same and they induce the same partition
		\begin{align*}
			&\qty{\{\varphi(h_1),\varphi(h_2)\},\{\varphi(h_3),\varphi(h_4)\},\dots, \{\varphi(h_{2e(H)-1}),\varphi(h_{2e(H)})\}  }\\
			=&\qty{\{\varphi'(h_1),\varphi'(h_2)\},\{\varphi'(h_3),\varphi'(h_4)\},\dots, \{\varphi'(h_{2e(H)-1}),\varphi'(h_{2e(H)})\}  }
		\end{align*}
		of the range for the half-edges. Then $\Phi(H)$ splits into equivalence classes $\{[\varphi]:\varphi\in\Phi(H) \}$, and we define a possible \emph{unlabeled} copy of $H$ to be such an equivalence class $[\varphi]$. One can think of each possible unlabeled copy of $H$ (an equivalence class $[\varphi]$) as a perfect matching of $2e(H)$ half-edges out of the half-edges $1,\dots, N$ which is ``isomorphic'' to $H$ as a multigraph.
		
		\medskip
		 Then one could generalize the notions introduced above to any possible unlabeled copy $\al$ of $H$ using a representative $\varphi$ of $\al$, i.e., $\al=[\varphi]$. For example, one could write $g\supset\al$ to mean that all the possible edges $\varphi(h_1)\varphi(h_2),\dots, \varphi(h_{2e(H)-1)}\varphi(h_{2e(H)})$ are actually formed in $g$, and one could define the indicator $I_\al$ that $\al$ is realized in $\mathfrak{g}$ as $I_\al\coloneqq I_{\varphi(h_1)\varphi(h_2)}\cdots I_{\varphi(h_{2e(H)-1)}\varphi(h_{2e(H)})}=\indi_{\mathfrak{g}\supset\al}$. The formula \eqref{success probability} for $p_\al=\e I_\al$ would continue to hold. 
	\end{rem}

	\subsection{Computing the bound}
	\label{subsec computing the bound}
	
	We compute and estimate each term in the bound \eqref{Stein coupling bound} for the sums of indicators in \eqref{possible unlabeled copies indicator representation}. We take $d=2$, $\Ga_{11}$ as the set of possible isolated edges, and $\Ga_{12}$ as the set of possible isolated 2-stars.  We take the normalizing constants as $s_1=s_2=\sqrt{n}$ (recall the definition \eqref{noramalized sums edge twostar} of $(W_\edge,W_\twostar)$). Let $n_1$ and $n_2$ be the numbers of vertices of degree 1 and degree 2, respectively. We take $r=2$, $\Ga_{21}$ as the set of possible self-loops and $\Ga_{22}$ as the set of possible double edges when considering the joint distribution $(W_1,W_2)=(W_\edge,W_\twostar,S_n,M_n)$, whereas $r=1$ and $\Ga_2$ as the set of possible self-loops and possible double edges when considering the conditional distribution $(W_\edge,W_\twostar)|S_n+M_n=0$. However, even in the latter case, for notational reasons it is still convenient to split $\Ga_2$ into two parts $\Ga_{21}$ and $\Ga_{22}$, corresponding to the set of possible self-loops and the set of possible double edges, respectively.

	\subsubsection{Covariance matrix $\Si_n$} We first compute the covariance matrix $\Si_n=\mqty(\si_{n,jk})=\cov(W_1)$, where $W_1=(W_\edge,W_\twostar)^\top$. For $\al,\be\in\Ga_1$, we say that $\al$ \textbf{intersects} $\be$ (or vice versa) and write $\al\cap \be\neq\varnothing$ if $\al$ and $\be$ share a vertex. We say that $\al$ and $\be$ are \textbf{disjoint} and write $\al\cap\be=\varnothing$ if they share no vertex. By $\e\qty[I_\al J_{\be\al}]=p_\al p_\be$, Lemmas \ref{intersecting different isolated trees} and \ref{disjoint multigraphs non creation} in Appendix \ref{stein coupling bound first term}, for $\al,\be\in\Ga_1$ we have
	\begin{align*}
		\cov\qty(I_\al,I_\be)=\e\qty[I_\al\qty(I_\be-J_{\be\al})]=\begin{cases}
			p_\al(1-p_\al)	&(\al=\be)\\
			-p_\al p_\be	&(\al\neq\be,\al\cap\be\neq\varnothing)\\
			\e\qty[I_\al I_\be\qty(1-J_{\be\al})]	&(\al\cap\be=\varnothing)
		\end{cases}.
	\end{align*}
	Furthermore, for $\al\in\Ga_{1j}$ and $\be\in\Ga_{1k}$ such that $\al\cap\be=\varnothing$, by \eqref{prob al be exist be destroyed} in Appendix \ref{subsec second variance},
	\begin{align*}
		\p\qty(I_\al=I_\be=1,J_{\be\al}=0)=\frac{1}{((N-1))_{e(H_j)+e(H_k)}}\qty(1-\prod_{\ell=1}^{e(H_j)}\qty(1-\frac{2e(H_k)}{N-2(e(H_j)-\ell)-1})),
	\end{align*}
where $H_j$ and $H_k$ denote the types of $\Ga_{1j}$ and $\Ga_{1k}$, respectively. (Note that $N\ge2\qty(e(H_j)+e(H_k))$ as long as a pair $\{\al,\be\}$ such that $\al\in\Ga_{1j}$, $\be\in\Ga_{1k}$ and $\al\cap\be=\varnothing$ exists.)	Thus,\footnote{To see that $\si_{n,jk}=\si_{n,kj}$, note that $\prod_{\ell=1}^{j}\qty(1-\frac{2k}{N-2(j-\ell)-1})=\frac{((N-2k-1))_{j}}{((N-1))_{j}}=\frac{((N-2j-1))_{k}}{((N-1))_{k}}=\prod_{\ell=1}^{k}\qty(1-\frac{2j}{N-2(k-\ell)-1})$ for any $N,j,k\ge1$ such that $N-2(j+k)+1\ge0$.}
	\begin{align*}
		n\si_{n,jk}=&\sum_{\al\in\Ga_{1j}}\sum_{\be\in\Ga_{1k}}\cov\qty(I_\al,I_\be)\\
		=&\indi_{j=k}\sum_{\al\in\Ga_{1j}}\frac{1}{((N-1))_{e(H_j)}}\qty(1-\frac{1}{((N-1))_{e(H_j)}})+\sum_{\al\in\Ga_{1j}}\sum_{\substack{\be\in\Ga_{1k}\setminus\{\al\}\\\be\cap\al\neq\varnothing}}\frac{-1}{((N-1))_{e(H_j)}((N-1))_{e(H_k)}}\\
		&\quad+\sum_{\al\in\Ga_{1j}}\sum_{\substack{\be\in\Ga_{1k}\\\be\cap\al=\varnothing}}\frac{1}{((N-1))_{e(H_j)+e(H_k)}}\qty(1-\prod_{\ell=1}^{e(H_j)}\qty(1-\frac{2e(H_k)}{N-2(e(H_j)-\ell)-1})).
	\end{align*}
	We can compute the cardinality of each sum as
	\begin{equation}\label{cardinality of each sum in covariance}
		\begin{split}
		\qty|\sum_{\al\in\Ga_{1j}}|=&\begin{cases}
			\frac{(n_1)_2}{2} &	(j=1)\\
			(n_1)_2n_2 &	(j=2)
		\end{cases},\\ \qty|\sum_{\al\in\Ga_{1j}}\sum_{\substack{\be\in\Ga_{1k}\\\be\cap\al=\varnothing}}|=&\begin{cases}
			\frac{(n_1)_4}{4}&(j=k=1)\\
			\frac{(n_1)_4n_2}{2}&(\text{$(j,k)=(1,2)$ or $(2,1)$})\\
			(n_1)_4(n_2)_2&(j=k=2)
		\end{cases},\\
		\qty|\sum_{\al\in\Ga_{1j}}\sum_{\substack{\be\in\Ga_{1k}\setminus\{\al\}\\\be\cap\al\neq\varnothing}}|=&\begin{cases}
			\frac{(n_1)_2}{2}\qty[\frac{(n_1)_2}{2}-1-\frac{(n_1-2)_2}{2}]&(j=k=1)\\
			\frac{(n_1)_2}{2}\qty[(n_1)_2-(n_1-2)_2]n_2&(\text{$(j,k)=(1,2)$ or $(2,1)$})\\
			(n_1)_2n_2\qty[(n_1)_2n_2-1-(n_1-2)_2(n_2-1)]	&(j=k=2)
		\end{cases}\\
		=&\begin{cases}
			(n_1)_3&(j=k=1)\\
			2(n_1)_3n_2+(n_1)_2n_2&(\text{$(j,k)=(1,2)$ or $(2,1)$})\\
		\begin{aligned}	&4(n_1)_3(n_2)_2+(n_1)_4n_2+2(n_1)_2(n_2)_2\\&\qquad+4(n_1)_3n_2+(n_1)_2n_2 \end{aligned}	&(j=k=2)
		\end{cases}.
		\end{split}
	\end{equation}
	Thus each covariance $\si_{n,jk}$ can be computed as follows:
	\begin{equation}\label{exact covariances}
		\begin{split}
		\si_{n,11}=&\frac{1}{n}\qty[\frac{(n_1)_2}{2}\frac{1}{N-1}\qty(1-\frac{1}{N-1})-\frac{(n_1)_3}{(N-1)^2}+\frac{(n_1)_4}{4}\frac{1}{((N-1))_2}\frac{2}{N-1}],\\
		\si_{n,12}=&\si_{n,21}=\frac{1}{n}\qty[-\frac{2(n_1)_3n_2+(n_1)_2n_2}{(N-1)((N-1))_2}+\frac{(n_1)_4n_2}{2}\frac{1}{((N-1))_3}\frac{4}{N-1}],\\
		\si_{n,22}=&\frac{1}{n}\bigg[(n_1)_2n_2\frac{1}{((N-1))_2}\qty(1-\frac{1}{((N-1))_2})\bigg.\\
		&\quad\bigg.-\frac{4(n_1)_3(n_2)_2+(n_1)_4n_2+2(n_1)_2(n_2)_2+4(n_1)_3n_2+(n_1)_2n_2}{\qty[((N-1))_2]^2}+\frac{(n_1)_4(n_2)_2}{((N-1))_4}\frac{8(N-4)}{((N-1))_2}\bigg].
		\end{split}
	\end{equation}
	Here, we understand that each term in the above expressions \eqref{exact covariances} is non-zero only if the cardinality of each sum in \eqref{cardinality of each sum in covariance} is positive. For example, the term
	\[\frac{(n_1)_4(n_2)_2}{((N-1))_4}\frac{8(N-4)}{((N-1))_2}\]
	in the expression of $\si_{22}$ is non-zero only when $(n_1)_4(n_2)_2>0$, whence $N\ge8$.
	
	\subsubsection{Main results}
	\label{subsubsec main results}
	
	We are ready to state our main theorem.
	\begin{thm}\label{configuration model clt}
		Suppose that $n\wedge N>15$, $n_1\ge1$ and $n_0+n_1\le n-2$. Let $c_{*}$ be any real number such that
		\begin{equation}\label{ratio constant}
			\frac{n_1\vee n_2}{n\wedge N-15}\le c_*.
		\end{equation}
		Define
		\begin{align*}
			\la^S_n=\frac{\sum_{i\in[n]}(d_i)_2}{2(N-1)},\quad \la^M_n=\frac{\sum_{1\le i<j\le n}(d_i)_2(d_j)_2}{2((N-1))_2},\quad\nu_n=\frac{\sum_{i\in[n]}(d_i)_2}{N-7},\quad\mu_n^{(3)}=\frac{\sum_{i\in[n]}(d_i)_3}{N-7},
		\end{align*}
		and
		\begin{align*}
			&\De_{\mathrm{simple}}	\coloneqq\\
			&\frac{7\nu_n^4+\qty(\mu_n^{(3)})^2+4\mu_n^{(3)}\nu_n^2+8\nu_n^3+4\mu_n^{(3)}\nu_n+2\mu_n^{(3)}+4\nu_n^2+2\la^S_n}{2(N-1)}+\frac{6\qty(\mu_n^{(3)})^2+8\mu_n^{(3)}\nu_n+\nu_n^2+4\la^M_n}{4((N-1))_2}.
		\end{align*}
		Also set
		\begin{align*}
			&C(c_*)=\frac{c_*}{2}\bigg[\sqrt{3\qty(\frac{1}{2}+\frac{3}{2}c_*^2+4c_*^3+\frac{7}{2}c_*^4+2c_*^6)}+\sqrt{2\qty(2c_*^3+5c_*^4+68c_*^5+72c_*^6+72c_*^8)}\bigg.\\
			&\qquad\qquad\quad+\sqrt{2\qty(2c_*^3+8c_*^4+4c_*^5+56c_*^6+72c_*^8)}\\
			&\qquad\qquad\quad+\bigg.\sqrt{3\qty(c_*+24c_*^4+64c_*^5+328c_*^6+256c_*^7+872c_*^8+2048c_*^{10})}\bigg]+\frac{4}{3}c_*^2+6c_*^3,\\
			&C'\qty(c_*,\la_n^S,\la_n^M)=\qty(c_*^2+6c_*^3)\la^S_n+\qty(2c_*^2+8c_*^3)\la^M_n+c_*^3.
		\end{align*}
		Let $Z_n$ be a $\mN_2(0,\Si_n)$-vector whose covariance matrix $\Si_n=\mqty(\si_{n,jk})$ equals $\cov(W_1)$, where $W_1=(W_\edge,W_\twostar)^\top$. (The entries $\si_{n,jk}$ are given in \eqref{exact covariances}.) Let $N_n$ be a $\Po(\la^S_n)\times\Po(\la^M_n)$-vector. Suppose that $Z_n$ and $N_n$ are independent.
		
				{\setlength{\leftmargini}{29.5pt}  	
			\begin{itemize}
				\setlength{\labelsep}{7pt}     
				\setlength{\itemsep}{3pt}      
				
				\item[\emph{a}.] \emph{(Joint approximation in the configuration model.)} Let $\mH_3$ be as in \eqref{H3} with $d=r=2$. We have
				
				\begin{align}
					&\sup_{h\in\mH_3}\Big|\e\qty[h\qty(W_\edge,W_\twostar,S_n,M_n)]-\e\left[h(Z_n,N_n)\right]\Big|\notag\\
					&\begin{multlined}
					\le\frac{C(c_*)}{\sqrt{n\wedge N-1}}+\qty(2\bigwedge\frac{3\pi}{2\sqrt{2}}\sqrt{\frac{1}{e\qty(\la^S_n\wedge\la^M_n)}})\frac{C'\qty(c_*,\la_n^S,\la_n^M)}{\sqrt{n}}\\+\qty(2\bigwedge\frac{4+4\log^+\qty(e\qty(\la^S_n\wedge\la^M_n))}{e\qty(\la^S_n\wedge\la^M_n)})\De_{\mathrm{simple}}.\end{multlined}\label{joint approximation in the configuration model bound}
				\end{align}
				
				\item[\emph{b}.] \emph{(Probability of simplicity of the configuration model.)}
				\begin{equation}\label{probability of simplicity bound}
					\qty|\p\qty(S_n+M_n=0)-e^{-\la^S_n-\la^M_n}|\le\qty(\frac{1}{2}\bigwedge\frac{1-e^{-\la^S_n-\la^M_n}}{\la^S_n+\la^M_n})\De_{\mathrm{simple}}.
				\end{equation}
				
				\item[\emph{c}.] \emph{(Normal approximation in uniform simple graphs.)} Let $\mC_3$ be as in \eqref{C3} with $d=2$. If
				\begin{align*}
					L_n\coloneqq &e^{-\la^S_n-\la^M_n}-\qty(\frac{1}{2}\bigwedge\frac{1-e^{-\la^S_n-\la^M_n}}{\la^S_n+\la^M_n})\De_{\mathrm{simple}}
				\end{align*}
				is positive, then we have
				\begin{align}
					&\sup_{g\in\mC_3,\|g\|\le1}\Big|\e\left[g\qty(W_\edge,W_\twostar)\middle|S_n+M_n=0\right]-\e\left[g(Z_n)\right]\Big|\notag\\
					\le&\frac{1}{L_n}\qty[\frac{C(c_*)}{\sqrt{n\wedge N-1}}+\qty(1\bigwedge\frac{3(1-e^{-\la^S_n-\la^M_n})}{2\qty(\la^S_n+\la^M_n)})\frac{C'\qty(c_*,\la_n^S,\la_n^M)}{\sqrt{n}}+\qty(1\bigwedge\frac{2(1-e^{-\la^S_n-\la^M_n})}{\la^S_n+\la^M_n})\De_{\mathrm{simple}}].\label{conditional clt bound 1}
				\end{align}
				
		\end{itemize} }
	\end{thm}
	\begin{rem}
	Some comments on Theorem \ref{configuration model clt} are warranted.
		{\setlength{\leftmargini}{29.5pt}  	
			\begin{enumerate}
				\setlength{\labelsep}{7pt}     
				\setlength{\itemsep}{3pt}      
				
				\item $n_1\ge1$ here (and in Proposition \ref{first term in stein coupling bound} below) is assumed to simplify the bound. $n_0+n_1\le n-2$ is to ensure that both $\la^S_n$ and $\la^M_n$ are positive.
				
				\item Since $\la^S_n\le\nu_n$ and $\la^M_n\le\nu_n^2$, $C'\qty(c_*,\la_n^S,\la_n^M)=O(\nu_n^2)$.
				
				\item Since $2\mu_n^{(3)}\nu_n^2\le\nu_n^4+(\mu_n^{(3)})^2$, $\De_{\mathrm{simple}}=O((\nu_n^4+(\mu_n^{(3)})^2)/N)$.\footnote{This is valid even if $\nu_n=O(1)$ or $\mu_n^{(3)}=O(1)$ (or both).} Therefore, (b) is consistent with \cite[Theorem 1.1]{AHH19}, in particular the bound \eqref{probability of simplicity bound} is of the same order as (1.13) ibidem. Note that in the definitions of $\nu_n$ and $\mu_n^{(3)}$, the numerators are divided by $N$ in \cite{AHH19}, which makes no difference in the order of magnitude.
				
				\item (c), as well as Theorem \ref{another conditional clt} below, are the first finite sample normal approximation results in uniform simple graphs with given vertex degrees. \cite{BR19} (in particular Theorem 2.1 ibidem) derives results only for the unconditional distribution of $\mathrm{CM}_n(\bm{d})$ and \cite{J20} (in particular Theorem 3.9 ibidem) proves only results on asymptotic normality by the method of moments.
				
		\end{enumerate} }
	\end{rem}
	
	\medskip
Before proving the theorem, we show convergence in distribution of the joint distribution $(W_\edge,W_\twostar,S_n,M_n)$ and the conditional distribution $(W_\edge,W_\twostar)|S_n+M_n=0$ under the standard assumption on the asymptotics of the degree sequence $\bm{d}$ in the literature.  Let
		$D_n$ denote the degree of a uniformly random vertex, i.e., $D_n$ is a random variable with the distribution
		\begin{align*}
			\p(D_n=k)=n_k/n,\quad k\ge0.
		\end{align*}
	 The distribution function of $D_n$	has the form
		\begin{align*}
			F_n(x)=\frac{1}{n}\sum_{i\in[n]}\indi_{d_i\le x}.
		\end{align*}
Because of this, the distribution of $D_n$ is also called the ``empirical degree distribution''.	We assume the following; (i) and (ii) are also assumed in \cite[Condition 1.3]{AHH19}. (iii) is \cite[(A6)]{J20}, which is assumed to exclude some less interesting cases in \cite[Theorems 3.2 and 3.9]{J20} (see Remark 3.5 ibidem).
		\begin{cond}[{\cite[(A1), (A2) and (A6)]{J20}}]\label{degree weak convergence positive mean uniform integrability}
			(i) $D_n$, the degree of a uniformly chosen vertex, converges in distribution to a random variable $D$ with a finite and positive mean $\mu\coloneqq\e[D]$. In other words, there exists a probability distribution $(p_k)_{k=0}^\infty$ such that
			\begin{align*}
				\frac{n_k}{n} \rightarrow p_k=\p(D=k),\quad\quad k\ge0,
			\end{align*}
			and $\mu=\sum_{k=0}^\infty kp_k\in(0,\infty)$. (ii) $\e[D_n]\rightarrow\e[D]=\mu$. Assuming (i), this is equivalent to $D_n$ being uniformly integrable. (iii) $p_1>0$ and $p_0+p_1<1$.
		\end{cond}
		Note that, under Condition \ref{degree weak convergence positive mean uniform integrability}(i)--(ii), 
		\begin{align*}
			\frac{N}{n}=\sum_{k=0}^{\infty}k\frac{n_k}{n}=\e[D_n]\rightarrow\mu\in(0,\infty).
		\end{align*}
		In particular, $N=O(n)$ and  $n=O(N)$ for sufficiently large $n$.  In order to deduce convergence of $\nu_n$ and $\Delta_{\mathrm{simple}}$, we further assume convergence of the second moment of the empirical degree distribution. This is also assumed in \cite[Theorems 1.5 and 1.7]{AHH19}.
		\begin{cond}[{\cite[(A3)]{J20}}]\label{convergence of the second moment of the empirical degree distribution}
$\e D_n^2\rightarrow \e D^2<\infty$. Assuming Condition \ref{degree weak convergence positive mean uniform integrability}(i), this is equivalent to $D_n^2$ being uniformly integrable.
		\end{cond}
	 The following corollary gives a Stein's method proof for some special cases of the weak convergence results in \cite[Theorems 3.2 and 3.9, Lemma 7.10]{J20} and partially answers the open problem mentioned in Remark 1.2 ibidem. (Note, however, that (A6) is not assumed in \cite[Lemma 7.10]{J20}.)
	
	\begin{cor}\label{weak convergence result}
		Under Conditions \ref{degree weak convergence positive mean uniform integrability} and \ref{convergence of the second moment of the empirical degree distribution}, $(W_\edge,W_\twostar,S_n,M_n)\stackrel{d}{\rightarrow}\mN_2(0,\Si)\times\Po(\nu/2)\times\Po(\nu^2/4)$ and $(W_\edge,W_\twostar)|S_n+M_n=0\stackrel{d}{\rightarrow}\mN_2(0,\Si)$, where $\nu=\e D(D-1)/\mu$ and
		\begin{equation}\label{asymptotic covariance matrix}
		\Si=\mqty(p_1^2\mu^{-1}/2-p_1^3\mu^{-2}+p_1^4\mu^{-3}/2&-2p_1^3p_2\mu^{-3}+2p_1^4p_2\mu^{-4}\\-2p_1^3p_2\mu^{-3}+2p_1^4p_2\mu^{-4}&p_1^2p_2\mu^{-2}-4p_1^3p_2^2\mu^{-4}-p_1^4p_2\mu^{-4}+8p_1^4p_2^2\mu^{-5}).
		\end{equation}
	\end{cor}
	\begin{proof}
	We first show that the bounds \eqref{joint approximation in the configuration model bound}--\eqref{conditional clt bound 1} go to zero as $n\rightarrow\infty$ under Conditions \ref{degree weak convergence positive mean uniform integrability} and \ref{convergence of the second moment of the empirical degree distribution}. Condition \ref{degree weak convergence positive mean uniform integrability}(i)--(iii) imply that there exists $n_*$ depending on the sequence $(D_n)$, or the ``sequence of the degree sequences'' $(\bm{d}_n)$, such that for all $n\ge n_*$,
	\begin{gather*}
		n>15,\quad N\ge\frac{\mu}{2}n>15,\quad	n_1\ge\frac{p_1}{2}n\ge1,\\
		\frac{n_0}{n}+\frac{n_1}{n}\le1-\frac{1}{2}(1-p_0-p_1)\le1-\frac{2}{n},\\
		\qty(1\wedge\frac{\mu}{2})n-15\ge\frac{1}{2}\qty(1\wedge\frac{\mu}{2})n.
	\end{gather*}
	Thus for all $n\ge n_*$, $n\wedge N>15$, $n_1\ge1$, $n_0+n_1\le n-2$ and 
	\begin{align*}
		\frac{n_1\vee n_2}{n\wedge N-15}\le\frac{n_1\vee n_2}{\qty(1\wedge\frac{\mu}{2})n-15}\le\frac{2}{1\wedge\frac{\mu}{2}}\frac{n_1\vee n_2}{n}\le \frac{4}{2\wedge\mu},
	\end{align*}
	verifying the condition \eqref{ratio constant} in Theorem \ref{configuration model clt} with $c_*=4/(2\wedge\mu)$. Thus the bounds \eqref{joint approximation in the configuration model bound} (with $c_*=4/(2\wedge\mu)$) and \eqref{probability of simplicity bound} hold for all $n\ge n_*$.
	
	\medskip
	Define $\nu=\e D(D-1)/\mu$. Then under Condition \ref{convergence of the second moment of the empirical degree distribution}, $\nu_n\rightarrow\nu$, $\la^S_n\rightarrow\nu/2$ and $\la^M_n\rightarrow\nu^2/4$. Hence $C'\qty(c_*,\la_n^S,\la_n^M)=O(1)$. Let $d_{\mathrm{max}}\coloneqq\max_{i\in[n]}d_i$ be the maximal degree.  Then $\nu_n\le\mu_n^{(3)}\le d_{\mathrm{max}}\nu_n$ holds and thus $\Delta_{\mathrm{simple}}=O\qty(\nu_n^4/n+d_{\mathrm{max}}^2\nu_n^2/n)=O(d_{\mathrm{max}}^2/n)$.  As noted in \cite[(2.6)]{J20}, the uniform integrability of $D_n^2$ in Condition \ref{convergence of the second moment of the empirical degree distribution} implies that $d_{\mathrm{max}}=o(n^{1/2})$ and thus $\Delta_{\mathrm{simple}}=o(1)$. Hence there exists $n_{**}$ no less than $n_*$ such that for all $n\ge n_{**}$, $C'\qty(c_*,\la_n^S,\la_n^M)\le \qty(c_*^2+6c_*^3)\nu/2+\qty(c_*^2+4c_*^3)\nu^2/2+c_*^3+1$, $L_n\ge e^{-\nu/2-\nu^2/4}/2>0$ and the bound \eqref{conditional clt bound 1} also holds. Therefore, under Conditions \ref{degree weak convergence positive mean uniform integrability} and \ref{convergence of the second moment of the empirical degree distribution}, the upper bounds and the left-hand side quantities in \eqref{joint approximation in the configuration model bound}--\eqref{conditional clt bound 1} all converge to zero.
	
	\medskip
	Under Condition \ref{degree weak convergence positive mean uniform integrability}(i)--(ii), the entries in $\Si_n$, \eqref{exact covariances}, converge to those in $\Si$, \eqref{asymptotic covariance matrix}. Therefore, the triangle inequalities and Theorem \ref{comparison bound joint normal-Poisson} conclude that, for independent random vectors $Z\sim\mN_2(0,\Si)$ and $N\sim\Po(\nu/2)\times\Po(\nu^2/4)$,
	\begin{align*}
\sup_{h\in\mH_3}\Big|\e\qty[h\qty(W_\edge,W_\twostar,S_n,M_n)]-\e\left[h(Z,N)\right]\Big|&\rightarrow0,\\
\sup_{g\in\mC_3,\|g\|\le1}\Big|\e\left[g\qty(W_\edge,W_\twostar)\middle|S_n+M_n=0\right]-\e\left[g(Z)\right]\Big|&\rightarrow0,
	\end{align*}
	under Conditions \ref{degree weak convergence positive mean uniform integrability} and \ref{convergence of the second moment of the empirical degree distribution}. This completes the proof.
	\end{proof}
	
	\medskip
	Theorem \ref{configuration model clt} follows from Theorem \ref{Stein coupling theorem} and the following five propositions regarding the moment terms involving $(I_\be)$ and $(J_{\be\al})$ in the abstract bound \eqref{Stein coupling bound}: Propositions \ref{first term in stein coupling bound}, \ref{second term in stein coupling bound}, \ref{third term in stein coupling bound}, \ref{fourth term in stein coupling bound} and \ref{fifth term in stein coupling bound}. All the proofs are in Appendix \ref{main proofs}, especially, in Appendix \ref{stein coupling bound first term}, Appendix \ref{stein coupling bound second term}, Appendix \ref{stein coupling bound third term}, Appendix \ref{stein coupling bound fourth term} and Appendix \ref{stein coupling bound fifth term}, respectively. 
	\begin{prp}\label{first term in stein coupling bound}
		Suppose that $N>15$ and $n_1\ge1$. Let $c_{*}$ be any real number such that
		\begin{align*}
			\frac{n_1\vee n_2}{N-15}\le c_*.
		\end{align*}
		Then, for the set of possible isolated edges $\Ga_{11}$ and the set of possible isolated 2-stars $\Ga_{12}$, we have 
		\begin{align*}
	\sum_{j,k=1}^2\sqrt{\Var\qty(\e\qty[\sum_{\al\in\Ga_{1j}}I_\al\sum_{\be\in\Ga_{1k}}(I_{\be}-J_{\be\al})\Big|W_1,W_2])}\le& C_1(c_*)\frac{n_1\vee n_2}{\sqrt{N-1}},
		\end{align*}
			where
		\begin{align*}
			C_1(c_*)=&\sqrt{3\qty(\frac{1}{2}+\frac{3}{2}c_*^2+4c_*^3+\frac{7}{2}c_*^4+2c_*^6)}+\sqrt{2\qty(2c_*^3+5c_*^4+68c_*^5+72c_*^6+72c_*^8)}\bigg.\\
			&\quad+\sqrt{2\qty(2c_*^3+8c_*^4+4c_*^5+56c_*^6+72c_*^8)}\\
			&\quad+\bigg.\sqrt{3\qty(c_*+24c_*^4+64c_*^5+328c_*^6+256c_*^7+872c_*^8+2048c_*^{10})}.
		\end{align*}
	\end{prp}

\begin{prp}\label{second term in stein coupling bound}
	Suppose that $N>3$ and let $c_\star$ be any real number such that
	\begin{align*}
		\frac{n_1\vee n_2}{N-3}\le c_\star.
	\end{align*}
	Then, for the set of possible isolated edges $\Ga_{11}$ and the set of possible isolated 2-stars $\Ga_{12}$, we have
	\begin{align*}
		\sum_{\al\in\Ga_{1}}\e\qty[I_\al\sum_{\be\in\Ga_{1}}\qty|I_\be-J_{\be\al}|\sum_{\ga\in\Ga_{1}}\qty|I_\ga-J_{\ga\al}|]\le&\qty(8c_\star+36c_\star^2)\qty(n_1\vee n_2).
	\end{align*}
\end{prp}
	\begin{prp}\label{third term in stein coupling bound}
		Suppose that $N>7$ and let $c_\sharp$ be any real number such that
		\begin{align*}
			\frac{n_1\vee n_2}{N-7}\le c_\sharp.
		\end{align*}
		Then, for the set of possible isolated edges and isolated 2-stars $\Ga_1$ and the set of possible self-loops and double edges $\Ga_2$, we have 
		\begin{align*}
			\sum_{\al\in\Ga_1}\sum_{\be\in\Ga_2}\e\qty[I_\al\qty|I_\be-J_{\be\al}|]\le \qty(c_\sharp^2+6c_\sharp^3)\la^S_n+\qty(2c_\sharp^2+8c_\sharp^3)\la^M_n+c_\sharp^3.
		\end{align*}
	\end{prp}

	\begin{prp}\label{fourth term in stein coupling bound}
		Suppose that $N>7$ and let $c_\sharp$ be any real number such that
		$c_\sharp\ge\frac{n_1\vee n_2}{N-7}$. Then, for the set of possible isolated edges and isolated 2-stars $\Ga_1$ and the set of possible self-loops and double edges $\Ga_2$, we have 
		\begin{align*}
			\sum_{\al\in\Ga_2}\sum_{\be\in\Ga_1}\e\qty[I_\al\qty|I_\be-J_{\be\al}|]=\sum_{\al\in\Ga_1}\sum_{\be\in\Ga_2}\e\qty[I_\al\qty|I_\be-J_{\be\al}|]	\le& \qty(c_\sharp^2+6c_\sharp^3)\la^S_n+\qty(2c_\sharp^2+8c_\sharp^3)\la^M_n+c_\sharp^3.
		\end{align*}
	\end{prp}

		\begin{prp}\label{fifth term in stein coupling bound}
		Suppose that $N>7$. Define
		\begin{align*}
			\nu_n=\frac{\sum_{i\in[n]}(d_i)_2}{N-7},\qquad\mu_n^{(3)}=\frac{\sum_{i\in[n]}(d_i)_3}{N-7}.
		\end{align*}
		Then, for the set of possible self-loops and double edges $\Ga_2$, we have 
		\begin{align*}
			&\sum_{\al\in\Ga_2} p_{\al}^2+\sum_{\al\in\Ga_2}\sum_{\be\in\Ga_2\setminus\{\al\}}\e\qty[I_\al|I_\beta-J_{\beta\al}|]\\
			\le&\frac{7\nu_n^4+\qty(\mu_n^{(3)})^2+4\mu_n^{(3)}\nu_n^2+8\nu_n^3+4\mu_n^{(3)}\nu_n+2\mu_n^{(3)}+4\nu_n^2+2\la^S_n}{2(N-1)}+\frac{6\qty(\mu_n^{(3)})^2+8\mu_n^{(3)}\nu_n+\nu_n^2+4\la^M_n}{4((N-1))_2}.
		\end{align*}
	\end{prp}
	Given these five propositions, the proof of Theorem \ref{configuration model clt} is immediate.
	\begin{proof}[Proof of Theorem \ref{configuration model clt}]
		(a) follows from Theorem \ref{Stein coupling theorem} and Propositions \ref{first term in stein coupling bound}, \ref{second term in stein coupling bound}, \ref{third term in stein coupling bound}, \ref{fourth term in stein coupling bound} and \ref{fifth term in stein coupling bound}, together with the smoothness estimates, Lemma \ref{smoothness x}, \eqref{smoothness first order difference of first derivative} in Lemma \ref{smoothness y} and \eqref{smoothness second order difference uniform} in Remark \ref{smoothness y uniform}.	(b) is an application of the one-dimensional Poisson approximation bound \eqref{one dim Poisson approximation zero probability bound} in Example \ref{one dimensional Poisson approximation}, combined with Proposition \ref{fifth term in stein coupling bound}.
		
		\medskip
			As for (c), we may assume that $Z_n$ is independent of all else. Let $g\in\mC_3$ with $\|g\|\le1$. Write $W_1=\qty(W_\edge,W_\twostar)$ and $W_2=S_n+M_n$. By the independence of $Z_n$ from $(W_2,N_n)$, as long as $\p\qty(W_2=0)>0$ we have
		\begin{align*}
			&\Big|\e\left[g(W_1)\middle|W_2=0\right]-\e\left[g(Z_n)\right]\Big|\\
			=&\frac{1}{\p(W_2=0)}\Big|\e\qty[g(W_1)\indi_{\{W_2=0\}}]-\e\qty[g(Z_n)\indi_{\{W_2=0\}}]\Big|\\
			\le&\frac{1}{\p(W_2=0)}\qty(\Big|\e\qty[g(W_1)\indi_{\{W_2=0\}}]-\e\qty[g(Z_n)\indi_{\{N_n=0\}}]\Big|+\Big|\e\qty[g(Z_n)\indi_{\{W_2=0\}}]-\e\qty[g(Z_n)\indi_{\{N_n=0\}}]\Big|)\\
			=&\frac{1}{\p(W_2=0)}\qty(\Big|\e\qty[g(W_1)\indi_{\{W_2=0\}}]-\e\qty[g(Z_n)\indi_{\{N_n=0\}}]\Big|+\Big|\e\qty[g(Z_n)]\Big(\p(W_2=0)-\p\qty(N_n=0)\Big)\Big|)\\
			\le&\frac{1}{\p(W_2=0)}\qty(\Big|\e\qty[g(W_1)\indi_{\{W_2=0\}}]-\e\qty[g(Z_n)\indi_{\{N_n=0\}}]\Big|+\Big|\p(W_2=0)-\p\qty(N_n=0)\Big|).
		\end{align*}
		 By taking $h(x,y)=g(x)\indi_{y=0}$ in Theorem \ref{Stein coupling theorem} together with the improved smoothness estimates \eqref{smoothness second order difference diagonal improved} and \eqref{smoothness first order difference of first derivative improved} in Lemma \ref{smoothness y}, we can bound the first term in the numerator. We may use part (b) to lower bound the probability $\p(W_2=0)$ and upper bound the second term in the numerator, from which (c) follows.
	\end{proof}
	
	\medskip
The quantitative error bound for the normal approximation in uniform simple graphs, Theorem \ref{configuration model clt}(c), \eqref{conditional clt bound 1}, is somewhat unsatisfactory in the sense that the Poisson approximation error $\Delta_{\mathrm{simple}}$ appears also in the numerator. As discussed above, even under the assumption that the second moment of the empirical degree distribution converges, $\Delta_{\mathrm{simple}}=O(d_{\mathrm{max}}^2/n)$ and this rate may be slower than $n^{-1/2}$ unless $d_{\rm{max}}=O(n^{1/4})$. (This condition is yet assumed in \cite[Theorem 2.1]{BR19}, which is about the unconditional $\mathrm{CM}_n(\bm{d})$.) By reconsidering the proof, we notice that (after assuming that $Z_n$ is independent of all else) it is enough to bound the quantity
\begin{equation}\label{direct approach to conditional clt}
\Big|\e\qty[g(W_1)\indi_{\{W_2=0\}}]-\e\qty[g(Z_n)\indi_{\{W_2=0\}}]\Big|,
\end{equation}
where $W_1=\qty(W_\edge,W_\twostar)$ and $W_2=S_n+M_n$. To do this directly, we take another approach; we employ the idea of parametrized Stein equation and solution in a similar spirit to \cite{BP14, Pim24,Tudor2025, TudorZurcher2025}. For $g\in\mC_3$ with $\|g\|\le1$, define a function $\tilde{f}_g:\R^2\rightarrow\R$ by
\begin{equation}
	\tilde{f}_g(x)=-\int_{0}^{1}\frac{1}{2(1-s)}\Big[S_sg(x)-\e[ g(Z_n)]\Big]\dd{s},
	\label{eq y fixed solution}
\end{equation}
where $S_s g(x)=\e[g\qty(\sqrt{1-s}x+\sqrt{s}Z_n)]$ is the Slepian interpolation of $g$. The integral in \eqref{eq y fixed solution} is absolutely convergent, $\tilde{f}_g$ is of class $C^3$ with estimates
\begin{align*}
	&|\tilde{f}_g|_{\mathrm{Lip}}\le1;\\
	&|\tilde{f}_g|_k\le\frac{1}{k},\quad k=2,3,
\end{align*}
and $\tilde{f}_g$ satisfies the equations
\begin{equation}\label{eq only x Stein equation}
	g(x)-\e\qty[g(Z_n)]=\sum_{j,k=1}^2\si_{n,jk}(\partial_{jk}\tilde{f}_g)(x)-\sum_{j=1}^2x_j (\partial_j\tilde{f}_g)(x),\quad\quad x\in\R^2.
\end{equation}
By multiplying both sides of \eqref{eq only x Stein equation} by $\indi_{y=0}$, $y\in\Z_+$, we obtain
\begin{equation}\label{eq y fixed Stein equation}
	g(x)\indi_{y=0}-\e\qty[g(Z_n)\indi_{y=0}]=\sum_{j,k=1}^2\si_{n,jk}(\partial_{jk}\tilde{f}_g)(x)\indi_{y=0}-\sum_{j=1}^2x_j (\partial_j\tilde{f}_g)(x)\indi_{y=0},\quad (x,y)\in\R^2\times\Z_+.
\end{equation}
Note that
\begin{equation*}
	\e\qty[g(Z_n)\indi_{W_2=0}]=\e\qty[\Big(\e\qty[g(Z_n)\indi_{y=0}]\Big)_{y=W_2(\omega)}]
\end{equation*}
by the independence of $Z_n$ and $W_2$. Thus, substituting $(x,y)=(W_1,W_2)$ in the parametrized Stein equation \eqref{eq y fixed Stein equation} and taking the expectations of both sides we have
 \begin{equation}\label{conditional clt Stein expectation expression}
 	\e\qty[g(W_1)\indi_{W_2=0}]-\e\qty[g(Z_n)\indi_{W_2=0}]=\e\qty[\sum_{j,k=1}^2\si_{n,jk}(\partial_{jk}\tilde{f}_g )(W_1)\indi_{W_2=0}-\sum_{j=1}^2W_{1j} (\partial_j\tilde{f}_g)(W_1)\indi_{W_2=0}],
 \end{equation}
 which can be used in order to bound \eqref{direct approach to conditional clt}. In particular, replacing $f_h(x,y)$ by $\tilde{f}_g(x)\indi_{y=0}$ we can follow the first half of the proof of Theorem \ref{Stein coupling theorem} (bounding \eqref{OU part}) to bound the right-hand side of \eqref{conditional clt Stein expectation expression} in magnitude. Since $	|\indi_{W_2=0}-\indi_{W_{2\al}=0}|\le\indi_{W_2\neq W_{2\al}}\le|W_{2}-W_{2\al}|$, we bound the counterpart of $R_3$ in \eqref{OU part decomposition} as
 \begin{align*}
 	\qty|\frac{1}{\sqrt{n}}\sum_{\al\in\Ga_{1}}\e\Big[I_\al(\partial_j\tilde{f}_g)(W_{1\al})\qty(\indi_{W_{2}=0}-\indi_{W_{2\al}=0})\Big]|\le&\frac{1}{\sqrt{n}}\sum_{\al\in\Ga_{1}}\e\qty[I_\al\|\partial_j\tilde{f}_g\|\qty|W_2-W_{2\al}|]\\
 	\le&\frac{1}{\sqrt{n}}\sum_{\al\in\Ga_{1}}\sum_{\be\in\Ga_{2}}\e\qty[I_\al\qty|I_\be-J_{\be\al}|].
 \end{align*}
  Therefore, arguing similarly, we obtain another quantitative bound for the conditional CLT:
	\begin{thm}[Normal approximation in uniform simple graphs, 2]\label{another conditional clt}
	Keep the same assumptions and notation as in Theorem \ref{configuration model clt}.	If $L_n$ is positive, then we have
		\begin{equation}\label{conditional clt bound 2}
			\sup_{g\in\mC_3,\|g\|\le1}\Big|\e\left[g\qty(W_\edge,W_\twostar)\middle|S_n+M_n=0\right]-\e\left[g(Z_n)\right]\Big|\le\frac{1}{L_n}\qty[\frac{C(c_*)}{\sqrt{n\wedge N-1}}+\frac{C'\qty(c_*,\la_n^S,\la_n^M)}{\sqrt{n}}].
		\end{equation}
	\end{thm}
Comparing this with Theorem \ref{configuration model clt}(c), \eqref{conditional clt bound 1}, we note that there is a trade-off: If we rely on the Stein equation and solution for joint normal-Poisson approximation, we obtain the magic factor $(\la_n^S+\la_n^M)^{-1}$ also for the cross term at the cost of the additional Poisson approximation error term in the bound, whereas if we use the parametrized Stein equation \eqref{eq y fixed Stein equation} and solution $\tilde{f}_g(x)\indi_{y=0}$, we can avoid including the Poisson approximation error term in the bound but we also lose the magic factor in the cross term.
\begin{cor}\label{pseudo rate of convergence smooth metric}
Let $\mC_3$ be as in \eqref{C3} with $d=2$. Let $Z_n$ be a $\mN_2(0,\Si_n)$-vector whose covariance matrix $\Si_n$ equals that of $(W_\edge,W_\twostar)^\top$. Under Conditions \ref{degree weak convergence positive mean uniform integrability} and \ref{convergence of the second moment of the empirical degree distribution}, there exists a constant $C$ depending on the sequence $(D_n)$ (or the sequence of the degree sequences $(\bm{d}_n)$) such that
	\begin{equation}\label{pseudo rate of convergence smooth metric equation}
\sup_{g\in\mC_3,\|g\|\le1}\Big|\e\left[g\qty(W_\edge,W_\twostar)\middle|S_n+M_n=0\right]-\e\left[g(Z_n)\right]\Big|\le Cn^{-1/2}.
	\end{equation}
\end{cor}
\begin{proof}
	As shown in the proof of Corollary \ref{weak convergence result}, Conditions \ref{degree weak convergence positive mean uniform integrability} and \ref{convergence of the second moment of the empirical degree distribution} imply that there exists $n_*$ depending on the sequence $(D_n)$ (or the sequence of the degree sequences $(\bm{d}_n)$) such that for all $n\ge n_*$:
			{\setlength{\leftmargini}{29.5pt}  	
		\begin{itemize}
			\setlength{\labelsep}{7pt}     
			\setlength{\itemsep}{3pt}      
			
			\item $n\wedge N>15$, $n_1\ge1$, $n_0+n_1\le n-2$ and
			\begin{align*}
				\frac{1}{n\wedge N-15}\le \frac{4}{2\wedge\mu}\frac{1}{n},
			\end{align*}
			which verifies the condition \eqref{ratio constant} with $c_*=4/(2\wedge\mu)$, and
			
			\item $C'\qty(c_*,\la_n^S,\la_n^M)\le \qty(c_*^2+6c_*^3)\nu/2+\qty(c_*^2+4c_*^3)\nu^2/2+c_*^3+1$ and $L_n\ge e^{-\nu/2-\nu^2/4}/2>0$, where $\nu=\e D(D-1)/\mu$ and $c_*=4/(2\wedge\mu)$.
			
	\end{itemize} }
Thus the inequality \eqref{conditional clt bound 2} holds with $c_*=4/(2\wedge\mu)$ for all $n\ge n_*$, and since $1/\sqrt{n\wedge N-1}\le \sqrt{c_*/n}$ with $c_*=4/(2\wedge\mu)$ for all $n\ge n_*$, there exists a constant $C$ depending on $\nu=\e D(D-1)/\mu$ and $c_*=4/(2\wedge\mu)$ such that \eqref{pseudo rate of convergence smooth metric equation} holds for $n\ge n_*$. For $n<n_*$, we trivially estimate
\begin{align*}
	\sup_{g\in\mC_3,\|g\|\le1}\Big|\e\left[g\qty(W_\edge,W_\twostar)\middle|S_n+M_n=0\right]-\e\left[g(Z_n)\right]\Big|\le 2\le2\sqrt{\frac{n_*}{n}}.
\end{align*}
Thus \eqref{pseudo rate of convergence smooth metric equation} holds for all $n$ if $C$ is replaced by $C\vee2\sqrt{n_*}$.
\end{proof}
\begin{rem}
We caution that the constant $C$ in \eqref{pseudo rate of convergence smooth metric equation} is not a universal constant. It depends not only on $\nu=\e D(D-1)/\mu$ and $c_*=4/(2\wedge\mu)$ but also on $n_*$ (and in fact grows at a rate $n_*^{1/2}$), which depends on the sequence $(D_n)$ (the sequence of the degree sequences $(\bm{d}_n)$), or the \emph{model} itself. This is because convergence of the distribution and the moments of $D_n$ assumed in Conditions \ref{degree weak convergence positive mean uniform integrability} and \ref{convergence of the second moment of the empirical degree distribution} is itself asymptotic. If we want to make it universal, we have to impose a more stringent condition on the convergence of $n_k/n\rightarrow p_k$, $\e D_n\rightarrow\mu$ and $\e D_n^2\rightarrow\e D^2$, e.g., $\mu/2<\e D_n=N/n<3\mu/2$ for $n\ge1000$.
\end{rem}
	
	\subsubsection{On moment calculation}
	
	Finally, we will sketch an instance of moment calculation required in the proof of Proposition \ref{first term in stein coupling bound} (Appendix \ref{stein coupling bound first term}), because the proof of Proposition \ref{first term in stein coupling bound} is technically involved compared to those of Propositions \ref{second term in stein coupling bound}--\ref{fifth term in stein coupling bound}. In the proof, especially in Appendix \ref{subsec third variance}, we are required to estimate the variance of the type
	\begin{equation}
		\label{variance for al be intersection case in the main text}
		\Var\qty(\sum_{\al\in\Ga_{1j}}\sum_{\substack{\be\in\Ga_{1k}\setminus\{\al\}\\ \be\cap\al\neq\varnothing }}I_\al J_{\be\al})=\sum_{\al,\al'\in\Ga_{1j}}\sum_{\substack{\be\in\Ga_{1k}\setminus\{\al\}\\ \be\cap\al\neq\varnothing }}\sum_{\substack{\be'\in\Ga_{1k}\setminus\{\al'\}\\ \be'\cap\al'\neq\varnothing }}\Big(\e\qty[I_\al I_{\al'} J_{\be\al}J_{\be'\al'}]-p_\al p_{\al'}p_\be p_{\be'}\Big),
	\end{equation}
	and the probability $\p\qty(I_\al=I_{\al'}=1,J_{\be\al}=J_{\be'\al'}=1)$ accordingly. Here we consider the following case (this is a strict subcase of Case (d1i) in Appendix \ref{subsec third variance}): $\al=(s_1s_2)\cup(s_3s_4)$, $\be=(s_4t_2)\cup(t_3t_4)$, $\al'=(s_1's_2')\cup(s_3's_4')$ and $\be'=(s_4't_2')\cup(t_3't_4')$, where all the half-edges are distinct, $s_2$ and $s_3$ are incident to $\al$'s degree 2 vertex, $t_2$ and $t_3$ are incident to $\be$'s degree 2 vertex, $s_2'$ and $s_3'$ are incident to $\al'$'s degree 2 vertex, and $t_2'$ and $t_3'$ are incident to $\be'$'s degree 2 vertex.
	
	\medskip
	 In order for $\be$ and $\be'$ to be created after switching $\al$ and $\al'$, respectively, we note that the edges $t_3t_4$ and $t_3't_4'$ must be present in the original configuration and the edge $s_3s_4$ (resp.\ $s_3's_4'$) must be switched so that the edge $s_4t_2$ (resp.\ $s_4't_2'$) is certainly created in the switching procedure with respect to $\al$ (resp.\ $\al'$). Such a switching occurs with probability at most $1/(N-3)$, corresponding to the case $s_3\wedge s_4<s_1\wedge s_2$ (resp.\  $s_3'\wedge s_4'<s_1'\wedge s_2'$), i.e., when we switch the edge $s_3s_4$ (resp.\ $s_3's_4'$) first in the switching procedure with respect to $\al$ (resp.\ $\al'$). Therefore,
	 \begin{align}
\p\qty(I_\al=I_{\al'}=1,J_{\be\al}=J_{\be'\al'}=1)\le&\p\qty(I_\al=I_{\al'}=I_{t_3t_4}=I_{t_3't_4'}=1)\times\frac{1}{N-3}\times\frac{1}{N-3}\notag\\
=&\frac{1}{((N-1))_6(N-3)^2}.\label{probability created example main text}
	 \end{align} 
For a quadruplet $\al,\al',\be,\be'$ with such a relationship, we note that there are ten free vertices, namely, six distinct degree 1 vertices and four distinct degree 2 vertices.  Thus, under Condition \ref{degree weak convergence positive mean uniform integrability}(i), the possibilities of such a quadruplet $\al,\al',\be,\be'$ is $O(n^{10})$. Under Condition \ref{degree weak convergence positive mean uniform integrability}(i)--(ii), for sufficiently large $n$, $n=O(N)$ and
\begin{align*}
	&\sum_{\al,\al',\be,\be'}\e\qty[I_\al I_{\al'} J_{\be\al}J_{\be'\al'}]-\sum_{\al,\al',\be,\be'}p_\al p_{\al'} p_\be p_{\be'}\\
	=&O(n^{10})\times\qty(\frac{1}{((N-1))_6(N-3)^2}-\frac{1}{(N-1)^4(N-3)^4})\\
	=&O(n^{10})\times O(n^{-9})\\
	=&O(n).
\end{align*}
In Appendix \ref{subsec third variance}, we show the estimate of the type \eqref{probability created example main text} more formally by reducing to estimating the cardinality of a set defined by the deterministic functions $f_\al$ and $f_{\al'}$ (cf.\  \eqref{probability to cardinality} and \eqref{subset of the product of Gbeta and Gbetaprime} ibidem). We also estimate the possibilities of such a quadruplet $\al,\al',\be,\be'$ more precisely without assuming Condition \ref{degree weak convergence positive mean uniform integrability}, leading to the finite sample bound in Proposition \ref{first term in stein coupling bound}. 
	
	\section{Concluding remarks}
	\label{concluding remarks}
	
	In this paper, we studied normal approximation to the numbers of isolated edges and 2-stars and Poisson approximation to the numbers of self-loops and double edges in the configuration model, both in a joint manner and conditioning on the multigraph being simple, the latter implying the first finite sample normal approximation results for uniform simple graphs with given vertex degrees (Theorem \ref{configuration model clt}(c) and Theorem \ref{another conditional clt}). As a preparation for this, we developed a new Stein's method framework for joint normal-Poisson approximation and a new coupling approach to sums of indicator random variables, which may be of independent interest and could have other potential applications.
	
	\medskip
Although this work confirms the prediction made in \cite[Remark 1.2]{J20} for the asymptotic normality of the numbers of isolated edges and isolated 2-stars (see Corollary \ref{weak convergence result}), we have not yet fully reproduced \cite[Theorems 3.2 and 3.9, Lemma 7.10]{J20}, which claim joint weak convergence of the numbers of general isolated trees and general isolated connected unicyclic multigraphs in any dimension. As noted in Remark \ref{unlabeled copy generalization}, the coupling construction via switching could possibly be generalized to general isolated trees. However, the current proof of Proposition \ref{first term in stein coupling bound} relies on a brute force approach to evaluate the variance of the type \eqref{variance for al be intersection case in the main text}, where we essentially perform the case analysis about how $\be,\be'$ are created by the switchings when $\al,\be,\al',\be'$ intersect each other (see Appendix \ref{subsec third variance}). We leave the problem of finding a fully generic proof for general isolated trees as an open problem. On the other hand, our new Stein's method developed in Sections \ref{new method} and \ref{indicator framework} can readily incorporate additional Poisson approximation coordinates for isolated connected unicyclic multigraphs, such as isolated triangles and isolated squares.
	
	\medskip
	With a stronger moment assumption on the degrees (which implies convergence of all moments of the empirical degree distribution), \cite{J20} also showed asymptotic normality of the number of copies (not necessarily isolated) of a given tree in the configuration model and the uniform simple graph (Theorem 3.10 ibidem). Although our proof heavily relies on explicit cardinality estimates about various sets of possible isolated copies (e.g., Lemma \ref{cardinality estimates}), the extension of the normal approximation results in this paper to the subgraph counts could be another research direction.
		\section*{Acknowledgements}
		The author would like to thank Adrian R\"{o}llin for introducing the author to the problem of the configuration model and Stein's method and for insightful comments, valuable suggestions and discussions. In particular, the idea of working on a kind of generator like \eqref{new generator} for joint normal-Poisson approximation is his, and Lemmas \ref{intersection limited creation} and \ref{disjoint limited destruction} are inspired by his observation that only a limited number of the indicators should be changed by the switchings. The author is grateful to Yuta Koike for valuable discussions, and in particular, for allowing the author to cite his \cite[Lemma 9.9 and Proposition 9.4]{Koike25}.

	\bibliographystyle{abbrv}
	\bibliography{reference}
	
	\newpage

	\appendix
	
	\section{Proof of \eqref{configuration model uniform over simple}}
\label{configuration model uniform over simple proof}

We adopt the following conventions: $1/(N-1)!!=1$ and $[N]=\varnothing$ if $N=0$, and there is only one bijection from $\varnothing$ onto $\varnothing$. Since each configuration is assigned probability $1/(N-1)!!$, it suffices to show that the number of possible configurations which give rise to a given simple graph $G$ with degree sequence $\bm{d}$ is $\prod_{i=1}^nd_i!$. We identify a matching of the $N$ half-edges (a partition of $[N]$ into $N/2$ pairs) with a \emph{matching function}, i.e., a bijection $\mu:[N]\rightarrow[N]$ such that for all $s\in[N]$, $\mu(s)\neq s$, and for all $s\neq t\in[N]$, $\mu(s)=t$ iff $\mu(t)=s$. In particular, we have $\mu(\mu(s))=s$ for all $s\in[N]$ for such a matching function $\mu$. Given a matching function $\mu$ and half-edges $s$ and $t$, $1_{\mu(s)=t}$ then denotes the indicator that $s$ and $t$ are matched in $\mu$. Write $G=\qty(x_{ij})_{1\le i,j\le n}$, where $x_{ii}$ is the number of self-loops in $v_i$ and $x_{ij}=x_{ji}$ is the number of edges between $v_i$ and $v_j$ for $i\neq j$. Since $G$ is simple, $x_{ii}=0$ and $x_{ij}$ is at most $1$ for $i\neq j$. Then the set of all matchings which are compatible with the given simple graph $G$ can be described as the following set:
\begin{align*}
	M_1\coloneqq\qty{\mu:\begin{aligned}
			&\text{$\mu$ is a matching function on $[N]$},\\
			&\sum_{s<t\in v_i}1_{\mu(s)=t}=0\ (1\le i\le n),\quad \sum_{s\in v_i,t\in v_j}1_{\mu(s)=t}=x_{ij}\ (1\le i<j\le n)
				 \end{aligned}
				 }.
\end{align*}
	For $1\le i\le n$, let $\Upsilon_i\coloneqq\qty{s\in[N]:s\in v_i}$ (the set of all half-edges incident to $v_i$) and $\mN_i\coloneqq\qty{j\in[n]\setminus\{i\}:x_{ij}=1}$. Note that $\qty|\Upsilon_i|=\qty|\mN_i|=d_i$ (and that $\Upsilon_i=\mN_i=\varnothing$ if $d_i=0$). Define
	\begin{align*}
		M_2\coloneqq\qty{\qty(f_1,\dots,f_n):\text{$f_i:\Upsilon_i\rightarrow\mN_i$ is a bijection, $1\le i\le n$}}.
	\end{align*}
Each element $\qty(f_1,\dots, f_n)\in M_2$ specifies which half-edge incident to $v_i$ is toward which vertex adjacent to $v_i$ in $G$. Then there is a natural bijection $F$ from $M_2$ onto $M_1$, showing that $\qty|M_1|=\qty|M_2|=\prod_{i=1}^nd_i!$. Precisely, for each $\qty(f_1,\dots,f_n)\in M_2$, a function $F\qty(f_1,\dots,f_n):[N]\rightarrow[N]$ is defined by matching $s\in v_i$ with the half-edge incident to $v_{f_i(s)}$ which is toward $v_i$, i.e., $F\qty(f_1,\dots,f_n)(s)\coloneqq\qty(f_{f_i(s)})^{-1}(i)$ for each $s\in v_i$ for each $1\le i\le n$. (Note that $i\in\mN_{f_i(s)}$ for $s\in v_i$, since $f_i(s)\in\mN_i$ and $x_{if_i(s)}=x_{f_i(s)i}=1$.) Formally checking that $F$ is into $M_1$ and that $F$ is injective and surjective is a routine matter and thus omitted. \qed 
	
	\section{Proofs for Section \ref{new method}}
	
	\subsection{Convergence in distribution in $\R^d\times\Z_+^r$}
	\begin{prp}\label{convergence in distribution in RZ}
		Let $(X_n,Y_n),n=1,2,\dots$ and $(X,Y)$ be random elements in $\R^d\times\Z_+^r$. Then $(X_n,Y_n)$ converges in distribution to $(X,Y)$ if and only if $\lim_n\e[h(X_n,Y_n)]=\e[h(X,Y)]$ for all $h\in\mH_3$.
	\end{prp}
	\begin{proof}
		We first show the ``only if'' implication. Since any $h\in\mH_3$ is bounded, it suffices to show that any $h\in\mH_3$ is continuous on $\R^d\times\Z_+^r$. Let $h\in\mH_3$, and let $(x_n,y_n)$ be a convergent sequence in $\R^d\times\Z_+^r$ converging to $(x,y)\in\R^d\times\Z_+^r$. Since $x_n\rightarrow x$ in $\R^d$, $y_n\rightarrow y$ in $\Z_+^r$ and $h(\cdot,y)$ is continuous on $\R^d$, for any $\epsilon>0$, there exists $n_\epsilon$ such that $y_n=y$ and $\qty|h(x_n,y)-h(x,y)|<\epsilon$ for $n\ge n_\epsilon$, thus $\qty|h(x_n,y_n)-h(x,y)|<\epsilon$ for $n\ge n_\epsilon$. This shows that $h$ is continuous on $\R^d\times\Z_+^r$.
		
		\medskip
		For the converse, suppose that $\lim_n\e[h(X_n,Y_n)]=\e[h(X,Y)]$ for all $h\in\mH_3$. We first show that $\lim_n\e[f(X_n)g(Y_n)]=\e[f(X)g(Y)]$ for every compactly supported, continuous real-valued function $f$ on $\R^d$ and every bounded, continuous real-valued function $g$ on $\Z_+^r$. Let $f$ be a compactly supported, continuous real-valued function on $\R^d$, and let $g$ be a bounded, continuous real-valued function on $\Z_+^r$. We may assume that $\|g\|\le1$. Let $\epsilon>0$. There exists a compactly supported, real-valued $C^\infty$-function $f_\epsilon$ on $\R^d$ such that $\|f-f_\epsilon\|\le\epsilon$. Let $M\coloneqq\qty(\|f_\epsilon\|\vee|f_\epsilon|_{\mathrm{Lip}}\vee|f_\epsilon|_2\vee|f_\epsilon|_3)+1$. Then $1\le M<\infty$. Note that
		\begin{align*}
			\qty|\e[f(X_n)g(Y_n)]-\e[f(X)g(Y)]|\le&\qty|\e[f_\epsilon(X_n)g(Y_n)]-\e[f_\epsilon(X)g(Y)]|+2\epsilon\\
			=&M\qty|\e\qty[M^{-1}f_\epsilon(X_n)g(Y_n)]-\e\qty[M^{-1}f_\epsilon(X)g(Y)]|+2\epsilon.
		\end{align*}
		Since $(x,y)\mapsto M^{-1}f_\epsilon(x) g(y)$ is in $\mH_3$, we have $\limsup_n\qty|\e[f(X_n)g(Y_n)]-\e[f(X)g(Y)]|\le2\epsilon$. Since $\epsilon>0$ is arbitrary, $\lim_n\e[f(X_n)g(Y_n)]=\e[f(X)g(Y)]$, and we are done.
		
		\medskip
		We next show that $\lim_n\e[f(X_n)g(Y_n)]=\e[f(X)g(Y)]$ for all bounded, continuous real-valued functions $f$ on $\R^d$ and $g$ on $\Z_+^r$. Let $f$ and $g$ be bounded, continuous real-valued functions on $\R^d$ and on $\Z_+^r$, respectively. We may assume that $\|f\|\le1$ and $\|g\|\le1$. Let $\epsilon>0$, and take $A>0$ so that $\p(|X|\le A)\ge1-\epsilon$. Take a continuous function $\phi:\R^d\rightarrow[0,1]$ which equals $1$ on $\{x\in\R^d:|x|\le A \}$ and $0$ on $\{x\in\R^d:|x|\ge A+1\}$. Since $\phi$ is a compactly supported, continuous real-valued function on $\R^d$, we have $\lim_n\e[\phi(X_n)]=\e[\phi(X)]\ge\p(|X|\le A)\ge1-\epsilon$. Note that
		\begin{align*}
			\qty|\e[f(X_n)g(Y_n)]-\e[f(X)g(Y)]|\le&\qty|\e[f(X_n)\phi(X_n)g(Y_n)]-\e[f(X)\phi(X)g(Y)]|\\
			&\quad\quad+\e[1-\phi(X_n)]+\e[1-\phi(X)]\\
			\le&\qty|\e[f(X_n)\phi(X_n)g(Y_n)]-\e[f(X)\phi(X)g(Y)]|\\
			&\quad\quad+\e[\phi(X)]-\e[\phi(X_n)]+2\epsilon.
		\end{align*}
		Since $f\phi$ is a compactly supported, continuous real-valued function on $\R^d$, we have $\limsup_n\linebreak\qty|\e[f(X_n)g(Y_n)]-\e[f(X)g(Y)]|\le2\epsilon$.
		Since $\epsilon>0$ is arbitrary, we have $\lim_n\e[f(X_n)g(Y_n)]=\e[f(X)g(Y)]$.
		
		\medskip
		We have shown that $\lim_n\e[f(X_n)g(Y_n)]=\e[f(X)g(Y)]$ for all bounded, continuous real-valued functions $f$ on $\R^d$ and $g$ on $\Z_+^r$. Then we see that $\lim_n\e[f(X_n)g_*(Y_n)]=\e[f(X)g_*(Y)]$ for all bounded, continuous real-valued functions $f$ on $\R^d$ and $g_*$ on $\R^r$ by restricting $g_*$ to $\Z_+^r$. Then it follows that $\lim_n\e[h_*(X_n,Y_n)]=\e[h_*(X,Y)]$ for all bounded, continuous real-valued function $h_*$ on $\R^d\times\R^r$ by regarding $(X_n,Y_n),n=1,2,\dots$ and $(X,Y)$ as random elements in $\R^d\times\R^r$ and using a characteristic function argument (the L\'{e}vy continuity theorem in $\R^{d+r}$). Finally, we show that $\lim_n\e[h(X_n,Y_n)]=\e[h(X,Y)]$ for all bounded, continuous real-valued function $h$ on $\R^d\times\Z_+^r$. Let $h$ be a bounded, continuous real-valued function on $\R^d\times\Z_+^r$. Note that $\qty(\R^d\times\R^r)\setminus\qty(\R^d\times\Z_+^r)=\R^d\times\qty(\R^r\setminus\Z_+^r)$ is an open subset of $\R^d\times\R^r$, thus $\R^d\times\Z_+^r$ is a closed subset of $\R^d\times\R^r$. By the Tietze extension theorem, there is a bounded, continuous real-valued function $h_*$ on $\R^d\times\R^r$ such that $h_*=h$ on $\R^d\times\Z_+^r$, and we have $\e[h(X_n,Y_n)]=\e[h_*(X_n,Y_n)]\rightarrow\e[h_*(X,Y)]=\e[h(X,Y)]$. This completes the proof.
	\end{proof}
	
	\subsection{The Slepian interpolation}
	\label{sec: slepian}
	
	We use the following form of the multivariate Stein identity (\cite[Lemma 9.9]{Koike25}), also known as the Gaussian integration by parts formula. Almost the same statements appear as \cite[Lemma 2]{CCK15} and \cite[Exercise 3.1.5]{NP12}. The point is that the formula holds even for a singular covariance matrix $\Si$. \cite[Lemma 2]{CCK15} refers the proof to \cite[Section A.6]{Tala03}, \cite{CGS11} and \cite{Stein81}. We provide the proof for the reader's convenience. See also \cite[Sections A.3 and A.4]{Tala11}.
	\begin{lem}[{\cite[Lemma 9.9]{Koike25}}]\label{multivariate Stein identity}
		Let $Z=(Z_1,\dots,Z_d)^\top$ be a centered Gaussian random vector in $\R^d$ (whose covariance matrix may be singular). Let $F:\R^d\rightarrow\R$ be a $C^1$-function such that $\e[|\partial_jF(Z)|]<\infty$ and $\e[|Z_jF(Z)|]<\infty$ for all $1\le j\le d$.\footnote{\cite[Lemma 2]{CCK15} is stated without the assumption ``$\e[|Z_jF(Z)|]<\infty$ for all $1\le j\le d$'', but as opposed to the one-dimensional case, it seems difficult to renounce this assumption in the multivariate case.} Then for every  $1\le j\le d$, 
		\begin{align*}
			\e[Z_jF(Z)]=\sum_{k=1}^d\e[Z_jZ_k]\e[\partial_kF(Z)].
		\end{align*}
	\end{lem}
	\begin{proof}
		The $d=1$ case is merely the statement of the one-dimensional Stein identity and assume that $d\ge2$. Fix $1\le j\le d$. We may assume that $\e [Z_j^2]>0$. Consider the random variables
		\begin{align*}
			Z'_k=Z_k-Z_j\frac{\e[Z_jZ_k]}{\e [Z_j^2]},\quad k\neq j.
		\end{align*}
		They satisfy $\e[Z_jZ_k']=0$, $k\neq j$; in particular, since $(Z_j,(Z_k')_{k\neq j})$ are jointly Gaussian, $Z_j$ is independent of the family $(Z_k')_{k\neq j}$. Somewhat imprecisely, we write a function of $d$ variables $f(z_1,\dots,z_d)$ as $f(z_j,(z_k)_{k\neq j})$. For $z_j\in\R$ and $(z_k')_{k\neq j}\in\R^{d-1}$, define $\tilde{F}(z_j,(z_k')_{k\neq j})$ by
		\begin{align*}
			\tilde{F}(z_j,(z_k')_{k\neq j})=F(z_j,(z_k'+z_j\e[Z_jZ_k]/E[Z_j^2])_{k\neq j}).
		\end{align*}
		$ \tilde{F}$ is a $C^1$-function and we have
		\begin{align*}
			\partial_j\tilde{F}(z_j,(z_k')_{k\neq j})=\sum_{\ell=1}^{d}\frac{\e[Z_jZ_\ell]}{\e [Z_j^2]}\partial_\ell F(z_j,(z_k'+z_j\e[Z_jZ_k]/E[Z_j^2])_{k\neq j});
		\end{align*}
		in particular $\e[|\partial_j\tilde{F}(Z_j,(Z_k')_{k\neq j})|]<\infty$ by the assumption. Then by the independence  
		\begin{align*}
			\e[|\partial_j\tilde{F}(Z_j,(Z_k')_{k\neq j})|]=\e[g((Z_k')_{k\neq j})],
		\end{align*}
		where 
		\begin{align*}
			g((z_k')_{k\neq j})=\e[|\partial_j\tilde{F}(Z_j,(z_k')_{k\neq j})|],\quad (z_k')_{k\neq j}\in\R^{d-1}.
		\end{align*}
		Moreover, $g((z_k')_{k\neq j})<\infty$ for $\mL_{(Z_k')_{k\neq j}}$-almost every $(z_k')_{k\neq j}$. Thus, for $\mL_{(Z_k')_{k\neq j}}$-almost every $(z_k')_{k\neq j}$, by invoking the one-dimensional Stein identity we have ($\e[|Z_j\tilde{F}(Z_j,(z_k')_{k\neq j})|]<\infty$ and)
		\begin{align*}
			\e[Z_j\tilde{F}(Z_j,(z_k')_{k\neq j})]=\e[Z_j^2]\e[\partial_j\tilde{F}(Z_j,(z_k')_{k\neq j})].
		\end{align*}
		Since $\e[|Z_jF(Z)|]=\e[|Z_j\tilde{F}(Z_j,(Z_k')_{k\neq j})|]<\infty$ and $\e[|\partial_j\tilde{F}(Z_j,(Z_k')_{k\neq j})|]<\infty$ by the assumption, using the independence again concludes
		\begin{align*}
			\e[Z_jF(Z)]=\e[Z_j\tilde{F}(Z_j,(Z_k')_{k\neq j})]=\e[Z_j^2]\e[\partial_j\tilde{F}(Z_j,(Z_k')_{k\neq j})]=\sum_{k=1}^d\e[Z_jZ_k]\e[\partial_kF(Z)].
		\end{align*}
	\end{proof}
	For $Z\sim\mN_d(0,\Si)$ with a possibly singular covariance matrix $\Si=\mqty(\si_{jk})$ and a function $h:\R^d\rightarrow\R$, define the Slepian interpolation of $h$ as
	\begin{align*}
		S_sh(x)=\e\qty[h(\sqrt{1-s}x+\sqrt{s}Z)] \quad\quad(x\in\R^d,\;s\in[0,1]),
	\end{align*}
	whenever the integral is defined. The next Lemma \ref{Slepian backward equation} shows that the Slepian interpolation of a $C^2$-test function having bounded first and second partial derivatives satisfies the ``backward equations'' (\eqref{Slepian backward equation equation} below). \eqref{Slepian backward equation equation} appears in \cite[Proposition 9.4]{Koike25}. (In the notation there, $s$ and $1-s$ are reversed.) This is implicit in the proof of \cite[Lemma 1, part 3]{Meckes09}, but  \cite[Lemma 1, part 3]{Meckes09} is stated with a $C^\infty$-test function. For $\Si$ being the identity matrix, it also appears with weaker smoothness assumptions on a test function in the literature, for example, in \cite[equation (2.6)]{G91}, and also in \cite{Raic2019arXiv, Raic19} with a different parameterization.
	\begin{lem}\label{Slepian backward equation}
		If $h:\R^d\rightarrow\R$ is a $C^2$-function having bounded first and second partial derivatives, then $S_s|h|(x)<\infty$ everywhere, $x\mapsto S_sh(x)$ is of class $C^2$, and $s\mapsto S_sh(x)$ is continuous on $[0,1]$ and continuously differentiable on $(0,1)$ with the derivative satisfying
		\begin{equation}\label{Slepian backward equation equation}
			\pdv{s}S_s h(x)=\frac{1}{2(1-s)}\qty(\sum_{j,k=1}^d\si_{jk}(\partial_{jk}S_s h)(x)-\sum_{j=1}^dx_j (\partial_jS_s h)(x)),\quad\quad s\in(0,1).
		\end{equation}
	\end{lem}
	\begin{proof}
		Note that $\|\grad{h}\|$ is finite by the assumption. We have
		\begin{align*}
			\qty|h(\sqrt{1-s}x+\sqrt{s}Z)|\le |h(0)|+\|\grad{h}\|(\sqrt{1-s}|x|+\sqrt{s}|Z|)\le|h(0)|+\|\grad{h}\|(|x|+|Z|)
		\end{align*}
		(cf.\ \cite[Remark 2.2]{Raic2019arXiv}).	Thus
		\begin{align*}
			\e	\qty[\qty|h(\sqrt{1-s}x+\sqrt{s}Z)|]	\le |h(0)|+\|\grad{h}\|(|x|+\sqrt{\e |Z|^2})=|h(0)|+\|\grad{h}\|(|x|+\sqrt{\tr\Si})<\infty.
		\end{align*}
		By taking a sequence $s_n\rightarrow s$ and dominated convergence,	one can check that $s\mapsto S_sh(x)$ is continuous on $[0,1]$.  Next, note that, for any $s\in(0,1)$ there are $0<s_1<s<s_2<1$ and
		\begin{equation*}\label{time derivative bound}
			\forall s'\in(s_1,s_2),\quad\;\qty|\qty(\frac{-x}{2\sqrt{1-s'}}+\frac{Z}{2\sqrt{s'}})^\top\grad{h}(\sqrt{1-s'}x+\sqrt{s'}Z)|\le\frac{\|\grad{h}\|}{2}\qty(\frac{|x|}{\sqrt{1-s_2}}+\frac{|Z|}{\sqrt{s_1}}),
		\end{equation*}
		and the dominating function is $\p$-integrable. Thus differentiating under the integral sign is justified and 
		\begin{equation}\label{time derivative}
			\begin{split}
				\forall s\in(0,1),\quad                 \pdv{s}S_s h(x)=& \frac{1}{2}\e\qty[\qty(\frac{-x}{\sqrt{1-s}}+\frac{Z}{\sqrt{s}})^\top\grad{h}(\sqrt{1-s}x+\sqrt{s}Z)]\\
				=&\frac{1}{2\sqrt{s}}\sum_{j=1}^{d}\e\qty[Z_j(\partial_j h)(\sqrt{1-s}x+\sqrt{s}Z) ]\\
				&\quad\quad-\frac{1}{2\sqrt{1-s}}\sum_{j=1}^{d}x_j\e\qty[(\partial_j h)(\sqrt{1-s}x+\sqrt{s}Z) ].
			\end{split}
		\end{equation}
		Under the assumption, differentiating with respect to $x$ under the integral sign is more easily justified and we have
		\begin{equation}\label{x derivative}
			\begin{split}
				(\partial_j S_s h)(x)&=\sqrt{1-s}\e\qty[(\partial_j h)(\sqrt{1-s}x+\sqrt{s}Z)],\\
				(\partial_{jk} S_s h)(x) &=(1-s)\e\qty[(\partial_{jk} h)(\sqrt{1-s}x+\sqrt{s}Z)].
			\end{split}
		\end{equation}
		One can see that $\qty(\partial_{j}S_sh)(x) $ and $\qty(\partial_{jk}S_sh)(x)  $ are continuous in $s$ on $[0,1]$ and continuous in $x$ on $\R^d$ by bounded convergence.	 Also, by Lemma \ref{multivariate Stein identity} we have
		\begin{equation}\label{stein identity applied}
			\e\qty[Z_j(\partial_j h)(\sqrt{1-s}x+\sqrt{s}Z)]=\sqrt{s}\sum_{k=1}^d\si_{jk}\e\qty[(\partial_{jk}h)(\sqrt{1-s}x+\sqrt{s}Z)].
		\end{equation}
		Combining \eqref{time derivative}, \eqref{x derivative} and \eqref{stein identity applied} gives \eqref{Slepian backward equation equation}.
	\end{proof}
	
	\subsection{Proof of the ``backward equations'' \eqref{h immigration death backward}}
	\label{sec: immigration death backward}
	
	Recall the formula \eqref{transition probability}
	\begin{align*}
		p_s(y,z)=\prod_{j=1}^{r}	\sum_{k=0}^{y_j\wedge z_j}\binom{y_j}{k}(1-s)^{k}s^{y_j-k}[\la_js]^{z_j-k}\frac{e^{-\la_js}}{(z_j-k)!}.
	\end{align*}
	For each $j$, define
	\begin{align*}
		p_s^{(j)}(y_j,z_j)\coloneqq\sum_{k=0}^{y_j\wedge z_j}\binom{y_j}{k}(1-s)^{k}s^{y_j-k}[\la_js]^{z_j-k}\frac{e^{-\la_js}}{(z_j-k)!},
	\end{align*}
	so that $p_s(y,z)=\prod_{j=1}^{r}p_s^{(j)}(y_j,z_j)$.
	\begin{lem}\label{immigration death backward equation}
		$s\mapsto p_s(y,z)$ is continuous on $[0,1]$ and continuously differentiable on $(0,1)$ with the derivative satisfying
		\begin{align*}
			(1-s)	\pdv{s}p_s(y,z)=\sum_{j=1}^{r}\qty[\la_j\Big(p_s(y+e_j,z)-p_s(y,z)\Big)+y_j\Big(p_s(y-e_j,z)-p_s(y,z)\Big) ],\quad s\in(0,1).
		\end{align*}
	\end{lem}
	\begin{proof}
		Note that
		\begin{align*}
			\pdv{s}p_s(y,z)=\sum_{j=1}^{r}\qty(\pdv{s}p_s^{(j)}(y_j,z_j)\prod_{k\neq j}p_s^{(k)}(y_k,z_k)).
		\end{align*}
		Let
		\begin{align*}
			P_s^{(j)}(y_j,z_j,k)=\binom{y_j}{k}(1-s)^{k}s^{y_j-k}[\la_js]^{z_j-k}\frac{e^{-\la_js}}{(z_j-k)!}
		\end{align*}
so that	$p_s^{(j)}(y_j,z_j)=\sum_{k=0}^{y_j\wedge z_j}P_s^{(j)}(y_j,z_j,k)$. A direct computation yields
		\begin{align}\label{q derivative}
			(1-s)\pdv{s}p_s^{(j)}(y_j,z_j)=\sum_{k=0}^{y_j\wedge z_j}P_s^{(j)}(y_j,z_j,k)\qty[-k-(1-s)\la_j+\frac{1-s}{s}(y_j+z_j-2k)].
		\end{align}
		Using $\binom{y_j+1}{k}=\binom{y_j}{k-1}+\binom{y_j}{k}$, one can see that 
		\begin{align*}
			P_s^{(j)}(y_j+1,z_j,k)=\la_j^{-1}\frac{1-s}{s}(z_j-k+1)P_s^{(j)}(y_j,z_j,k-1)+sP_s^{(j)}(y_j,z_j,k),
		\end{align*}
		so that
		\begin{align*}
			p_s^{(j)}(y_j+1,z_j)&=\sum_{k=0}^{(y_j+1)\wedge z_j}P_s^{(j)}(y_j+1,z_j,k)\\
			&=\sum_{k=1}^{(y_j+1)\wedge z_j}P_s^{(j)}(y_j,z_j,k-1)\la_j^{-1}\frac{1-s}{s}(z_j-k+1)+\sum_{k=0}^{y_j\wedge z_j}P_s^{(j)}(y_j,z_j,k)s\\
			&=\sum_{k=0}^{y_j\wedge z_j}P_s^{(j)}(y_j,z_j,k)\la_j^{-1}\frac{1-s}{s}(z_j-k)+\sum_{k=0}^{y_j\wedge z_j}P_s^{(j)}(y_j,z_j,k)s\\
			&=\sum_{k=0}^{y_j\wedge z_j}P_s^{(j)}(y_j,z_j,k)\qty[\la_j^{-1}\frac{1-s}{s}(z_j-k)+s].
		\end{align*}
		Note that it is innocuous to include the $k=z_j$ term in the first sum in the third line thanks to the factor $(z_j-k)$. Thus
		\begin{align}\label{q plus one}
			\la_j\Big(p_s^{(j)}(y_j+1,z_j)-p_s^{(j)}(y_j,z_j))\Big)=\sum_{k=0}^{y_j\wedge z_j}P_s^{(j)}(y_j,z_j,k)\qty[\frac{1-s}{s}(z_j-k)-\la_j (1-s)].
		\end{align}
		On the other hand, since $P_s^{(j)}(y_j-1,z_j,k)=y_j^{-1}s^{-1}(y_j-k)P_s^{(j)}(y_j,z_j,k)$, we have
		\begin{equation}\label{q minus one}
			\begin{split}
				y_j\Big(p_s^{(j)}(y_j-1,z_j)-p_s^{(j)}(y_j,z_j)\Big)&=\sum_{k=0}^{y_j\wedge z_j}P_s^{(j)}(y_j,z_j,k)\qty[\frac{1}{s}(y_j-k)-y_j]\\
				&=\sum_{k=0}^{y_j\wedge z_j}P_s^{(j)}(y_j,z_j,k)\qty[-k+\frac{1-s}{s}(y_j-k)].
			\end{split}
		\end{equation}
		Again the factor $(y_j-k)$ compensates the $k=y_j$ term. \eqref{q plus one} and \eqref{q minus one} sum to \eqref{q derivative}, namely,
		\begin{align*}
			(1-s)\pdv{s}p_s^{(j)}(y_j,z_j)=	\la_j\Big(p_s^{(j)}(y_j+1,z_j)-p_s^{(j)}(y_j,z_j))\Big) +		y_j\Big(p_s^{(j)}(y_j-1,z_j)-p_s^{(j)}(y_j,z_j)\Big).
		\end{align*}
		Therefore, noting that
		\begin{align*}
			p_s(y+e_j,z)&=p_s^{(j)}(y_j+1,z_j)\prod_{k\neq j}p_s^{(k)}(y_k,z_k),\\
			p_s(y,z)&=p_s^{(j)}(y_j,z_j)\prod_{k\neq j}p_s^{(k)}(y_k,z_k),\\
			p_s(y-e_j,z)&=p_s^{(j)}(y_j-1,z_j)\prod_{k\neq j}p_s^{(k)}(y_k,z_k),
		\end{align*}
		we have
		\begin{align*}
			&(1-s)\pdv{s}p_s(y,z)\\
			=&\sum_{j=1}^{r}\qty((1-s)\pdv{s}p_s^{(j)}(y_j,z_j)\prod_{k\neq j}p_s^{(k)}(y_k,z_k))\\
			=&\sum_{j=1}^{r}\qty[\la_j\Big(p_s(y+e_j,z)-p_s(y,z)\Big)+y_j\Big(p_s(y-e_j,z)-p_s(y,z)\Big) ].
		\end{align*}
	\end{proof}
	
	\subsection{Proofs for the smoothness of the solution}
	\label{sec: smoothness of the solution}

	\subsubsection{Proof of Lemma \ref{smoothness x}}
		Since $|h(\cdot,D_y(s)+Z_0(s))|_{\mathrm{Lip}}\le1$, for any $x,x'\in\R^d$,
		\begin{align*}
			\qty|f_h(x,y)-f_h(x',y)|&\le\int_{0}^{1}\frac{1}{2(1-s)}\e\qty[|h(\cdot,D_y(s)+Z_0(s))|_{\mathrm{Lip}}\sqrt{1-s}|x-x'|]\dd{s}\\
			&\le|x-x'|\int_{0}^{1}\frac{1}{2\sqrt{1-s}}\dd{s}=|x-x'|.
		\end{align*}
		This shows 	$|f_h(\cdot,y)|_{\mathrm{Lip}}\le1$. By differentiating with respect to $x$ under the integral sign, for any $x\in\R^d$, $k\in\{1,2,3\}$ and $j_1,\dots,j_k\in\{1,\dots,d\}$, we have
		\begin{gather*}
			(\partial_{j_1\cdots j_k} T_s h)(x,y) =(1-s)^{k/2}\e\qty[(\partial_{j_1\cdots j_k} h)\qty(\sqrt{1-s}x+\sqrt{s}Z,D_y(s)+Z_0(s))], \label{Ts x derivative kth} \\
			\qty|(\partial_{j_1\cdots j_k} T_s h)(x,y)|\le (1-s)^{k/2}. \label{Ts x derivative bound kth}
		\end{gather*} 
		$(\partial_{j_1\cdots j_k} T_s h)(x,y) $ is continuous in $x$ on $\R^d$ by bounded convergence. Since
		\begin{align*}
			\int_{0}^{1}\frac{1}{2(1-s)}\qty|(\partial_{j_1\cdots j_k}T_sh)(x,y)|\dd{s}\le\frac{1}{2}\int_{0}^{1}(1-s)^{k/2-1}\dd{s}=\frac{1}{k},
		\end{align*}
		we may differentiate under the integral sign
		\begin{align}
			\qty(\partial_{j_1\cdots j_k}f_h)(x,y)=&-\int_{0}^{1}\frac{1}{2(1-s)}\qty(\partial_{j_1\cdots j_k}T_sh)(x,y)\dd{s} \notag \\
			=&-\int_{0}^{1}\frac{(1-s)^{k/2-1}}{2}\e\qty[(\partial_{j_1\cdots j_k} h)\qty(\sqrt{1-s}x+\sqrt{s}Z,D_y(s)+Z_0(s))]\dd{s}.
			\label{solution kth derivative}
		\end{align}
		By taking a sequence $x_n\rightarrow x$ and dominated convergence, we see that $\qty(\partial_{j_1\cdots j_k}f_h)(\cdot,y):\R^d \rightarrow\R$ is continuous. $|f_h(\cdot,y)|_k\le\frac{1}{k}$ also follows.  The last part of the lemma is immediate from the proof. \qed

	\subsubsection{Proof of Lemma \ref{smoothness y}}
	We have shown \eqref{smoothness first order difference of first derivative} and \eqref{smoothness first order difference of first derivative improved} in the main text. We first prove \eqref{smoothness second order difference diagonal}, \eqref{smoothness second order difference off diagonal}, \eqref{smoothness second order difference diagonal improved} and \eqref{smoothness second order difference off diagonal improved} in a similar fashion. The proof is done by a suitable adaptation of the proof of \cite[Lemma 10.2.10]{BHJ92}. If $j\neq k$, for $s\in[0,1)$,
		\begin{align*}
			&D_{y+e_j}(s)=D_y(s)+e_j\indi_{\{\zeta_{j,y_j+1}>-\log(1-s) \}},\\
			&D_{y+e_k}(s)=D_y(s)+e_k\indi_{\{\zeta_{k,y_k+1}>-\log(1-s) \}},\\
			&D_{y+e_k+e_j}(s)=D_y(s)+e_k\indi_{\{\zeta_{k,y_k+1}>-\log(1-s) \}}+e_j\indi_{\{\zeta_{j,y_j+1}>-\log(1-s) \}}.
		\end{align*}
		So,	if $j\neq k$, for $s\in[0,1)$,
		\begin{align*}
			&\left.h\qty(\sqrt{1-s}x+\sqrt{s}Z,D_{y+e_k+e_j}(s)+Z_0(s))-h\qty(\sqrt{1-s}x+\sqrt{s}Z,D_{y+e_j}(s)+Z_0(s)) \right.\\
			&\quad\quad	\quad	\left.   -h\qty(\sqrt{1-s}x+\sqrt{s}Z,D_{y+e_k}(s)+Z_0(s))+h\qty(\sqrt{1-s}x+\sqrt{s}Z,D_{y}(s)+Z_0(s))   \right.\\
			=&\indi_{\{\zeta_{j,y_j+1}>-\log(1-s) \}}\indi_{\{\zeta_{k,y_k+1}>-\log(1-s)\}}\\
			&\times\bigg[ h\qty(\sqrt{1-s}x+\sqrt{s}Z,D_{y}(s)+Z_0(s)+e_k+e_j)-h\qty(\sqrt{1-s}x+\sqrt{s}Z,D_{y}(s)+Z_0(s)+e_j) \bigg.\\
			&\quad\quad\quad\quad\quad	\bigg.   -h\qty(\sqrt{1-s}x+\sqrt{s}Z,D_{y}(s)+Z_0(s)+e_k)+h\qty(\sqrt{1-s}x+\sqrt{s}Z,D_{y}(s)+Z_0(s))   \bigg]\\
			=&\indi_{\{\zeta_{j,y_j+1}>-\log(1-s) \}}\indi_{\{\zeta_{k,y_k+1}>-\log(1-s)\}}\De_{jk}h\qty(\sqrt{1-s}x+\sqrt{s}Z,D_y(s)+Z_0(s)),
		\end{align*}
		and hence
		\begin{align*}
			&(\De_{jk}T_sh)(x,y) \\
			=&\e\bigg[h\qty(\sqrt{1-s}x+\sqrt{s}Z,D_{y+e_k+e_j}(s)+Z_0(s))-h\qty(\sqrt{1-s}x+\sqrt{s}Z,D_{y+e_j}(s)+Z_0(s)) \bigg.\\
			&\quad\quad\quad	\quad	\bigg.   -h\qty(\sqrt{1-s}x+\sqrt{s}Z,D_{y+e_k}(s)+Z_0(s))+h\qty(\sqrt{1-s}x+\sqrt{s}Z,D_{y}(s)+Z_0(s))   \bigg]\\
			=&\p\qty(\zeta_{j,y_j+1}>-\log(1-s) )\p\qty(\zeta_{k,y_k+1}>-\log(1-s))\e\qty[\De_{jk}h\qty(\sqrt{1-s}x+\sqrt{s}Z,D_y(s)+Z_0(s))]\\
			=&(1-s)^2\e\qty[\De_{jk}h\qty(\sqrt{1-s}x+\sqrt{s}Z,D_y(s)+Z_0(s))].
		\end{align*}
		If $j=k$, note that, for $s\in[0,1)$,
		\begin{align*}
			&D_{y+e_j}(s)=D_y(s)+e_j\indi_{\{\zeta_{j,y_j+1}>-\log(1-s) \}}\stackrel{d}{=}D_y(s)+e_j\indi_{\{\zeta_{j,y_j+2}>-\log(1-s) \}},\\
			&D_{y+2e_j}(s)=D_y(s)+e_j\qty(\indi_{\{\zeta_{j,y_j+2}>-\log(1-s) \}}+\indi_{\{\zeta_{j,y_j+1}>-\log(1-s) \}}),
		\end{align*}
		and
		\begin{align*}
			&\e\bigg[h\qty(\sqrt{1-s}x+\sqrt{s}Z,D_{y+2e_j}(s)+Z_0(s))-2h\qty(\sqrt{1-s}x+\sqrt{s}Z,D_{y+e_j}(s)+Z_0(s)) \bigg.\\
			&\quad\quad\quad\quad\quad\quad\quad\quad\quad	\bigg.   +h\qty(\sqrt{1-s}x+\sqrt{s}Z,D_{y}(s)+Z_0(s))   \bigg]\\
			=&\e\bigg[h\qty(\sqrt{1-s}x+\sqrt{s}Z,D_{y+2e_j}(s)+Z_0(s))-h\qty(\sqrt{1-s}x+\sqrt{s}Z,D_{y+e_j}(s)+Z_0(s)) \bigg.\\
			&	\bigg.  -h\qty(\sqrt{1-s}x+\sqrt{s}Z,D_y(s)+e_j\indi_{\{\zeta_{j,y_j+2}>-\log(1-s) \}}+Z_0(s)) +h\qty(\sqrt{1-s}x+\sqrt{s}Z,D_{y}(s)+Z_0(s))   \bigg].
		\end{align*}
		Hence
		\begin{align*}
			(\De_{jj}T_sh)(x,y)=(1-s)^2\e\qty[\De_{jj}h\qty(\sqrt{1-s}x+\sqrt{s}Z,D_y(s)+Z_0(s))]
		\end{align*}
		for $s\in[0,1)$ can be shown similarly. For any pair $\{j,k\}$, the identity
		\begin{align*}
			(\De_{jk}T_sh)(x,y)=(1-s)^2\e\qty[\De_{jk}h\qty(\sqrt{1-s}x+\sqrt{s}Z,D_y(s)+Z_0(s))]
		\end{align*}
		extends to hold	for $s=1$. It thus follows that, for any pair $\{j,k\}$, 
		\begin{align*}
			\De_{jk}f_h(x,y)=&-\int_{0}^{1}\frac{1}{2(1-s)} (\De_{jk}T_sh)(x,y)\dd{s}\\
			=&-\int_{0}^{1}\frac{1-s}{2}\e\qty[\De_{jk}h\qty(\sqrt{1-s}x+\sqrt{s}Z,D_y(s)+Z_0(s))]\dd{s}.
		\end{align*}
		First consider the $j\neq k$ case. Let $Z_{0,-\{j,k\}}(s)\coloneqq Z_0(s)-Z_{0,j}(s)e_j-Z_{0,k}(s)e_k$. Defining $\bar{h}:\R^d\times\Z^r_+\times\Z_+\times\Z_+\rightarrow[-1,1]$ by
		\begin{align*}
			\bar{h}(x,y,m,n)=h(x,y+me_j+ne_k),
		\end{align*}
		we have 
		\begin{align*}
			&\De_{jk}f_h(x,y)\\
			=&-\int_{0}^{1}\frac{1-s}{2}\e\bigg[\bar{h}\qty(\sqrt{1-s}x+\sqrt{s}Z,D_y(s)+Z_{0,-\{j,k\}}(s),Z_{0,j}(s)+1,Z_{0,k}(s)+1)\bigg. \\
			&	\quad\quad\quad\quad\quad\quad\quad\left.	-\bar{h}\qty(\sqrt{1-s}x+\sqrt{s}Z,D_y(s)+Z_{0,-\{j,k\}}(s),Z_{0,j}(s)+1,Z_{0,k}(s))\right.\\
			&\quad\quad\quad\quad\quad\quad\quad\quad\left.-\bar{h}\qty(\sqrt{1-s}x+\sqrt{s}Z,D_y(s)+Z_{0,-\{j,k\}}(s),Z_{0,j}(s),Z_{0,k}(s)+1) \right.  \\
			&\quad\quad\quad\quad\quad\quad\quad\quad\quad\quad	\bigg.	+\bar{h}\qty(\sqrt{1-s}x+\sqrt{s}Z,D_y(s)+Z_{0,-\{j,k\}}(s),Z_{0,j}(s),Z_{0,k}(s)) \bigg]\dd{s}.
		\end{align*}
		For $s\in(0,1]$, given that $(Z,D_y(s)+Z_{0,-\{j,k\}}(s))=(z,y')$, the expectation is (with the convention that $\Po\qty(s\la_j)\{-1\}=\Po\qty(s\la_k)\{-1\}=0$)
		\begin{align}\label{second difference bound j neq k conditional expectation}
			\sum_{m,n\ge0}	\bar{h}(\sqrt{1-s}x+\sqrt{s}z,y',m,n)\Big[\Po\qty(s\la_j)\{m\}-\Po\qty(s\la_j)\{m-1\}\Big]\Big[\Po\qty(s\la_k)\{n\}-\Po\qty(s\la_k)\{n-1\}\Big],
		\end{align}
		since $Z_{0,j}(s)\sim\Po(s\la_j)$ and $Z_{0,k}(s)\sim\Po(s\la_k)$ for $s\in(0,1]$, $(Z_{0,j}(s),Z_{0,k}(s))$ are independent of $(Z,D_y(s)+Z_{0,-\{j,k\}}(s))$, and $Z_{0,j}(s)$ and $Z_{0,k}(s)$ are independent. Hence it is bounded in magnitude by
		\begin{align*}
			4\max_{m\ge0}\Po\qty(s\la_j)\{m\}\max_{n\ge0}\Po\qty(s\la_k)\{n\}	\le4\min\{1,(2es\la_j)^{-1/2}\}\min\{1,(2es\la_k)^{-1/2}\},
		\end{align*}
		where where the first bound is due to Lemma \ref{Poisson smoothness} below and the last inequality is \cite[Proposition A.2.7]{BHJ92}. Thus 
		\begin{align*}
			\qty|\De_{jk}f_h(x,y)|\le\int_{0}^{1}2(1-s)\min\{1,(2es\la_j)^{-1/2}\}\min\{1,(2es\la_k)^{-1/2}\}\dd{s}
		\end{align*}
		and \eqref{smoothness second order difference off diagonal} follows by elementary computations. To see
		\begin{align*}
			\qty|\De_{jk}f_h(x,y)|\le\frac{1+\log^+(2e(\la_j\wedge\la_k))}{e(\la_j\wedge\la_k)}
		\end{align*}
		for any $j\neq k$,	note that
		\begin{align*}
			&\int_{0}^{1}2(1-s)\min\{1,(2es\la_j)^{-1/2}\}\min\{1,(2es\la_k)^{-1/2}\}\dd{s}\\
			\le&\int_{0}^{\frac{1}{2e(\la_j\wedge\la_k)}\wedge1}2(1-s)\dd{s}+\int_{\frac{1}{2e(\la_j\wedge\la_k)}\wedge1}^1\frac{1-s}{s}\frac{1}{e\sqrt{\la_j\la_k}}\dd{s}\\
			\le&2\qty(\frac{1}{2e(\la_j\wedge\la_k)}\wedge1)+\frac{\log(2e(\la_j\wedge\la_k))}{e(\la_j\wedge\la_k)}\indi_{\frac{1}{2e(\la_j\wedge\la_k)}<1}\\
			=&2\times\indi_{\frac{1}{2e(\la_j\wedge\la_k)}\ge1}+\frac{1+\log(2e(\la_j\wedge\la_k))}{e(\la_j\wedge\la_k)}\indi_{\frac{1}{2e(\la_j\wedge\la_k)}<1}\\
			\le&\frac{1+\log^+(2e(\la_j\wedge\la_k))}{e(\la_j\wedge\la_k)}.
		\end{align*}
		If $h$ is such that $h(x',y')=0$ for $y'\neq0$, $\bar{h}(x',y',m,n)=0$ unless $y'=m=n=0$ and \eqref{second difference bound j neq k conditional expectation} is bounded in magnitude by $\Po\qty(s\la_j)\{0\}\Po\qty(s\la_k)\{0\}=e^{-(\la_j+\la_k)s}$. This gives the improved bound in \eqref{smoothness second order difference off diagonal improved}:
		\begin{align*}
			\qty|\De_{jk}f_h(x,y)|\le\int_{0}^{1}\frac{1-s}{2}e^{-(\la_j+\la_k)s}\dd{s}\le\min\qty{\frac{1}{4},\frac{1-e^{-(\la_j+\la_k)}}{2(\la_j+\la_k)}}.
		\end{align*}
		
		\smallskip
		For the $j=k$ case, let $Z_{0,-j}(s)\coloneqq Z_0(s)-Z_{0,j}(s)e_j$. We may assume that, on the same probability space $(\Omega,\mF,\p)$, the following additional independent random elements are defined, independently of everything else:
		{\setlength{\leftmargini}{29.5pt}  	
			\begin{itemize}
				\setlength{\labelsep}{7pt}    
				\setlength{\itemsep}{3pt}     	
				
				\item Independent stochastic processes $\{Y_1(s),s\in[0,1]\}$ and $\{Y_2(s),s\in[0,1]\}$ on $Z_+$ such that $Y_1(0)=Y_2(0)=0$ and $Y_1(s),Y_2(s)\sim\Po(s\la_j/2)$ for $s\in(0,1]$.
				
		\end{itemize} }
		Then $Z_0(s)\stackrel{d}{=}(Y_1(s)+Y_2(s))e_j+Z_{0,-j}(s)$ for all $s\in[0,1]$. By replacing $Z_{0,-\{j,k\}}(s),Z_{0,j}(s),Z_{0,k}(s)$ with $Z_{0,-j}(s),Y_1(s),Y_2(s)$, respectively, the same proof as above with $k=j$ shows
		\begin{align*}
			\qty|\De_{jj}f_h(x,y)|\le\int_{0}^{1}2(1-s)\qty(\min\{1,(es\la_j)^{-1/2}\})^2\dd{s},
		\end{align*}
		and	\eqref{smoothness second order difference diagonal} follows by elementary computations. To see
		\begin{align*}
			\qty|\De_{jj}f_h(x,y)|\le\frac{2+2\log^+(e\la_j)}{e\la_j},
		\end{align*}
		note that
		\begin{align*}
			\int_{0}^{1}2(1-s)\qty(\min\{1,(es\la_j)^{-1/2}\})^2\dd{s}=&\int_{0}^{\frac{1}{e\la_j}\wedge1}2(1-s)\dd{s}+\int_{\frac{1}{e\la_j}\wedge1}^{1}\frac{2}{e\la_j}\frac{1-s}{s}\dd{s} \\
			\le&2\qty[\qty(\frac{1}{e\la_j}\wedge1)+\frac{\log(e\la_j)}{e\la_j}\indi_{\frac{1}{e\la_j}<1}]\\
			=&2\qty[\indi_{\frac{1}{e\la_j}\ge1}+\frac{1+\log(e\la_j)}{e\la_j}\indi_{\frac{1}{e\la_j}<1}]\\
			\le&2\times\frac{1+\log^+(e\la_j)}{e\la_j}.
		\end{align*}
		If $h$ is such that $h(x',y')=0$ for $y'\neq0$, we obtain the improved bound in \eqref{smoothness second order difference diagonal improved}:
		\begin{align*}
			\qty|\De_{jj}f_h(x,y)|\le\int_{0}^{1}\frac{1-s}{2}e^{-\la_js}\dd{s}\le\min\qty{\frac{1}{4},\frac{1-e^{-\la_j}}{2\la_j}}.
		\end{align*}
		
		\medskip
		The proofs of \eqref{smoothness first order difference}, \eqref{smoothness first order difference of second derivative}, \eqref{smoothness first order difference improved} and \eqref{smoothness first order difference of second derivative improved} can be done similarly. For \eqref{smoothness first order difference of second derivative} and \eqref{smoothness first order difference of second derivative improved}, note that
		\begin{align*}
			(\partial_{jk} f_h)(x,y)&=-\int_{0}^{1}\frac{1}{2}\e\qty[(\partial_{jk}h)\qty(\sqrt{1-s}x+\sqrt{s}Z,D_y(s)+Z_0(s))]\dd{s}
		\end{align*}
		(cf.\ \eqref{solution kth derivative}) and $\|\partial_{jk}h\|\le 1$ holds. We obtain
		\begin{align*}
			\qty|\De_jf_h(x,y)|&\le \int_{0}^{1}\min\{ 1, (2es\la_j)^{-1/2} \}\dd{s}\le\min\qty{1,\sqrt{\frac{2}{e\la_j}}},   \\
			\qty|\De_\ell(\partial_{jk}f_h)(x,y)|&\le\int_{0}^{1}(1-s)\min\{ 1, (2es\la_\ell)^{-1/2} \}\dd{s}\le\min\qty{\frac{1}{2},\frac{4}{3\sqrt{2}}\sqrt{\frac{1}{e\la_\ell}}}.
		\end{align*}
		If furthermore $h$ is such that $h(x',y')=0$ for $y'\neq0$, $\partial_{jk}h(x',y')=0$ for $y'\neq0$. These bounds improve as
		\begin{align*}
			\qty|\De_jf_h(x,y)|&\le \int_{0}^{1}\frac{1}{2}e^{-\la_js}\dd{s}\le\min\qty{\frac{1}{2},\frac{1-e^{-\la_j}}{2\la_j}},   \\
			\qty|\De_\ell(\partial_{jk}f_h)(x,y)|&\le\int_{0}^{1}\frac{1-s}{2}e^{-\la_\ell s}\dd{s}\le\min\qty{\frac{1}{4},\frac{1-e^{-\la_\ell}}{2\la_\ell}}.
		\end{align*}

\medskip
Finally, the last part of the lemma is immediate from the definition \eqref{eq new solution} of $f_h$ (and the expression \eqref{new interpolation} of $T_sh$) and the representations \eqref{solution kth derivative} of the derivatives, once we note that under the assumption on $h$, $\partial_jh$ and $\partial_{jk}h$ also have the same property. \qed
	
	\begin{lem}[Smoothness of a Poisson distribution]\label{Poisson smoothness}
		Let $\Po(\la)$ be a one-dimensional Poisson measure on $\Z_+$ with mean $\la>0$ and $ \flr{\la}$ be the integer part of $\la$. Then, with the convention that $\Po(\la)\{-1\}=0$, 
		\begin{align*}
			\sum_{k=0}^{\infty}\qty|\Po(\la)\{k\}-\Po(\la)\{k-1\}|=2\Po(\la)\{\flr{\la} \}.
		\end{align*}
	\end{lem}
	\begin{proof}
		 Note that
		\begin{align*}
			\Po(\la)\{k\}-\Po(\la)\{k-1\}=\frac{\la^ke^{-\la}}{k!}-\frac{\la^{k-1}e^{-\la}}{(k-1)!} =\frac{\la-k}{\la}\frac{\la^{k}e^{-\la}}{k!},
		\end{align*}
		which is nonnegative if $k\le\flr{\la}$ and negative if $k>\flr{\la}$. Therefore, for any $m>\flr{\la}$,
		\begin{align*}
			&\sum_{k=0}^{m}\qty|\Po(\la)\{k\}-\Po(\la)\{k-1\}|\\
			=&\sum_{k=0}^{\flr{\la}}\qty(\Po(\la)\{k\}-\Po(\la)\{k-1\})+\sum_{k=\flr{\la}+1}^{m}\qty(\Po(\la)\{k-1\}-\Po(\la)\{k\})\\
			=&\Po(\la)\{\flr{\la}\}+\Po(\la)\{\flr{\la}\}-\Po(\la)\{m\}\\
			=&2\Po(\la)\{\flr{\la}\}-\Po(\la)\{m\}.
		\end{align*}
		Since $\Po(\la)\{m\}\rightarrow0$ as $m\rightarrow\infty$, letting $m\rightarrow\infty$ in both sides completes the proof.
	\end{proof}
	
	\section{Proofs for Section \ref{configuration model}}
	\label{main proofs}
	
		Let $\al$ and $\be$ be either a member of $\Ga$ or a possible edge. We say that $\al$ \textbf{intersects} $\be$ (or vice versa) and write $\al\cap \be\neq\varnothing$ if $\al$ and $\be$ share a vertex. We say that $\al$ and $\be$ are \textbf{disjoint} and write $\al\cap\be=\varnothing$ if they share no vertex. Note that if $\al$ and $\be$ share a half-edge, then $\al$ and $\be$ also share the vertex to which that half-edge is incident. If either $\al$ or $\be$, say $\al$, is in $\Ga_{1}$ (i.e., an isolated one), then the converse of this is also true: If $\al$ and $\be$ share a vertex, then $\al$ shares all the half-edges of $\be$ which are incident to that vertex.
		
		\medskip
	 In the following lemma, we collect the cardinality estimates about various sets of $\be\in\Ga$ intersecting given $\al\in\Ga_1$ which are useful in the subsequent proofs (other than Appendix \ref{stein coupling bound fifth term}; the proof of Proposition \ref{fifth term in stein coupling bound} is mainly done by exactly counting the number of possibilities of a pair of $\al,\be\in\Ga_2$ sharing a or no half-edge and we will do it there). The proof of this lemma is deferred to Appendix \ref{proof cardinality estimates}. Recall that $\Ga_{11}$, $\Ga_{12}$, $\Ga_{21}$ and $\Ga_{22}$ are the sets of possible isolated edges, possible isolated 2-stars, possible self-loops and possible double edges, respectively.
			\begin{lem}[Cardinality estimates]\label{cardinality estimates}\ 
						{\setlength{\leftmargini}{29.5pt}  	
					\begin{itemize}
						\setlength{\labelsep}{7pt}     
						\setlength{\itemsep}{3pt}      
						
						\item[\emph{a}.] \begin{align*}
							\qty|\Ga_{11}|=\frac{(n_1)_2}{2};\quad\qty|\Ga_{12}|=(n_1)_2n_2;\quad\qty|\Ga_{21}|=\frac{\sum_{i\in[n]}(d_i)_2}{2};\quad\qty|\Ga_{22}|=\frac{\sum_{1\le i<j\le n}(d_i)_2(d_j)_2}{2}.
						\end{align*}
						
						\item[\emph{b}.] For $\al\in\Ga_{1j}$,						
						\begin{align*}
							\Big|\qty{\be\in\Ga_{1k}\setminus\{\al\}:\be\cap\al\neq\varnothing}\Big|\le\begin{cases}
								2n_1&(j=k=1)\\
								4n_1n_2&(j=1,k=2)\\
								2n_1&(j=2,k=1)\\
								4n_1n_2+(n_1)_2&(j=k=2)
							\end{cases}.
						\end{align*}
						
						\item[\emph{c}.] For $\al\in\Ga_{11}$,
						\begin{align*}
							&\Big|\qty{\be\in\Ga_{12}\setminus\{\al\}:\text{\emph{$\al$ and $\be$ share two degree 1 vertices}}}\Big|\le2n_2.
						\end{align*}
						
						\item[\emph{d}.] For $\al\in\Ga_{12}$, let
						\begin{align*}
							\La^\al=\qty{\be\in\Ga_{12}\setminus\{\al\}:\text{\emph{$\al$ and $\be$ share one degree 1 vertex only}}}.
						\end{align*}
						Then $\qty|\La^\al|\le4n_1n_2$ and
							$\qty|\qty{\be\in\Ga_{12}\setminus\{\al\}:\be\cap\al\neq\varnothing}\setminus\La^\al|\le(n_1)_2+2n_2$.
						
						\item[\emph{e}.] For $\al\in\Ga_{11}$, $\qty|\qty{\be\in\Ga_{2}:\be\cap\al\neq\varnothing}|=0$. For $\al\in\Ga_{12}$,
						\begin{align*}
							\Big|\qty{\be\in\Ga_{21}:\be\cap\al\neq\varnothing}\Big|=1,\quad\quad
							\Big|\qty{\be\in\Ga_{22}:\be\cap\al\neq\varnothing}\Big|\le\sum_{i\in[n]}(d_i)_2.
						\end{align*}
						
				\end{itemize} }
		\end{lem}

		\subsection{Proof of Proposition \ref{first term in stein coupling bound}}
		\label{stein coupling bound first term}
		
		Note that
		\begin{align*}
			\Var\qty(\e\qty[\sum_{\al\in\Ga_{1j}}I_\al\sum_{\be\in\Ga_{1k}}(I_{\be}-J_{\be\al})\Big|W_1,W_2])\le\Var\qty(\sum_{\al\in\Ga_{1j}}I_\al\sum_{\be\in\Ga_{1k}}(I_{\be}-J_{\be\al})).
		\end{align*} 
		We estimate
		\begin{equation}\label{first term unconditional variance}
			\Var\qty(\sum_{\al\in\Ga_{1j}}I_\al\sum_{\be\in\Ga_{1k}}(I_{\be}-J_{\be\al}))
		\end{equation}
		as follows. We will rely on the following simple observations. The proofs are deferred to Appendix \ref{proofs}.
		\begin{lem}\label{intersecting different isolated trees}
			Suppose that $\al$ and $\be$ are either a possible isolated edge or a possible isolated 2-star, i.e., $\al,\be\in\Ga_1$. If $\al\cap\be\neq\varnothing$ and $\al\neq\be$, then there exists a half-edge $s$ shared by $\al$ and $\be$ which is paired to different half-edges in $\al$ and in $\be$, respectively. In particular, if $\al\cap\be\neq\varnothing$ and $\al\neq\be$, then $I_\al I_\be=0$.
		\end{lem}
		\begin{lem}\label{disjoint multigraphs non creation}
			Suppose that $\al,\be\in\Ga$ and let $H$ be the type of $\al$. If $\al$ and $\be$ share no half-edge, then for every $g\in\mG_\al\setminus\mG_\be$ and $b=\qty(b_\ell)_{\ell=1}^{e(H)}\in\prod_{\ell=1}^{e(H)}\qty[N-2(e(H)-\ell)-1]$, $f_\al(g,b)\not\supset\be$. In particular, if $\al$ and $\be$ share no half-edge and $I_\be=0$, then $J_{\be\al}=0$.
		\end{lem}

		\medskip
		Let $H_j$ and $H_k$ denote the types of $\Ga_{1j}$ and $\Ga_{1k}$, respectively. First, using the Cauchy--Schwarz inequality for $\R^2$ or $\R^3$ inside the expectation yields
		\begin{align}
			&\Var\qty(\sum_{\al\in\Ga_{1j}}I_\al\sum_{\be\in\Ga_{1k}}(I_{\be}-J_{\be\al}))\notag \\
			\le&\qty(2+\indi_{j=k})\left[\indi_{j=k}\Var\qty(\sum_{\al\in\Ga_{1j}}I_\al(1-J_{\al\al}))+\Var\qty(\sum_{\al\in\Ga_{1j}}\sum_{\substack{\be\in\Ga_{1k} \\ \be\cap\al=\varnothing }}I_\al I_\be(1-J_{\be\al})) \right. \notag  \\
			&\qquad\qquad\qquad\qquad\qquad\left.+\Var\qty(\sum_{\al\in\Ga_{1j}}\sum_{\substack{\be\in\Ga_{1k}\setminus\{\al\}\\\be\cap\al\neq\varnothing}}I_\al J_{\be\al})\right]. 
			\label{variance decomposition}
		\end{align}
		
		\subsubsection{The first variance in \eqref{variance decomposition}} Write
		\begin{equation}\label{first variance}
			\Var\qty(\sum_{\al\in\Ga_{1j}}I_\al(1-J_{\al\al}))=\sum_{\al,\al'\in\Ga_{1j}}\e\qty[I_\al I_{\al'}(1-J_{\al\al})(1-J_{\al'\al'})]-\qty(\sum_{\al\in\Ga_{1j}}p_\al(1-p_\al))^2.
		\end{equation}
		The following observation, which is almost obvious from the definition of $f_\al$, is useful.
		\begin{lem}\label{destroyed al al}
			Suppose that $\al\in\Ga$ and let $H$ be the type of $\al$.  Let $s_{\ell,1}s_{\ell,2}$, $1\le\ell\le e(H)$ denote the $e(H)$ edges of $\al$ arranged in ascending order such that $s_{1,1}\wedge s_{1,2}<\cdots<s_{e(H),1}\wedge s_{e(H),2}$. For each $1\le\ell\le e(H)$, let $r_{\al,\ell}\in[N-2(e(H)-\ell)-1]$ denote the rank (in ascending order) of $s_{\ell,1}\vee s_{\ell,2}$ among
			\begin{align*}
				\{1,\dots,N \}\setminus\qty(\{s_{\ell,1}\wedge s_{\ell,2} \}\uplus\biguplus_{m=\ell+1}^{e(H)}\{s_{m,1},s_{m,2} \}).
			\end{align*}
			If $g\in\mG_\al$, then $f_{\al}\qty(g,r_{\al,1},\dots,r_{\al,e(H)})=g\supset\al$.
		\end{lem}
		For $\al,\al'\in\Ga_{1j}$ with $\al\cap\al'=\varnothing$, let $\qty(r_{\al,\ell})_{\ell=1}^{e(H_j)},\qty(r_{\al',\ell})_{\ell=1}^{e(H_j)}\in\prod_{\ell=1}^{e(H_j)}[N-2(e(H_j)-\ell)-1]$ be as in Lemma \ref{destroyed al al}. By the independence, we have
		\begin{align*}
			&\p\qty(I_\al=I_{\al'}=1,J_{\al\al}=J_{\al'\al'}=0)\\
			\le&\p\qty(I_\al=I_{\al'}=1,B_\al\neq (r_{\al,\ell})_{\ell=1}^{e(H_j)},B_{\al'}\neq \qty(r_{\al',\ell})_{\ell=1}^{e(H_j)})\\
			=&\p\qty(I_\al=I_{\al'}=1)\qty(1-\p\qty(B_\al= (r_{\al,\ell})_{\ell=1}^{e(H_j)}))\qty(1-\p\qty(B_{\al'}= (r_{\al',\ell})_{\ell=1}^{e(H_j)}))\\
			=&\frac{1}{((N-1))_{2e(H_j)}}\qty(1-p_\al)^2.
		\end{align*}
		By Lemma \ref{intersecting different isolated trees} and decomposing the sum into two parts, we upper bound \eqref{first variance} as
		\begin{align*}
			\Var\qty(\sum_{\al\in\Ga_{1j}}I_\al(1-J_{\al\al}))\le&\sum_{\al\in\Ga_{1j}}\Big(\e\qty[I_\al(1-J_{\al\al})]-p_\al^2(1-p_\al)^2\Big)\\
			&\quad+\sum_{\substack{\al,\al'\in\Ga_{1j}\\\al\cap\al'=\varnothing}}\Big(\e\qty[I_\al I_{\al'}(1-J_{\al\al})(1-J_{\al'\al'})]-p_\al p_{\al'}(1-p_\al)(1-p_{\al'})\Big)\\
			=&\sum_{\al\in\Ga_{1j}}p_\al(1-p_\al)\Big(1-p_\al(1-p_\al)\Big)\\
			&\quad+\sum_{\substack{\al,\al'\in\Ga_{1j}\\\al\cap\al'=\varnothing}}\Big(\p\qty(I_\al=I_{\al'}=1,J_{\al\al}=J_{\al'\al'}=0)-p_\al^2(1-p_\al)^2\Big)\\
			\le&\sum_{\al\in\Ga_{1j}}\frac{1}{((N-1))_{e(H_j)}}\qty(1-p_\al)\Big(1-p_\al(1-p_\al)\Big)\\
			&\qquad+\sum_{\substack{\al,\al'\in\Ga_{1j}\\\al\cap\al'=\varnothing}}\qty(\frac{1}{((N-1))_{2e(H_j)}}-p_\al^2)\qty(1-p_\al)^2\\
			=&\sum_{\al\in\Ga_{1j}}\frac{1}{((N-1))_{e(H_j)}}\qty(1-p_\al)\Big(1-p_\al(1-p_\al)\Big)\\
			&\qquad+\sum_{\substack{\al,\al'\in\Ga_{1j}\\\al\cap\al'=\varnothing}}\frac{1}{((N-1))_{2e(H_j)}}\qty(1-\frac{((N-1))_{2e(H_j)}}{[((N-1))_{e(H_j)}]^2})\qty(1-p_\al)^2.
		\end{align*}
		By Bernoulli's inequality,\footnote{$(1+x_1)\cdots(1+x_r)\ge 1+x_1+\dots+x_r$ for all real numbers $x_1,\dots,x_r>-1$ with the same sign.} we have
		\begin{align*}
			1-\frac{((N-1))_{2e(H_j)}}{[((N-1))_{e(H_j)}]^2}=&1-\frac{((N-2e(H_j)-1))_{e(H_j)}}{((N-1))_{e(H_j)}}\\
			\le& \sum_{\ell=1}^{e(H_j)}\qty(1-\frac{N-2(e(H_j)+\ell-1)-1}{N-2(\ell-1)-1})\\
			=&\sum_{\ell=1}^{e(H_j)}\frac{2e(H_j)}{N-2(\ell-1)-1}.
		\end{align*}
		Thus, we upper bound the first variance in \eqref{variance decomposition} as
		\begin{align}
			\Var\qty(\sum_{\al\in\Ga_{1j}}I_\al(1-J_{\al\al}))\le&\sum_{\al\in\Ga_{1j}}\frac{1}{((N-1))_{e(H_j)}}+\sum_{\substack{\al,\al'\in\Ga_{1j}\\\al\cap\al'=\varnothing}}\frac{1}{((N-1))_{2e(H_j)}}\sum_{\ell=1}^{e(H_j)}\frac{2e(H_j)}{N-2(\ell-1)-1}\notag\\
			\le&\begin{cases}
				\frac{(n_1)_2}{2(N-1)}+\frac{\qty((n_1)_2)^2}{2((N-1))_2(N-1)}&(j=1)\\
				\frac{(n_1)_2n_2}{((N-1))_2}+\frac{8\qty((n_1)_2n_2)^2}{((N-1))_4(N-3)}&(j=2)
			\end{cases}.\label{first variance bound}
		\end{align}

		\medskip
		\subsubsection{The second variance in \eqref{variance decomposition}} 
				\label{subsec second variance}
		
		Write
		\begin{equation}\label{second variance}
			\begin{split}
				&\Var\qty(\sum_{\al\in\Ga_{1j}}\sum_{\substack{\be\in\Ga_{1k}\\ \be\cap\al=\varnothing }}I_\al I_\be(1-J_{\be\al}))\\
				=&\sum_{\al\in\Ga_{1j}}\sum_{\substack{\be\in\Ga_{1k}\\ \be\cap\al=\varnothing }}\sum_{\al'\in\Ga_{1j}}\sum_{\substack{\be'\in\Ga_{1k}\\ \be'\cap\al'=\varnothing }}\e\qty[I_\al I_\be I_{\al'} I_{\be'}(1-J_{\be\al})(1-J_{\be'\al'})]\\
				&\quad\quad\quad\quad\quad\quad\quad\quad\quad\quad\quad\quad\quad\quad\quad\quad-\qty(\sum_{\al\in\Ga_{1j}}\sum_{\substack{\be\in\Ga_{1k}\\ \be\cap\al=\varnothing }}\e\qty[I_\al I_\be(1-J_{\be\al})])^2.
			\end{split}
		\end{equation}
		The following observation is useful. The formal proof will be given in Appendix \ref{proofs}.
		\begin{lem}\label{destroyed al be disjoint}
			Suppose that $\al,\be\in\Ga$ and that $\al$ and $\be$ share no half-edge. Let $H$ and $K$ be the types of $\al$ and $\be$, respectively. Let $s_{\ell,1}s_{\ell,2}$, $1\le\ell\le e(H)$ denote the $e(H)$ edges of $\al$ arranged in ascending order such that $s_{1,1}\wedge s_{1,2}<\cdots<s_{e(H),1}\wedge s_{e(H),2}$. For each $1\le\ell\le e(H)$, let $R_{\al,\be,\ell}\subset[N-2(e(H)-\ell)-1]$ with $|R_{\al,\be,\ell}|=2e(K)$ denote the set of the ranks (in ascending order) of the $2e(K)$ half-edges of $\be$ among
			\begin{align*}
				\{1,\dots,N \}\setminus\qty(\{s_{\ell,1}\wedge s_{\ell,2} \}\uplus\biguplus_{m=\ell+1}^{e(H)}\{s_{m,1},s_{m,2} \}).
			\end{align*}
			Let $g\in\mG_\al\cap\mG_\be$ and $b=\qty(b_{\ell})_{\ell=1}^{e(H)}\in\prod_{\ell=1}^{e(H)}[N-2(e(H)-\ell)-1]$. If $b_{\ell}\notin R_{\al,\be,\ell}$ for all $1\le\ell\le e(H)$, then $f_{\al}\qty(g,b)\supset\be$. Conversely, if $b_{\ell}\in R_{\al,\be,\ell}$ for some $1\le\ell\le e(H)$, then $f_{\al}\qty(g,b)\not\supset\be$.
		\end{lem}
		Let $\al,\al'\in\Ga_{1j}$ and $\be,\be'\in\Ga_{1k}$ be such that $\al\cap\be=\al'\cap\be'=\varnothing$. For each of $1\le\ell\le e(H_j)$, let $R_{\al,\be,\ell},R_{\al',\be',\ell}\subset[N-2(e(H_j)-\ell)-1]$ be as in Lemma \ref{destroyed al be disjoint}. By the independence, we have
		\begin{align}
			\p\qty(I_\al=I_\be=1,J_{\be\al}=0)=&\p\qty(\qty{I_\al=I_\be=1}\cap\bigcup_{\ell=1}^{e(H_j)}\qty{ B_{\al,\ell}\in R_{\al,\be,\ell}})\notag\\
			=&\p\qty(I_\al=I_\be=1)\qty(1-\p\qty(\bigcap_{\ell=1}^{e(H_j)}\qty{B_{\al,\ell}\notin R_{\al,\be,\ell}}))\notag\\
			=&\frac{1}{((N-1))_{e(H_j)+e(H_k)}}\qty(1-\prod_{\ell=1}^{e(H_j)}\qty(1-\frac{2e(H_k)}{N-2(e(H_j)-\ell)-1})).
			\label{prob al be exist be destroyed}
		\end{align}
		If $\al\cap\be'=\be\cap\be'=\varnothing$, then $R_{\al,\be,\ell}\cap R_{\al,\be',\ell}=\varnothing$. By the independence, we have
		\begin{align*}
			&\p\qty(I_\al=I_\be=I_{\be'}=1,J_{\be\al}=J_{\be'\al}=0)\\
			=&\p\qty(\qty{I_\al=I_\be=I_{\be'}=1}\cap\bigcup_{\ell=1}^{e(H_j)}\qty{ B_{\al,\ell}\in R_{\al,\be,\ell}}\cap \bigcup_{\ell=1}^{e(H_j)}\qty{ B_{\al,\ell}\in R_{\al,\be',\ell}})\\
			=&\p\qty(I_\al=I_\be=I_{\be'}=1)\\
			&\times\qty(1-\p\qty(\bigcap_{\ell=1}^{e(H_j)}\qty{B_{\al,\ell}\notin R_{\al,\be,\ell}})-\p\qty(\bigcap_{\ell=1}^{e(H_j)}\qty{B_{\al,\ell}\notin R_{\al,\be',\ell}})+\p\qty(\bigcap_{\ell=1}^{e(H_j)}\qty{B_{\al,\ell}\notin R_{\al,\be,\ell}\uplus R_{\al,\be',\ell}}))\\
			=&\frac{1}{((N-1))_{e(H_j)+2e(H_k)}}\qty(1-2\prod_{\ell=1}^{e(H_j)}\qty(1-\frac{2e(H_k)}{N-2(e(H_j)-\ell)-1})+\prod_{\ell=1}^{e(H_j)}\qty(1-\frac{4e(H_k)}{N-2(e(H_j)-\ell)-1})).
		\end{align*}
		If $\al\cap\al'=\al'\cap\be=\varnothing$, by the independence, we have
		\begin{align*}
			&\p\qty(I_\al=I_{\al'}=I_\be=1,J_{\be\al}=J_{\be\al'}=0)\\
			=&\p\qty(\qty{I_\al=I_{\al'}=I_\be=1}\cap\bigcup_{\ell=1}^{e(H_j)}\qty{ B_{\al,\ell}\in R_{\al,\be,\ell}}\cap \bigcup_{\ell=1}^{e(H_j)}\qty{ B_{\al',\ell}\in R_{\al',\be,\ell}})\\
			=&\p\qty(I_\al=I_{\al'}=I_{\be}=1)\qty(1-\p\qty(\bigcap_{\ell=1}^{e(H_j)}\qty{B_{\al,\ell}\notin R_{\al,\be,\ell}}))\qty(1-\p\qty(\bigcap_{\ell=1}^{e(H_j)}\qty{B_{\al',\ell}\notin R_{\al',\be,\ell}}))\\
			=&\frac{1}{((N-1))_{2e(H_j)+e(H_k)}}\qty(1-\prod_{\ell=1}^{e(H_j)}\qty(1-\frac{2e(H_k)}{N-2(e(H_j)-\ell)-1}))^2.
		\end{align*}
		If $\al,\al',\be,\be'$ are all disjoint, by the independence, we have
		\begin{align*}
			&\p\qty(I_\al=I_\be=I_{\al'}=I_{\be'}=1,J_{\be\al}=J_{\be'\al'}=0)\\
			=&\p\qty(\qty{I_\al=I_\be=I_{\al'}=I_{\be'}=1}\cap\bigcup_{\ell=1}^{e(H_j)}\qty{ B_{\al,\ell}\in R_{\al,\be,\ell}}\cap \bigcup_{\ell=1}^{e(H_j)}\qty{ B_{\al',\ell}\in R_{\al',\be',\ell}})\\
			=&\p\qty(I_\al=I_\be=I_{\al'}=I_{\be'}=1)\qty(1-\p\qty(\bigcap_{\ell=1}^{e(H_j)}\qty{B_{\al,\ell}\notin R_{\al,\be,\ell}}))\qty(1-\p\qty(\bigcap_{\ell=1}^{e(H_j)}\qty{B_{\al',\ell}\notin R_{\al',\be',\ell}}))\\
			=&\frac{1}{((N-1))_{2e(H_j)+2e(H_k)}}\qty(1-\prod_{\ell=1}^{e(H_j)}\qty(1-\frac{2e(H_k)}{N-2(e(H_j)-\ell)-1}))^2.
		\end{align*}
		Furthermore, by Bernoulli's inequality, we have
		\begin{align}\label{Bernoulli inequality}
			1-\prod_{\ell=1}^{e(H_j)}\qty(1-\frac{2e(H_k)}{N-2(e(H_j)-\ell)-1})\le\sum_{\ell=1}^{e(H_j)}\frac{2e(H_k)}{N-2(e(H_j)-\ell)-1}.
		\end{align}
		Combined with, e.g., \cite[p.17, Eq.\ (1.17)]{Stanley12}, it gives
		\begin{align*}
			&1-2\prod_{\ell=1}^{e(H_j)}\qty(1-\frac{2e(H_k)}{N-2(e(H_j)-\ell)-1})+\prod_{\ell=1}^{e(H_j)}\qty(1-\frac{4e(H_k)}{N-2(e(H_j)-\ell)-1})\\
			\le&\sum_{\ell=1}^{e(H_j)}\frac{4e(H_k)}{N-2(e(H_j)-\ell)-1}-1+\prod_{\ell=1}^{e(H_j)}\qty(1-\frac{4e(H_k)}{N-2(e(H_j)-\ell)-1})\\
			=&\sum_{\substack{A\subset[e(H_j)]\\|A|\ge2}}\qty(-4e(H_k))^{|A|}\prod_{\ell\in A}\frac{1}{N-2(e(H_j)-\ell)-1}.
		\end{align*}
		In view of Lemma \ref{intersecting different isolated trees}, together with the restriction $\al\cap\be=\al'\cap\be'=\varnothing$, it suffices to consider the following mutually exclusive seven cases in order to upper bound \eqref{second variance}:
		{\setlength{\leftmargini}{29.5pt}  	
			\begin{enumerate}
					\renewcommand{\labelenumi}{(\roman{enumi})}
				
				\setlength{\labelsep}{7pt}     
				\setlength{\itemsep}{3pt}      
				
				\item $\al=\al'$, $\be=\be'$,  $\al\cap\be=\varnothing$.
				
				\item $\al=\al'$, $\be\cap\be'=\al\cap\be=\al\cap\be'=\varnothing$.
				
				\item $j=k$, $\al\cap\al'=\varnothing$, $\be=\al'$, $\be'=\al$.
				
				\item $j=k$, $\al\cap\al'=\varnothing$, $\be'=\al$, $\be\cap\al'=\al\cap\be=\varnothing$.
				
				\item $j=k$, $\al\cap\al'=\varnothing$, $\be=\al'$, $\be'\cap\al=\al'\cap\be'=\varnothing$.
				
				\item $\al\cap\al'=\varnothing$, $\be=\be'$, $\al\cap\be=\al'\cap\be=\varnothing$.
				
				\item $\al,\al',\be,\be'$ are all disjoint.
				
		\end{enumerate} }
		
		\medskip
		For Case (i), we have
		\begin{align*}
			&\sum_{\al\in\Ga_{1j}}\sum_{\substack{\be\in\Ga_{1k}\\\be\cap\al=\varnothing}}\qty(\e\qty[I_\al I_\be (1-J_{\be\al})]-\qty(\e\qty[I_\al I_\be(1-J_{\be\al})])^2)\\
			=&\sum_{\al\in\Ga_{1j}}\sum_{\substack{\be\in\Ga_{1k}\\\be\cap\al=\varnothing}}\frac{1}{((N-1))_{e(H_j)+e(H_k)}}\qty(1-\prod_{\ell=1}^{e(H_j)}\qty(1-\frac{2e(H_k)}{N-2(e(H_j)-\ell)-1}))\\
			&\quad\quad\quad\quad\times\qty(1-	\p\qty(I_\al=I_\be=1,J_{\be\al}=0))\\
			\le&\sum_{\al\in\Ga_{1j}}\sum_{\substack{\be\in\Ga_{1k}\\\be\cap\al=\varnothing}}\frac{1}{((N-1))_{e(H_j)+e(H_k)}}\sum_{\ell=1}^{e(H_j)}\frac{2e(H_k)}{N-2(e(H_j)-\ell)-1}.
		\end{align*}
		For Case (ii), we have
		\begin{align*}
			&\mathop{\sum_{\al\in\Ga_{1j}}\sum_{\be,\be'\in\Ga_{1k}}}_{\text{$\al,\be,\be'$ are all disjoint}}\e\qty[I_\al I_\be  I_{\be'}(1-J_{\be\al})(1-J_{\be'\al})]\\
			\le&\mathop{\sum_{\al\in\Ga_{1j}}\sum_{\be,\be'\in\Ga_{1k}}}_{\text{$\al,\be,\be'$ are all disjoint}}\frac{1}{((N-1))_{e(H_j)+2e(H_k)}}\sum_{\substack{A\subset[e(H_j)]\\|A|\ge2}}\qty(-4e(H_k))^{|A|}\prod_{\ell\in A}\frac{1}{N-2(e(H_j)-\ell)-1}.
		\end{align*}
		For Case (iii), similarly to Case (i),
		\begin{align*}
			&\sum_{\substack{\al,\al'\in\Ga_{1j}\\ \al\cap\al'=\varnothing}}\qty(\e\qty[I_\al I_\be (1-J_{\be\al})]-\qty(\e\qty[I_\al I_\be(1-J_{\be\al})])^2)\\
			\le&\sum_{\substack{\al,\al'\in\Ga_{1j}\\ \al\cap\al'=\varnothing}}\frac{1}{((N-1))_{2e(H_j)}}\sum_{\ell=1}^{e(H_j)}\frac{2e(H_j)}{N-2(e(H_j)-\ell)-1}.
		\end{align*}
		For Case (vi), we have
		\begin{align*}
			&\mathop{\sum_{\al,\al'\in\Ga_{1j}}\sum_{\be\in\Ga_{1k}}}_{\text{$\al,\al',\be$ are all disjoint}}\e\qty[I_\al I_{\al'}  I_{\be}(1-J_{\be\al})(1-J_{\be\al'})]\\
			=&\mathop{\sum_{\al,\al'\in\Ga_{1j}}\sum_{\be\in\Ga_{1k}}}_{\text{$\al,\al',\be$ are all disjoint}}\frac{1}{((N-1))_{2e(H_j)+e(H_k)}}\qty(1-\prod_{\ell=1}^{e(H_j)}\qty(1-\frac{2e(H_k)}{N-2(e(H_j)-\ell)-1}))^2\\
			\le&\mathop{\sum_{\al,\al'\in\Ga_{1j}}\sum_{\be\in\Ga_{1k}}}_{\text{$\al,\al',\be$ are all disjoint}}\frac{1}{((N-1))_{2e(H_j)+e(H_k)}}\qty(\sum_{\ell=1}^{e(H_j)}\frac{2e(H_k)}{N-2(e(H_j)-\ell)-1})^2.
		\end{align*}
		For Case (iv), similarly to Case (vi),
		\begin{align*}
			&\sum_{\substack{\al,\al',\be\in\Ga_{1j}\\ \text{$\al,\al',\be$ are all disjoint} }}\e\qty[I_\al I_{\al'}  I_{\be}(1-J_{\be\al})(1-J_{\be\al'})]\\
			=&\sum_{\substack{\al,\al',\be\in\Ga_{1j}\\ \text{$\al,\al',\be$ are all disjoint} }}\frac{1}{((N-1))_{3e(H_j)}}\qty(1-\prod_{\ell=1}^{e(H_j)}\qty(1-\frac{2e(H_j)}{N-2(e(H_j)-\ell)-1}))^2\\
			\le&\sum_{\substack{\al,\al',\be\in\Ga_{1j}\\ \text{$\al,\al',\be$ are all disjoint} }}\frac{1}{((N-1))_{3e(H_j)}}\qty(\sum_{\ell=1}^{e(H_j)}\frac{2e(H_j)}{N-2(e(H_j)-\ell)-1})^2.
		\end{align*}
		For Case (v), similarly to Case (vi) and Case (iv),
		\begin{align*}
			&\sum_{\substack{\al,\al',\be'\in\Ga_{1j}\\ \text{$\al,\al',\be'$ are all disjoint} }}\e\qty[I_\al I_{\al'}  I_{\be'}(1-J_{\be'\al})(1-J_{\be'\al'})]\\
			=&\sum_{\substack{\al,\al',\be'\in\Ga_{1j}\\ \text{$\al,\al',\be'$ are all disjoint} }}\frac{1}{((N-1))_{3e(H_j)}}\qty(1-\prod_{\ell=1}^{e(H_j)}\qty(1-\frac{2e(H_j)}{N-2(e(H_j)-\ell)-1}))^2\\
			\le&\sum_{\substack{\al,\al',\be'\in\Ga_{1j}\\ \text{$\al,\al',\be'$ are all disjoint} }}\frac{1}{((N-1))_{3e(H_j)}}\qty(\sum_{\ell=1}^{e(H_j)}\frac{2e(H_j)}{N-2(e(H_j)-\ell)-1})^2.
		\end{align*}
		Finally, for Case (vii), we have
		\begin{align*}
			&\mathop{\sum_{\al,\al'\in\Ga_{1j}}\sum_{\be,\be'\in\Ga_{1k}}}_{\text{$\al,\al',\be,\be'$ are all disjoint}}\Big(\e\qty[I_\al I_\be I_{\al'} I_{\be'}(1-J_{\be\al})(1-J_{\be'\al'})]-\e\qty[I_\al I_\be(1-J_{\be\al})]\e\qty[ I_{\al'} I_{\be'}(1-J_{\be'\al'})]\Big)\\
			=&\mathop{\sum_{\al,\al'\in\Ga_{1j}}\sum_{\be,\be'\in\Ga_{1k}}}_{\text{$\al,\al',\be,\be'$ are all disjoint}}\qty(\frac{1}{((N-1))_{2e(H_j)+2e(H_k)}}-\frac{1}{[((N-1))_{e(H_j)+e(H_k)}]^2})\\
			&\quad\quad\quad\quad\times\qty(1-\prod_{\ell=1}^{e(H_j)}\qty(1-\frac{2e(H_k)}{N-2(e(H_j)-\ell)-1}))^2\\
			\le&\mathop{\sum_{\al,\al'\in\Ga_{1j}}\sum_{\be,\be'\in\Ga_{1k}}}_{\text{$\al,\al',\be,\be'$ are all disjoint}}\frac{1}{((N-1))_{2e(H_j)+2e(H_k)}}\qty(1-\frac{((N-1))_{2e(H_j)+2e(H_k)}}{[((N-1))_{e(H_j)+e(H_k)}]^2})\qty(\sum_{\ell=1}^{e(H_j)}\frac{2e(H_k)}{N-2(e(H_j)-\ell)-1})^2\\
			\le&\mathop{\sum_{\al,\al'\in\Ga_{1j}}\sum_{\be,\be'\in\Ga_{1k}}}_{\text{$\al,\al',\be,\be'$ are all disjoint}}\frac{1}{((N-1))_{2e(H_j)+2e(H_k)}}\qty(\sum_{\ell=1}^{e(H_j)+e(H_k)}\frac{2\qty(e(H_j)+e(H_k))}{N-2(\ell-1)-1})\\
			&\quad\quad\quad\quad\times\qty(\sum_{\ell=1}^{e(H_j)}\frac{2e(H_k)}{N-2(e(H_j)-\ell)-1})^2.
		\end{align*}
		If $j=k$, \eqref{second variance} is bounded above by the sum over all of Cases (i)--(vii). If $j\neq k$, it is bounded by the sum over Cases (i), (ii), (vi) and (vii). Thus, we upper bound the second variance in \eqref{variance decomposition} as
		\begin{align}
			&\Var\qty(\sum_{\al\in\Ga_{1j}}\sum_{\substack{\be\in\Ga_{1k}\\ \be\cap\al=\varnothing }}I_\al I_\be(1-J_{\be\al}))\notag\\
			\le&\begin{cases}
				\frac{\qty((n_1)_2)^2}{((N-1))_2(N-1)}+\frac{\frac{3}{2}\qty((n_1)_2)^3}{((N-1))_3(N-1)^2}+\frac{2\qty((n_1)_2)^4}{((N-1))_4(N-3)(N-1)^2}&(j=k=1)\\
			\frac{2\qty((n_1)_2)^2n_2}{((N-1))_3(N-1)}+\frac{4\qty((n_1)_2)^3n_2}{((N-1))_4(N-1)^2}+\frac{72\qty((n_1)_2)^4{n_2}^2}{((N-1))_6(N-5)(N-1)^2}	&(j=1,k=2)\\
			\frac{2\qty((n_1)_2)^2n_2}{((N-1))_3(N-3)}+\frac{4\qty((n_1)_2)^3n_2}{((N-1))_4((N-1))_2}+\frac{8\qty((n_1)_2)^3{n_2}^2}{((N-1))_5(N-3)^2}+\frac{72\qty((n_1)_2)^4{n_2}^2}{((N-1))_6(N-5)(N-3)^2}	&(j=2,k=1)\\
		\frac{16\qty((n_1)_2n_2)^2}{((N-1))_4(N-3)}+\frac{64\qty((n_1)_2n_2)^3}{((N-1))_6((N-1))_2}+\frac{192\qty((n_1)_2n_2)^3}{((N-1))_6(N-3)^2}+\frac{2048\qty((n_1)_2n_2)^4}{((N-1))_8(N-7)(N-3)^2}	&(j=k=2)
			\end{cases}.\label{second variance bound}
		\end{align}

		\subsubsection{The third variance in \eqref{variance decomposition}}
		\label{subsec third variance}
		
		 By Lemma \ref{intersecting different isolated trees} and decomposing the sum into two parts, we upper bound
		\begin{align}
			&\Var\qty(\sum_{\al\in\Ga_{1j}}\sum_{\substack{\be\in\Ga_{1k}\setminus\{\al\}\\ \be\cap\al\neq\varnothing }}I_\al J_{\be\al}) \notag  \\
			=&\sum_{\al\in\Ga_{1j}}\sum_{\substack{\be\in\Ga_{1k}\setminus\{\al\}\\ \be\cap\al\neq\varnothing }}\sum_{\al'\in\Ga_{1j}}\sum_{\substack{\be'\in\Ga_{1k}\setminus\{\al'\}\\ \be'\cap\al'\neq\varnothing }}\e\qty[I_\al I_{\al'} J_{\be\al}J_{\be'\al'}]-\qty(\sum_{\al\in\Ga_{1j}}\sum_{\substack{\be\in\Ga_{1k}\setminus\{\al\}\\ \be\cap\al\neq\varnothing}}p_\al p_\be)^2 \notag \\
			\le&\sum_{\al\in\Ga_{1j}}\sum_{\substack{\be,\be'\in\Ga_{1k}\setminus\{\al\}\\ \be\cap\al\neq\varnothing,\be'\cap\al\neq\varnothing }}\e\qty[I_\al J_{\be\al}J_{\be'\al} ]-\sum_{\al\in\Ga_{1j}}\sum_{\substack{\be,\be'\in\Ga_{1k}\setminus\{\al\}\\ \be\cap\al\neq\varnothing,\be'\cap\al\neq\varnothing }}p_\al p_\be p_{\be'} \notag \\
			&+\sum_{\substack{\al,\al'\in\Ga_{1j}\\ \al\cap\al'=\varnothing}}\sum_{\substack{\be\in\Ga_{1k}\setminus\{\al\}\\ \be\cap\al\neq\varnothing }}\sum_{\substack{\be'\in\Ga_{1k}\setminus\{\al'\}\\ \be'\cap\al'\neq\varnothing }}\e\qty[I_\al I_{\al'} J_{\be\al}J_{\be'\al'}]-\sum_{\substack{\al,\al'\in\Ga_{1j}\\ \al\cap\al'=\varnothing}}\sum_{\substack{\be\in\Ga_{1k}\setminus\{\al\}\\ \be\cap\al\neq\varnothing }}\sum_{\substack{\be'\in\Ga_{1k}\setminus\{\al'\}\\ \be'\cap\al'\neq\varnothing }}p_\al p_{\al'}p_\be p_{\be'}.   \label{third variance upper bound}
		\end{align}
		Note that $\e\qty[I_\al J_{\be\al}J_{\be'\al} ]=p_\al\e\qty[J_{\be\al}J_{\be'\al}|I_\al=1 ]=p_\al\e\qty[I_{\be}I_{\be'}]$. Thus, by Lemma \ref{intersecting different isolated trees}, we upper bound the first line of \eqref{third variance upper bound} as
		\begin{align}
			&\sum_{\al\in\Ga_{1j}}\sum_{\substack{\be,\be'\in\Ga_{1k}\setminus\{\al\}\\ \be\cap\al\neq\varnothing,\be'\cap\al\neq\varnothing }}\e\qty[I_\al J_{\be\al}J_{\be'\al} ]-\sum_{\al\in\Ga_{1j}}\sum_{\substack{\be,\be'\in\Ga_{1k}\setminus\{\al\}\\ \be\cap\al\neq\varnothing,\be'\cap\al\neq\varnothing }}p_\al p_\be p_{\be'}\label{third variance first line}\\
			\le&\sum_{\al\in\Ga_{1j}}\sum_{\substack{\be\in\Ga_{1k}\setminus\{\al\}\\ \be\cap\al\neq\varnothing }}\qty(p_\al\e\qty[I_\be]-p_\al p_\be)+\sum_{\al\in\Ga_{1j}}\sum_{\substack{\be,\be'\in\Ga_{1k}\setminus\{\al\}\\ \be\cap\al\neq\varnothing,\be'\cap\al\neq\varnothing,\be\cap\be'=\varnothing }}\qty(p_\al\e\qty[I_{\be}I_{\be'}]-p_\al p_\be p_{\be'})\notag\\
			=&\sum_{\al\in\Ga_{1j}}\sum_{\substack{\be,\be'\in\Ga_{1k}\setminus\{\al\}\\ \be\cap\al\neq\varnothing,\be'\cap\al\neq\varnothing,\be\cap\be'=\varnothing }}\frac{1}{((N-1))_{e(H_j)}((N-1))_{2e(H_k)}}\qty(1-\frac{((N-1))_{2e(H_k)}}{\qty[((N-1))_{e(H_k)}]^2})\notag\\
			\le&\sum_{\al\in\Ga_{1j}}\sum_{\substack{\be,\be'\in\Ga_{1k}\setminus\{\al\}\\ \be\cap\al\neq\varnothing,\be'\cap\al\neq\varnothing,\be\cap\be'=\varnothing }}\frac{1}{((N-1))_{e(H_j)}((N-1))_{2e(H_k)}}\sum_{\ell=1}^{e(H_k)}\frac{2e(H_k)}{N-2(\ell-1)-1}.\notag
		\end{align}
		By Lemma \ref{cardinality estimates}(b), we have
		\begin{align*}
			\qty|\sum_{\al\in\Ga_{1j}}\sum_{\substack{\be,\be'\in\Ga_{1k}\setminus\{\al\}\\ \be\cap\al\neq\varnothing,\be'\cap\al\neq\varnothing,\be\cap\be'=\varnothing }}|\le\begin{cases}
			2(n_1)_2{n_1}^2	&(j=k=1)\\
			8(n_1)_2{n_1}^2{n_2}^2	&(j=1,k=2)\\
			4(n_1)_2n_2{n_1}^2	&(j=2,k=1)\\
			(n_1)_2n_2\qty(4n_1n_2+(n_1)_2)^2	&(j=k=2)
			\end{cases}.
		\end{align*}
	Thus we obtain upper bounds on \eqref{third variance first line} as follows:
		\begin{align}
			&\frac{4(n_1)_2{n_1}^2}{(N-1)^2((N-1))_2}&&(j=k=1)\label{third variance first line j1k1}\\ &\frac{64(n_1)_2{n_1}^2{n_2}^2}{((N-1))_2((N-1))_4}&&(j=1,k=2)\label{third variance first line j1k2}\\
			&\frac{8(n_1)_2n_2{n_1}^2}{\qty[((N-1))_2]^2(N-1)}&&(j=2,k=1)\label{third variance first line j2k1}\\
			&\frac{8(n_1)_2n_2\qty(4n_1n_2+(n_1)_2)^2}{((N-1))_2((N-1))_4(N-3)}&&(j=k=2).\label{third variance first line j2k2}
		\end{align}

			\medskip
		We now turn to the second line of \eqref{third variance upper bound}. By Lemma \ref{f al ell f al inverse}, for $\al\in\Ga$ of type $H$, the function $f_\al$ in \eqref{coupling construction f al} has the inverse $f_\al^{-1}$. For a perfect matching $g$ of the $N$ half-edges, we write $f_{\al}^{-1}(g)=\qty(\qty(f_\al^{-1}(g))_1,\qty(f_\al^{-1}(g))_2)$ with $\qty(f_\al^{-1}(g))_1\in\mG_\al$ and
		$\qty(f_\al^{-1}(g))_2\in\prod_{\ell=1}^{e(H)}\{1,\dots, N-2(e(H)-\ell)-1 \}$. Then, for $\al,\al'\in\Ga_{1j}$, $\al\cap\al'=\varnothing$, $\be\in\Ga_{1k}\setminus\{\al\}$, $\be\cap\al\neq\varnothing$ and $\be'\in\Ga_{1k}\setminus\{\al'\}$, $\be'\cap\al'\neq\varnothing$, we have
		\begin{equation}\label{probability to cardinality}
			\begin{split}
			&\p(I_\al=I_{\al'}=1,J_{\be\al}=J_{\be'\al'}=1)\\
			=&\p\qty(\bigcup_{(g,g')\in\mG_\be\times\mG_{\be'}}\qty{I_\al=I_{\al'}=1,\mathfrak{g}_\al=g,\mathfrak{g}_{\al'}=g'})\\
			\le&\sum_{\substack{(g,g')\in\mG_\be\times\mG_{\be'}\\ \mathrm{s.t.\ }\qty(f_\al^{-1}(g))_1=\qty(f_{\al'}^{-1}(g'))_1}}\p\qty(\mathfrak{g}=\qty(f_\al^{-1}(g))_1=\qty(f_{\al'}^{-1}(g'))_1,B_{\al}=\qty(f_\al^{-1}(g))_2,B_{\al'}=\qty(f_{\al'}^{-1}(g'))_2)\\
			=&\sum_{\substack{(g,g')\in\mG_\be\times\mG_{\be'}\\ \mathrm{s.t.\ }\qty(f_\al^{-1}(g))_1=\qty(f_{\al'}^{-1}(g'))_1}}\frac{1}{(N-1)!!}\frac{1}{\qty[((N-1))_{e(H_j)}]^2}.
			\end{split}
		\end{equation}
		Let
		\begin{equation}\label{subset of the product of Gbeta and Gbetaprime}
			\mH_{\al,\al',\be,\be'}\coloneqq\qty{(g,g')\in\mG_\be\times\mG_{\be'}:\qty(f_\al^{-1}(g))_1=\qty(f_{\al'}^{-1}(g'))_1}.
		\end{equation}
		So we have to carefully estimate the cardinality $|\mH_{\al,\al',\be,\be'}|$. Since $\Ga_{1j}$ and $\Ga_{1k}$ can be of two distinct types, corresponding to isolated edges and isolated 2-stars, this gives rise to four
		different cases. The results we need to estimate the cardinality of the set $\mH_{\al,\al',\be,\be'}$ in \eqref{subset of the product of Gbeta and Gbetaprime} are summarized into the following lemmas. The proofs of these lemmas are based on careful and lengthy analyses about the corresponding restrictions of $f_\al$ in \eqref{coupling construction f al} and will be deferred to Appendix \ref{proofs}. Also, in those proofs of the lemmas, when we talk about the ``rank'' of a half-edge it is always the rank in ascending order among a set of half-edges.
		\begin{lem}\label{surjection edge edge}
			Let $\al$ and $\be$ be possible isolated edges such that $\al\cap\be\neq\varnothing$ and $\al\neq\be$. Then there exists a surjection $h_{\al,\be}$ from $\mG_\al$ onto $\mG_\be$ such that $h_{\al,\be}\qty(\qty(f^{-1}_\al(g_\be))_1)=g_\be$ for $g_\be\in\mG_\be$.
		\end{lem}
		\begin{lem}\label{surjection al edge be 2star}
			Let $\al=s_1s_2$ be a possible isolated edge and let $\be=(t_1t_2)\cup( t_3t_4)$ be a possible isolated 2-star such that $t_2$ and $t_3$ are incident to its degree 2 vertex. Suppose that $\al\cap\be\neq\varnothing$ and $\al\neq\be$.
			{\setlength{\leftmargini}{29.5pt}  	
				\begin{itemize}
					\setlength{\labelsep}{7pt}     
					\setlength{\itemsep}{3pt}      
					
					\item[\emph{a}.] Suppose that $\al$ and $\be$ share one degree 1 vertex only and $s_2=t_1$. Then the edge $t_3t_4$ exists in $\qty(f^{-1}_\al(g_\be))_1$ for $g_\be\in\mG_\be$ and there exists a surjection $h_{\al,\be}$ from $\mG_\al\cap\mG_{t_3t_4}$ onto $\mG_\be$ such that $h_{\al,\be}\qty(\qty(f^{-1}_\al(g_\be))_1)=g_\be$ for $g_\be\in\mG_\be$.
					
					\item[\emph{b}.] Suppose that $\al$ and $\be$ share two degree 1 vertices and assume that $s_1=t_1$ and $s_2=t_4$. Then the self-loop $t_2t_3$ exists in $\qty(f^{-1}_\al(g_\be))_1$ for $g_\be\in\mG_\be$ and there exists a surjection $h_{\al,\be}$ from $\mG_\al\cap\mG_{t_2t_3}$ onto $\mG_\be$ such that $h_{\al,\be}\qty(\qty(f^{-1}_\al(g_\be))_1)=g_\be$ for $g_\be\in\mG_\be$.

			\end{itemize} }
		\end{lem}
		\begin{lem}\label{surjection al 2star be edge}
			Let $\al$ be a possible isolated 2-star and let $\be$ be a possible isolated edge such that $\al\cap\be\neq\varnothing$ and $\al\neq\be$. Then there exists a surjection $h_{\al,\be}$ from $\mG_{\al}\times[N-1]$ onto $\mG_\be$ such that for any $g_\be\in\mG_\be$, $h_{\al,\be}\qty(\qty(f^{-1}_\al(g_\be))_1,c_\al(g_\be))=g_\be$ for some $c_\al(g_\be)\in[N-1]$.
		\end{lem}
		\begin{lem}\label{surjection 2star 2star}
			Let $\al$ and $\be$ be possible isolated 2-stars such that $\al\cap\be\neq\varnothing$ and $\al\neq\be$. Write $\al=(s_1s_2)\cup (s_3s_4)$ and $\be=(t_1t_2)\cup(t_3t_4)$, where $s_2$ and $s_3$ are incident to $\al$'s degree 2 vertex and $t_2$ and $t_3$ are incident to $\be$'s degree 2 vertex.
			{\setlength{\leftmargini}{29.5pt}  	
				\begin{itemize}
					\setlength{\labelsep}{7pt}     
					\setlength{\itemsep}{3pt}      
					
					\item[\emph{a}.] Suppose that $\al$ and $\be$ share one degree 1 vertex only and assume that $s_4=t_1$ without loss of generality. Then the edge $t_3t_4$ exists in $\qty(f^{-1}_\al(g_\be))_1$ for $g_\be\in\mG_\be$ and there exists a surjection $h_{\al,\be}$ from $\qty(\mG_\al\cap\mG_{t_3t_4})\times[N-1]$ onto $\mG_\be$ such that for any $g_\be\in\mG_\be$, $h_{\al,\be}\qty(\qty(f^{-1}_\al(g_\be))_1,c_\al(g_\be))=g_\be$ for some $c_\al(g_\be)\in[N-1]$.
					
					\item[\emph{b}.] For all the other cases about $\al\cap\be\neq\varnothing,\al\neq\be$ than (a), there exists a surjection $h_{\al,\be}$ from $\mG_{\al}$ onto $\mG_\be$ such that $h_{\al,\be}\qty(\qty(f^{-1}_\al(g_\be))_1)=g_\be$ for $g_\be\in\mG_\be$.
					
			\end{itemize} }
		\end{lem}
		
		\medskip
		Now we are ready to estimate
		\begin{equation}\label{fouth moment term}
			\sum_{\substack{\al,\al'\in\Ga_{1j}\\ \al\cap\al'=\varnothing}}\sum_{\substack{\be\in\Ga_{1k}\setminus\{\al\}\\ \be\cap\al\neq\varnothing }}\sum_{\substack{\be'\in\Ga_{1k}\setminus\{\al'\}\\ \be'\cap\al'\neq\varnothing }}\e\qty[I_\al I_{\al'} J_{\be\al}J_{\be'\al'}]-\sum_{\substack{\al,\al'\in\Ga_{1j}\\ \al\cap\al'=\varnothing}}\sum_{\substack{\be\in\Ga_{1k}\setminus\{\al\}\\ \be\cap\al\neq\varnothing }}\sum_{\substack{\be'\in\Ga_{1k}\setminus\{\al'\}\\ \be'\cap\al'\neq\varnothing }}p_\al p_{\al'}p_\be p_{\be'}
		\end{equation}
		for each of the four cases: (a) $\Ga_{1j}$ is possible isolated edges, $\Ga_{1k}$ is possible isolated edges ($j=k=1$). (b) $\Ga_{1j}$ is possible isolated edges, $\Ga_{1k}$ is possible isolated 2-stars ($j=1,k=2$). (c) $\Ga_{1j}$ is possible isolated 2-stars, $\Ga_{1k}$ is possible isolated edges ($j=2,k=1$). (d) $\Ga_{1j}$ is possible isolated 2-stars, $\Ga_{1k}$ is possible isolated 2-stars ($j=k=2$). 
		
		\medskip
		\textbf{(a) $\Ga_{1j}$ is possible isolated edges, $\Ga_{1k}$ is possible isolated edges ($j=k=1$).}
		
		\medskip
		Let $h_{\al,\be}:\mG_\al\rightarrow\mG_\be$ and $h_{\al',\be'}:\mG_{\al'}\rightarrow\mG_{\be'}$ be the surjections in Lemma \ref{surjection edge edge} such that $h_{\al,\be}\qty(\qty(f^{-1}_\al(g_\be))_1)=g_\be$ for $g_\be\in\mG_\be$ and $h_{\al',\be'}\qty(\qty(f^{-1}_{\al'}(g_{\be'}))_1)=g_{\be'}$ for $g_{\be'}\in\mG_{\be'}$. Define a map $h_{\al,\al',\be,\be'}:\mG_{\al}\cap\mG_{\al'}\rightarrow\mG_{\be}\times\mG_{\be'}$ by $h_{\al,\al',\be,\be'}(g)\coloneqq\qty(h_{\al,\be}(g), h_{\al',\be'}(g))$.  For any $(g_\be,g_{\be'})\in\mH_{\al,\al',\be,\be'}$, letting $g\coloneqq\qty(f_\al^{-1}(g_\be))_1=\qty(f_{\al'}^{-1}(g_{\be'}))_1\in\mG_{\al}\cap\mG_{\al'}$ yields $h_{\al,\al',\be,\be'}(g)=(g_\be,g_{\be'})$. Therefore, $h_{\al,\al',\be,\be'}$ is a surjection onto $\mH_{\al,\al',\be,\be'}$ and we have
		\begin{align*}
			\qty|\mH_{\al,\al',\be,\be'}|\le\qty|\mG_{\al}\cap\mG_{\al'}|=(N-5)!!.
		\end{align*}
		Consequently, we have
		\begin{align*}
			\p(I_\al=I_{\al'}=1,J_{\be\al}=J_{\be'\al'}=1)\le\frac{1}{(N-1)(N-3)}\frac{1}{(N-1)^2}
		\end{align*}
		in this case. This, combined with
		\begin{align*}
\qty|\sum_{\substack{\al,\al'\in\Ga_{11}\\ \al\cap\al'=\varnothing}}\sum_{\substack{\be\in\Ga_{11}\setminus\{\al\}\\ \be\cap\al\neq\varnothing }}\sum_{\substack{\be'\in\Ga_{11}\setminus\{\al'\}\\ \be'\cap\al'\neq\varnothing }}|\le\underbrace{\qty(\frac{(n_1)_2}{2})^2}_{\al,\al'}\times\underbrace{\qty(2n_1)^2}_{\be,\be'}=\qty((n_1)_2)^2{n_1}^2
		\end{align*}
		(cf.\ Lemma \ref{cardinality estimates}(b)), leads to the estimate
		\begin{align}
			&\sum_{\substack{\al,\al'\in\Ga_{11}\\ \al\cap\al'=\varnothing}}\sum_{\substack{\be\in\Ga_{11}\setminus\{\al\}\\ \be\cap\al\neq\varnothing }}\sum_{\substack{\be'\in\Ga_{11}\setminus\{\al'\}\\ \be'\cap\al'\neq\varnothing }}\e\qty[I_\al I_{\al'} J_{\be\al}J_{\be'\al'}]-\sum_{\substack{\al,\al'\in\Ga_{11}\\ \al\cap\al'=\varnothing}}\sum_{\substack{\be\in\Ga_{11}\setminus\{\al\}\\ \be\cap\al\neq\varnothing }}\sum_{\substack{\be'\in\Ga_{11}\setminus\{\al'\}\notag\\ \be'\cap\al'\neq\varnothing }}p_\al p_{\al'} p_\be p_{\be'}\\
			\le&\sum_{\substack{\al,\al'\in\Ga_{11}\\ \al\cap\al'=\varnothing}}\sum_{\substack{\be\in\Ga_{11}\setminus\{\al\}\\ \be\cap\al\neq\varnothing }}\sum_{\substack{\be'\in\Ga_{11}\setminus\{\al'\}\\ \be'\cap\al'\neq\varnothing }}\qty(\frac{1}{(N-1)^3(N-3)}-\frac{1}{(N-1)^4})\notag\\
			=&\sum_{\substack{\al,\al'\in\Ga_{11}\\ \al\cap\al'=\varnothing}}\sum_{\substack{\be\in\Ga_{11}\setminus\{\al\}\\ \be\cap\al\neq\varnothing }}\sum_{\substack{\be'\in\Ga_{11}\setminus\{\al'\}\\ \be'\cap\al'\neq\varnothing }}\frac{1}{(N-1)^3(N-3)}\qty(1-\frac{(N-1)^3(N-3)}{(N-1)^4})\notag\\
			\le& \sum_{\substack{\al,\al'\in\Ga_{11}\\ \al\cap\al'=\varnothing}}\sum_{\substack{\be\in\Ga_{11}\setminus\{\al\}\\ \be\cap\al\neq\varnothing }}\sum_{\substack{\be'\in\Ga_{11}\setminus\{\al'\}\\ \be'\cap\al'\neq\varnothing }}\frac{1}{(N-1)^3(N-3)}\frac{2}{N-1} \notag\\
			\le&\frac{2\qty((n_1)_2)^2{n_1}^2}{((N-1))_2(N-1)^3}.\label{third variance second line case a}
		\end{align}
		
		\medskip
		\textbf{(b) $\Ga_{1j}$ is possible isolated edges, $\Ga_{1k}$ is possible isolated 2-stars ($j=1,k=2$).}

		\medskip
		Let $\al=s_1s_2$, $\be=(t_1t_2)\cup (t_3t_4)$, $\al'=s_1's_2'$ and $\be'=(t_1't_2')\cup (t_3't_4')$, where $t_2$ and $t_3$ are incident to the degree 2 vertex of $\be$, and $t_2'$ and $t_3'$ are incident to the degree 2 vertex of $\be'$. There are two cases to consider about $\be\cap\al\neq\varnothing,\be\neq\al$. One is that $\al$ and $\be$ share one degree 1 vertex only and the other is that $\al$ and $\be$ share two degree 1 vertices. The same about $\be'\cap\al',\be'\neq\al'$, so there are four cases to consider in total. 
		
		\medskip
		\emph{Case} (b1): \emph{$\al$ and $\be$ share one degree 1 vertex only, $\al'$ and $\be'$ share one degree 1 vertex only.}
		
		\medskip 
		Suppose that $s_2=t_1$ and $s_2'=t_1'$ without loss of generality. Let $h_{\al,\be}:\mG_{\al}\cap\mG_{t_3t_4}\rightarrow\mG_\be$ and $h_{\al',\be'}:\mG_{\al'}\cap\mG_{t_3't_4'}\rightarrow\mG_{\be'}$ be the surjections in Lemma \ref{surjection al edge be 2star}(a) such that $h_{\al,\be}\qty(\qty(f^{-1}_\al(g_\be))_1)=g_\be$ for $g_\be\in\mG_\be$ and $h_{\al',\be'}\qty(\qty(f^{-1}_{\al'}(g_{\be'}))_1)=g_{\be'}$ for $g_{\be'}\in\mG_{\be'}$. Define a map $h_{\al,\al',\be,\be'}:\mG_{\al}\cap\mG_{\al'}\cap\mG_{t_3t_4}\cap\mG_{t_3't_4'}\rightarrow\mG_{\be}\times\mG_{\be'}$ by $h_{\al,\al',\be,\be'}(g)\coloneqq\qty(h_{\al,\be}(g), h_{\al',\be'}(g))$.  
		
		\medskip
		For $(g_\be,g_{\be'})\in\mH_{\al,\al',\be,\be'}$, $\qty(f_\al^{-1}(g_\be))_1=\qty(f_{\al'}^{-1}(g_{\be'}))_1$ must contain all of $\al,\al',t_3t_4,t_3't_4'$. First note that $\al\cap\al'=\al\cap (t_3t_4)=\al'\cap(t_3't_4')=\varnothing$. If $(t_3t_4)\cap\al'\neq\varnothing$, then $t_4$ must be equal to either $s_1'$ or $s_2'$ since $t_4,s_1',s_2'$ are incident to their degree 1 vertices and $t_3$ is incident to the degree 2 vertex, but then $\al'$ and $t_3t_4$ cannot coexist. Similarly,  If $(t_3't_4')\cap\al\neq\varnothing$, then $t_4'$ must be equal to either $s_1$ or $s_2$ since $t_4',s_1,s_2$ are incident to their degree 1 vertices and $t_3'$ is incident to the degree 2 vertex, but then $\al$ and $t_3't_4'$ cannot coexist. Thus we may assume that $(t_3t_4)\cap\al'=(t_3't_4')\cap\al=\varnothing$, since otherwise $\mH_{\al,\al',\be,\be'}$ will be empty. In order for $t_3t_4$ and $t_3't_4' $ to be able to coexist so that $\mH_{\al,\al',\be,\be'}$ is nonempty, either Case (b1i): $\{t_3,t_4 \}\cap\{t_3',t_4' \}=\varnothing$ (as the subsets of half-edges\footnote{If $t_3$ and $t_3'$ are incident to the same degree 2 vertex, $t_3\neq t_3'$ and $t_4\neq t_4'$, then $\{t_3,t_4 \}\cap\{t_3',t_4' \}=\varnothing$ but $(t_3t_4)\cap(t_3't_4')\neq\varnothing$ because those possible edges share that degree 2 vertex.}) or Case (b1ii): $\{t_3,t_4\}=\{t_3',t_4' \}$ (and thus $t_3t_4=t_3't_4'$) must hold.
		
		\medskip
		First consider Case (b1i): $\{t_3,t_4 \}\cap\{t_3',t_4' \}=\varnothing$. Then $h_{\al,\al',\be,\be'}$ is a surjection from $\mG_{\al}\cap\mG_{\al'}\cap\mG_{t_3t_4}\cap\mG_{t_3't_4'}$
		onto $\mH_{\al,\al',\be,\be'}$. Indeed, for any $(g_\be,g_{\be'})\in\mH_{\al,\al',\be,\be'}$, letting $g\coloneqq\qty(f_\al^{-1}(g_\be))_1=\qty(f_{\al'}^{-1}(g_{\be'}))_1\in\mG_{\al}\cap\mG_{\al'}\cap\mG_{t_3t_4}\cap\mG_{t_3't_4'}$ yields $h_{\al,\al',\be,\be'}(g)=(g_\be,g_{\be'})$. Therefore,
		\begin{align*}
			\qty|\mH_{\al,\al',\be,\be'}|\le\qty|\mG_{\al}\cap\mG_{\al'}\cap\mG_{t_3t_4}\cap\mG_{t_3't_4'}|=(N-9)!!.
		\end{align*}
		Consequently,
		\begin{align*}
			\p(I_\al=I_{\al'}=1,J_{\be\al}=J_{\be'\al'}=1)\le\frac{1}{((N-1))_4}\frac{1}{(N-1)^2}
		\end{align*}
		in this case. Since we may estimate the possibilities of a quadruplet $\al,\al',\be,\be'$ in this class as (using Lemma \ref{cardinality estimates}(b) for the possibilities of $\be,\be'$)
		\begin{align*}
				\qty|\sum_{\substack{\al,\al'\in\Ga_{11}\\ \al\cap\al'=\varnothing}}\sum_{\substack{\be,\be'\in\Ga_{12}\\ \text{Case (b1i)} }}|\le\underbrace{\qty(\frac{(n_1)_2}{2})^2}_{\al,\al'}\times\underbrace{\qty(4n_1 n_2)^2}_{\be,\be'}=4\qty((n_1)_2)^2{n_1}^2{n_2}^2,
		\end{align*}
		this leads to the estimate
		\begin{align*}
			&\sum_{\substack{\al,\al'\in\Ga_{11}\\ \al\cap\al'=\varnothing}}\sum_{\substack{\be,\be'\in\Ga_{12}\\ \text{Case (b1i)} }}\e\qty[I_\al I_{\al'} J_{\be\al}J_{\be'\al'}]-\sum_{\substack{\al,\al'\in\Ga_{11}\\ \al\cap\al'=\varnothing}}\sum_{\substack{\be,\be'\in\Ga_{12}\\ \text{Case (b1i)} }}p_\al p_{\al'} p_\be p_{\be'}\\
			\le&\sum_{\substack{\al,\al'\in\Ga_{11}\\ \al\cap\al'=\varnothing}}\sum_{\substack{\be,\be'\in\Ga_{12}\\ \text{Case (b1i)} }}\qty(\frac{1}{((N-1))_4(N-1)^2}-\frac{1}{(N-1)^4(N-3)^2})\\
			=&\sum_{\substack{\al,\al'\in\Ga_{11}\\ \al\cap\al'=\varnothing}}\sum_{\substack{\be,\be'\in\Ga_{12}\\ \text{Case (b1i)} }}\frac{1}{((N-1))_4(N-1)^2}\qty(1-\frac{(N-5)(N-7)}{(N-1)(N-3)})\\
			\le&\sum_{\substack{\al,\al'\in\Ga_{11}\\ \al\cap\al'=\varnothing}}\sum_{\substack{\be,\be'\in\Ga_{12}\\ \text{Case (b1i)} }}\frac{1}{((N-1))_4(N-1)^2}\frac{8}{N-3}\\
			\le&\frac{32\qty((n_1)_2)^2{n_1}^2{n_2}^2}{((N-1))_4(N-1)^2(N-3)}.
		\end{align*}
		
		\medskip
		For Case (b1ii): $\{t_3,t_4\}=\{t_3',t_4' \}$, $h_{\al,\al',\be,\be'}$ similarly defines a surjection from $\mG_{\al}\cap\mG_{\al'}\cap\mG_{t_3t_4}$ onto $\mH_{\al,\al',\be,\be'}$. Thus
		\begin{align*}
			\qty|\mH_{\al,\al',\be,\be'}|\le\qty|\mG_{\al}\cap\mG_{\al'}\cap\mG_{t_3t_4}|=(N-7)!!,
		\end{align*}
		and we have
		\begin{align*}
			\p(I_\al=I_{\al'}=1,J_{\be\al}=J_{\be'\al'}=1)\le\frac{1}{((N-1))_3}\frac{1}{(N-1)^2}
		\end{align*}
		in this case. One of the degree 1 vertices of $\be'$ is one of the degree 1 vertices of $\al'$. Since $\{t_3,t_4\}=\{t_3',t_4' \}$, the other degree 1 vertex of $\be'$ is the degree 1 vertex of $\be$ not shared by $\al$, and the degree 2 vertex of $\be'$ is that of $\be$. Therefore, if $\al,\al',\be$ are given, there are only two possibilities about the vertices of $\be'$ in this case; namely, two possibilities about the degree 1 vertex from $\al'$. Keeping this in mind, we estimate the possibilities of a quadruplet $\al,\al',\be,\be'$ in this class as (also using Lemma \ref{cardinality estimates}(b) for the possibilities of $\be$)
		\begin{align*}
				\qty|\sum_{\substack{\al,\al'\in\Ga_{11}\\ \al\cap\al'=\varnothing}}\sum_{\substack{\be,\be'\in\Ga_{12}\\ \text{Case (b1ii)} }}|\le&\underbrace{\qty(\frac{(n_1)_2}{2})^2}_{\al,\al'}\times\underbrace{4 n_1 n_2}_{\be}\times\underbrace{2\times2}_{\be'}=4\qty((n_1)_2)^2n_1n_2.
		\end{align*}
	Thus we estimate
		\begin{align*}
			\sum_{\substack{\al,\al'\in\Ga_{11}\\ \al\cap\al'=\varnothing}}\sum_{\substack{\be,\be'\in\Ga_{12}\\ \text{Case (b1ii)} }}\e\qty[I_\al I_{\al'} J_{\be\al}J_{\be'\al'}]\le&\sum_{\substack{\al,\al'\in\Ga_{11}\\ \al\cap\al'=\varnothing}}\sum_{\substack{\be,\be'\in\Ga_{12}\\ \text{Case (b1ii)} }}\frac{1}{((N-1))_3(N-1)^2}\\
			\le&  \frac{4\qty((n_1)_2)^2n_1n_2}{((N-1))_3(N-1)^2}.
		\end{align*}
		
		\medskip
		Therefore, we obtain an estimate for Case (b1) as
		\begin{align*}
			&\sum_{\substack{\al,\al'\in\Ga_{11}\\ \al\cap\al'=\varnothing}}\sum_{\substack{\be,\be'\in\Ga_{12}\\ \text{Case (b1)} }}\e\qty[I_\al I_{\al'} J_{\be\al}J_{\be'\al'}]-\sum_{\substack{\al,\al'\in\Ga_{11}\\ \al\cap\al'=\varnothing}}\sum_{\substack{\be,\be'\in\Ga_{12}\\ \text{Case (b1)} }}p_\al p_{\al'} p_\be p_{\be'}\\
			\le&\underbrace{\frac{32\qty((n_1)_2)^2{n_1}^2{n_2}^2}{((N-1))_4(N-1)^2(N-3)}}_{\text{Case (b1i)}}+\underbrace{\frac{4\qty((n_1)_2)^2n_1n_2}{((N-1))_3(N-1)^2}}_{\text{Case (b1ii)}}.
		\end{align*}
		
		\medskip
		\emph{Case} (b2): \emph{$\al$ and $\be$ share one degree 1 vertex only, $\al'$ and $\be'$ share two degree 1 vertices.}
		
		\smallskip 
		Suppose that $s_2=t_1$, $s_1'=t_1'$ and $s_2'=t_4'$ without loss of generality. Let $h_{\al,\be}:\mG_{\al}\cap\mG_{t_3t_4}\rightarrow\mG_\be$ be the surjection in Lemma \ref{surjection al edge be 2star}(a) such that $h_{\al,\be}\qty(\qty(f^{-1}_\al(g_\be))_1)=g_\be$ for $g_\be\in\mG_\be$. Let $h_{\al',\be'}:\mG_{\al'}\cap\mG_{t_2't_3'}\rightarrow\mG_{\be'}$ be the surjection in Lemma \ref{surjection al edge be 2star}(b) such that $h_{\al',\be'}\qty(\qty(f^{-1}_{\al'}(g_{\be'}))_1)=g_{\be'}$ for $g_{\be'}\in\mG_{\be'}$. Define a map $h_{\al,\al',\be,\be'}:\mG_{\al}\cap\mG_{\al'}\cap\mG_{t_3t_4}\cap\mG_{t_2't_3'}\rightarrow\mG_{\be}\times\mG_{\be'}$ by $h_{\al,\al',\be,\be'}(g)\coloneqq\qty(h_{\al,\be}(g), h_{\al',\be'}(g))$.  
		
		\medskip
		For $(g_\be,g_{\be'})\in\mH_{\al,\al',\be,\be'}$, $\qty(f_\al^{-1}(g_\be))_1=\qty(f_{\al'}^{-1}(g_{\be'}))_1$ must contain all of $\al,\al',t_3t_4,t_2't_3'$. First note that $\al\cap\al'=\al\cap(t_3t_4)=\al'\cap(t_2't_3') =\al\cap(t_2't_3')=\varnothing$. If $(t_3t_4)\cap\al'\neq\varnothing$, $t_4$ must be equal to either $s_1'$ or $s_2'$ since $t_4,s_1',s_2'$ are incident to their degree 1 vertices and $t_3$ is incident to the degree 2 vertex, but then $t_3t_4$ and $\al'$ cannot coexist. Similarly, if $(t_3t_4)\cap (t_2't_3' )\neq\varnothing$, then $t_3$ must be equal to either $t_2'$ or $t_3'$ because $t_3$ is incident to $\be$'s degree 2 vertex, $t_2'\neq t_3'$ are incident to $\be'$'s degree 2 vertex, and $t_4$ is incident to a degree 1 vertex. But then $t_3t_4$ and $t_2't_3'$ cannot coexist. Thus we may assume that $(t_3t_4)\cap\al'=(t_3t_4)\cap(t_2't_3')=\varnothing$, since otherwise $\mH_{\al,\al',\be,\be'}$ will be empty. Since $\al,\al',t_3t_4,t_2't_3'$ are now all disjoint, $h_{\al,\al',\be,\be'}$ is a surjection from $\mG_{\al}\cap\mG_{\al'}\cap\mG_{t_3t_4}\cap\mG_{t_2't_3'}$ onto $\mH_{\al,\al',\be,\be'}$. Indeed, for any $(g_\be,g_{\be'})\in\mH_{\al,\al',\be,\be'}$, letting $g\coloneqq\qty(f_\al^{-1}(g_\be))_1=\qty(f_{\al'}^{-1}(g_{\be'}))_1\in\mG_{\al}\cap\mG_{\al'}\cap\mG_{t_3t_4}\cap\mG_{t_2't_3'}$ yields $h_{\al,\al',\be,\be'}(g)=(g_\be,g_{\be'})$.  Therefore, we have
		\begin{align*}
			\qty|\mH_{\al,\al',\be,\be'}|\le\qty|\mG_{\al}\cap\mG_{\al'}\cap\mG_{t_3t_4}\cap\mG_{t_2't_3'}|=(N-9)!!
		\end{align*}
		and
		\begin{align*}
			\p(I_\al=I_{\al'}=1,J_{\be\al}=J_{\be'\al'}=1)\le\frac{1}{((N-1))_4}\frac{1}{(N-1)^2}
		\end{align*}
		in this case. Since we may estimate the possibilities of a quadruplet $\al,\al',\be,\be'$ in this class as (using Lemma \ref{cardinality estimates}(b) for the possibilities of $\be$ and Lemma \ref{cardinality estimates}(c) for the possibilities of $\be'$)
			\begin{align*}
				\qty|\sum_{\substack{\al,\al'\in\Ga_{11}\\ \al\cap\al'=\varnothing}}\sum_{\substack{\be,\be'\in\Ga_{12}\\ \text{Case (b2)} }}|\le\underbrace{\qty(\frac{(n_1)_2}{2})^2}_{\al,\al'}\times\underbrace{4n_1n_2}_{\be}\times\underbrace{2 n_2}_{\be'}=2\qty((n_1)_2)^2n_1{n_2}^2,
		\end{align*}
		 this leads to the estimate
		\begin{align*}
			&\sum_{\substack{\al,\al'\in\Ga_{11}\\ \al\cap\al'=\varnothing}}\sum_{\substack{\be,\be'\in\Ga_{12}\\ \text{Case (b2)} }}\e\qty[I_\al I_{\al'} J_{\be\al}J_{\be'\al'}]-\sum_{\substack{\al,\al'\in\Ga_{11}\\ \al\cap\al'=\varnothing}}\sum_{\substack{\be,\be'\in\Ga_{12}\\ \text{Case (b2)} }}p_\al p_{\al'} p_\be p_{\be'}\\
			\le&\sum_{\substack{\al,\al'\in\Ga_{11}\\ \al\cap\al'=\varnothing}}\sum_{\substack{\be,\be'\in\Ga_{12}\\ \text{Case (b2)} }}\qty(\frac{1}{((N-1))_4(N-1)^2}-\frac{1}{(N-1)^4(N-3)^2})\\
			=&\sum_{\substack{\al,\al'\in\Ga_{11}\\ \al\cap\al'=\varnothing}}\sum_{\substack{\be,\be'\in\Ga_{12}\\ \text{Case (b2)} }}\frac{1}{((N-1))_4(N-1)^2}\qty(1-\frac{(N-5)(N-7)}{(N-1)(N-3)})\\
			\le& \sum_{\substack{\al,\al'\in\Ga_{11}\\ \al\cap\al'=\varnothing}}\sum_{\substack{\be,\be'\in\Ga_{12}\\ \text{Case (b2)} }}\frac{1}{((N-1))_4(N-1)^2}\frac{8}{N-3}\\
			\le&\frac{16\qty((n_1)_2)^2n_1{n_2}^2}{((N-1))_4(N-1)^2(N-3)}.
		\end{align*}

		\medskip
		\emph{Case} (b3): \emph{$\al$ and $\be$ share two degree 1 vertices, $\al'$ and $\be'$ share one degree 1 vertex only.}
		
		\medskip
		By symmetry, we obtain the same estimate as in Case (b2):
		\begin{align*}
			\sum_{\substack{\al,\al'\in\Ga_{11}\\ \al\cap\al'=\varnothing}}\sum_{\substack{\be,\be'\in\Ga_{12}\\ \text{Case (b3)} }}\e\qty[I_\al I_{\al'} J_{\be\al}J_{\be'\al'}]-\sum_{\substack{\al,\al'\in\Ga_{11}\\ \al\cap\al'=\varnothing}}\sum_{\substack{\be,\be'\in\Ga_{12}\\ \text{Case (b3)} }}p_\al p_{\al'} p_\be p_{\be'}\le\frac{16\qty((n_1)_2)^2n_1{n_2}^2}{((N-1))_4(N-1)^2(N-3)}.
		\end{align*}
		
		\medskip
		\emph{Case} (b4):	\emph{$\al$ and $\be$ share two degree 1 vertices, $\al'$ and $\be'$ share two degree 1 vertices.}
		
		\smallskip
		Suppose that $s_1=t_1$, $s_2=t_4$, $s_1'=t_1'$ and $s_2'=t_4'$ without loss of generality. Let $h_{\al,\be}:\mG_{\al}\cap\mG_{t_2t_3}\rightarrow\mG_\be$ and $h_{\al',\be'}:\mG_{\al'}\cap\mG_{t_2't_3'}\rightarrow\mG_{\be'}$ be the surjections in Lemma \ref{surjection al edge be 2star}(b) such that $h_{\al,\be}\qty(\qty(f^{-1}_\al(g_\be))_1)=g_\be$ for $g_\be\in\mG_\be$ and $h_{\al',\be'}\qty(\qty(f^{-1}_{\al'}(g_{\be'}))_1)=g_{\be'}$ for $g_{\be'}\in\mG_{\be'}$. Define a map $h_{\al,\al',\be,\be'}:\mG_{\al}\cap\mG_{\al'}\cap\mG_{t_2t_3}\cap\mG_{t_2't_3'}\rightarrow\mG_{\be}\times\mG_{\be'}$ by $h_{\al,\al',\be,\be'}(g)\coloneqq\qty(h_{\al,\be}(g), h_{\al',\be'}(g))$. 
		
		\medskip
		For $(g_\be,g_{\be'})\in\mH_{\al,\al',\be,\be'}$, $\qty(f_\al^{-1}(g))_1=\qty(f_{\al'}^{-1}(g'))_1$ must contain all of $\al,\al',t_2t_3,t_2't_3'$. First note that $\al\cap\al'=\al\cap(t_2t_3) =\al'\cap(t_2't_3')=\al\cap(t_2't_3')=\al'\cap(t_2t_3)=\varnothing$. Either Case (b4i): $\{t_2,t_3\}\cap\{t_2',t_3'\}=\varnothing$ or Case (b4ii):  $\{t_2,t_3\}\cap\{t_2',t_3'\}\neq\varnothing$ happens.  
		
		\medskip
		First consider Case (b4i): $\{t_2,t_3 \}\cap\{t_2',t_3' \}=\varnothing$. Then $(t_2t_3)\cap(t_2't_3')=\varnothing$ and $h_{\al,\al',\be,\be'}$ is a surjection from $\mG_{\al}\cap\mG_{\al'}\cap\mG_{t_2t_3}\cap\mG_{t_2't_3'}$ onto $\mH_{\al,\al',\be,\be'}$. Indeed, for any $(g_\be,g_{\be'})\in\mH_{\al,\al',\be,\be'}$, letting $g\coloneqq\qty(f_\al^{-1}(g_\be))_1=\qty(f_{\al'}^{-1}(g_{\be'}))_1\in\mG_{\al}\cap\mG_{\al'}\cap\mG_{t_2t_3}\cap\mG_{t_2't_3'}$ yields $h_{\al,\al',\be,\be'}(g)=(g_\be,g_{\be'})$. Therefore, we have
		\begin{align*}
			\qty|\mH_{\al,\al',\be,\be'}|\le\qty|\mG_{\al}\cap\mG_{\al'}\cap\mG_{t_2t_3}\cap\mG_{t_2't_3'}|=(N-9)!!
		\end{align*}
		and
		\begin{align*}
			\p(I_\al=I_{\al'}=1,J_{\be\al}=J_{\be'\al'}=1)\le\frac{1}{((N-1))_4}\frac{1}{(N-1)^2}
		\end{align*}
		in this case. Since we may estimate the possibilities of a quadruplet $\al,\al',\be,\be'$ in this class as (using Lemma \ref{cardinality estimates}(c) for the possibilities of $\be,\be'$)
			\begin{align*}
				\qty|\sum_{\substack{\al,\al'\in\Ga_{11}\\ \al\cap\al'=\varnothing}}\sum_{\substack{\be,\be'\in\Ga_{12}\\ \text{Case (b4i)} }}|\le\underbrace{\qty(\frac{(n_1)_2}{2})^2}_{\al,\al'}\times\underbrace{ \qty(2n_2)^2}_{\be,\be'}=\qty((n_1)_2)^2{n_2}^2,
		\end{align*}
	this leads to the estimate
		\begin{align*}
			&\sum_{\substack{\al,\al'\in\Ga_{11}\\ \al\cap\al'=\varnothing}}\sum_{\substack{\be,\be'\in\Ga_{12}\\ \text{Case (b4i)} }}\e\qty[I_\al I_{\al'} J_{\be\al}J_{\be'\al'}]-\sum_{\substack{\al,\al'\in\Ga_{11}\\ \al\cap\al'=\varnothing}}\sum_{\substack{\be,\be'\in\Ga_{12}\\ \text{Case (b4i)} }}p_\al p_{\al'} p_\be p_{\be'}\\
			\le&\sum_{\substack{\al,\al'\in\Ga_{11}\\ \al\cap\al'=\varnothing}}\sum_{\substack{\be,\be'\in\Ga_{12}\\ \text{Case (b4i)} }}\qty(\frac{1}{((N-1))_4(N-1)^2}-\frac{1}{(N-1)^4(N-3)^2})\\
			=&\sum_{\substack{\al,\al'\in\Ga_{11}\\ \al\cap\al'=\varnothing}}\sum_{\substack{\be,\be'\in\Ga_{12}\\ \text{Case (b4i)} }}\frac{1}{((N-1))_4(N-1)^2}\qty(1-\frac{(N-5)(N-7)}{(N-1)(N-3)})\\
			\le&\sum_{\substack{\al,\al'\in\Ga_{11}\\ \al\cap\al'=\varnothing}}\sum_{\substack{\be,\be'\in\Ga_{12}\\ \text{Case (b4i)} }}\frac{1}{((N-1))_4}\frac{1}{(N-1)^2}\frac{8}{N-3}\\
			\le&\frac{8\qty((n_1)_2)^2{n_2}^2}{((N-1))_4(N-1)^2(N-3)}.
		\end{align*}
		For Case (b4ii): $\{t_2,t_3\}\cap\{t_2',t_3' \}\neq\varnothing$, $\be$'s degree 2 vertex and $\be'$'s degree 2 vertex are equal to each other and $\{t_2,t_3\}=\{t_2',t_3'\}$ (equivalent to $t_2t_3=t_2't_3'$\footnote{$(t_2,t_3)=(t_3',t_2')$ can happen.}) holds. $h_{\al,\al',\be,\be'}$ similarly defines a surjection from $\mG_{\al}\cap\mG_{\al'}\cap\mG_{t_2t_3}$ onto $\mH_{\al,\al',\be,\be'}$. We have
		\begin{align*}
			\qty|\mH_{\al,\al',\be,\be'}|\le\qty|\mG_{\al}\cap\mG_{\al'}\cap\mG_{t_2t_3}|=(N-7)!!
		\end{align*}
		and
		\begin{align*}
			\p(I_\al=I_{\al'}=1,J_{\be\al}=J_{\be'\al'}=1)\le\frac{1}{((N-1))_3}\frac{1}{(N-1)^2}
		\end{align*}
		in this case. The two degree 1 vertices of $\be'$ are those of $\al'$. Since $\{t_2,t_3\}=\{t_2',t_3' \}$, the degree 2 vertex of $\be'$ is that of $\be$ in this case. Therefore, if $\al,\al',\be$ are given, there is no degree of freedom about the vertices of $\be'$ in this case. Keeping this in mind, we estimate the possibilities of a quadruplet $\al,\al',\be,\be'$ in this class as (also using Lemma \ref{cardinality estimates}(c) for the possibilities of $\be$)
		\begin{align*}
				\qty|\sum_{\substack{\al,\al'\in\Ga_{11}\\ \al\cap\al'=\varnothing}}\sum_{\substack{\be,\be'\in\Ga_{12}\\ \text{Case (b4ii)} }}|\le&\underbrace{\qty(\frac{(n_1)_2}{2})^2}_{\al,\al'}\times\underbrace{ 2n_2}_{\be}\times\underbrace{2}_{\be'}=\qty((n_1)_2)^2n_2.
		\end{align*}
	Thus we estimate
		\begin{align*}
			\sum_{\substack{\al,\al'\in\Ga_{11}\\ \al\cap\al'=\varnothing}}\sum_{\substack{\be,\be'\in\Ga_{12}\\ \text{Case (b4ii)} }}\e\qty[I_\al I_{\al'} J_{\be\al}J_{\be'\al'}]\le&\sum_{\substack{\al,\al'\in\Ga_{11}\\ \al\cap\al'=\varnothing}}\sum_{\substack{\be,\be'\in\Ga_{12}\\ \text{Case (b4ii)} }}\frac{1}{((N-1))_3(N-1)^2}\\
			\le&  \frac{\qty((n_1)_2)^2n_2}{((N-1))_3(N-1)^2}.
		\end{align*}
		
		\medskip
		Therefore, we obtain an estimate for Case (b4) as
		\begin{align*}
			&	\sum_{\substack{\al,\al'\in\Ga_{11}\\ \al\cap\al'=\varnothing}}\sum_{\substack{\be,\be'\in\Ga_{12}\\ \text{Case (b4)} }}\e\qty[I_\al I_{\al'} J_{\be\al}J_{\be'\al'}]-\sum_{\substack{\al,\al'\in\Ga_{11}\\ \al\cap\al'=\varnothing}}\sum_{\substack{\be,\be'\in\Ga_{12}\\ \text{Case (b4)} }}p_\al p_{\al'} p_\be p_{\be'}\\
			\le& \underbrace{\frac{8\qty((n_1)_2)^2{n_2}^2}{((N-1))_4(N-1)^2(N-3)}}_{\text{Case (b4i)}}+ \underbrace{\frac{\qty((n_1)_2)^2n_2}{((N-1))_3(N-1)^2}}_{\text{Case (b4ii)}}.
		\end{align*}
		
		\medskip
		From Cases (b1)--(b4), we conclude that
		\begin{equation}\label{third variance second line case b}
			\begin{split}
			&\sum_{\substack{\al,\al'\in\Ga_{11}\\ \al\cap\al'=\varnothing}}\sum_{\substack{\be\in\Ga_{12}\setminus\{\al\}\\ \be\cap\al\neq\varnothing }}\sum_{\substack{\be'\in\Ga_{12}\setminus\{\al'\}\\ \be'\cap\al'\neq\varnothing }}\e\qty[I_\al I_{\al'} J_{\be\al}J_{\be'\al'}]-\sum_{\substack{\al,\al'\in\Ga_{11}\\ \al\cap\al'=\varnothing}}\sum_{\substack{\be\in\Ga_{12}\setminus\{\al\}\\ \be\cap\al\neq\varnothing }}\sum_{\substack{\be'\in\Ga_{12}\setminus\{\al'\}\\ \be'\cap\al'\neq\varnothing }}p_\al p_{\al'} p_\be p_{\be'}\\
			\le&\underbrace{\frac{32\qty((n_1)_2)^2{n_1}^2{n_2}^2}{((N-1))_4(N-1)^2(N-3)}+\frac{4\qty((n_1)_2)^2n_1n_2}{((N-1))_3(N-1)^2}}_{\text{Case (b1)}}+\underbrace{\frac{32\qty((n_1)_2)^2n_1{n_2}^2}{((N-1))_4(N-1)^2(N-3)}}_{\text{Case (b2)+Case(b3)}}\\
			&\quad\quad\quad+\underbrace{\frac{8\qty((n_1)_2)^2{n_2}^2}{((N-1))_4(N-1)^2(N-3)}+ \frac{\qty((n_1)_2)^2n_2}{((N-1))_3(N-1)^2}}_{\text{Case (b4)}}.
			\end{split}
		\end{equation}
		
		\medskip
		\textbf{(c) $\Ga_{1j}$ is possible isolated 2-stars, $\Ga_{1k}$ is possible isolated edges ($j=2,k=1$).}
		
		\medskip
		Let $h_{\al,\be}:\mG_{\al}\times[N-1]\rightarrow\mG_\be$ be the surjection in Lemma \ref{surjection al 2star be edge} with the property: For any $g_\be\in\mG_\be$, there exists $c_\al(g_\be)\in[N-1]$ such that $h_{\al,\be}\qty(\qty(f^{-1}_\al(g_\be))_1,c_\al(g_\be))=g_\be$. Similarly, let $h_{\al',\be'}:\mG_{\al'}\times[N-1]\rightarrow\mG_{\be'}$ be the surjection in Lemma \ref{surjection al 2star be edge} with the property: For any $g_{\be'}\in\mG_{\be'}$, there exists $c_{\al'}(g_{\be'})\in[N-1]$ such that $h_{\al',\be'}\qty(\qty(f^{-1}_{\al'}(g_{\be'}))_1,c_{\al'}(g_{\be'}))=g_{\be'}$. Define a map $h_{\al,\al',\be,\be'}:\qty(\mG_{\al}\cap\mG_{\al'})\times[N-1]\times[N-1]\rightarrow\mG_{\be}\times\mG_{\be'}$ by $h_{\al,\al',\be,\be'}(g,c,c')\coloneqq\qty(h_{\al,\be}(g,c), h_{\al',\be'}(g,c'))$.  For any $(g_\be,g_{\be'})\in\mH_{\al,\al',\be,\be'}$, letting $g\coloneqq\qty(f_\al^{-1}(g_\be))_1=\qty(f_{\al'}^{-1}(g_{\be'}))_1\in\mG_{\al}\cap\mG_{\al'}$ yields $h_{\al,\al',\be,\be'}(g,c_\al(g_\be),c_{\al'}(g_{\be'}))=\qty(h_{\al,\be}(g,c_\al(g_\be)), h_{\al',\be'}(g,c_{\al'}(g_{\be'})))=(g_\be,g_{\be'})$. Therefore, $h_{\al,\al',\be,\be'}$ is a surjection onto $\mH_{\al,\al',\be,\be'}$ and we have
		\begin{align*}
			\qty|\mH_{\al,\al',\be,\be'}|\le\qty|\qty(\mG_{\al}\cap\mG_{\al'})\times[N-1]\times[N-1]|=(N-9)!!(N-1)^2.
		\end{align*}
		Consequently, we have
		\begin{align*}
			\p(I_\al=I_{\al'}=1,J_{\be\al}=J_{\be'\al'}=1)\le\frac{(N-1)^2}{((N-1))_4}\frac{1}{(N-1)^2(N-3)^2}
		\end{align*}
		in this case. This, combined with
	\begin{align*}
				\qty|\sum_{\substack{\al,\al'\in\Ga_{12}\\ \al\cap\al'=\varnothing}}\sum_{\substack{\be\in\Ga_{11}\setminus\{\al\}\\ \be\cap\al\neq\varnothing }}\sum_{\substack{\be'\in\Ga_{11}\setminus\{\al'\}\\ \be'\cap\al'\neq\varnothing }}|\le\underbrace{\qty((n_1)_2n_2)^2}_{\al,\al'}\times\underbrace{\qty(2 n_1)^2}_{\be,\be'}=4\qty((n_1)_2)^2{n_1}^2{n_2}^2
		\end{align*}
	(cf.\ Lemma \ref{cardinality estimates}(b)), leads to the estimate
		\begin{align}
			&\sum_{\substack{\al,\al'\in\Ga_{12}\\ \al\cap\al'=\varnothing}}\sum_{\substack{\be\in\Ga_{11}\setminus\{\al\}\\ \be\cap\al\neq\varnothing }}\sum_{\substack{\be'\in\Ga_{11}\setminus\{\al'\}\\ \be'\cap\al'\neq\varnothing }}\e\qty[I_\al I_{\al'} J_{\be\al}J_{\be'\al'}]-\sum_{\substack{\al,\al'\in\Ga_{12}\\ \al\cap\al'=\varnothing}}\sum_{\substack{\be\in\Ga_{11}\setminus\{\al\}\\ \be\cap\al\neq\varnothing }}\sum_{\substack{\be'\in\Ga_{11}\setminus\{\al'\}\\ \be'\cap\al'\neq\varnothing }}p_\al p_{\al'} p_\be p_{\be'}\notag\\
			\le&\sum_{\substack{\al,\al'\in\Ga_{12}\\ \al\cap\al'=\varnothing}}\sum_{\substack{\be\in\Ga_{11}\setminus\{\al\}\\ \be\cap\al\neq\varnothing }}\sum_{\substack{\be'\in\Ga_{11}\setminus\{\al'\}\\ \be'\cap\al'\neq\varnothing }}\qty(\frac{1}{((N-1))_4(N-3)^2}-\frac{1}{(N-1)^4(N-3)^2})\notag\\
			=&\sum_{\substack{\al,\al'\in\Ga_{12}\\ \al\cap\al'=\varnothing}}\sum_{\substack{\be\in\Ga_{11}\setminus\{\al\}\\ \be\cap\al\neq\varnothing }}\sum_{\substack{\be'\in\Ga_{11}\setminus\{\al'\}\\ \be'\cap\al'\neq\varnothing }}\frac{1}{((N-1))_4(N-3)^2}\qty(1-\frac{((N-1))_4(N-3)^2}{(N-1)^4(N-3)^2})\notag\\
			=&\sum_{\substack{\al,\al'\in\Ga_{12}\\ \al\cap\al'=\varnothing}}\sum_{\substack{\be\in\Ga_{11}\setminus\{\al\}\\ \be\cap\al\neq\varnothing }}\sum_{\substack{\be'\in\Ga_{11}\setminus\{\al'\}\\ \be'\cap\al'\neq\varnothing }}\frac{1}{((N-1))_4(N-3)^2}\qty(1-\frac{(N-3)(N-5)(N-7)}{(N-1)^3})\notag\\
			\le&  \frac{4\qty((n_1)_2)^2{n_1}^2{n_2}^2}{((N-1))_4(N-3)^2}\frac{12}{N-1}\notag\\
			=&  \frac{48\qty((n_1)_2)^2{n_1}^2{n_2}^2}{((N-1))_4(N-3)^2(N-1)}.\label{third variance second line case c}
		\end{align}

		\medskip
		\textbf{(d) $\Ga_{1j}$ is possible isolated 2-stars, $\Ga_{1k}$ is possible isolated 2-stars ($j=k=2$).}
		
		\smallskip
		Let $\al=(s_1s_2)\cup(s_3s_4)$, $\be=(t_1t_2)\cup(t_3t_4)$, $\al'=(s_1's_2')\cup(s_3's_4')$ and $\be'=(t_1't_2')\cup(t_3't_4')$, where $s_2$ and $s_3$ are incident to $\al$'s degree 2 vertex, $t_2$ and $t_3$ are incident to $\be$'s degree 2 vertex, $s_2'$ and $s_3'$ are incident to $\al'$'s degree 2 vertex, and $t_2'$ and $t_3'$ are incident to $\be'$'s degree 2 vertex. In view of Lemma \ref{surjection 2star 2star}, there are four possible cases in total about $\al\cap\be\neq\varnothing,\al\neq\be$ and $\al'\cap\be'\neq\varnothing,\al'\neq\be'$. We begin with the most complicated case.
		
		\medskip
		\emph{Case} (d1): \emph{$\al$ and $\be$ share one degree 1 vertex only, $\al'$ and $\be'$ share one degree 1 vertex only.}

		\medskip
		Suppose that $s_4=t_1$ and $s_4'=t_1'$ without loss of generality.  Let $h_{\al,\be}:\qty(\mG_{\al}\cap\mG_{t_3t_4})\times[N-1]\rightarrow\mG_\be$ be the surjection in Lemma \ref{surjection 2star 2star}(a) with the property: For any $g_\be\in\mG_\be$, there exists $c_\al(g_\be)\in[N-1]$ such that $h_{\al,\be}\qty(\qty(f^{-1}_\al(g_\be))_1,c_\al(g_\be))=g_\be$. Similarly, let $h_{\al',\be'}:\qty(\mG_{\al'}\cap\mG_{t_3't_4'})\times[N-1]\rightarrow\mG_{\be'}$ be the surjection in Lemma \ref{surjection 2star 2star}(a) with the property: For any $g_{\be'}\in\mG_{\be'}$, there exists $c_{\al'}(g_{\be'})\in[N-1]$ such that $h_{\al',\be'}\qty(\qty(f^{-1}_{\al'}(g_{\be'}))_1,c_{\al'}(g_{\be'}))=g_{\be'}$. Define a map $h_{\al,\al',\be,\be'}:\qty(\mG_{\al}\cap\mG_{\al'}\cap\mG_{t_3t_4}\cap\mG_{t_3't_4'})\times[N-1]\times[N-1]\rightarrow\mG_{\be}\times\mG_{\be'}$ by $h_{\al,\al',\be,\be'}(g,c,c')\coloneqq\qty(h_{\al,\be}(g,c), h_{\al',\be'}(g,c'))$.
		
		\medskip
		For $(g_\be,g_{\be'})\in\mH_{\al,\al',\be,\be'}$, $\qty(f_\al^{-1}(g))_1=\qty(f_{\al'}^{-1}(g'))_1$ must contain all of $\al,\al',t_3t_4,t_3't_4'$. Note that $\al\cap\al'=\al\cap(t_3t_4)=\al'\cap(t_3't_4')=\varnothing$.
		
		\medskip
		If $(t_3t_4)\cap\al'\neq\varnothing$, $t_3t_4$ and $\al'$ must share a vertex. When $t_3t_4$ and $\al'$ share a degree 1 vertex, $t_4$ must be equal to either $s_1'$ or $s_4'$. In the case when $t_4=s_1'$ (resp.\ $t_4=s_4'$), we may assume that $t_3=s_2'$ and $t_3t_4=s_1's_2'$ (resp.\ $t_3=s_3'$ and $t_3t_4=s_3's_4'$), because otherwise $t_3t_4$ and $\al'$ cannot coexist and $\mH_{\al,\al',\be,\be'}$ will be empty. When $t_3t_4$ and $\al'$ share a degree 2 vertex, $t_3$ must be equal to either $s_2'$ or $s_3'$. In the case when $t_3=s_2'$ (resp.\ $t_3=s_3'$), we may assume that $t_4=s_1'$ and $t_3t_4=s_1's_2'$ (resp.\ $t_4=s_4'$ and $t_3t_4=s_3's_4'$), because otherwise $t_3t_4$ and $\al'$ cannot coexist and $\mH_{\al,\al',\be,\be'}$ will be empty. To sum up, if $(t_3t_4)\cap\al'\neq\varnothing$, we may assume that $t_3t_4$ coincides with an edge of $\al'$, and then $(t_3t_4)\cap(t_3't_4')=\varnothing$ follows by $\al'\cap(t_3't_4')=\varnothing$. Similarly, if $(t_3't_4')\cap\al\neq\varnothing$, we may assume that $t_3't_4'$ coincides with an edge of $\al$, and then $(t_3t_4)\cap(t_3't_4')=\varnothing$ follows by $\al\cap(t_3t_4)=\varnothing$.
		
		\medskip
		If $\{t_3,t_4\}\cap\{t_3',t_4' \}\neq\varnothing$, we may assume that $\{t_3,t_4\}=\{t_3',t_4'\}$ (and thus $t_3t_4=t_3't_4'$), because otherwise $t_3t_4$ and $t_3't_4'$ cannot coexist and $\mH_{\al,\al',\be,\be'}$ will be empty.
		
		\medskip
		From these observations, we realize that it suffices to consider the following five subcases:
		{\setlength{\leftmargini}{29.5pt}  	
			\begin{itemize}
				\setlength{\labelsep}{7pt}     
				\setlength{\itemsep}{3pt}     
				
				\item[] Case (d1i): $\{t_3,t_4\}\cap\{t_3',t_4'\}=\varnothing$ (as the subsets of half-edges\footnote{If $t_3$ and $t_3'$ are incident to the same degree 2 vertex, $t_3\neq t_3'$ and $t_4\neq t_4'$, then $\{t_3,t_4 \}\cap\{t_3',t_4' \}=\varnothing$ but $(t_3t_4)\cap(t_3't_4')\neq\varnothing$ because those possible edges share that degree 2 vertex.}), $(t_3t_4)\cap\al'=(t_3't_4')\cap\al=\varnothing$.
				
				\item[] Case (d1ii): $\{t_3,t_4\}=\{t_3',t_4'\}$, $(t_3t_4)\cap\al'=(t_3't_4')\cap\al=\varnothing$.
				
				\item[] Case (d1iii):  $t_3t_4$ is an edge of $\al'$, $(t_3't_4')\cap\al=\varnothing$.
				
				\item[] Case (d1iv): $t_3't_4'$ is an edge of $\al$, $(t_3t_4)\cap\al'=\varnothing$.
				
				\item[] Case (d1v):  $t_3t_4$ is an edge of $\al'$, $t_3't_4'$ is an edge of $\al$.

		\end{itemize} }
		
		\medskip
		For Case (d1i), $h_{\al,\al',\be,\be'}$ defines a surjection from $\qty(\mG_{\al}\cap\mG_{\al'}\cap\mG_{t_3t_4}\cap\mG_{t_3't_4'})\times[N-1]\times[N-1]$ onto $\mH_{\al,\al',\be,\be'}$. Indeed, for any $(g_\be,g_{\be'})\in\mH_{\al,\al',\be,\be'}$, letting $g\coloneqq\qty(f_\al^{-1}(g_\be))_1=\qty(f_{\al'}^{-1}(g_{\be'}))_1\in\mG_{\al}\cap\mG_{\al'}\cap\mG_{t_3t_4}\cap\mG_{t_3't_4'}$ yields $h_{\al,\al',\be,\be'}(g,c_\al(g_\be),c_{\al'}(g_{\be'}))=\qty(h_{\al,\be}(g,c_\al(g_\be)), h_{\al',\be'}(g,c_{\al'}(g_{\be'})))=(g_\be,g_{\be'})$. Therefore, we have
		\begin{align*}
			\qty|\mH_{\al,\al',\be,\be'}|\le&\left|\qty(\mG_{\al}\cap\mG_{\al'}\cap\mG_{t_3t_4}\cap\mG_{t_3't_4'})\times[N-1]\times[N-1] \right|\\
			=&(N-13)!!(N-1)^2
		\end{align*}
		and
		\begin{align*}
			\p(I_\al=I_{\al'}=1,J_{\be\al}=J_{\be'\al'}=1)\le\frac{(N-1)^2}{((N-1))_6}\frac{1}{(N-1)^2(N-3)^2}
		\end{align*}
		in this case. Since we may estimate the possibilities of a quadruplet $\al,\al',\be,\be'$ in this class as (using Lemma \ref{cardinality estimates}(d) for the possibilities of $\be,\be'$)
		\begin{align*}
				\qty|\sum_{\substack{\al,\al'\in\Ga_{12}\\ \al\cap\al'=\varnothing}}\sum_{\substack{\be,\be'\in\Ga_{12}\\ \text{Case (d1i)} }}|\le&\underbrace{\qty((n_1)_2n_2)^2}_{\al,\al'}\times\underbrace{\qty(4n_1 n_2)^2}_{\be,\be'}=16\qty((n_1)_2)^2{n_1}^2{n_2}^4,
		\end{align*}
		 this leads to the estimate
		\begin{align*}
			&\sum_{\substack{\al,\al'\in\Ga_{12}\\ \al\cap\al'=\varnothing}}\sum_{\substack{\be,\be'\in\Ga_{12}\\ \text{Case (d1i)} }}\e\qty[I_\al I_{\al'} J_{\be\al}J_{\be'\al'}]-\sum_{\substack{\al,\al'\in\Ga_{12}\\ \al\cap\al'=\varnothing}}\sum_{\substack{\be,\be'\in\Ga_{12}\\ \text{Case (d1i)} }}p_\al p_{\al'} p_\be p_{\be'}\\
			\le&\sum_{\substack{\al,\al'\in\Ga_{12}\\ \al\cap\al'=\varnothing}}\sum_{\substack{\be,\be'\in\Ga_{12}\\ \text{Case (d1i)} }}\qty(\frac{1}{((N-1))_6(N-3)^2}-\frac{1}{(N-1)^4(N-3)^4})\\
			=&\sum_{\substack{\al,\al'\in\Ga_{12}\\ \al\cap\al'=\varnothing}}\sum_{\substack{\be,\be'\in\Ga_{12}\\ \text{Case (d1i)} }}\frac{1}{((N-1))_6(N-3)^2}\qty(1-\frac{((N-1))_6(N-3)^2}{(N-1)^4(N-3)^4})\\
			=&\sum_{\substack{\al,\al'\in\Ga_{12}\\ \al\cap\al'=\varnothing}}\sum_{\substack{\be,\be'\in\Ga_{12}\\ \text{Case (d1i)} }}\frac{1}{((N-1))_6(N-3)^2}\qty(1-\frac{(N-5)(N-7)(N-9)(N-11)}{(N-1)^3(N-3)})\\
			\le&  \frac{16\qty((n_1)_2)^2{n_1}^2{n_2}^4}{((N-1))_6(N-3)^2}\frac{26}{N-3}\\
			\le&  \frac{416\qty((n_1)_2)^2{n_1}^2{n_2}^4}{((N-1))_6(N-3)^3}.
		\end{align*}
		
		\medskip
		For Case (d1ii), $h_{\al,\al',\be,\be'}$ defines a surjection from $\qty(\mG_{\al}\cap\mG_{\al'}\cap\mG_{t_3t_4})\times[N-1]\times[N-1]$ onto $\mH_{\al,\al',\be,\be'}$. Indeed, for any $(g_\be,g_{\be'})\in\mH_{\al,\al',\be,\be'}$, letting $g\coloneqq\qty(f_\al^{-1}(g_\be))_1=\qty(f_{\al'}^{-1}(g_{\be'}))_1\in\mG_{\al}\cap\mG_{\al'}\cap\mG_{t_3t_4}$ yields $h_{\al,\al',\be,\be'}(g,c_\al(g_\be),c_{\al'}(g_{\be'}))=\qty(h_{\al,\be}(g,c_\al(g_\be)), h_{\al',\be'}(g,c_{\al'}(g_{\be'})))=(g_\be,g_{\be'})$. Therefore, we have
		\begin{align*}
			\qty|\mH_{\al,\al',\be,\be'}|\le&\qty|\qty(\mG_{\al}\cap\mG_{\al'}\cap\mG_{t_3t_4})\times[N-1]\times[N-1]|\\
			=&(N-11)!!(N-1)^2
		\end{align*}
		and
		\begin{align*}
			\p(I_\al=I_{\al'}=1,J_{\be\al}=J_{\be'\al'}=1)\le\frac{(N-1)^2}{((N-1))_5}\frac{1}{(N-1)^2(N-3)^2}
		\end{align*}
		in this case. One of the degree 1 vertices of $\be'$ is one of the degree 1 vertices of $\al'$. Since $\{t_3,t_4 \}=\{t_3',t_4' \}$, the other degree 1 vertex of $\be'$ is the degree 1 vertex of $\be$ not shared by $\al$, and the degree 2 vertex of $\be'$ is that of $\be$. Therefore, if $\al,\al',\be$ are given, there are only two possibilities about the vertices of $\be'$ in this case; namely, two possibilities about the degree 1 vertex from $\al'$. Keeping this in mind, we estimate the possibilities of a quadruplet $\al,\al',\be,\be'$ in this class as (also using Lemma \ref{cardinality estimates}(d) for the possibilities of $\be$)
		\begin{align*}
				\qty|\sum_{\substack{\al,\al'\in\Ga_{12}\\ \al\cap\al'=\varnothing}}\sum_{\substack{\be,\be'\in\Ga_{12}\\ \text{Case (d1ii)} }}|\le&\underbrace{\qty((n_1)_2n_2)^2}_{\al,\al'}\times\underbrace{4 n_1 n_2}_{\be}\times\underbrace{2\times2}_{\be'}=16\qty((n_1)_2)^2n_1{n_2}^3.
		\end{align*}
	Thus we estimate
		\begin{align*}
			\sum_{\substack{\al,\al'\in\Ga_{12}\\ \al\cap\al'=\varnothing}}\sum_{\substack{\be,\be'\in\Ga_{12}\\ \text{Case (d1ii)} }}\e\qty[I_\al I_{\al'} J_{\be\al}J_{\be'\al'}]\le&\sum_{\substack{\al,\al'\in\Ga_{12}\\ \al\cap\al'=\varnothing}}\sum_{\substack{\be,\be'\in\Ga_{12}\\ \text{Case (d1ii)} }}\frac{1}{((N-1))_5(N-3)^2}\\
			\le&  \frac{16\qty((n_1)_2)^2n_1{n_2}^3}{((N-1))_5(N-3)^2}.
		\end{align*}
		
		\medskip
		For Case (d1iii), $h_{\al,\al',\be,\be'}$ defines a surjection from $\qty(\mG_{\al}\cap\mG_{\al'}\cap\mG_{t_3't_4'})\times[N-1]\times[N-1]$ onto $\mH_{\al,\al',\be,\be'}$. Indeed, for any $(g_\be,g_{\be'})\in\mH_{\al,\al',\be,\be'}$, letting $g\coloneqq\qty(f_\al^{-1}(g_\be))_1=\qty(f_{\al'}^{-1}(g_{\be'}))_1\in\mG_{\al}\cap\mG_{\al'}\cap\mG_{t_3't_4'}$ yields $h_{\al,\al',\be,\be'}(g,c_\al(g_\be),c_{\al'}(g_{\be'}))=\qty(h_{\al,\be}(g,c_\al(g_\be)), h_{\al',\be'}(g,c_{\al'}(g_{\be'})))=(g_\be,g_{\be'})$. Therefore, we have
		\begin{align*}
			\qty|\mH_{\al,\al',\be,\be'}|\le&\qty|\qty(\mG_{\al}\cap\mG_{\al'}\cap\mG_{t_3't_4'})\times[N-1]\times[N-1]|\\
			=&(N-11)!!(N-1)^2
		\end{align*}
		and
		\begin{align*}
			\p(I_\al=I_{\al'}=1,J_{\be\al}=J_{\be'\al'}=1)\le\frac{(N-1)^2}{((N-1))_5}\frac{1}{(N-1)^2(N-3)^2}
		\end{align*}
		in this case. One of the degree 1 vertices of $\be$ is one of the degree 1 vertices of $\al$. Since $t_3t_4$ is an edge of $\al'$, the other degree 1 vertex of $\be$ is one of the degree 1 vertices of $\al'$, and the degree 2 vertex of $\be$ is that of $\al'$. Thus, if $\al,\al'$ are given, there are only four possibilities about the vertices of $\be$ in this case; namely, there are two possibilities about the degree 1 vertex from $\al$ and two possibilities about the degree 1 vertex from $\al'$. Keeping this in mind, we estimate the possibilities of a quadruplet $\al,\al',\be,\be'$ in this class as (also using Lemma \ref{cardinality estimates}(d) for the possibilities of $\be'$)
		\begin{align*}
				\qty|\sum_{\substack{\al,\al'\in\Ga_{12}\\ \al\cap\al'=\varnothing}}\sum_{\substack{\be,\be'\in\Ga_{12}\\ \text{Case (d1iii)} }}|\le&\underbrace{\qty((n_1)_2n_2)^2}_{\al,\al'}\times\underbrace{2\times2\times2}_{\be}\times\underbrace{4 n_1 n_2}_{\be'}=32\qty((n_1)_2)^2n_1{n_2}^3.
		\end{align*}
		Thus we estimate
		\begin{align*}
			\sum_{\substack{\al,\al'\in\Ga_{12}\\ \al\cap\al'=\varnothing}}\sum_{\substack{\be,\be'\in\Ga_{12}\\ \text{Case (d1iii)} }}\e\qty[I_\al I_{\al'} J_{\be\al}J_{\be'\al'}]\le&\sum_{\substack{\al,\al'\in\Ga_{12}\\ \al\cap\al'=\varnothing}}\sum_{\substack{\be,\be'\in\Ga_{12}\\ \text{Case (d1iii)} }}\frac{1}{((N-1))_5(N-3)^2}\\
			\le&  \frac{32\qty((n_1)_2)^2n_1{n_2}^3}{((N-1))_5(N-3)^2}.
		\end{align*}

		\medskip
		For Case (d1iv), we have the same estimate as in Case (d1iii). The argument is completely parallel to that of Case (d1iii) and thus omitted.
		\begin{align*}
			\sum_{\substack{\al,\al'\in\Ga_{12}\\ \al\cap\al'=\varnothing}}\sum_{\substack{\be,\be'\in\Ga_{12}\\ \text{Case (d1iv)} }}\e\qty[I_\al I_{\al'} J_{\be\al}J_{\be'\al'}]
			\le&   \frac{32\qty((n_1)_2)^2n_1{n_2}^3}{((N-1))_5(N-3)^2}.
		\end{align*}
		
		\medskip
		For Case (d1v), $h_{\al,\al',\be,\be'}$ defines a surjection from $\qty(\mG_{\al}\cap\mG_{\al'})\times[N-1]\times[N-1]$ onto $\mH_{\al,\al',\be,\be'}$. Indeed, for any $(g_\be,g_{\be'})\in\mH_{\al,\al',\be,\be'}$, letting $g\coloneqq\qty(f_\al^{-1}(g_\be))_1=\qty(f_{\al'}^{-1}(g_{\be'}))_1\in\mG_{\al}\cap\mG_{\al'}$ yields $h_{\al,\al',\be,\be'}(g,c_\al(g_\be),c_{\al'}(g_{\be'}))=\qty(h_{\al,\be}(g,c_\al(g_\be)), h_{\al',\be'}(g,c_{\al'}(g_{\be'})))=(g_\be,g_{\be'})$. Therefore, we have
		\begin{align*}
			&\qty|\mH_{\al,\al',\be,\be'} |	\le\qty|\qty(\mG_{\al}\cap\mG_{\al'})\times[N-1]\times[N-1]|=(N-9)!!(N-1)^2
		\end{align*}
		and
		\begin{align*}
			\p(I_\al=I_{\al'}=1,J_{\be\al}=J_{\be'\al'}=1)\le\frac{(N-1)^2}{((N-1))_4}\frac{1}{(N-1)^2(N-3)^2}
		\end{align*}
		in this case.  Since $t_3t_4$ is an edge of $\al'$, if $\al,\al'$ are given, there are only four possibilities about the vertices of $\be$; namely, there are two possibilities about the degree 1 vertex from $\al$ and two possibilities about the degree 1 vertex from $\al'$. Similarly, since $t_3't_4'$ is an edge of $\al$, if $\al,\al'$ are given, there are only four possibilities about the vertices of $\be'$; namely, there are two possibilities about the degree 1 vertex from $\al'$ and two possibilities about the degree 1 vertex from $\al$. Keeping this in mind, we estimate the possibilities of a quadruplet $\al,\al',\be,\be'$ in this class as
		\begin{align*}
				\qty|\sum_{\substack{\al,\al'\in\Ga_{12}\\ \al\cap\al'=\varnothing}}\sum_{\substack{\be,\be'\in\Ga_{12}\\ \text{Case (d1v)} }}|\le&\underbrace{\qty((n_1)_2n_2)^2}_{\al,\al'}\times\underbrace{2\times2\times2}_{\be}\times\underbrace{2\times2\times2}_{\be'}=64\qty((n_1)_2)^2{n_2}^2.
		\end{align*}
		Thus we estimate
		\begin{align*}
			\sum_{\substack{\al,\al'\in\Ga_{12}\\ \al\cap\al'=\varnothing}}\sum_{\substack{\be,\be'\in\Ga_{12}\\ \text{Case (d1v)} }}\e\qty[I_\al I_{\al'} J_{\be\al}J_{\be'\al'}]\le&\sum_{\substack{\al,\al'\in\Ga_{12}\\ \al\cap\al'=\varnothing}}\sum_{\substack{\be,\be'\in\Ga_{12}\\ \text{Case (d1v)} }}\frac{1}{((N-1))_4(N-3)^2}\\
			\le& \frac{64\qty((n_1)_2)^2{n_2}^2}{((N-1))_4(N-3)^2}.
		\end{align*}

		\medskip
		Therefore, we obtain an estimate for Case (d1) as
		\begin{align*}
			&\sum_{\substack{\al,\al'\in\Ga_{12}\\ \al\cap\al'=\varnothing}}\sum_{\substack{\be,\be'\in\Ga_{12}\\ \text{Case (d1)} }}\e\qty[I_\al I_{\al'} J_{\be\al}J_{\be'\al'}]-\sum_{\substack{\al,\al'\in\Ga_{12}\\ \al\cap\al'=\varnothing}}\sum_{\substack{\be,\be'\in\Ga_{12}\\ \text{Case (d1)} }}p_\al p_{\al'} p_\be p_{\be'}\\
			\le& \underbrace{\frac{416\qty((n_1)_2)^2{n_1}^2{n_2}^4}{((N-1))_6(N-3)^3}}_{\text{Case (d1i)}}+\underbrace{ \frac{80\qty((n_1)_2)^2n_1{n_2}^3}{((N-1))_5(N-3)^2}}_{\text{Case (d1ii)+Case (d1iii)+Case (d1iv)}}+\underbrace{\frac{64\qty((n_1)_2)^2{n_2}^2}{((N-1))_4(N-3)^2}}_{\text{Case (d1v)}}.
		\end{align*}
		
		\medskip
		\emph{Case} (d2): \emph{$\al$ and $\be$ share one degree 1 vertex only, $\al'$ and $\be'$ are as in Lemma \ref{surjection 2star 2star}(b).}
		
		\medskip
		Suppose that $s_4=t_1$ without loss of generality.  Let $h_{\al,\be}:\qty(\mG_{\al}\cap\mG_{t_3t_4})\times[N-1]\rightarrow\mG_\be$ be the surjection in Lemma \ref{surjection 2star 2star}(a) with the property: For any $g_\be\in\mG_\be$, there exists $c_\al(g_\be)\in[N-1]$ such that $h_{\al,\be}\qty(\qty(f^{-1}_\al(g_\be))_1,c_\al(g_\be))=g_\be$. Let $h_{\al',\be'}:\mG_{\al'}\rightarrow\mG_{\be'}$ be the surjection in Lemma \ref{surjection 2star 2star}(b) such that $h_{\al',\be'}\qty(\qty(f^{-1}_{\al'}(g_{\be'}))_1)=g_{\be'}$ for $g_{\be'}\in\mG_{\be'}$. Define a map $h_{\al,\al',\be,\be'}:\qty(\mG_{\al}\cap\mG_{\al'}\cap\mG_{t_3t_4})\times[N-1]\rightarrow\mG_{\be}\times\mG_{\be'}$ by $h_{\al,\al',\be,\be'}(g,c)\coloneqq\qty(h_{\al,\be}(g,c), h_{\al',\be'}(g))$.
		
		\medskip
		For $(g_\be,g_{\be'})\in\mH_{\al,\al',\be,\be'}$, $\qty(f_\al^{-1}(g))_1=\qty(f_{\al'}^{-1}(g'))_1$ must contain all of $\al,\al',t_3t_4$. Note that $\al\cap\al'=\al\cap(t_3t_4)=\varnothing$.
		
		\medskip
		If $(t_3t_4)\cap\al'\neq\varnothing$, $t_3t_4$ and $\al'$ must share a vertex. When $t_3t_4$ and $\al'$ share a degree 1 vertex, $t_4$ must be equal to either $s_1'$ or $s_4'$. In the case when $t_4=s_1'$ (resp.\ $t_4=s_4'$), we may assume that $t_3=s_2'$ and $t_3t_4=s_1's_2'$ (resp.\ $t_3=s_3'$ and $t_3t_4=s_3's_4'$), because otherwise $t_3t_4$ and $\al'$ cannot coexist and $\mH_{\al,\al',\be,\be'}$ will be empty. When $t_3t_4$ and $\al'$ share a degree 2 vertex, $t_3$ must be equal to either $s_2'$ or $s_3'$. In the case when $t_3=s_2'$ (resp.\ $t_3=s_3'$), we may assume that $t_4=s_1'$ and $t_3t_4=s_1's_2'$ (resp.\ $t_4=s_4'$ and $t_3t_4=s_3's_4'$), because otherwise $t_3t_4$ and $\al'$ cannot coexist and $\mH_{\al,\al',\be,\be'}$ will be empty. To sum up, if $(t_3t_4)\cap\al'\neq\varnothing$, we may assume that $t_3t_4$ coincides with an edge of $\al'$.
		
		\medskip
		Thus it suffices to consider the following two subcases: Case (d2i): $(t_3t_4)\cap\al'=\varnothing$ and Case (d2ii): $t_3t_4$ is an edge of $\al'$.
		
		\medskip
		For Case (d2i), $h_{\al,\al',\be,\be'}$ defines a surjection from $\qty(\mG_{\al}\cap\mG_{\al'}\cap\mG_{t_3t_4})\times[N-1]$ onto $\mH_{\al,\al',\be,\be'}$. Indeed, for any $(g_\be,g_{\be'})\in\mH_{\al,\al',\be,\be'}$, letting $g\coloneqq\qty(f_\al^{-1}(g_\be))_1=\qty(f_{\al'}^{-1}(g_{\be'}))_1\in\mG_{\al}\cap\mG_{\al'}\cap\mG_{t_3t_4}$ yields $h_{\al,\al',\be,\be'}(g,c_\al(g_\be))=\qty(h_{\al,\be}(g,c_\al(g_\be)), h_{\al',\be'}(g))=(g_\be,g_{\be'})$. Therefore, we have
		\begin{align*}
			\qty|\mH_{\al,\al',\be,\be'}|\le\left|\qty(\mG_{\al}\cap\mG_{\al'}\cap\mG_{t_3t_4})\times[N-1] \right|=(N-11)!!(N-1)
		\end{align*}
		and
		\begin{align*}
			\p(I_\al=I_{\al'}=1,J_{\be\al}=J_{\be'\al'}=1)\le\frac{(N-1)}{((N-1))_5}\frac{1}{(N-1)^2(N-3)^2}
		\end{align*}
		in this case. Since we may estimate the possibilities of a quadruplet $\al,\al',\be,\be'$ in this class as (using Lemma \ref{cardinality estimates}(d) for the possibilities of $\be,\be'$)
			\begin{align*}
				\qty|\sum_{\substack{\al,\al'\in\Ga_{12}\\ \al\cap\al'=\varnothing}}\sum_{\substack{\be,\be'\in\Ga_{12}\\ \text{Case (d2i)} }}|\le&\underbrace{\qty((n_1)_2n_2)^2}_{\al,\al'}\times\underbrace{4n_1 n_2}_{\be}\times\underbrace{\qty((n_1)_2+2n_2)}_{\be'}=4\qty((n_1)_2)^2n_1{n_2}^3\qty((n_1)_2+2n_2),
		\end{align*}
		 this leads to the estimate
		\begin{align*}
			&\sum_{\substack{\al,\al'\in\Ga_{12}\\ \al\cap\al'=\varnothing}}\sum_{\substack{\be,\be'\in\Ga_{12}\\ \text{Case (d2i)} }}\e\qty[I_\al I_{\al'} J_{\be\al}J_{\be'\al'}]-\sum_{\substack{\al,\al'\in\Ga_{12}\\ \al\cap\al'=\varnothing}}\sum_{\substack{\be,\be'\in\Ga_{12}\\ \text{Case (d2i)} }}p_\al p_{\al'} p_\be p_{\be'}\\
			\le&\sum_{\substack{\al,\al'\in\Ga_{12}\\ \al\cap\al'=\varnothing}}\sum_{\substack{\be,\be'\in\Ga_{12}\\ \text{Case (d2i)} }}\qty(\frac{1}{((N-1))_5(N-1)(N-3)^2}-\frac{1}{(N-1)^4(N-3)^4})\\
			=&\sum_{\substack{\al,\al'\in\Ga_{12}\\ \al\cap\al'=\varnothing}}\sum_{\substack{\be,\be'\in\Ga_{12}\\ \text{Case (d2i)} }}\frac{1}{((N-1))_5(N-1)(N-3)^2}\qty(1-\frac{((N-1))_5(N-1)(N-3)^2}{(N-1)^4(N-3)^4})\\
			=&\sum_{\substack{\al,\al'\in\Ga_{12}\\ \al\cap\al'=\varnothing}}\sum_{\substack{\be,\be'\in\Ga_{12}\\ \text{Case (d2i)} }}\frac{1}{((N-1))_5(N-1)(N-3)^2}\qty(1-\frac{(N-5)(N-7)(N-9)}{(N-1)^2(N-3)})\\
			\le& \frac{4\qty((n_1)_2)^2n_1{n_2}^3\qty((n_1)_2+2n_2)}{((N-1))_5(N-1)(N-3)^2}\frac{16}{N-3}\\
			=& \frac{64\qty((n_1)_2)^2n_1{n_2}^3\qty((n_1)_2+2n_2)}{((N-1))_5(N-1)(N-3)^3}.
		\end{align*}
		
		\medskip
		For Case (d2ii), $h_{\al,\al',\be,\be'}$ defines a surjection from $\qty(\mG_{\al}\cap\mG_{\al'})\times[N-1]$ onto $\mH_{\al,\al',\be,\be'}$. Indeed, for any $(g_\be,g_{\be'})\in\mH_{\al,\al',\be,\be'}$, letting $g\coloneqq\qty(f_\al^{-1}(g_\be))_1=\qty(f_{\al'}^{-1}(g_{\be'}))_1\in\mG_{\al}\cap\mG_{\al'}$ yields $h_{\al,\al',\be,\be'}(g,c_\al(g_\be))=\qty(h_{\al,\be}(g,c_\al(g_\be)), h_{\al',\be'}(g))=(g_\be,g_{\be'})$. Therefore, we have
		\begin{align*}
			\qty|\mH_{\al,\al',\be,\be'}|\le\qty|\qty(\mG_{\al}\cap\mG_{\al'})\times[N-1]|=(N-9)!!(N-1)
		\end{align*}
		and
		\begin{align*}
			\p(I_\al=I_{\al'}=1,J_{\be\al}=J_{\be'\al'}=1)\le\frac{(N-1)}{((N-1))_4}\frac{1}{(N-1)^2(N-3)^2}
		\end{align*}
		in this case. One of the degree 1 vertices of $\be$ is one of the degree 1 vertices of $\al$. Since $t_3t_4$ is an edge of $\al'$, the other degree 1 vertex of $\be$ is one of the degree 1 vertices of $\al'$, and the degree 2 vertex of $\be$ is that of $\al'$. Thus, if $\al,\al'$ are given, there are only four possibilities about the vertices of $\be$ in this case; namely, there are two possibilities about the degree 1 vertex from $\al$ and two possibilities about the degree 1 vertex from $\al'$. Keeping this in mind, we estimate the possibilities of a quadruplet $\al,\al',\be,\be'$ in this class as (also using Lemma \ref{cardinality estimates}(d) for the possibilities of $\be'$)
		\begin{align*}
				\qty|\sum_{\substack{\al,\al'\in\Ga_{12}\\ \al\cap\al'=\varnothing}}\sum_{\substack{\be,\be'\in\Ga_{12}\\ \text{Case (d2ii)} }}|\le&\underbrace{\qty((n_1)_2n_2)^2}_{\al,\al'}\times\underbrace{2\times2\times2}_{\be}\times\underbrace{\qty((n_1)_2+2n_2)}_{\be'}=8\qty((n_1)_2)^2{n_2}^2\qty((n_1)_2+2n_2).
		\end{align*}
		Thus we estimate
		\begin{align*}
			\sum_{\substack{\al,\al'\in\Ga_{12}\\ \al\cap\al'=\varnothing}}\sum_{\substack{\be,\be'\in\Ga_{12}\\ \text{Case (d2ii)} }}\e\qty[I_\al I_{\al'} J_{\be\al}J_{\be'\al'}]\le&\sum_{\substack{\al,\al'\in\Ga_{12}\\ \al\cap\al'=\varnothing}}\sum_{\substack{\be,\be'\in\Ga_{12}\\ \text{Case (d2ii)} }}\frac{1}{((N-1))_4(N-1)(N-3)^2}\\
			\le&  \frac{8\qty((n_1)_2)^2{n_2}^2\qty((n_1)_2+2n_2)}{((N-1))_4(N-1)(N-3)^2}.
		\end{align*}
		
		\medskip
		Therefore, we obtain an estimate for Case (d2) as
		\begin{align*}
			&\sum_{\substack{\al,\al'\in\Ga_{12}\\ \al\cap\al'=\varnothing}}\sum_{\substack{\be,\be'\in\Ga_{12}\\ \text{Case (d2)} }}\e\qty[I_\al I_{\al'} J_{\be\al}J_{\be'\al'}]-\sum_{\substack{\al,\al'\in\Ga_{12}\\ \al\cap\al'=\varnothing}}\sum_{\substack{\be,\be'\in\Ga_{12}\\ \text{Case (d2)} }}p_\al p_{\al'} p_\be p_{\be'}\\
			\le&\underbrace{\frac{64\qty((n_1)_2)^2n_1{n_2}^3\qty((n_1)_2+2n_2)}{((N-1))_5(N-1)(N-3)^3}}_{\text{Case (d2i)}}+\underbrace{\frac{8\qty((n_1)_2)^2{n_2}^2\qty((n_1)_2+2n_2)}{((N-1))_4(N-1)(N-3)^2}}_{\text{Case (d2ii)}}.
		\end{align*}
		
		\medskip
		\emph{Case} (d3): \emph{$\al$ and $\be$ are as in Lemma \ref{surjection 2star 2star}(b), $\al'$ and $\be'$ share one degree 1 vertex only.}
		
		\medskip
		By symmetry, we obtain the same estimate as in Case (d2):
		\begin{align*}
			&\sum_{\substack{\al,\al'\in\Ga_{12}\\ \al\cap\al'=\varnothing}}\sum_{\substack{\be,\be'\in\Ga_{12}\\ \text{Case (d3)} }}\e\qty[I_\al I_{\al'} J_{\be\al}J_{\be'\al'}]-\sum_{\substack{\al,\al'\in\Ga_{12}\\ \al\cap\al'=\varnothing}}\sum_{\substack{\be,\be'\in\Ga_{12}\\ \text{Case (d3)} }}p_\al p_{\al'} p_\be p_{\be'}\\
			\le&\frac{64\qty((n_1)_2)^2n_1{n_2}^3\qty((n_1)_2+2n_2)}{((N-1))_5(N-1)(N-3)^3}+\frac{8\qty((n_1)_2)^2{n_2}^2\qty((n_1)_2+2n_2)}{((N-1))_4(N-1)(N-3)^2}.
		\end{align*}
		
		\medskip
		\emph{Case} (d4): \emph{$\al$ and $\be$ are as in Lemma \ref{surjection 2star 2star}(b), $\al'$ and $\be'$ are as in Lemma \ref{surjection 2star 2star}(b).}
		
		\medskip
		Let $h_{\al,\be}:\mG_\al\rightarrow\mG_\be$ and $h_{\al',\be'}:\mG_{\al'}\rightarrow\mG_{\be'}$ be the surjections in Lemma \ref{surjection 2star 2star}(b) such that $h_{\al,\be}\qty(\qty(f^{-1}_\al(g_\be))_1)=g_\be$ for $g_\be\in\mG_\be$ and $h_{\al',\be'}\qty(\qty(f^{-1}_{\al'}(g_{\be'}))_1)=g_{\be'}$ for $g_{\be'}\in\mG_{\be'}$. Define a map $h_{\al,\al',\be,\be'}:\mG_{\al}\cap\mG_{\al'}\rightarrow\mG_{\be}\times\mG_{\be'}$ by $h_{\al,\al',\be,\be'}(g)\coloneqq\qty(h_{\al,\be}(g), h_{\al',\be'}(g))$.  For any $(g_\be,g_{\be'})\in\mH_{\al,\al',\be,\be'}$, letting $g\coloneqq\qty(f_\al^{-1}(g_\be))_1=\qty(f_{\al'}^{-1}(g_{\be'}))_1\in\mG_{\al}\cap\mG_{\al'}$ yields $h_{\al,\al',\be,\be'}(g)=(g_\be,g_{\be'})$. Therefore, $h_{\al,\al',\be,\be'}$ is a surjection onto $\mH_{\al,\al',\be,\be'}$. We have
		\begin{align*}
			\qty|\mH_{\al,\al',\be,\be'}|\le\qty|\mG_{\al}\cap\mG_{\al'}|=(N-9)!!
		\end{align*}
		and
		\begin{align*}
			\p(I_\al=I_{\al'}=1,J_{\be\al}=J_{\be'\al'}=1)\le\frac{1}{((N-1))_4}\frac{1}{(N-1)^2(N-3)^2}
		\end{align*}
		in this case. Since we may estimate the possibilities of a quadruplet $\al,\al',\be,\be'$ in this class as (using Lemma \ref{cardinality estimates}(d) for the possibilities of $\be,\be'$)
		\begin{align*}
				\qty|\sum_{\substack{\al,\al'\in\Ga_{12}\\ \al\cap\al'=\varnothing}}\sum_{\substack{\be,\be'\in\Ga_{12}\\ \text{Case (d4)} }}|\le&\underbrace{\qty((n_1)_2n_2)^2}_{\al,\al'}\times\underbrace{\qty((n_1)_2+2n_2)^2}_{\be,\be'}=\qty((n_1)_2)^2{n_2}^2\qty((n_1)_2+2n_2)^2,
		\end{align*}
	this leads to the estimate
		\begin{align*}
			&\sum_{\substack{\al,\al'\in\Ga_{12}\\ \al\cap\al'=\varnothing}}\sum_{\substack{\be,\be'\in\Ga_{12}\\ \text{Case (d4)} }}\e\qty[I_\al I_{\al'} J_{\be\al}J_{\be'\al'}]-\sum_{\substack{\al,\al'\in\Ga_{12}\\ \al\cap\al'=\varnothing}}\sum_{\substack{\be,\be'\in\Ga_{12}\\ \text{Case (d4)} }}p_\al p_{\al'} p_\be p_{\be'}\\
			\le&\sum_{\substack{\al,\al'\in\Ga_{12}\\ \al\cap\al'=\varnothing}}\sum_{\substack{\be,\be'\in\Ga_{12}\\ \text{Case (d4)} }}\qty(\frac{1}{((N-1))_4(N-1)^2(N-3)^2}-\frac{1}{(N-1)^4(N-3)^4})\\
			=&\sum_{\substack{\al,\al'\in\Ga_{12}\\ \al\cap\al'=\varnothing}}\sum_{\substack{\be,\be'\in\Ga_{12}\\ \text{Case (d4)} }}\frac{1}{((N-1))_4(N-1)^2(N-3)^2}\qty(1-\frac{((N-1))_4(N-1)^2(N-3)^2}{(N-1)^4(N-3)^4})\\
			=&\sum_{\substack{\al,\al'\in\Ga_{12}\\ \al\cap\al'=\varnothing}}\sum_{\substack{\be,\be'\in\Ga_{12}\\ \text{Case (d4)} }}\frac{1}{((N-1))_4(N-1)^2(N-3)^2}\qty(1-\frac{(N-5)(N-7)}{(N-1)(N-3)})\\
			\le& \frac{\qty((n_1)_2)^2{n_2}^2\qty((n_1)_2+2n_2)^2}{((N-1))_4(N-1)^2(N-3)^2}\frac{8}{N-3}\\
			=&\frac{8\qty((n_1)_2)^2{n_2}^2\qty((n_1)_2+2n_2)^2}{((N-1))_4(N-1)^2(N-3)^3}.
		\end{align*}

		\medskip
		From Cases (d1)--(d4), we conclude that
		\begin{equation}\label{third variance second line case d}
			\begin{split}
			&\sum_{\substack{\al,\al'\in\Ga_{12}\\ \al\cap\al'=\varnothing}}\sum_{\substack{\be\in\Ga_{12}\setminus\{\al\}\\ \be\cap\al\neq\varnothing }}\sum_{\substack{\be'\in\Ga_{12}\setminus\{\al'\}\\ \be'\cap\al'\neq\varnothing }}\e\qty[I_\al I_{\al'} J_{\be\al}J_{\be'\al'}]-\sum_{\substack{\al,\al'\in\Ga_{12}\\ \al\cap\al'=\varnothing}}\sum_{\substack{\be\in\Ga_{12}\setminus\{\al\}\\ \be\cap\al\neq\varnothing }}\sum_{\substack{\be'\in\Ga_{12}\setminus\{\al'\}\\ \be'\cap\al'\neq\varnothing }}p_\al p_{\al'} p_\be p_{\be'}\\
			\le&\underbrace{\frac{416\qty((n_1)_2)^2{n_1}^2{n_2}^4}{((N-1))_6(N-3)^3}+\frac{80\qty((n_1)_2)^2n_1{n_2}^3}{((N-1))_5(N-3)^2}+\frac{64\qty((n_1)_2)^2{n_2}^2}{((N-1))_4(N-3)^2}}_{\text{Case (d1)}}\\
			&+\underbrace{\frac{128\qty((n_1)_2)^2n_1{n_2}^3\qty((n_1)_2+2n_2)}{((N-1))_5(N-1)(N-3)^3}+\frac{16\qty((n_1)_2)^2{n_2}^2\qty((n_1)_2+2n_2)}{((N-1))_4(N-1)(N-3)^2} }_{\text{Case (d2)+Case (d3)}}+\underbrace{\frac{8\qty((n_1)_2)^2{n_2}^2\qty((n_1)_2+2n_2)^2}{((N-1))_4(N-1)^2(N-3)^3}}_{\text{Case (d4)}}.
			\end{split}
		\end{equation}

\medskip
Recalling \eqref{third variance upper bound} and putting all this together, we upper bound the third variance in \eqref{variance decomposition} as
\begin{equation}\label{third variance final bound}
\Var\qty(\sum_{\al\in\Ga_{1j}}\sum_{\substack{\be\in\Ga_{1k}\setminus\{\al\}\\ \be\cap\al\neq\varnothing }}I_\al J_{\be\al})	\le\begin{cases}
	\eqref{third variance first line j1k1}+\eqref{third variance second line case a}	&(j=k=1)\\
		\eqref{third variance first line j1k2}+\eqref{third variance second line case b}	&(j=1,k=2)\\
		\eqref{third variance first line j2k1}+\eqref{third variance second line case c}	&(j=2,k=1)\\
		\eqref{third variance first line j2k2}+\eqref{third variance second line case d}	&(j=k=2)
	\end{cases}.
\end{equation}
		
		\subsubsection{Conclusion: Upper bounds on \eqref{first term unconditional variance}}
		Now assume that $N>15$ and $n_1\ge1$. Let $c_*$ be any real number such that
		\begin{equation}\label{absolute constant n1 n2 N}
			\frac{n_1\vee n_2}{N-15}\le c_*.
		\end{equation}
		Below, we collect the results obtained thus far and simplify them under the assumption \eqref{absolute constant n1 n2 N}.
		
		\medskip
		\underline{The first variance in \eqref{variance decomposition}}: By \eqref{first variance bound},
		\begin{align*}
			\Var\qty(\sum_{\al\in\Ga_{1j}}I_\al(1-J_{\al\al}))\le\begin{cases}
				\qty(\frac{1}{2}+\frac{1}{2}c_*^2)\frac{{n_1}^2}{N-1} &(j=1)\\
			\qty(c_*+8c_*^4)\frac{\qty(n_1\vee n_2)^2}{N-1}	&(j=2)
			\end{cases}.
		\end{align*}
		
		\medskip
		\underline{The second variance in \eqref{variance decomposition}}:
By \eqref{second variance bound},
		\begin{align*}
			\Var\qty(\sum_{\al\in\Ga_{1j}}\sum_{\substack{\be\in\Ga_{1k}\\ \be\cap\al=\varnothing }}I_\al I_\be(1-J_{\be\al}))
			\le&\begin{cases}
			\qty(c_*^2+\frac{3}{2}c_*^4+2c_*^6)\frac{{n_1}^2}{N-1}	&(j=k=1)\\
			\qty(2c_*^3+4c_*^5+72c_*^8)\frac{\qty(n_1\vee n_2)^2}{N-1}	&(j=1,k=2)\\
			\qty(2c_*^3+4c_*^5+8c_*^6+72c_*^8)\frac{\qty(n_1\vee n_2)^2}{N-1}	&(j=2,k=1)\\
			\qty(16c_*^4+256c_*^7+2048c_*^{10})\frac{\qty(n_1\vee n_2)^2}{N-1}	&(j=k=2)
			\end{cases}.
		\end{align*}

		\medskip
		\underline{The third variance in \eqref{variance decomposition}}: By \eqref{third variance final bound} (we use $1\le n_1$ for some terms in all of the cases and also $4n_1n_2+(n_1)_2\le5\qty(n_1\vee n_2)^2$ and $(n_1)_2+2n_2\le3\qty(n_1\vee n_2)^2$ in the case $j=k=2$),
		\begin{align*}
			\Var\qty(\sum_{\al\in\Ga_{1j}}\sum_{\substack{\be\in\Ga_{1k}\setminus\{\al\}\\\be\cap\al\neq\varnothing}}I_\al J_{\be\al})\le&\begin{cases}	
			\qty(4c_*^3+2c_*^4)\frac{{n_1}^2}{N-1}&	(j=k=1)\\
		\qty(64c_*^5+72c_*^6+5c_*^4)\frac{\qty(n_1\vee n_2)^2}{N-1}	&	(j=1,k=2)\\
		\qty(8c_*^4+48c_*^6)\frac{\qty(n_1\vee n_2)^2}{N-1}	&	(j=2,k=1)\\
		\qty(328c_*^6+872c_*^8+64c_*^5)\frac{\qty(n_1\vee n_2)^2}{N-1}	&	(j=k=2)
			\end{cases}.
		\end{align*}

		\medskip
		Now we obtain upper bounds on \eqref{first term unconditional variance} using \eqref{variance decomposition}:
		\begin{align*}
			&\Var\qty(\sum_{\al\in\Ga_{1j}}I_\al\sum_{\be\in\Ga_{1k}}(I_{\be}-J_{\be\al}))\\
			\le&\begin{cases}
			3\qty[\qty(\frac{1}{2}+\frac{1}{2}c_*^2)+\qty(c_*^2+\frac{3}{2}c_*^4+2c_*^6)+\qty(4c_*^3+2c_*^4)]\frac{{n_1}^2}{N-1}	&(j=k=1)\\
			2\qty[\qty(2c_*^3+4c_*^5+72c_*^8)+\qty(64c_*^5+72c_*^6+5c_*^4)]\frac{\qty(n_1\vee n_2)^2}{N-1}	&(j=1,k=2)\\
			2\qty[\qty(2c_*^3+4c_*^5+8c_*^6+72c_*^8)+\qty(8c_*^4+48c_*^6)]\frac{\qty(n_1\vee n_2)^2}{N-1}	&(j=2,k=1)\\
			3\qty[\qty(c_*+8c_*^4)+\qty(16c_*^4+256c_*^7+2048c_*^{10})+\qty(328c_*^6+872c_*^8+64c_*^5)]\frac{\qty(n_1\vee n_2)^2}{N-1}	&(j=k=2)
			\end{cases}
			\\
			=&\begin{cases}
			3\qty[\frac{1}{2}+\frac{3}{2}c_*^2+4c_*^3+\frac{7}{2}c_*^4+2c_*^6]\frac{{n_1}^2}{N-1}	&(j=k=1)\\
			2\qty[2c_*^3+5c_*^4+68c_*^5+72c_*^6+72c_*^8]\frac{\qty(n_1\vee n_2)^2}{N-1}	&(j=1,k=2)\\
			2\qty[2c_*^3+8c_*^4+4c_*^5+56c_*^6+72c_*^8]\frac{\qty(n_1\vee n_2)^2}{N-1}	&(j=2,k=1)\\
		3\qty[c_*+24c_*^4+64c_*^5+328c_*^6+256c_*^7+872c_*^8+2048c_*^{10}]\frac{\qty(n_1\vee n_2)^2}{N-1}		&(j=k=2)
			\end{cases}.
		\end{align*}
		This completes the proof of Proposition \ref{first term in stein coupling bound}.

		\subsection{Proof of Proposition \ref{second term in stein coupling bound}}
	\label{stein coupling bound second term}
		
Let $\al\in\Ga_{1j}$. Decompose
		\begin{align*}
			\sum_{\be\in\Ga_{1}}\qty|I_\be-J_{\be\al}|=\indi_{j=k}I_\al(1-J_{\al\al})+\sum_{\substack{\be\in\Ga_{1}\setminus\{\al\}\\\be\cap\al\neq\varnothing}}\qty|I_\be-J_{\be\al}|+\sum_{\substack{\be\in\Ga_{1}\\\be\cap\al=\varnothing}}\qty|I_\be-J_{\be\al}|.
		\end{align*}
		Note that $J_{\be\al}=I_\be$ if $I_\al=0$. By Lemma \ref{intersecting different isolated trees},
		\begin{align*}
			\sum_{\substack{\be\in\Ga_{1}\setminus\{\al\}\\\be\cap\al\neq\varnothing}}\qty|I_\be-J_{\be\al}|=\sum_{\substack{\be\in\Ga_{1}\setminus\{\al\}\\\be\cap\al\neq\varnothing}}I_\al J_{\be\al}.
		\end{align*}
		We use the following simple observation that only a limited number of $\be$'s intersecting $\al$ will be created by \emph{each} instance of the switching procedure with respect to $\al$. The proof is given in Appendix \ref{proof 2}.
		\begin{lem}
			\label{intersection limited creation}
			Suppose that $\al\in\Ga_1$ and let $H$ be the type of $\al$.  Then we have
			\begin{align*}
				\sum_{\substack{\be\in\Ga_1\setminus\{\al\}\\\be\cap\al\neq\varnothing}}J_{\be\al}\le v(H).
			\end{align*}
		\end{lem}
		On the other hand,	by Lemma \ref{disjoint multigraphs non creation},
		\begin{align*}
			\sum_{\substack{\be\in\Ga_{1}\\\be\cap\al=\varnothing}}\qty|I_\be-J_{\be\al}|=\sum_{\substack{\be\in\Ga_{1}\\\be\cap\al=\varnothing}}I_\be(1-J_{\be\al}).
		\end{align*}
		Similarly, only a limited number of $\be$'s disjoint from $\al$ will be destroyed by each instance of the switching procedure with respect to $\al$. The proof is given in Appendix \ref{proof 2}.
		\begin{lem}
			\label{disjoint limited destruction}
				Suppose that $\al\in\Ga_1$ and let $H$ be the type of $\al$. Then we have
			\begin{align*}
				\sum_{\substack{\be\in\Ga_1\\\be\cap\al=\varnothing}}I_\be(1-J_{\be\al})\le e(H).
			\end{align*}
		\end{lem}
		By Lemmas \ref{intersection limited creation} and \ref{disjoint limited destruction}, we have
		\begin{align*}
			\sum_{\be\in\Ga_{1}}\qty|I_\be-J_{\be\al}|=&\indi_{j=k}I_\al(1-J_{\al\al})+	\sum_{\substack{\be\in\Ga_{1}\\\be\cap\al\neq\varnothing}}I_\al J_{\be\al}+\sum_{\substack{\be\in\Ga_{1}\\\be\cap\al=\varnothing}}I_\be(1-J_{\be\al})\\
			\le &1+v(H_j)+e(H_j)=2v(H_j).
		\end{align*}
		So we have
		\begin{align*}
			\sum_{\al\in\Ga_{1j}}I_\al\sum_{\be\in\Ga_{1}}\qty|I_\be-J_{\be\al}|\sum_{\ga\in\Ga_{1}}\qty|I_\ga-J_{\ga\al}|\le&4\qty(v(H_j))^2\sum_{\al\in\Ga_{1j}}I_\al
		\end{align*}	
		and 
		\begin{align*}
			\sum_{\al\in\Ga_{1j}}\e\qty[I_\al\sum_{\be\in\Ga_{1}}\qty|I_\be-J_{\be\al}|\sum_{\ga\in\Ga_{1}}\qty|I_\ga-J_{\ga\al}|]\le&4\qty(v(H_j))^2|\Ga_{1j}|p_\al\\
			=&\begin{cases}
			\frac{8(n_1)_2}{N-1}&(j=1)\\
			\frac{36(n_1)_2n_2}{((N-1))_2}		&(j=2)
			\end{cases}.
		\end{align*}
		Thus, using any real number $c_\star\ge	(n_1\vee n_2)/(N-3)$, we upper bound
\begin{align*}
	\sum_{\al\in\Ga_{1}}\e\qty[I_\al\sum_{\be\in\Ga_{1}}\qty|I_\be-J_{\be\al}|\sum_{\ga\in\Ga_{1}}\qty|I_\ga-J_{\ga\al}|]\le&\frac{8(n_1)_2}{N-1}+\frac{36(n_1)_2n_2}{((N-1))_2}\\
	\le&\qty(8c_\star+36c_\star^2)\qty(n_1\vee n_2).
\end{align*}
	This completes the proof of Proposition \ref{second term in stein coupling bound}.

	\subsection{Proof of Proposition \ref{third term in stein coupling bound}}
	\label{stein coupling bound third term}

	Though we do not restrict possible self-loops and possible double edges to be isolated ones, the conclusion of Lemma \ref{intersecting different isolated trees} continues to hold for $\al$ being either a possible isolated edge or a possible isolated 2-star and $\be$ being either a possible self-loop or a possible double edge; the proof is given in Appendix \ref{proof 3}.
	\begin{lem}\label{intersecting isolated tree and selfloop or double edge}
		Let $\al$ be either a possible isolated edge or a possible isolated 2-star, and let $\be$ be either a possible self-loop or a possible double edge, i.e., $\al\in\Ga_{1}$ and $\be\in\Ga_{2}$. If $\al\cap\be\neq\varnothing$, then $I_\al I_\be=0$.
	\end{lem}
	Thus, by Lemmas \ref{disjoint multigraphs non creation} and \ref{intersecting isolated tree and selfloop or double edge}, we have
	\begin{align*}
		&\sum_{\al\in\Ga_{1j}}\sum_{\be\in\Ga_{2k}}\e\qty[I_\al\qty|I_\be-J_{\be\al}|]\\
		=&\sum_{\al\in\Ga_{1j}}\sum_{\substack{\be\in\Ga_{2k}\\\be\cap\al=\varnothing}}\e\qty[I_\al I_\be\qty(1-J_{\be\al})]+\sum_{\al\in\Ga_{1j}}\sum_{\substack{\be\in\Ga_{2k}\\\be\cap\al\neq\varnothing}}\e\qty[I_\al J_{\be\al}]\\
		\le&\sum_{\al\in\Ga_{1j}}\sum_{\substack{\be\in\Ga_{2k}\\\be\cap\al=\varnothing}}\frac{1}{((N-1))_{e(H_j)+e(K_k)}}\sum_{\ell=1}^{e(H_j)}\frac{2e(K_k)}{N-2(e(H_j)-\ell)-1}\bigg.\notag\\
		&\quad\quad\quad\quad\bigg.+\sum_{\al\in\Ga_{1j}}\sum_{\substack{\be\in\Ga_{2k}\\\be\cap\al\neq\varnothing}}\frac{1}{((N-1))_{e(H_j)}((N-1))_{e(K_k)}},
	\end{align*}
	where $H_j$ and $K_k$ denote the types of $\Ga_{1j}$ and $\Ga_{2k}$, respectively. Here, note that the upper bound when $\al\cap\be=\varnothing$
	\begin{align}
	\p\qty(I_\al=I_\be=1,J_{\be\al}=0)\le&\frac{1}{((N-1))_{e(H_j)+e(K_k)}}\qty(1-\prod_{\ell=1}^{e(H_j)}\qty(1-\frac{2e(K_k)}{N-2(e(H_j)-\ell)-1}))\notag\\
		\le&\frac{1}{((N-1))_{e(H_j)+e(K_k)}}\sum_{\ell=1}^{e(H_j)}\frac{2e(K_k)}{N-2(e(H_j)-\ell)-1} \label{prob al be exist be destroyed final}
	\end{align}
	(cf.\ \eqref{prob al be exist be destroyed} and \eqref{Bernoulli inequality} from Appendix \ref{subsec second variance}) continues to hold for $\al$ being either a possible isolated edge or a possible isolated 2-star and $\be$ being either a possible self-loop or a possible double edge. Using Lemma \ref{cardinality estimates}(e), we obtain an upper bound for each of the cases:
	
	\medskip
	\underline{$j=k=1$}.
	
	\begin{align*}
		\sum_{\al\in\Ga_{11}}\sum_{\be\in\Ga_{21}}\e\qty[I_\al\qty|I_\be-J_{\be\al}|]\le&\frac{(n_1)_2}{2}\frac{\sum_{i\in[n]}(d_i)_2}{2}\frac{1}{((N-1))_2}\frac{2}{N-1}\\
		=&\frac{(n_1)_2}{((N-1))_2}\la^S_n.
	\end{align*}
	
	\medskip
	\underline{$j=1,k=2$}.

	\begin{align*}
		\sum_{\al\in\Ga_{11}}\sum_{\be\in\Ga_{22}}\e\qty[I_\al\qty|I_\be-J_{\be\al}|]\le&\frac{(n_1)_2}{2}\frac{\sum_{1\le i<j\le n}(d_i)_2(d_j)_2}{2}\frac{1}{((N-1))_3}\frac{4}{N-1}\\
		=&\frac{2(n_1)_2}{(N-1)(N-5)}\la^M_n.
	\end{align*}
	
	\medskip
	\underline{$j=2,k=1$}.

	\begin{align*}
		&\sum_{\al\in\Ga_{12}}\sum_{\be\in\Ga_{21}}\e\qty[I_\al\qty|I_\be-J_{\be\al}|]\\
		\le&(n_1)_2n_2\frac{\sum_{i\in[n]}(d_i)_2}{2}\frac{1}{((N-1))_3}2\qty(\frac{1}{N-3}+\frac{1}{N-1})+\frac{(n_1)_2n_2\times1}{((N-1))_2(N-1)}\\
		\le&(n_1)_2n_2\frac{\sum_{i\in[n]}(d_i)_2}{2}\frac{1}{((N-1))_3}\frac{4}{N-3}+\frac{(n_1)_2n_2}{((N-1))_2(N-1)}\\
		=&\frac{4(n_1)_2n_2}{(N-3)^2(N-5)}\la^S_n+\frac{(n_1)_2n_2}{((N-1))_2(N-1)}.
	\end{align*}
	
	\medskip
	\underline{$j=k=2$}.

	\begin{align*}
		&\sum_{\al\in\Ga_{12}}\sum_{\be\in\Ga_{22}}\e\qty[I_\al\qty|I_\be-J_{\be\al}|]\\
		\le&(n_1)_2n_2\frac{\sum_{1\le i<j\le n}(d_i)_2(d_j)_2}{2}\frac{1}{((N-1))_4}4\qty(\frac{1}{N-3}+\frac{1}{N-1})+\frac{(n_1)_2n_2\sum_{i\in[n]}(d_i)_2}{((N-1))_2((N-1))_2}\\
		\le&(n_1)_2n_2\frac{\sum_{1\le i<j\le n}(d_i)_2(d_j)_2}{2}\frac{1}{((N-1))_4}\frac{8}{N-3}+\frac{(n_1)_2n_2\sum_{i\in[n]}(d_i)_2}{((N-1))_2((N-1))_2}\\
		=&\frac{8(n_1)_2n_2}{((N-3))_3}\la^M_n+\frac{2(n_1)_2n_2}{((N-1))_2(N-3)}\la^S_n.
	\end{align*}
	
	\medskip
	Therefore, using any real number $c_\sharp\ge(n_1\vee n_2)/(N-7)$, we upper bound
	\begin{align*}
		\sum_{\al\in\Ga_1}\sum_{\be\in\Ga_2}\e\qty[I_\al\qty|I_\be-J_{\be\al}|]=&\sum_{j,k=1}^{2}\sum_{\al\in\Ga_{1j}}\sum_{\be\in\Ga_{2k}}\e\qty[I_\al\qty|I_\be-J_{\be\al}|]\\
		\le& c_\sharp^2\la^S_n+2c_\sharp^2\la^M_n+\qty(4c_\sharp^3\la^S_n+c_\sharp^3)+\qty(8c_\sharp^3\la^M_n+2c_\sharp^3\la^S_n)\\
		=&\qty(c_\sharp^2+6c_\sharp^3)\la^S_n+\qty(2c_\sharp^2+8c_\sharp^3)\la^M_n+c_\sharp^3,
	\end{align*}
	which proves Proposition \ref{third term in stein coupling bound}.

	\subsection{Proof of Proposition \ref{fourth term in stein coupling bound}}
	\label{stein coupling bound fourth term}

	 We resort to the following symmetry about our coupling \eqref{Stein coupling conditional dist}. This is an analogue of \cite[Lemma 3.1]{AHH19}, which is about the size-bias coupling approach (cf.\ \eqref{size bias coupling conditional dist}).
	\begin{lem}[Symmetry]
				\label{symmetry}
		For the coupling satisfying \eqref{Stein coupling conditional dist}, we have
		\begin{align*}
			\e\qty[I_\al\qty(I_\be-J_{\be\al})]=\e\qty[I_\be\qty(I_\al-J_{\al\be})]=\cov(I_\al,I_\be),\quad\al,\be\in\Ga.
		\end{align*}
	\end{lem}
	\begin{proof}
		This is immediate from $\e\qty[I_\al J_{\be\al}]=p_\al\e\qty[J_{\be\al}\middle|I_\al=1]=p_\al p_\be=p_\be\e\qty[J_{\al\be}\middle|I_\be=1]=\e\qty[I_\be J_{\al\be}]$.
	\end{proof}
	Now assume that $\al$ is either a possible self-loop or a possible double edge (i.e., $\al\in\Ga_2$) and $\be$ is either a possible isolated edge or a possible isolated 2-star (i.e., $\be\in\Ga_1$). If $\al\cap\be=\varnothing$, then by Lemma \ref{disjoint multigraphs non creation},
	\begin{align*}
		\e\qty[I_\al\qty|I_\be-J_{\be\al}|]=&\e\qty[I_\al I_\be\qty(1-J_{\be\al})]=	\e\qty[I_\al\qty(I_\be-J_{\be\al})],\\
		\e\qty[I_\be\qty|I_\al-J_{\al\be}|]=&\e\qty[I_\be I_\al\qty(1-J_{\al\be})]=	\e\qty[I_\be\qty(I_\al-J_{\al\be})].
	\end{align*}
	On the other hand, if $\al\cap\be\neq\varnothing$, then by Lemma \ref{intersecting isolated tree and selfloop or double edge},
	\begin{align*}
		\e\qty[I_\al\qty|I_\be-J_{\be\al}|]=&\e\qty[I_\al J_{\be\al}]=-\e\qty[I_\al\qty(I_\be-J_{\be\al})],\\
		\e\qty[I_\be\qty|I_\al-J_{\al\be}|]=&\e\qty[I_\be J_{\al\be}]=-\e\qty[I_\be\qty(I_\al-J_{\al\be})].
	\end{align*}
	Thus, for $\al\in\Ga_2$ and $\be\in\Ga_1$,
	\begin{align*}
		\e\qty[I_\al\qty|I_\be-J_{\be\al}|]=\qty|\e\qty[I_\al\qty(I_\be-J_{\be\al})]|=\qty|\e\qty[I_\be\qty(I_\al-J_{\al\be})]|=\e\qty[I_\be\qty|I_\al-J_{\al\be}|].
	\end{align*}
	The second equality is by Lemma \ref{symmetry}. Therefore,
	\begin{align*}
		\sum_{\al\in\Ga_2}\sum_{\be\in\Ga_1}\e\qty[I_\al\qty|I_\be-J_{\be\al}|]=	\sum_{\be\in\Ga_1}\sum_{\al\in\Ga_2}\e\qty[I_\be\qty|I_\al-J_{\al\be}|]=\sum_{\al\in\Ga_1}\sum_{\be\in\Ga_2}\e\qty[I_\al\qty|I_\be-J_{\be\al}|].
	\end{align*}
	A similar argument appears in \cite[Section 3, Case (c)]{AHH19} or in the context of \emph{monotone coupling} (in the sense of \cite[Definition 2.1.1]{BHJ92} and \cite[DEFINITION in Section 4]{J94}).

		\subsection{Proof of Proposition \ref{fifth term in stein coupling bound}}
	\label{stein coupling bound fifth term}

	Obviously,
	\begin{align*}
		\sum_{\al\in\Ga_2} p_{\al}^2=\frac{\sum_{i\in[n]}(d_i)_2}{2(N-1)^2}+\frac{\sum_{1\le i<j\le n}(d_i)_2(d_j)_2}{2\qty[((N-1))_2]^2}=\frac{\la^S_n}{N-1}+\frac{\la^M_n}{((N-1))_2}.
	\end{align*}
	
	\medskip
	We turn to upper bound the more involved term
	\begin{align*}
		\sum_{\al\in\Ga_{2}}\sum_{\be\in\Ga_{2}\setminus\{\al\}}\e\qty[I_\al|I_\beta-J_{\beta\al}|].
	\end{align*}
	Since $\al$ and $\be$ can be of two distinct types, corresponding to self-loops and double edges, this gives rise to four different cases. For notational reasons it is convenient to further define
	\begin{align*}
		\mu_n^{(4)}=\frac{\sum_{i\in[n]}(d_i)_4}{N-7}.
	\end{align*}
	Obviously, $\mu_n^{(4)}/(N-7)\le\nu_n^2$ holds.
	
	\medskip
	\textbf{(a) $\al$ and $\be$ are possible self-loops.}

	\medskip
	Let $\al\neq\be$ be possible self-loops. First suppose that $\al$ and $\be$ share a half-edge. Then (since $\al\neq\be$) $I_\al I_\be=0$ and $\e\qty[I_\al|I_\beta-J_{\beta\al}|]=\e\qty[I_\al J_{\beta\al}]=p_\al p_\be$. Noting that $\al$ and $\be$ are incident to the same vertex in this case, we compute the possibilities of such a pair of $\al$ and $\be$ as
	\begin{align}\label{intersecting self loops possibilities}
		\underbrace{\sum_{i\in[n]}}_{\text{vertex}}\underbrace{\frac{(d_i)_2}{2}}_{\al}\times\underbrace{2\times(d_i-2)}_{\be}=\sum_{i\in[n]}(d_i)_3.
	\end{align}
	Next suppose that $\al$ and $\be$ do not share any half-edge (but may share a vertex). Then by Lemma \ref{disjoint multigraphs non creation}, $J_{\be\al}=0$ if $I_\be=0$, and thus $\e\qty[I_\al|I_\beta-J_{\beta\al}|]=\e\qty[I_\al I_\be\qty(1-J_{\be\al})]$. We have (cf.\ \eqref{prob al be exist be destroyed}, \eqref{Bernoulli inequality} and \eqref{prob al be exist be destroyed final} from the previous subsections)
	\begin{align*}
		\p\qty(I_\al=I_\be=1,J_{\be\al}=0)\le\frac{1}{((N-1))_2}\frac{2}{N-1}.
	\end{align*}
	By subtracting \eqref{intersecting self loops possibilities} from the total, the possibilities of a pair of $\al$ and $\be$ in this class can be computed as follows:
	\begin{align*}
		&\sum_{i\in[n]}\sum_{j\in[n]}\frac{(d_i)_2(d_j)_2}{4}-\sum_{i\in[n]}\frac{(d_i)_2}{2}-\sum_{i\in[n]}(d_i)_3\\
		=&\sum_{1\le i\neq j\le n}\frac{(d_i)_2(d_j)_2}{4}+\sum_{i\in[n]}\qty(\frac{\qty((d_i)_2)^2}{4}-\frac{(d_i)_2}{2}-(d_i)_3)\\
		=&\sum_{1\le i\neq j\le n}\frac{(d_i)_2(d_j)_2}{4}+\sum_{i\in[n]}\frac{(d_i)_4}{4}\quad(\text{true also for $d_i\le 2$}).
	\end{align*}
	Note that the last two terms correspond to (i) the possibilities of a pair of $\al$ and $\be$ that are incident to two distinct vertices and (ii) the possibilities of a pair of $\al$ and $\be$ that are incident to the same vertex but share no half-edge, respectively. Thus
	\begin{align*}
		\sum_{\al\in\Ga_{21}}\sum_{\be\in\Ga_{21}\setminus\{\al\}}\e\qty[I_\al|I_\beta-J_{\beta\al}|]\le&\frac{\sum_{1\le i\neq j\le n}(d_i)_2(d_j)_2+\sum_{i\in[n]}(d_i)_4}{2((N-1))_2(N-1)}+\frac{\sum_{i\in[n]}(d_i)_3}{(N-1)^2}\\
		\le&\frac{\nu_n^2}{2(N-1)}+\frac{\mu_n^{(4)}}{2(N-1)^2}+\frac{\mu_n^{(3)}}{N-1}\\
		\le&\frac{\nu_n^2+\mu_n^{(3)}}{N-1}.
	\end{align*}
	
	\medskip
	\textbf{(b) $\al$ is a possible self-loop, $\be$ is a possible double edge.}
	
	\medskip
	Let $\al$ be a possible self-loop and $\be$ be a possible double edge. Suppose first that $\al$ and $\be$ share no half-edge (but may share a vertex). Then by Lemma \ref{disjoint multigraphs non creation}, $J_{\be\al}=0$ if $I_\be=0$, and thus $\e\qty[I_\al|I_\beta-J_{\beta\al}|]=\e\qty[I_\al I_\be\qty(1-J_{\be\al})]$. We have (cf.\ \eqref{prob al be exist be destroyed}, \eqref{Bernoulli inequality} and \eqref{prob al be exist be destroyed final} from the previous subsections)
	\begin{align*}
		\p\qty(I_\al=I_\be=1,J_{\be\al}=0)\le\frac{1}{((N-1))_3}\frac{4}{N-1}.
	\end{align*}
	The possibilities of such a pair of $\al$ and $\be$ can be computed as follows:
	\begin{align}
		&\underbrace{\sum_{i\in[n]}}_{\text{$\al$'s vertex}}\underbrace{\frac{(d_i)_2}{2}}_{\al}\underbrace{\qty(\sum_{\substack{1\le j<k\le n\\j\neq i,k\neq i}}\frac{(d_j)_2(d_k)_2}{2}+\sum_{\substack{j\in[n]\\j\neq i}}\frac{(d_i-2)_2(d_j)_2}{2})}_{\be}\label{disjoint al self loop be double edge raw}\\
		=&\sum_{i\in[n]}\sum_{\substack{1\le j<k\le n\\j\neq i,k\neq i}}\frac{(d_i)_2(d_j)_2(d_k)_2}{4}+\sum_{1\le i\neq j\le n}\frac{(d_i)_4(d_j)_2}{4}\notag\\
		=&\sum_{\substack{i,j,k\in[n]\\i\neq j,j\neq k,k\neq i}}\frac{(d_i)_2(d_j)_2(d_k)_2}{8}+\sum_{1\le i\neq j\le n}\frac{(d_i)_4(d_j)_2}{4}.\label{disjoint al self loop be double edge final}
	\end{align}
	Note that the first term in \eqref{disjoint al self loop be double edge final} corresponds to the case where $\al$ and $\be$ are incident to three distinct vertices, whereas the second term ibidem corresponds to the case where they are incident to only two different vertices.
	
	\medskip
	Next suppose that $\al$ and $\be$ share a half-edge. Then $I_\al I_\be=0$ and $\e\qty[I_\al|I_\beta-J_{\beta\al}|]=\e\qty[I_\al J_{\beta\al}]=p_\al p_\be$. By subtracting \eqref{disjoint al self loop be double edge raw} from the total, the possibilities of a pair of $\al$ and $\be$ in this class can be computed as follows: 
	\begin{align*}
		&\sum_{i\in[n]}\frac{(d_i)_2}{2}\sum_{1\le j<k\le n}\frac{(d_j)_2(d_k)_2}{2}-	\sum_{i\in[n]}\frac{(d_i)_2}{2}\qty(\sum_{\substack{1\le j<k\le n\\j\neq i,k\neq i}}\frac{(d_j)_2(d_k)_2}{2}+\sum_{\substack{j\in[n]\\j\neq i}}\frac{(d_i-2)_2(d_j)_2}{2})\\
		=&\sum_{i\in[n]}\frac{(d_i)_2}{2}\qty(\sum_{1\le j<k\le n}\frac{(d_j)_2(d_k)_2}{2}-\sum_{\substack{1\le j<k\le n\\j\neq i,k\neq i}}\frac{(d_j)_2(d_k)_2}{2}-\sum_{\substack{j\in[n]\\j\neq i}}\frac{(d_i-2)_2(d_j)_2}{2})\\
		=&\sum_{i\in[n]}\frac{(d_i)_2}{2}\qty(\sum_{i<k\le n}\frac{(d_i)_2(d_k)_2}{2}+\sum_{1\le j<i}\frac{(d_j)_2(d_i)_2}{2}-\sum_{\substack{j\in[n]\\j\neq i}}\frac{(d_i-2)_2(d_j)_2}{2})\\
		=&\sum_{i\in[n]}\frac{(d_i)_2}{2}\qty(\sum_{\substack{j\in[n]\\j\neq i}}\frac{(d_i)_2(d_j)_2}{2}-\sum_{\substack{j\in[n]\\j\neq i}}\frac{(d_i-2)_2(d_j)_2}{2})\\
		=&\sum_{i\in[n]}\frac{(d_i)_2}{2}\sum_{\substack{j\in[n]\\j\neq i}}\frac{2(2d_i-3)_+(d_j)_2}{2}\quad(\text{true also for $d_i\le2$})\\
		=&\sum_{1\le i\neq j\le n}\frac{(d_i)_2(2d_i-3)_+(d_j)_2}{2}=\sum_{1\le i\neq j\le n}\frac{2(d_i)_3(d_j)_2+(d_i)_2(d_j)_2}{2}\quad(\text{true also for $d_i\le2$}).
	\end{align*}
	Note that last two terms correspond to the cases where a pair of $\al$ and $\be$ is composed of $3+2$ distinct half-edges and $2+2$ distinct half-edges, respectively ($\al$ is considered incident to vertex $v_i$). Thus
	\begin{align*}
		&\sum_{\al\in\Ga_{21}}\sum_{\be\in\Ga_{22}}\e\qty[I_\al|I_\beta-J_{\beta\al}|]\\
		\le&\frac{\frac{1}{2}\sum_{\substack{i,j,k\in[n]\\i\neq j,j\neq k,k\neq i}}(d_i)_2(d_j)_2(d_k)_2+\sum_{1\le i\neq j\le n}(d_i)_4(d_j)_2}{((N-1))_3(N-1)}+\frac{\sum_{1\le i\neq j\le n}\qty(2(d_i)_3(d_j)_2+(d_i)_2(d_j)_2)}{2(N-1)((N-1))_2}\\
		\le&\frac{\nu_n^3}{2(N-1)}+\frac{\mu_n^{(4)}\nu_n}{(N-1)^2}+\frac{2\mu_n^{(3)}\nu_n+\nu_n^2}{2(N-1)}\\
		\le&\frac{3\nu_n^3+2\mu_n^{(3)}\nu_n+\nu_n^2}{2(N-1)}.
	\end{align*}
	
	\medskip
	\textbf{(c) $\al$ is a possible double edge, $\be$ is a possible self-loop.}
	
	\medskip
	The contribution from this case is \emph{equal} to the contribution from Case (b). By replacing the case where $\al\cap\be=\varnothing$ (resp.\ $\al\cap\be\neq\varnothing$) with the case where $\al$ and $\be$ share no (resp.\ a) half-edge, the same proof using the symmetry about our coupling \eqref{Stein coupling conditional dist} (Lemma \ref{symmetry}) as in Appendix \ref{stein coupling bound fourth term} applies here.
	\begin{align*}
		\sum_{\al\in\Ga_{22}}\sum_{\be\in\Ga_{21}}\e\qty[I_\al|I_\beta-J_{\beta\al}|]=\sum_{\al\in\Ga_{21}}\sum_{\be\in\Ga_{22}}\e\qty[I_\al|I_\beta-J_{\beta\al}|]
		\le\frac{3\nu_n^3+2\mu_n^{(3)}\nu_n+\nu_n^2}{2(N-1)}.
	\end{align*}
	
	\medskip
	\textbf{(d) $\al$ and $\be$ are possible double edges.}
	
	\medskip
	Let $\al\neq\be$ be possible double edges. Suppose first that $\al$ and $\be$ share no half-edge (but may share a vertex). Then by Lemma \ref{disjoint multigraphs non creation}, $J_{\be\al}=0$ if $I_\be=0$, and thus $\e\qty[I_\al|I_\beta-J_{\beta\al}|]=\e\qty[I_\al I_\be\qty(1-J_{\be\al})]$. We have (cf.\ \eqref{prob al be exist be destroyed}, \eqref{Bernoulli inequality} and \eqref{prob al be exist be destroyed final} from the previous subsections)
	\begin{align*}
		\p\qty(I_\al=I_\be=1,J_{\be\al}=0)\le\frac{1}{((N-1))_4}4\qty(\frac{1}{N-3}+\frac{1}{N-1})\le\frac{8}{((N-1))_4(N-3)}.
	\end{align*}
	The possibilities of such a pair of $\al$ and $\be$ can be computed as follows:
	\begin{align}
		&\underbrace{\sum_{1\le i<j\le n}}_{\text{$\al$'s vertices}}\underbrace{\frac{(d_i)_2(d_j)_2}{2}}_{\al}\notag\\
		&\quad\times\underbrace{\qty(\sum_{\substack{1\le k<\ell\le n\\\{k,\ell\}\cap\{i,j\}=\varnothing}}\frac{(d_k)_2(d_\ell)_2}{2}+\sum_{\substack{k\in[n]\\k\neq i,k\neq j}}\qty(\frac{(d_i-2)_2(d_k)_2}{2}+\frac{(d_j-2)_2(d_k)_2}{2})+\frac{(d_i-2)_2(d_j-2)_2}{2})}_{\be}\label{disjoint double edges raw}\\
		=&\sum_{1\le i<j\le n}\sum_{\substack{1\le k<\ell\le n\\\{k,\ell\}\cap\{i,j\}=\varnothing}}\frac{(d_i)_2(d_j)_2(d_k)_2(d_\ell)_2}{4}+\sum_{1\le i<j\le n}\sum_{\substack{k\in[n]\\k\neq i,k\neq j}}\qty(\frac{(d_i)_4(d_j)_2(d_k)_2}{4}+\frac{(d_i)_2(d_j)_4(d_k)_2}{4})\notag\\
		&\quad\quad\quad\quad\quad\quad\quad\quad+\sum_{1\le i<j\le n}\frac{(d_i)_4(d_j)_4}{4}\notag\\
		=&\sum_{\substack{i,j,k,\ell\in[n]\\\text{$i,j,k,\ell$: distinct}}}\frac{(d_i)_2(d_j)_2(d_k)_2(d_\ell)_2}{16}+\sum_{\substack{i,j,k\in[n]\\i\neq j,j\neq k,k\neq i}}\frac{(d_i)_4(d_j)_2(d_k)_2}{4}+\sum_{1\le i\neq j\le n}\frac{(d_i)_4(d_j)_4}{8}.\label{disjoint double edges final}
	\end{align}
	Note that the three terms in \eqref{disjoint double edges final} correspond to the cases where $\al$ and $\be$ are incident to four distinct vertices, only three distinct vertices, and only two different vertices, respectively.
	
	\medskip
	Now assume that $\al$ and $\be$ share a half-edge. There is only one case where $\al\neq\be$ are still compatible: $\al$ and $\be$ share a common edge. We can compute the moment $\e\qty[I_\al|I_\beta-J_{\beta\al}|]$ precisely in this case. The proof will be given in Appendix \ref{proof 5}.
	\begin{lem}\label{two double edges share a common edge}
		Suppose that two distinct possible double edges $\al\neq\be$ share a common edge. Then we have
		\begin{align*}
			\e\qty[I_\al|I_\beta-J_{\beta\al}|]=\e\qty[I_\al I_\be\qty(1-J_{\be\al})]=\frac{1}{((N-1))_3}\qty(1-\frac{N-5}{((N-1))_2}).
		\end{align*}
	\end{lem}
	Since we cannot improve the rate in $N$, we rather use a trivial bound $\e\qty[I_\al|I_\beta-J_{\beta\al}|]\le1/((N-1))_3$. The possibilities of a pair of $\al$ and $\be$ sharing a common edge can be computed as follows:
	\begin{align}\label{two double edges share a common edge cardinality}
		\underbrace{\sum_{1\le i<j\le n}}_{\text{two distinct vertices}}\underbrace{\frac{(d_i)_2(d_j)_2}{2}}_{\al}\times\underbrace{2\times(d_i-2)_1(d_j-2)_1}_{\be}=\sum_{1\le i<j\le n}(d_i)_3(d_j)_3=\sum_{1\le i\neq j\le n}\frac{(d_i)_3(d_j)_3}{2}.
	\end{align}
	
	\medskip
	In all the other configurations of $\al\neq\be$, $\al$ and $\be$ are incompatible: $\al$ and $\be$ share a half-edge, but no common edge, thus the partner of the said half-edge in $\al$ is different from the partner in $\be$. Hence $I_\al I_\be=0$ and $\e\qty[I_\al|I_\beta-J_{\beta\al}|]=\e\qty[I_\al J_{\beta\al}]=p_\al p_\be$. By subtracting \eqref{disjoint double edges raw} and \eqref{two double edges share a common edge cardinality} from the total, we shall count the possibilities of a pair of $\al$ and $\be$ in this class: 
	\begin{align*}
		&\sum_{1\le i<j\le n}\frac{(d_i)_2(d_j)_2}{2}\sum_{1\le k<\ell\le n}\frac{(d_k)_2(d_\ell)_2}{2}-\sum_{1\le i<j\le n}\frac{(d_i)_2(d_j)_2}{2}\\
		&-\sum_{1\le i<j\le n}\frac{(d_i)_2(d_j)_2}{2}\qty(\sum_{\substack{1\le k<\ell\le n\\\{k,\ell\}\cap\{i,j\}=\varnothing}}\frac{(d_k)_2(d_\ell)_2}{2}+\sum_{\substack{k\in[n]\\k\neq i,k\neq j}}\frac{\qty((d_i-2)_2+(d_j-2)_2)(d_k)_2}{2}+\frac{(d_i-2)_2(d_j-2)_2}{2})\\
		&-\sum_{1\le i<j\le n}(d_i)_3(d_j)_3\\
		=&\sum_{1\le i<j\le n}\frac{(d_i)_2(d_j)_2}{2}\left(\sum_{1\le k<\ell\le n}\frac{(d_k)_2(d_\ell)_2}{2}-\sum_{\substack{1\le k<\ell\le n\\\{k,\ell\}\cap\{i,j\}=\varnothing}}\frac{(d_k)_2(d_\ell)_2}{2}\right.\\
		&\quad\quad\quad\left.-\sum_{\substack{k\in[n]\\k\neq i,k\neq j}}\frac{\qty((d_i-2)_2+(d_j-2)_2)(d_k)_2}{2}-\frac{(d_i-2)_2(d_j-2)_2}{2}-2(d_i-2)_1(d_j-2)_1-1\right).
	\end{align*}
	First, note that
	\begin{align*}
		&\sum_{1\le k<\ell\le n}\frac{(d_k)_2(d_\ell)_2}{2}-\sum_{\substack{1\le k<\ell\le n\\\{k,\ell\}\cap\{i,j\}=\varnothing}}\frac{(d_k)_2(d_\ell)_2}{2}-\sum_{\substack{k\in[n]\\k\neq i,k\neq j}}\frac{\qty((d_i-2)_2+(d_j-2)_2)(d_k)_2}{2}\\
		=&\sum_{\substack{k\in[n]\\k\neq i,k\neq j}}\frac{2(2d_i-3)_+(d_k)_2}{2}+\sum_{\substack{k\in[n]\\k\neq i,k\neq j}}\frac{2(2d_j-3)_+(d_k)_2}{2}+\frac{(d_i)_2(d_j)_2}{2}.
	\end{align*}
	Then, 
	\begin{align}
		&\sum_{1\le i<j\le n}\frac{(d_i)_2(d_j)_2}{2}\qty(\sum_{\substack{k\in[n]\\k\neq i,k\neq j}}\frac{2(2d_i-3)_+(d_k)_2}{2}+\sum_{\substack{k\in[n]\\k\neq i,k\neq j}}\frac{2(2d_j-3)_+(d_k)_2}{2})\notag\\
		=&\sum_{1\le i<j\le n}\sum_{\substack{k\in[n]\\k\neq i,k\neq j}}\frac{(d_i)_2(2d_i-3)_+(d_j)_2(d_k)_2}{2}+\sum_{1\le i<j\le n}\sum_{\substack{k\in[n]\\k\neq i,k\neq j}}\frac{(d_j)_2(2d_j-3)_+(d_i)_2(d_k)_2}{2}\notag\\
		=&\sum_{1\le i<j\le n}\sum_{\substack{k\in[n]\\k\neq i,k\neq j}}\frac{2(d_i)_3(d_j)_2(d_k)_2+(d_i)_2(d_j)_2(d_k)_2}{2}+\sum_{1\le i<j\le n}\sum_{\substack{k\in[n]\\k\neq i,k\neq j}}\frac{2(d_j)_3(d_i)_2(d_k)_2+(d_j)_2(d_i)_2(d_k)_2}{2}\notag\\
		=&\qty(\sum_{1\le i<j\le n}\sum_{\substack{k\in[n]\\k\neq i,k\neq j}}+\sum_{1\le j<i\le n}\sum_{\substack{k\in[n]\\k\neq i,k\neq j}})\frac{2(d_i)_3(d_j)_2(d_k)_2+(d_i)_2(d_j)_2(d_k)_2}{2}\notag\\
		=&\sum_{\substack{i,j,k\in[n]\\i\neq j,j\neq k,k\neq i}}\frac{2(d_i)_3(d_j)_2(d_k)_2+(d_i)_2(d_j)_2(d_k)_2}{2}.\label{incompatible double edges three distinct vertices}
	\end{align}
	Note that these two terms correspond to the cases where $\al$ and $\be$ share vertex $v_i$ only and (i) contain three distinct half-edges incident to vertex $v_i$, or (ii) contain two distinct half-edges incident to vertex $v_i$, respectively. For the remaining terms, we can show that (as long as $d_i\wedge d_j\ge2$)
	\begin{align*}
		&\frac{(d_i)_2(d_j)_2}{2}-\frac{(d_i-2)_2(d_j-2)_2}{2}-2(d_i-2)_1(d_j-2)_1-1\\
		=&\begin{cases}
			2d_i^2d_j+2d_id_j^2-3d_i^2-14d_id_j-3d_j^2+19d_i+19d_j-27	&(d_i\ge3,d_j\ge3)\\
			d_i(d_i-1)-1&(d_i\ge3,d_j=2)\\
			d_j(d_j-1)-1&(d_i=2,d_j\ge3)\\
			1	&(d_i=d_j=2)
		\end{cases}\\
		=&2(d_i-2)_2(d_j-2)_1+2(d_i-2)_1(d_j-2)_2+(d_i-2)_2+(d_j-2)_2\\
		&\quad+6(d_i-2)_1(d_j-2)_1+4(d_i-2)_1+4(d_j-2)_1+1.
	\end{align*}
	To guess this formula, note that this is enumerating all the other possibilities about $\be$ which is incompatible with a given $\al$: $\al$ and $\be$ share vertices $v_i$ and $v_j$ using
	{\setlength{\leftmargini}{29.5pt}  	
		\begin{enumerate}
			\renewcommand{\labelenumi}{(\roman{enumi})}
			
			\setlength{\labelsep}{7pt}     
			\setlength{\itemsep}{3pt}      
			
			\item four half-edges from $v_i$ and three from $v_j$,
			
			\item three from $v_i$ and four from $v_j$,
			
			\item four from $v_i$ and two from $v_j$,
			
			\item two from $v_i$ and four from $v_j$,
			
			\item three from $v_i$ and three from $v_j$ (note that we have to exclude the case where $\al$ and $\be$ share a common edge),
			
			\item three from $v_i$ and two from $v_j$,
			
			\item two from $v_i$ and three from $v_j$, and
			
			\item two from $v_i$ and two from $v_j$ (only pairings of the half-edges differ in $\al$ and $\be$).
			
	\end{enumerate} }

	\medskip
	Thus, 
	\begin{align*}
		&\frac{(d_i)_2(d_j)_2}{2}\qty(\frac{(d_i)_2(d_j)_2}{2}-\frac{(d_i-2)_2(d_j-2)_2}{2}-2(d_i-2)_1(d_j-2)_1-1)\\
		=&\frac{\splitfrac{2(d_i)_4(d_j)_3+2(d_i)_3(d_j)_4+(d_i)_4(d_j)_2+(d_i)_2(d_j)_4}{+6(d_i)_3(d_j)_3+4(d_i)_3(d_j)_2+4(d_i)_2(d_j)_3+(d_i)_2(d_j)_2}}{2}
	\end{align*}
	and
	\begin{align}
		&\sum_{1\le i<j\le n}\frac{(d_i)_2(d_j)_2}{2}\qty(\frac{(d_i)_2(d_j)_2}{2}-\frac{(d_i-2)_2(d_j-2)_2}{2}-2(d_i-2)_1(d_j-2)_1-1)\notag\\
		=&\sum_{1\le i\neq j\le n}\frac{2(d_i)_4(d_j)_3+(d_i)_4(d_j)_2+3(d_i)_3(d_j)_3+4(d_i)_3(d_j)_2+\frac{1}{2}(d_i)_2(d_j)_2}{2}.\label{incompatible double edges two distinct vertices}
	\end{align}
	The number of the possibilities of a pair of incompatible $\al$ and $\be$ is the sum of \eqref{incompatible double edges three distinct vertices} and \eqref{incompatible double edges two distinct vertices}.
	
	\medskip
	Putting all this together, we obtain an estimate for Case (d) as
	\begin{align*}
		&\sum_{\al\in\Ga_{22}}\sum_{\be\in\Ga_{22}\setminus\{\al\}}\e\qty[I_\al|I_\beta-J_{\beta\al}|]\\
		\le&\frac{\frac{1}{2}\sum_{\substack{i,j,k,\ell\in[n]\\\text{$i,j,k,\ell$: distinct}}}(d_i)_2(d_j)_2(d_k)_2(d_\ell)_2+2\sum_{\substack{i,j,k\in[n]\\i\neq j,j\neq k,k\neq i}}(d_i)_4(d_j)_2(d_k)_2+\sum_{1\le i\neq j\le n}(d_i)_4(d_j)_4}{((N-1))_4(N-3)}.   \\
		&+\frac{\sum_{1\le i\neq j\le n}(d_i)_3(d_j)_3}{2((N-1))_3}    \\
		&+\frac{\sum_{\substack{i,j,k\in[n]\\i\neq j,j\neq k,k\neq i}}\qty(2(d_i)_3(d_j)_2(d_k)_2+(d_i)_2(d_j)_2(d_k)_2)}{2((N-1))_2((N-1))_2}\\
		&+\frac{\sum_{1\le i\neq j\le n}\qty(2(d_i)_4(d_j)_3+(d_i)_4(d_j)_2+3(d_i)_3(d_j)_3+4(d_i)_3(d_j)_2+\frac{1}{2}(d_i)_2(d_j)_2)}{2((N-1))_2((N-1))_2}\\
		\le&\frac{\nu_n^4}{2(N-1)}+\frac{2\mu_n^{(4)}\nu_n^2}{((N-1))_2}+\frac{\qty(\mu_n^{(4)})^2}{((N-1))_2(N-3)}+\frac{\qty(\mu_n^{(3)})^2}{2(N-1)}+\frac{2\mu_n^{(3)}\nu_n^2+\nu_n^3}{2(N-1)}\\
		&\quad+\frac{2\mu_n^{(4)}\mu_n^{(3)}+\mu_n^{(4)}\nu_n+3\qty(\mu_n^{(3)})^2+4\mu_n^{(3)}\nu_n+\frac{1}{2}\nu_n^2}{2(N-1)^2}\\
		\le&\frac{\nu_n^4}{2(N-1)}+\frac{2\nu_n^4}{N-1}+\frac{\nu_n^4}{N-1}+\frac{\qty(\mu_n^{(3)})^2}{2(N-1)}+\frac{2\mu_n^{(3)}\nu_n^2+\nu_n^3}{2(N-1)}\\
		&\quad+\frac{2\mu_n^{(3)}\nu_n^2+\nu_n^3}{2(N-1)}+\frac{3\qty(\mu_n^{(3)})^2+4\mu_n^{(3)}\nu_n+\frac{1}{2}\nu_n^2}{2(N-1)^2}\\
		=&\frac{7\nu_n^4+\qty(\mu_n^{(3)})^2+4\mu_n^{(3)}\nu_n^2+2\nu_n^3}{2(N-1)}+\frac{3\qty(\mu_n^{(3)})^2+4\mu_n^{(3)}\nu_n+\frac{1}{2}\nu_n^2}{2(N-1)^2}.
	\end{align*}
	
	\medskip
	Therefore, summing over Cases (a)--(d), we obtain
	\begin{align*}
		&\sum_{\al\in\Ga_{2}}\sum_{\be\in\Ga_{2}\setminus\{\al\}}\e\qty[I_\al|I_\beta-J_{\beta\al}|]\\
		\le&\frac{\nu_n^2+\mu_n^{(3)}}{N-1}+\frac{3\nu_n^3+2\mu_n^{(3)}\nu_n+\nu_n^2}{N-1}+\frac{7\nu_n^4+\qty(\mu_n^{(3)})^2+4\mu_n^{(3)}\nu_n^2+2\nu_n^3}{2(N-1)}+\frac{3\qty(\mu_n^{(3)})^2+4\mu_n^{(3)}\nu_n+\frac{1}{2}\nu_n^2}{2(N-1)^2}\\
		=&\frac{7\nu_n^4+\qty(\mu_n^{(3)})^2+4\mu_n^{(3)}\nu_n^2+8\nu_n^3+4\mu_n^{(3)}\nu_n+2\mu_n^{(3)}+4\nu_n^2}{2(N-1)}+\frac{6\qty(\mu_n^{(3)})^2+8\mu_n^{(3)}\nu_n+\nu_n^2}{4(N-1)^2}.
	\end{align*}
	This completes the proof of Proposition \ref{fifth term in stein coupling bound}.

	\section{Proofs of auxiliary lemmas}
	\label{proof of auxiliary lemmas}
	
	\subsection{Proof of Lemma \ref{cardinality estimates}}
	\label{proof cardinality estimates}
	
	(a) This is already mentioned at the beginning of Section \ref{configuration model}, when $\Ga_{11}$, $\Ga_{12}$, $\Ga_{21}$ and $\Ga_{22}$ are introduced.
	
	\medskip
(b) \underline{$j=k=1$}. Let $\al$ be a possible isolated edge and let $v_1^\al$ and $v_2^\al$ denote its degree 1 vertices. Then $\qty{\{v_i,v_j\}:i\neq j,d_i=d_j=1,\{v_i,v_j\}\cap\{v_1^\al,v_2^\al\}\neq\varnothing}\subset\bigcup_{i=1}^2\qty{\{v_i^\al,v_j\}:v_i^\al\neq v_j,d_j=1}$. Thus the number of possible isolated edges intersecting $\al$ is at most $2n_1$.

\medskip
\underline{$j=1,k=2$}. Let $\al$ be a possible isolated edge and let $v_1^\al$ and $v_2^\al$ denote its degree 1 vertices. Then $\qty{\{v_i,v_j,v_k\}:i\neq j,d_i=d_j=1,d_k=2,\{v_i,v_j\}\cap\{v_1^\al,v_2^\al\}\neq\varnothing}\subset\bigcup_{i=1}^2\{\{v_i^\al,v_j,v_k\}:v_i^\al\neq v_j,d_j=1,d_k=2\}$. Thus the number of possibilities about the vertices of a possible isolated 2-star intersecting $\al$ is at most $2n_1n_2$. Given two degree 1 vertices and one degree 2 vertex, there are two ways to permute the half-edges incident to the degree 2 vertex. Thus the number of possible isolated 2-stars intersecting $\al$ is at most $4n_1n_2$.

\medskip
\underline{$j=2,k=1$}. Let $\al$ be a possible isolated 2-star and let $v_{1}^{\al,1}$ and $v_{2}^{\al,1}$ denote its degree 1 vertices. Then $\{\{v_i,v_j\}:i\neq j,d_i=d_j=1,\{v_i,v_j\}\cap\{v_1^{\al,1},v_2^{\al,1}\}\neq\varnothing\}\subset\bigcup_{i=1}^2\{\{v_i^{\al,1},v_j\}:v_i^\al\neq v_j,d_j=1\}$. Thus the number of possible isolated edges intersecting $\al$ is at most $2n_1$. Note that the case where $\al,\be$ share two degree 1 vertices is already counted within $2n_1$.

\medskip
\underline{$j=k=2$}. Let $\al$ be a possible isolated 2-star. Let $v_{1}^{\al,1}$ and $v_{2}^{\al,1}$ denote the degree 1 vertices of $\al$ and let $v^{\al,2}$ denote the degree 2 vertex of $\al$. Then
\begin{align*}
&\qty{\{v_i,v_j,v_k\}:i\neq j,d_i=d_j=1,d_k=2,\{v_i,v_j,v_k\}\cap\{v_1^{\al,1},v_2^{\al,1},v^{\al,2}\}\neq\varnothing}\\
\subset&\qty(\bigcup_{i=1}^2\qty{\{v_i^{\al,1},v_j,v_k\}:v_i^\al\neq v_j,d_j=1,d_k=2})\cup\qty{\{v_i,v_j,v^{\al,2}\}:i\neq j,d_i=d_j=1}.
\end{align*}
Thus the number of possibilities about the vertices of a possible isolated 2-star intersecting $\al$ is at most $2n_1n_2+(n_1)_2/2$. Given two degree 1 vertices and one degree 2 vertex, there are two ways to permute the half-edges incident to the degree 2 vertex. Thus the number of possible isolated 2-stars intersecting $\al$ is at most $4n_1n_2+(n_1)_2$.

\medskip
(c) Since the two degree 1 vertices of $\be$ must be those of $\al$, there are only $n_2$ possibilities about the vertices of $\be$, namely those about the degree 2 vertex. Given two degree 1 vertices and one degree 2 vertex, there are two ways to permute the half-edges incident to the degree 2 vertex. Thus the number of possible isolated 2-stars sharing two degree 1 vertices with $\al$ is at most $2n_2$.

\medskip
(d) Let $\al$ be a possible isolated 2-star. Let $v_{1}^{\al,1}$ and $v_{2}^{\al,1}$ denote the degree 1 vertices of $\al$ and let $v^{\al,2}$ denote the degree 2 vertex of $\al$. Then
\begin{align*}
	&\qty{\{v_i,v_j,v_k\}:i\neq j,d_i=d_j=1,d_k=2,|\{v_i,v_j\}\cap\{v_1^{\al,1},v_2^{\al,1}\}|=1,v_k\neq v^{\al,2}}\\
	\subset&\bigcup_{i=1}^2\qty{\{v_i^{\al,1},v_j,v_k\}:v_i^\al\neq v_j,d_j=1,d_k=2}.
\end{align*}
Thus the number of possibilities about the vertices of a member $\be\in\La^\al$ is at most $2n_1n_2$. Given two degree 1 vertices and one degree 2 vertex, there are two ways to permute the half-edges incident to the degree 2 vertex. Thus $\qty|\La^\al|\le4n_1n_2$.

\medskip
Next, note that for any member $\be\in\qty{\be\in\Ga_{12}\setminus\{\al\}:\be\cap\al\neq\varnothing}\setminus\La^\al$, either (i) $\al$ and $\be$ share two degree 1 vertices, or (ii) $\al$ and $\be$ share one degree 2 vertex (or both). Obviously,
\begin{align*}
	&\qty|\qty{\{v_i,v_j,v_k\}:i\neq j,d_i=d_j=1,d_k=2,\{v_i,v_j\}=\{v_1^{\al,1},v_2^{\al,1}\}}|=n_2,\\
		&\qty|\qty{\{v_i,v_j,v_k\}:i\neq j,d_i=d_j=1,d_k=2,v_k= v^{\al,2}}|=\frac{(n_1)_2}{2}.
\end{align*}
Thus the number of possibilities about the vertices of a member $\be\in\qty{\be\in\Ga_{12}\setminus\{\al\}:\be\cap\al\neq\varnothing}\setminus\La^\al$ is at most $n_2+(n_1)_2/2$. Given two degree 1 vertices and one degree 2 vertex, there are two ways to permute the half-edges incident to the degree 2 vertex. Thus $|\{\be\in\Ga_{12}\setminus\{\al\}:\be\cap\al\neq\varnothing\}\setminus\La^\al|\le 2n_2+(n_1)_2$.

\medskip
(e)  	If $\al$ is a possible isolated edge, then $\al$ consists only of degree 1 vertices, whereas a possible self-loop or a possible double edge consists of vertices of degree higher than or equal to 2. Therefore there is no $\be\in\Ga_{2}$ intersecting $\al$.

\medskip
Suppose that $\al$ is a possible isolated 2-star. The vertex of a possible self-loop intersecting $\al$ must be $\al$'s degree 2 vertex, which is assigned only two half-edges. Thus there is only one possible self-loop intersecting $\al$. One of the vertices of a possible double edge intersecting $\al$ must be $\al$'s degree 2 vertex. Given the other vertex $v_i$, there are $(d_i)_2/2$ choices for the two half-edges of a possible double edge intersecting $\al$ which are incident to $v_i$, and there are two possible pairings of the two chosen half-edges incident to $v_i$ to the two half-edges incident to $\al$'s degree 2 vertex. Thus there are at most $\sum_{i\in[n]}(d_i)_2/2\times2$ possibilities about a possible double edge intersecting $\al$. \qed
	
	\subsection{Proofs for Appendix \ref{stein coupling bound first term}}
	\label{proofs}
	
	\subsubsection{Proof of Lemma \ref{intersecting different isolated trees}}
	We only consider the case where $\al,\be\in\Ga_{12}$, i.e., both $\al$ and $\be$ are possible isolated 2-stars. The other cases are easier. Suppose that $\al\cap\be\neq \varnothing$ and $\al\neq\be$. Suppose for contradiction that any half-edge $s$ shared by $\al$ and $\be$ has the same partner in $\al$ and in $\be$. Since $\al\cap\be\neq\varnothing$, $\al$ and $\be$ share a vertex.
	
	\medskip
	Suppose first that $\al$ and $\be$ share a degree 2 vertex. Let $t_1,t_2$ denote the two half-edges incident to that degree 2 vertex. $\al$ and $\be$ share $t_1$ and $t_2$ (since $\al$ and $\be$ are isolated ones). By what we have supposed for contradiction, $t_j$ has the same partner in $\al$ and in $\be$, which is denoted $s_j$, $j=1,2$. $s_j$ is incident to a degree 1 vertex and $\al$ and $\be$ also share that degree 1 vertex, $j=1,2$. All this together means that $\al=\be=(s_1t_1)\cup(t_2s_2)$, contradicting $\al\neq\be$.
	
	\medskip
	Suppose next that $\al$ and $\be$ share a degree 1 vertex, whose half-edge is denoted $s$. Then by what we have supposed for contradiction, $s$ has the same partner in $\al$ and in $\be$, which is denoted $t$, and in particular, $\al$ and $\be$ inevitably share the degree 2 vertex to which $t$ is incident. The preceding argument now shows that $\al=\be$, again contradicting $\al\neq\be$. \qed

	\subsubsection{Proof of Lemma \ref{disjoint multigraphs non creation}}
	\label{proof of disjoint multigraphs non creation}
	
	Let $s_{\ell,1}s_{\ell,2}$, $1\le\ell\le e(H)$ denote the $e(H)$ edges of $\al$ arranged in ascending order such that $s_{1,1}\wedge s_{1,2}<\cdots<s_{e(H),1}\wedge s_{e(H),2}$. Suppose that $\al$ and $\be$ share no half-edge, and let $g\in\mG_\al\setminus\mG_\be$ and $b=\qty(b_\ell)_{\ell=1}^{e(H)}\in\prod_{\ell=1}^{e(H)}\qty[N-2(e(H)-\ell)-1]$. Define recursively $g_0\coloneqq g\in\mG_\al$ and $g_\ell\coloneqq f_{\al,\ell}(g_{\ell-1},b_\ell)\in\mG_{\al,\ell}$, $\ell=1,\dots,e(H)$. Then $f_{\al}(g,b)=g_{e(H)}$ by definition. For each $1\le\ell\le e(H)$, let $t_{\ell,1}$ denote the $b_\ell$-th smallest half-edge among
	\begin{align*}
		\{1,\dots,N \}\setminus\qty(\{s_{\ell,1}\wedge s_{\ell,2} \}\uplus\biguplus_{m=\ell+1}^{e(H)}\{s_{m,1},s_{m,2} \}),
	\end{align*}
	and let $t_{\ell,2}$ denote the partner of $t_{\ell,1}$ in $g_{\ell-1}$. We prove that $g_{\ell}\not\supset\be$ for all $0\le\ell\le e(H)$ by induction on $\ell$. $g_0=g\not\supset\be$. By induction hypothesis, $g_{\ell-1}\not\supset\be$, and in particular there is an edge $uv$ of $\be$ which is not in $g_{\ell-1}$. Since $\al$ and $\be$ share no half-edge, neither $s_{\ell,1}$ nor $s_{\ell,2}$ is a half-edge of $\be$. Thus neither $(s_{\ell,1}\wedge s_{\ell,2})t_{\ell,1}$ nor $(s_{\ell,1}\vee s_{\ell,2})t_{\ell,2}$ equals $uv$, which implies
	\begin{align*}
		g_\ell=f_{\al,\ell}(g_{\ell-1},b_\ell)=&\qty(g_{\ell-1}\setminus\qty{ \{s_{\ell,1},s_{\ell,2}\},\{t_{\ell,1},t_{\ell,2}\} })\uplus\qty{ \{s_{\ell,1}\wedge s_{\ell,2},t_{\ell,1} \},\{s_{\ell,1}\vee s_{\ell,2},t_{\ell,2}\} }\\
		\not\supset&\be.
	\end{align*}
	Therefore, $g_{\ell}\not\supset\be$ for all $0\le\ell\le e(H)$, whence $f_{\al}(g,b)=g_{e(H)}\not\supset\be$. 
	
	\medskip
	Finally, suppose that $\al$ and $\be$ share no half-edge and $I_\be=0$. If $I_\al=0$, then $J_{\be\al}=I_\be=0$. If $I_\al=1$, then $\mathfrak{g}\in\mG_\al\setminus\mG_\be$ and we have $\mathfrak{g}_\al=f_\al(\mathfrak{g},B_\al)\not\supset\be$, i.e., $J_{\be\al}=0$. \qed
	
	\subsubsection{Proof of Lemma \ref{destroyed al be disjoint}}
	Let $g\in\mG_\al\cap\mG_\be$ and $b=\qty(b_{\ell})_{\ell=1}^{e(H)}\in\prod_{\ell=1}^{e(H)}[N-2(e(H)-\ell)-1]$. Define recursively $g_0\coloneqq g\in\mG_\al$ and $g_\ell\coloneqq f_{\al,\ell}(g_{\ell-1},b_\ell)\in\mG_{\al,\ell}$, $\ell=1,\dots,e(H)$. Then $f_{\al}(g,b)=g_{e(H)}$ by definition. For each $1\le\ell\le e(H)$, let $t_{\ell,1}$ denote the $b_\ell$-th smallest half-edge among
	\begin{equation}\label{destroyed lemma choice}
		\{1,\dots,N \}\setminus\qty(\{s_{\ell,1}\wedge s_{\ell,2} \}\uplus\biguplus_{m=\ell+1}^{e(H)}\{s_{m,1},s_{m,2} \}),
	\end{equation}
	and let $t_{\ell,2}$ denote the partner of $t_{\ell,1}$ in $g_{\ell-1}$. First suppose that $b_{\ell}\notin R_{\al,\be,\ell}$ for all $1\le\ell\le e(H)$.  We prove that $g_{\ell}\supset\be$ for all $0\le\ell\le e(H)$ by induction on $\ell$. $g_0\supset\be$. By induction hypothesis, $g_{\ell-1}\supset\be$. Since $b_\ell\notin R_{\al,\be,\ell}$, $t_{\ell,1}$ is not a half-edge of $\be$. Thus $t_{\ell,1}t_{\ell,2}$ is not an edge of $\be$. Since $\al$ and $\be$ share no half-edge, $s_{\ell,1}s_{\ell,2}$ is also not an edge of $\be$. Thus $g_{\ell-1}\setminus\qty{ \{s_{\ell,1},s_{\ell,2}\},\{t_{\ell,1},t_{\ell,2}\} }$ still contains all the edges of $\be$. Thus
	\begin{align*}
		g_\ell=f_{\al,\ell}(g_{\ell-1},b_\ell)=&\qty(g_{\ell-1}\setminus\qty{ \{s_{\ell,1},s_{\ell,2}\},\{t_{\ell,1},t_{\ell,2}\} })\uplus\qty{ \{s_{\ell,1}\wedge s_{\ell,2},t_{\ell,1} \},\{s_{\ell,1}\vee s_{\ell,2},t_{\ell,2}\} }\\
		\supset&\be.
	\end{align*}
	So $g_{\ell}\supset\be$ for all $0\le\ell\le e(H)$, whence $f_{\al}(g,b)=g_{e(H)}\supset\be$.
	
	\medskip
	Conversely, suppose this time that there is $1\le\ell_0\le e(H)$ such that $b_{\ell_0}\in R_{\al,\be,\ell_0}$. We prove that $g_{\ell}\not\supset\be$ for all $\ell_0\le\ell\le e(H)$ by induction on $\ell$. Since $t_{\ell_0,1}$ is the $b_{\ell_0}$-th smallest half-edge among \eqref{destroyed lemma choice}, it is a half-edge of $\be$. Since $\al$ and $\be$ share no half-edge, $s_{\ell_0,1}\wedge s_{\ell_0,2}$ is not a half-edge of $\be$. In $g_{\ell_0}$,
	\begin{align*}
		g_{\ell_0}=&f_{\al,\ell_0}(g_{\ell_0-1},b_{\ell_0})\\
		=&\qty(g_{\ell_0-1}\setminus\qty{ \{s_{\ell_0,1},s_{\ell_0,2}\},\{t_{\ell_0,1},t_{\ell_0,2}\} })\uplus\qty{ \{s_{\ell_0,1}\wedge s_{\ell_0,2},t_{\ell_0,1} \},\{s_{\ell_0,1}\vee s_{\ell_0,2},t_{\ell_0,2}\} },
	\end{align*}
	$t_{\ell_0,1}$ is paired to $s_{\ell_0,1}\wedge s_{\ell_0,2}$, so $\be$'s edge containing $t_{\ell_0,1}$ is not formed in $g_{\ell_0}$. Thus $g_{\ell_0}\not\supset\be$. By induction hypothesis, $g_{\ell-1}\not\supset\be$, and in particular there is an edge $uv$ of $\be$ which is not in $g_{\ell-1}$. Since $\al$ and $\be$ share no half-edge, neither $s_{\ell,1}$ nor $s_{\ell,2}$ is a half-edge of $\be$. Thus neither $(s_{\ell,1}\wedge s_{\ell,2})t_{\ell,1}$ nor $(s_{\ell,1}\vee s_{\ell,2})t_{\ell,2}$ equals $uv$, which implies
	\begin{align*}
		g_\ell=f_{\al,\ell}(g_{\ell-1},b_\ell)=&\qty(g_{\ell-1}\setminus\qty{ \{s_{\ell,1},s_{\ell,2}\},\{t_{\ell,1},t_{\ell,2}\} })\uplus\qty{ \{s_{\ell,1}\wedge s_{\ell,2},t_{\ell,1} \},\{s_{\ell,1}\vee s_{\ell,2},t_{\ell,2}\} }\\
		\not\supset&\be.
	\end{align*}
	Therefore, $g_{\ell}\not\supset\be$ for all $\ell_0\le\ell\le e(H)$, whence $f_{\al}(g,b)=g_{e(H)}\not\supset\be$. \qed

	\subsubsection{Proof of Lemma \ref{surjection edge edge}}
	Let $\al$ and $\be$ be possible isolated edges such that $\al\cap\be\neq\varnothing$ and $\al\neq\be$. Write $\al=s_1s_2$ and $\be=t_1t_2$, and assume that $s_2=t_1$ without loss of generality. Let $g\in\mG_\al$. Define $b_1=b_{1}(g)\in[N-1]$ by
	\begin{align*}
		b_{1}(g)\coloneqq\begin{cases}
		\begin{aligned}&	\text{the rank of the partner of $t_2$ in $g$}\\&\quad\text{among the $N-1$ half-edges other than $s_1$}\end{aligned}	&  (s_1<s_2)  \\
			\text{the rank of $t_2$ among the $N-1$ half-edges other than $s_2$}	&  (s_2<s_1)
		\end{cases}.
	\end{align*}
	Then define a map $h_{\al,\be}:\mG_\al\rightarrow\mG$ by $h_{\al,\be}(g)\coloneqq f_\al(g,b_{1}(g))$.
	Namely, we restrict $f_\al$ such that the bond between $s_1$ and $s_2$ is surely unpaired and the edge $\be=t_1t_2=s_2t_2$ is surely created. Let $g_\be\in\mG_\be$. Then $\qty(f^{-1}_\al(g_\be))_2$ is the rank of the partner of $s_1\wedge s_2$ in $g_\be$ among the $N-1$ half-edges other than $s_1\wedge s_2$,  and the partner of $s_1$ in $g_\be$ is the partner of $t_2$ in $\qty(f^{-1}_\al(g_\be))_1$. Thus $\qty(f^{-1}_\al(g_\be))_2=b_1\qty(\qty(f^{-1}_\al(g_\be))_1)$ and $h_{\al,\be}\qty(\qty(f^{-1}_\al(g_\be))_1)=f_\al\qty(\qty(f^{-1}_\al(g_\be))_1,\qty(f^{-1}_\al(g_\be))_2)=g_\be$. This shows that $h_{\al,\be}$ is a surjection onto $\mG_\be$. \qed
	
	\subsubsection{Proof of Lemma \ref{surjection al edge be 2star}}
	(a) Let $g\in\mG_\al$. Define $b_1=b_{1}(g)\in[N-1]$ by
	\begin{align*}
		b_{1}(g)\coloneqq\begin{cases}\begin{aligned}&
			\text{the rank of the partner of $t_2$ in $g$ }\\&\quad\text{among the $N-1$ half-edges other than $s_1$}\end{aligned}	&  (s_1<s_2)  \\
			\text{the rank of $t_2$ among the $N-1$ half-edges other than $s_2$}	&  (s_2<s_1)
		\end{cases}.
	\end{align*}
	Then define a map $h_{\al,\be}:\mG_\al\rightarrow\mG$ by $h_{\al,\be}(g)\coloneqq f_\al(g,b_{1}(g))$.
	Namely, we restrict $f_\al$ such that the bond between $s_1$ and $s_2$ is surely unpaired and one of the edges of $\be$, $t_1t_2=s_2t_2$ is surely created. Let $g_\be\in\mG_\be$. One of the edges of $\be$, $t_3t_4$ remains in $\qty(f^{-1}_\al(g_\be))_1$. $\qty(f^{-1}_\al(g_\be))_2$ is the rank of the partner of $s_1\wedge s_2$ in $g_\be$ among the $N-1$ half-edges other than $s_1\wedge s_2$,  and the partner of $s_1$ in $g_\be$ is the partner of $t_2$ in $\qty(f^{-1}_\al(g_\be))_1$. Thus $\qty(f^{-1}_\al(g_\be))_2=b_1\qty(\qty(f^{-1}_\al(g_\be))_1)$ and $h_{\al,\be}\qty(\qty(f^{-1}_\al(g_\be))_1)=f_\al\qty(\qty(f^{-1}_\al(g_\be))_1,\qty(f^{-1}_\al(g_\be))_2)=g_\be$. This shows that $h_{\al,\be}$ is a surjection from $\mG_\al\cap\mG_{t_3t_4}$ onto $\mG_\be$.
	
	\medskip
	(b) Let $g\in\mG_\al$. Define $b_1\in[N-1]$ by
	\begin{align*}
		b_{1}\coloneqq\begin{cases}
			\text{the rank of $t_2$ among the $N-1$ half-edges other than $s_1$}	&  (s_1<s_2)  \\
			\text{the rank of $t_3$ among the $N-1$ half-edges other than $s_2$}	&  (s_2<s_1)
		\end{cases}.
	\end{align*}
	Then define a map $h_{\al,\be}:\mG_\al\rightarrow\mG$ by $h_{\al,\be}(g)\coloneqq f_\al(g,b_{1})$. Namely, we restrict $f_\al$ such that the bond between $s_1$ and $s_2$ is surely unpaired and the edge $t_1t_2 =s_1t_2$ is surely created if $s_1<s_2$, whereas the edge $t_3t_4=t_3s_2$ is surely created if $s_2<s_1$. Note that if the self-loop $t_2t_3$ exists in $g$, then both of the edges $t_1t_2=s_1t_2$ and $t_3t_4=t_3s_2$ will be created simultaneously by $h_{\al,\be}$, whether $s_1<s_2$ or $s_2<s_1$.
	
	\medskip
	Let $g_\be\in\mG_\be$. Then in $\qty(f^{-1}_\al(g_\be))_1$, $t_2$ and $t_3$ are paired to each other to form a self-loop $t_2t_3$. $\qty(f^{-1}_\al(g_\be))_2$ is the rank of the partner of $s_1\wedge s_2$ in $g_\be$ among the $N-1$ half-edges other than $s_1\wedge s_2$, and the partners of $s_1=t_1$ and $s_2=t_4$ in $g_\be$ are $t_2$ and $t_3$, respectively. Thus $\qty(f^{-1}_\al(g_\be))_2=b_1$ and $h_{\al,\be}\qty(\qty(f^{-1}_\al(g_\be))_1)=f_\al\qty(\qty(f^{-1}_\al(g_\be))_1,\qty(f^{-1}_\al(g_\be))_2)=g_\be$. This shows that $h_{\al,\be}$ is a surjection from $\mG_\al\cap\mG_{t_2t_3}$ onto $\mG_\be$. \qed

	\subsubsection{Proof of Lemma \ref{surjection al 2star be edge}}
	Let $\al$ be a possible isolated 2-star and let $\be$ be a possible isolated edge such that $\al\cap\be\neq\varnothing$ and $\al\neq\be$. Write $\al=(s_1s_2)\cup (s_3s_4)$ and $\be=t_1t_2$, where $s_2$ and $s_3$ are incident to $\al$'s degree 2 vertex. There are two possible cases to consider: (a) $\al$ and $\be$ share one degree 1 vertex only, and (b) $\al$ and $\be$ share two degree 1 vertices. Without loss of generality, we assume that $s_4=t_1$ in Case (a) and we assume that $s_1=t_1$ and $s_4=t_2$ in Case (b).
	
	\medskip
	\underline{Case (a)}. First consider the case of $s_3\wedge s_4<s_1\wedge s_2$. Let $g\in\mG_\al$ and $b_2\in[N-1]$. Define $b_1=b_1(g)\in[N-3]$ by
	\begin{align*}
		b_1(g)\coloneqq\begin{cases}\begin{aligned}&
			\text{the rank of the partner of $t_2$ in $g$}\\&\quad\text{among the $N-3$ half-edges other than $s_1,s_2,s_3$}\end{aligned}	& (s_3<s_4) \\
			\text{the rank of $t_2$ among the $N-3$ half-edges other than $s_1,s_2,s_4$}	& (s_4<s_3)
		\end{cases}.
	\end{align*}
	Then define a map $h_{\al,\be}:\mG_\al\times[N-1]\rightarrow\mG$ by $h_{\al,\be}(g,b_2)\coloneqq f_\al(g,b_1(g),b_2)=f_{\al,2}\qty(f_{\al,1}\qty(g,b_1(g)),b_2)$. Namely, we restrict $f_{\al,1}$ such that it surely breaks up the bond between $s_3$ and $s_4$ and creates the edge $t_1t_2=s_4t_2$. Let $g_\be\in\mG_\be$. Let $\qty(f^{-1}_{\al,2}\qty(g_\be))_1\in\mG_{\al,1}=\mG_{s_1s_2}$ be the first coordinate of $f^{-1}_{\al,2}\qty(g_\be)$.  The second coordinate of $f^{-1}_{\al,1}\qty(\qty(f^{-1}_{\al,2}\qty(g_\be))_1)$,
	\begin{align*}
		\qty(f^{-1}_{\al,1}\qty(\qty(f^{-1}_{\al,2}\qty(g_\be))_1))_2\in[N-3],
	\end{align*}
	is the rank of the partner of $s_3\wedge s_4$ in $\qty(f^{-1}_{\al,2}\qty(g_\be))_1$ among the $N-3$ half-edges other than $s_3\wedge s_4,s_1,s_2$. The edge $t_1t_2=s_4t_2$ remains in $\qty(f^{-1}_{\al,2}\qty(g_\be))_1$, i.e., the partner of $s_4$ in $\qty(f^{-1}_{\al,2}\qty(g_\be))_1$ is $t_2$. The partner of $s_3$ in $\qty(f^{-1}_{\al,2}\qty(g_\be))_1$ equals the partner of $t_2$ in the first coordinate of $f^{-1}_{\al,1}\qty(\qty(f^{-1}_{\al,2}\qty(g_\be))_1)$,
	\begin{align*}
		\qty(f^{-1}_{\al,1}\qty(\qty(f^{-1}_{\al,2}\qty(g_\be))_1))_1=\qty(f^{-1}_\al(g_\be))_1\in\mG_\al.
	\end{align*}
	Thus we have
	\begin{align*}
		\qty(f^{-1}_{\al,1}\qty(\qty(f^{-1}_{\al,2}\qty(g_\be))_1))_2=b_1\qty(\qty(f^{-1}_\al(g_\be))_1).
	\end{align*}
	By taking $g=\qty(f^{-1}_\al(g_\be))_1\in\mG_\al$ and $b_2=\qty(f^{-1}_{\al,2}\qty(g_\be))_2\in[N-1]$, we have
	\begin{align*}
		\qty(b_1\qty(g),b_2)=\qty(\qty(f^{-1}_{\al,1}\qty(\qty(f^{-1}_{\al,2}\qty(g_\be))_1))_2,\qty(f^{-1}_{\al,2}\qty(g_\be))_2)	=\qty(f^{-1}_\al(g_\be))_2\in[N-3]\times[N-1]
	\end{align*}
	and $h_{\al,\be}(g,b_2)=f_\al\qty(\qty(f^{-1}_\al(g_\be))_1,\qty(f^{-1}_\al(g_\be))_2)=g_\be$. This shows that $h_{\al,\be}$ is a surjection from $\mG_\al\times[N-1]$ onto $\mG_\be$ with the desired property.
	
	\medskip
	Next consider the case of $s_1\wedge s_2<s_3\wedge s_4$. Let $g\in\mG_\al$ and $b_1\in[N-3]$. Define $b_2=b_2(g,b_1)\in[N-1]$ by
	\begin{align*}
		b_2(g,b_1)\coloneqq\begin{cases}\begin{aligned}&
			\text{the rank of the partner of $t_2$ in $f_{\al,1}(g,b_1)$}\\&\quad\text{among the $N-1$ half-edges other than $s_3$}\end{aligned}	& (s_3<s_4) \\
			\text{the rank of $t_2$ among the $N-1$ half-edges other than $s_4$}	& (s_4<s_3)
		\end{cases}.
	\end{align*}
	Then define a map $h_{\al,\be}:\mG_\al\times[N-3]\rightarrow\mG$ by $h_{\al,\be}(g,b_1)\coloneqq f_\al\qty(g,b_1,b_2(g,b_1))=f_{\al,2}\qty(f_{\al,1}\qty(g,b_1),b_2(g,b_1))$.
	Namely, we restrict $f_{\al,2}$ such that it surely breaks up the bond between $s_3$ and $s_4$ and creates the edge $t_1t_2=s_4t_2$. Let $g_\be\in\mG_\be$. Take $g\in\mG_\al$ as $g=\qty(f^{-1}_\al(g_\be))_1$. Let $\qty(f^{-1}_{\al,2}\qty(g_\be))_1\in\mG_{\al,1}=\mG_{s_3s_4}$ be the first coordinate of $f^{-1}_{\al,2}\qty(g_\be)$. Take $b_1=b_1(g_\be)\in[N-3]$ as the second coordinate of $f^{-1}_{\al,1}\qty(\qty(f^{-1}_{\al,2}\qty(g_\be))_1)$, 
	\begin{align*}
		b_1=\qty(f^{-1}_{\al,1}\qty(\qty(f^{-1}_{\al,2}\qty(g_\be))_1))_2.
	\end{align*}
	Then $f_{\al,1}\qty(g,b_1)=\qty(f^{-1}_{\al,2}\qty(g_\be))_1$. The second coordinate of $f^{-1}_{\al,2}\qty(g_\be)$, $\qty(f^{-1}_{\al,2}\qty(g_\be))_2\in[N-1]$, is the rank of the partner of $s_3\wedge s_4$ in $g_\be$ among the $N-1$ half-edges other than $s_3\wedge s_4$. The partner of $s_3$ in $g_\be$ is the partner of $t_2$ in $\qty(f^{-1}_{\al,2}\qty(g_\be))_1$. The partner of $s_4=t_1$ in $g_\be$ is $t_2$. Thus $\qty(f^{-1}_{\al,2}\qty(g_\be))_2=b_2(g,b_1)$ and we have
	\begin{align*}
		\qty(b_1,b_2(g,b_1))=\qty(\qty(f^{-1}_{\al,1}\qty(\qty(f^{-1}_{\al,2}\qty(g_\be))_1))_2,\qty(f^{-1}_{\al,2}\qty(g_\be))_2)=\qty(f^{-1}_\al\qty(g_\be))_2\in[N-3]\times[N-1]
	\end{align*}
	and $h_{\al,\be}(g,b_1)=f_\al\qty(\qty(f^{-1}_\al(g_\be))_1,\qty(f^{-1}_\al(g_\be))_2)=g_\be$.  This shows that $h_{\al,\be}$ is a surjection from $\mG_\al\times[N-3]$ onto $\mG_\be$. By extending the domain of $h_{\al,\be}$ to $\mG_\al\times[N-1]$ (taking $h_{\al,\be}(g,b_1)$ for $b_1=N-2,N-1$ to be, say, $g$), we prove that $h_{\al,\be}$ is a surjection from $\mG_\al\times[N-1]$ onto $\mG_\be$ with the desired property.
	
	\medskip	
	\underline{Case (b)}. We may assume that $s_1\wedge s_2<s_3\wedge s_4$ in this case without loss of generality. Let $g\in\mG_\al$ and $b_1\in[N-3]$. Define $b_2=b_2(g,b_1)\in[N-1]$ by
	\begin{align*}
		b_2(g,b_1)\coloneqq\begin{cases}\begin{aligned}&
			\text{the rank of the partner of $s_1$ in $f_{\al,1}(g,b_1)$}\\&\quad\text{among the $N-1$ half-edges other than $s_3$}\end{aligned}	& (s_3<s_4) \\
			\text{the rank of $s_1$ among the $N-1$ half-edges other than $s_4$}	& (s_4<s_3)
		\end{cases}.
	\end{align*}
	Then define a map $h_{\al,\be}:\mG_\al\times[N-3]\rightarrow\mG$ by $h_{\al,\be}(g,b_1)\coloneqq f_\al\qty(g,b_1,b_2(g,b_1))=f_{\al,2}\qty(f_{\al,1}\qty(g,b_1),b_2(g,b_1))$.
	Namely, we restrict $f_{\al,2}$ such that it surely breaks up the bond between $s_3$ and $s_4$ and creates the edge $t_1t_2=s_1s_4$. Let $g_\be\in\mG_\be=\mG_{s_1s_4}$. Take $g\in\mG_\al$ as $g=\qty(f^{-1}_\al(g_\be))_1$. Let $\qty(f^{-1}_{\al,2}\qty(g_\be))_1\in\mG_{\al,1}=\mG_{s_3s_4}$ be the first coordinate of $f^{-1}_{\al,2}\qty(g_\be)$. Take $b_1=b_1(g_\be)\in[N-3]$ as the second coordinate of $f^{-1}_{\al,1}\qty(\qty(f^{-1}_{\al,2}\qty(g_\be))_1)$, \begin{align*}
		b_1=\qty(f^{-1}_{\al,1}\qty(\qty(f^{-1}_{\al,2}\qty(g_\be))_1))_2.
	\end{align*}
	Then $f_{\al,1}\qty(g,b_1)=\qty(f^{-1}_{\al,2}\qty(g_\be))_1$. The second coordinate of $f^{-1}_{\al,2}\qty(g_\be)$, $\qty(f^{-1}_{\al,2}\qty(g_\be))_2\in[N-1]$, is the rank of the partner of $s_3\wedge s_4$ in $g_\be$ among the $N-1$ half-edges other than $s_3\wedge s_4$, and the partner of $s_3$ in $g_\be$ is the partner of $s_1$ in $\qty(f^{-1}_{\al,2}\qty(g_\be))_1$. Thus $\qty(f^{-1}_{\al,2}\qty(g_\be))_2=b_2(g,b_1)$ and we have
	\begin{align*}
		\qty(b_1,b_2(g,b_1))=\qty(\qty(f^{-1}_{\al,1}\qty(\qty(f^{-1}_{\al,2}\qty(g_\be))_1))_2,\qty(f^{-1}_{\al,2}\qty(g_\be))_2)=\qty(f^{-1}_\al\qty(g_\be))_2\in[N-3]\times[N-1]
	\end{align*}
	and $h_{\al,\be}(g,b_1)=f_\al\qty(\qty(f^{-1}_\al(g_\be))_1,\qty(f^{-1}_\al(g_\be))_2)=g_\be$.  This shows that $h_{\al,\be}$ is a surjection from $\mG_\al\times[N-3]$ onto $\mG_\be$. By extending the domain of $h_{\al,\be}$ to $\mG_\al\times[N-1]$ (taking $h_{\al,\be}(g,b_1)$ for $b_1=N-2,N-1$ to be, say, $g$), we prove that $h_{\al,\be}$ is a surjection from $\mG_\al\times[N-1]$ onto $\mG_\be$ with the desired property.
	
	\medskip
	Therefore, there exists a desired surjection $h_{\al,\be}$ for each of Cases (a) and (b), completing the proof of the lemma. \qed

	\subsubsection{Proof of Lemma \ref{surjection 2star 2star}}
	(a) First consider the case of $s_3\wedge s_4<s_1\wedge s_2$.  Let $g\in\mG_\al$ and $b_2\in[N-1]$. Define $b_1=b_1(g)\in[N-3]$ by
	\begin{align*}
		b_1(g)\coloneqq\begin{cases}\begin{aligned}&
			\text{the rank of the partner of $t_2$ in $g$}\\&\quad\text{among the $N-3$ half-edges other than $s_1,s_2,s_3$}\end{aligned}	& (s_3<s_4) \\
			\text{the rank of $t_2$ among the $N-3$ half-edges other than $s_1,s_2,s_4$}	& (s_4<s_3)
		\end{cases}.
	\end{align*}
	Then define a map $h_{\al,\be}:\mG_\al\times[N-1]\rightarrow\mG$ by $h_{\al,\be}(g,b_2)\coloneqq f_\al(g,b_1(g),b_2)=f_{\al,2}\qty(f_{\al,1}\qty(g,b_1(g)),b_2)$. Namely, we restrict $f_{\al,1}$ such that it surely breaks up the bond between $s_3$ and $s_4$ and creates the edge $t_1t_2=s_4t_2$. Let $g_\be\in\mG_\be$. Let $\qty(f^{-1}_{\al,2}\qty(g_\be))_1\in\mG_{\al,1}=\mG_{s_1s_2}$ be the first coordinate of $f^{-1}_{\al,2}\qty(g_\be)$.  The second coordinate of $f^{-1}_{\al,1}\qty(\qty(f^{-1}_{\al,2}\qty(g_\be))_1)$,
	\begin{align*}
		\qty(f^{-1}_{\al,1}\qty(\qty(f^{-1}_{\al,2}\qty(g_\be))_1))_2\in[N-3],
	\end{align*}
	is the rank of the partner of $s_3\wedge s_4$ in $\qty(f^{-1}_{\al,2}\qty(g_\be))_1$ among the $N-3$ half-edges other than $s_3\wedge s_4,s_1,s_2$. The edge $t_1t_2=s_4t_2$ remains in $\qty(f^{-1}_{\al,2}\qty(g_\be))_1$, i.e., the partner of $s_4$ in $\qty(f^{-1}_{\al,2}\qty(g_\be))_1$ is $t_2$. The partner of $s_3$ in $\qty(f^{-1}_{\al,2}\qty(g_\be))_1$ equals the partner of $t_2$ in the first coordinate of $f^{-1}_{\al,1}\qty(\qty(f^{-1}_{\al,2}\qty(g_\be))_1)$,
	\begin{align*}
		\qty(f^{-1}_{\al,1}\qty(\qty(f^{-1}_{\al,2}\qty(g_\be))_1))_1=\qty(f^{-1}_\al(g_\be))_1\in\mG_\al.
	\end{align*}
	Thus we have
	\begin{align*}
		\qty(f^{-1}_{\al,1}\qty(\qty(f^{-1}_{\al,2}\qty(g_\be))_1))_2=b_1\qty(\qty(f^{-1}_\al(g_\be))_1).
	\end{align*}
	By taking $g=\qty(f^{-1}_\al(g_\be))_1\in\mG_\al$ and $b_2=\qty(f^{-1}_{\al,2}\qty(g_\be))_2\in[N-1]$, we have
	\begin{align*}
		\qty(b_1\qty(g),b_2)=\qty(\qty(f^{-1}_{\al,1}\qty(\qty(f^{-1}_{\al,2}\qty(g_\be))_1))_2,\qty(f^{-1}_{\al,2}\qty(g_\be))_2)	=\qty(f^{-1}_\al(g_\be))_2\in[N-3]\times[N-1]
	\end{align*}
	and $h_{\al,\be}(g,b_2)=f_\al\qty(\qty(f^{-1}_\al(g_\be))_1,\qty(f^{-1}_\al(g_\be))_2)=g_\be$. It is straightforward to check that the edge $t_3t_4$ remains in $g=\qty(f^{-1}_\al(g_\be))_1$.\footnote{Let $u_{2,1},u_{2,2}$ denote the partners of $s_1\wedge s_2$ and $s_1\vee s_2$ in $g_\be$, respectively. Since $t_3$ and $t_4$ are paired to each other in $g_\be$, $\{u_{2,1},u_{2,2}\}\cap\{t_3,t_4\}=\varnothing$. Thus \begin{align*}
				\qty(f^{-1}_{\al,2}(g_\be))_1=
			\qty(g_\be\setminus\qty{\{ s_{1}\wedge s_{2},u_{2,1}\},\{s_{1}\vee s_{2},u_{2,2}\}})\uplus\qty{\{s_{1},s_{2}\},\{u_{2,1},u_{2,2} \} }
	\end{align*}
	still contains the edge $t_3t_4$. Let $u_{1,1},u_{1,2}$ denote the partners of $s_3\wedge s_4$ and $s_3\vee s_4$ in $	\qty(f^{-1}_{\al,2}(g_\be))_1$, respectively. Since $t_3$ and $t_4$ are paired to each other in $\qty(f^{-1}_{\al,2}(g_\be))_1$, $\{u_{1,1},u_{1,2}\}\cap\{t_3,t_4\}=\varnothing$. Thus
\begin{align*}
		\qty(f^{-1}_{\al}(g_\be))_{1}=\qty(f^{-1}_{\al,1}\qty(\qty(f^{-1}_{\al,2}(g_\be))_1))_1=
	\qty(\qty(f^{-1}_{\al,2}(g_\be))_1\setminus\qty{\{ s_{3}\wedge s_{4},u_{1,1}\},\{s_{3}\vee s_{4},u_{1,2}\}})\uplus\qty{\{s_{3},s_{4}\},\{u_{1,1},u_{1,2} \} }
	\end{align*}
still contains the edge $t_3t_4$.} This shows that $h_{\al,\be}$ is a surjection from $\qty(\mG_\al\cap\mG_{t_3t_4})\times[N-1]$ onto $\mG_\be$ with the desired property.
	
	\medskip
	Next consider the case of $s_1\wedge s_2<s_3\wedge s_4$. Let $g\in\mG_\al$ and $b_1\in[N-3]$. Define $b_2=b_2(g,b_1)\in[N-1]$ by
	\begin{align*}
		b_2(g,b_1)\coloneqq\begin{cases}\begin{aligned}&
			\text{the rank of the partner of $t_2$ in $f_{\al,1}(g,b_1)$ }\\&\quad\text{among the $N-1$ half-edges other than $s_3$}\end{aligned}	& (s_3<s_4) \\
			\text{the rank of $t_2$ among the $N-1$ half-edges other than $s_4$}	& (s_4<s_3)
		\end{cases}.
	\end{align*}
	Then define a map $h_{\al,\be}:\mG_\al\times[N-3]\rightarrow\mG$ by $h_{\al,\be}(g,b_1)\coloneqq f_\al\qty(g,b_1,b_2(g,b_1))=f_{\al,2}\qty(f_{\al,1}\qty(g,b_1),b_2(g,b_1))$.
	Namely, we restrict $f_{\al,2}$ such that it surely breaks up the bond between $s_3$ and $s_4$ and creates the edge $t_1t_2=s_4t_2$. Let $g_\be\in\mG_\be$. Take $g\in\mG_\al$ as $g=\qty(f^{-1}_\al(g_\be))_1$. Let $\qty(f^{-1}_{\al,2}\qty(g_\be))_1\in\mG_{\al,1}=\mG_{s_3s_4}$ be the first coordinate of $f^{-1}_{\al,2}\qty(g_\be)$. Take $b_1=b_1(g_\be)\in[N-3]$ as the second coordinate of $f^{-1}_{\al,1}\qty(\qty(f^{-1}_{\al,2}\qty(g_\be))_1)$, 
	\begin{align*}
		b_1=\qty(f^{-1}_{\al,1}\qty(\qty(f^{-1}_{\al,2}\qty(g_\be))_1))_2.
	\end{align*}
	Then $f_{\al,1}\qty(g,b_1)=\qty(f^{-1}_{\al,2}\qty(g_\be))_1$. The second coordinate of $f^{-1}_{\al,2}\qty(g_\be)$, $\qty(f^{-1}_{\al,2}\qty(g_\be))_2\in[N-1]$, is the rank of the partner of $s_3\wedge s_4$ in $g_\be$ among the $N-1$ half-edges other than $s_3\wedge s_4$. The partner of $s_3$ in $g_\be$ is the partner of $t_2$ in $\qty(f^{-1}_{\al,2}\qty(g_\be))_1$. The partner of $s_4=t_1$ in $g_\be$ is $t_2$. Thus $\qty(f^{-1}_{\al,2}\qty(g_\be))_2=b_2(g,b_1)$ and we have
	\begin{align*}
		\qty(b_1,b_2(g,b_1))=\qty(\qty(f^{-1}_{\al,1}\qty(\qty(f^{-1}_{\al,2}\qty(g_\be))_1))_2,\qty(f^{-1}_{\al,2}\qty(g_\be))_2)=\qty(f^{-1}_\al\qty(g_\be))_2\in[N-3]\times[N-1]
	\end{align*}
	and $h_{\al,\be}(g,b_1)=f_\al\qty(\qty(f^{-1}_\al(g_\be))_1,\qty(f^{-1}_\al(g_\be))_2)=g_\be$. It is straightforward to check that the edge $t_3t_4$ remains in $g=\qty(f^{-1}_\al(g_\be))_1$.\footnote{Let $u_{2,1},u_{2,2}$ denote the partners of $s_3\wedge s_4$ and $s_3\vee s_4$ in $g_\be$, respectively. Since $t_3$ and $t_4$ are paired to each other in $g_\be$, $\{u_{2,1},u_{2,2}\}\cap\{t_3,t_4\}=\varnothing$. Thus \begin{align*}
			\qty(f^{-1}_{\al,2}(g_\be))_1=
			\qty(g_\be\setminus\qty{\{ s_{3}\wedge s_{4},u_{2,1}\},\{s_{3}\vee s_{4},u_{2,2}\}})\uplus\qty{\{s_{3},s_{4}\},\{u_{2,1},u_{2,2} \} }
		\end{align*}
		still contains the edge $t_3t_4$. Let $u_{1,1},u_{1,2}$ denote the partners of $s_1\wedge s_2$ and $s_1\vee s_2$ in $	\qty(f^{-1}_{\al,2}(g_\be))_1$, respectively. Since $t_3$ and $t_4$ are paired to each other in $\qty(f^{-1}_{\al,2}(g_\be))_1$, $\{u_{1,1},u_{1,2}\}\cap\{t_3,t_4\}=\varnothing$. Thus
		\begin{align*}
			\qty(f^{-1}_{\al}(g_\be))_{1}=\qty(f^{-1}_{\al,1}\qty(\qty(f^{-1}_{\al,2}(g_\be))_1))_1=
			\qty(\qty(f^{-1}_{\al,2}(g_\be))_1\setminus\qty{\{ s_{1}\wedge s_{2},u_{1,1}\},\{s_{1}\vee s_{2},u_{1,2}\}})\uplus\qty{\{s_{1},s_{2}\},\{u_{1,1},u_{1,2} \} }
		\end{align*}
		still contains the edge $t_3t_4$.} This shows that $h_{\al,\be}$ is a surjection from $\qty(\mG_\al\cap\mG_{t_3t_4})\times[N-3]$ onto $\mG_\be$. By extending the domain of $h_{\al,\be}$ to $\qty(\mG_\al\cap\mG_{t_3t_4})\times[N-1]$ (taking $h_{\al,\be}(g,b_1)$ for $b_1=N-2,N-1$ to be, say, $g$), we prove that $h_{\al,\be}$ is a surjection from $\qty(\mG_\al\cap\mG_{t_3t_4})\times[N-1]$ onto $\mG_\be$ with the desired property.
	
	\medskip
	(b) By taking symmetry into account, there are in total five cases to consider other than (a) about $\al\cap\be,\al\neq\be$. In the following, we keep the convention that $s_2$ and $s_3$ (if appear) are incident to $\al$'s degree 2 vertex and $t_2$ and $t_3$ (if appear) are incident to $\be$'s degree 2 vertex. 
	{\setlength{\leftmargini}{29.5pt}  	
		\begin{itemize}
			\setlength{\labelsep}{7pt}     
			\setlength{\itemsep}{3pt}      
			
			\item[(b1)] $\al$ and $\be$ share two degree 1 vertices only, $\al=(s_1s_2)\cup(s_3s_4)$, $\be=(s_1t_2)\cup(t_3s_4)$, where $s_1,s_2,s_3,s_4,t_2,t_3$ are distinct.
			
			\item[(b2)] $\al$ and $\be$ share one degree 2 vertex only, $\al=(s_1s_2)\cup(s_3s_4)$, $\be=(t_1s_2)\cup(s_3t_4)$, where $s_1,s_2,s_3,s_4,t_1,t_4$ are distinct.
			
			\item[(b3)] $\al$ and $\be$ share one degree 1 vertex and one degree 2 vertex only, $\al=(s_1s_2)\cup(s_3s_4)$, $\be=(s_1s_2)\cup(s_3t_4)$, where $s_1,s_2,s_3,s_4,t_4$ are distinct.
			
			\item[(b4)]	$\al$ and $\be$ share one degree 1 vertex and one degree 2 vertex only, $\al=(s_1s_2)\cup(s_3s_4)$, $\be=(s_1s_3)\cup(s_2t_4)$, where $s_1,s_2,s_3,s_4,t_4$ are distinct.
			
			\item[(b5)] $\al$ and $\be$ share two degree 1 vertices and one degree 2 vertex, $\al=(s_1s_2)\cup(s_3s_4)$, $\be=(s_1s_3)\cup(s_2s_4)$ (where of course $s_1,s_2,s_3,s_4$ are distinct).
			
	\end{itemize} }
	
	\medskip
	\underline{Case (b1)}. We may assume that $s_1\wedge s_2<s_3\wedge s_4$ without loss of generality in this case. Let $g\in\mG_\al$. Define $b_1=b_1(g)\in[N-3]$ by
	\begin{align*}
		b_1(g)\coloneqq\begin{cases}
			\text{the rank of $t_2$ among the $N-3$ half-edges other than $s_1,s_3,s_4$}&(s_1<s_2)\\
		\begin{aligned}&	\text{the rank of the partner of $t_2$ in $g$ }\\&\quad\text{among the $N-3$ half-edges other than $s_2,s_3,s_4$} \end{aligned}&(s_2<s_1)
		\end{cases}.
	\end{align*}
	Then define $b_2=b_2(g)\in[N-1]$ by
	\begin{align*}
		b_2(g)\coloneqq\begin{cases}\begin{aligned}&
			\text{the rank of the partner of $t_3$ in $f_{\al,1}(g,b_1(g))$}\\&\quad\text{among the $N-1$ half-edges other than $s_3$}\end{aligned}&(s_3<s_4)\\
			\text{the rank of $t_3$ among the $N-1$ half-edges other than $s_4$} &(s_4<s_3)
		\end{cases}.
	\end{align*}
	Then define a map $h_{\al,\be}:\mG_\al\rightarrow\mG$ by $h_{\al,\be}(g)\coloneqq f_\al(g,b_1(g),b_2(g))=f_{\al,2}\qty(f_{\al,1}(g,b_1(g)),b_2(g))$. Namely, we restrict $f_{\al,1}$ such that it surely breaks up the bond between $s_1$ and $s_2$ and creates the edge $s_1t_2$, and restrict $f_{\al,2}$ such that it surely breaks up the bond between $s_3$ and $s_4$ and creates the edge $t_3s_4$. Let $g_\be\in\mG_\be$. Let $g=\qty(f^{-1}_\al(g_\be))_1\in\mG_\al$. Let $\qty(f^{-1}_{\al,2}(g_\be))_1\in\mG_{\al,1}=\mG_{s_3s_4}$ be the first coordinate of $f^{-1}_{\al,2}(g_\be)$. The second coordinate of $f^{-1}_{\al,1}\qty(\qty(f^{-1}_{\al,2}\qty(g_\be))_1)$,
	\begin{align*}
		\qty(f^{-1}_{\al,1}\qty(\qty(f^{-1}_{\al,2}\qty(g_\be))_1))_2\in[N-3],
	\end{align*}
	is the rank of the partner of $s_1\wedge s_2$ in $\qty(f^{-1}_{\al,2}(g_\be))_1$ among the $N-3$ half-edges other than $s_1\wedge s_2,s_3,s_4$. The edge $s_1t_2$ still exists in $\qty(f^{-1}_{\al,2}(g_\be))_1$, i.e., the partner of $s_1$ in $\qty(f^{-1}_{\al,2}(g_\be))_1$ is $t_2$. The partner of $s_2$ in $\qty(f^{-1}_{\al,2}(g_\be))_1$ is paired to $t_2$ in $g$. Thus,
	\begin{align*}
		b_1(g)=\qty(f^{-1}_{\al,1}\qty(\qty(f^{-1}_{\al,2}\qty(g_\be))_1))_2
	\end{align*}
	holds. Hence $f_{\al,1}\qty(g,b_1(g))=\qty(f^{-1}_{\al,2}\qty(g_\be))_1$ holds. The second coordinate of $f^{-1}_{\al,2}\qty(g_\be)$, $\qty(f^{-1}_{\al,2}\qty(g_\be))_2\in[N-1]$, is the rank of the partner of $s_3\wedge s_4$ in $g_\be$ among the $N-1$ half-edges other than $s_3\wedge s_4$. $s_4$ is paired to $t_3$ in $g_\be$, and the partner of $s_3$ in $g_\be$ is paired to $t_3$ in $\qty(f^{-1}_{\al,2}\qty(g_\be))_1=f_{\al,1}\qty(g,b_1(g))$. Thus, $b_2(g)=\qty(f^{-1}_{\al,2}\qty(g_\be))_2$ holds. Finally, we get
	\begin{align*}
		h_{\al,\be}\qty(g)=f_{\al,2}\qty(f_{\al,1}(g,b_1(g)),b_2(g))=f_{\al,2}\qty(\qty(f^{-1}_{\al,2}\qty(g_\be))_1,\qty(f^{-1}_{\al,2}\qty(g_\be))_2)=g_\be.
	\end{align*}
	This shows that $h_{\al,\be}$ is a surjection from $\mG_\al$ onto $\mG_\be$ with the desired property.
	
	\medskip
	\underline{Case (b2)}. We may assume that $s_1\wedge s_2<s_3\wedge s_4$ without loss of generality in this case. Let $g\in\mG_\al$. Define $b_1=b_1(g)\in[N-3]$ by
	\begin{align*}
		b_1(g)\coloneqq\begin{cases}\begin{aligned}&
			\text{the rank of the partner of $t_1$ in $g$}\\&\quad\text{among the $N-3$ half-edges other than $s_1,s_3,s_4$}\end{aligned}&(s_1<s_2)\\
			\text{the rank of $t_1$ among the $N-3$ half-edges other than $s_2,s_3,s_4$} &(s_2<s_1)
		\end{cases}.
	\end{align*}
	Then define $b_2=b_2(g)\in[N-1]$ by
	\begin{align*}
		b_2(g)\coloneqq\begin{cases}
			\text{the rank of $t_4$ among the $N-1$ half-edges other than $s_3$}&(s_3<s_4)\\
		\begin{aligned}	&\text{the rank of the partner of $t_4$ in $f_{\al,1}(g,b_1(g))$}\\&\quad\text{among the $N-1$ half-edges other than $s_4$}\end{aligned} &(s_4<s_3)
		\end{cases}.
	\end{align*}
	Then define a map $h_{\al,\be}:\mG_\al\rightarrow\mG$ by $h_{\al,\be}(g)\coloneqq f_\al(g,b_1(g),b_2(g))=f_{\al,2}\qty(f_{\al,1}(g,b_1(g)),b_2(g))$. Namely, we restrict $f_{\al,1}$ such that it surely breaks up the bond between $s_1$ and $s_2$ and creates the edge $t_1s_2$, and restrict $f_{\al,2}$ such that it surely breaks up the bond between $s_3$ and $s_4$ and creates the edge $s_3t_4$. Let $g_\be\in\mG_\be$. Let $g=\qty(f^{-1}_\al(g_\be))_1\in\mG_\al$. Let $\qty(f^{-1}_{\al,2}(g_\be))_1\in\mG_{\al,1}=\mG_{s_3s_4}$ be the first coordinate of $f^{-1}_{\al,2}(g_\be)$. The second coordinate of $f^{-1}_{\al,1}\qty(\qty(f^{-1}_{\al,2}\qty(g_\be))_1)$,
	\begin{align*}
		\qty(f^{-1}_{\al,1}\qty(\qty(f^{-1}_{\al,2}\qty(g_\be))_1))_2\in[N-3],
	\end{align*}
	is the rank of the partner of $s_1\wedge s_2$ in $\qty(f^{-1}_{\al,2}(g_\be))_1$ among the $N-3$ half-edges other than $s_1\wedge s_2,s_3,s_4$. The edge $t_1s_2$ still exists in $\qty(f^{-1}_{\al,2}(g_\be))_1$, i.e., the partner of $s_2$ in $\qty(f^{-1}_{\al,2}(g_\be))_1$ is $t_1$. The partner of $s_1$ in $\qty(f^{-1}_{\al,2}(g_\be))_1$ is paired to $t_1$ in $g$. Thus,
	\begin{align*}
		b_1(g)=\qty(f^{-1}_{\al,1}\qty(\qty(f^{-1}_{\al,2}\qty(g_\be))_1))_2
	\end{align*}
	holds. Hence $f_{\al,1}\qty(g,b_1(g))=\qty(f^{-1}_{\al,2}\qty(g_\be))_1$ holds. The second coordinate of $f^{-1}_{\al,2}\qty(g_\be)$, $\qty(f^{-1}_{\al,2}\qty(g_\be))_2\in[N-1]$, is the rank of the partner of $s_3\wedge s_4$ in $g_\be$ among the $N-1$ half-edges other than $s_3\wedge s_4$. $s_3$ is paired to $t_4$ in $g_\be$, and the partner of $s_4$ in $g_\be$ is paired to $t_4$ in $\qty(f^{-1}_{\al,2}\qty(g_\be))_1=f_{\al,1}\qty(g,b_1(g))$. Thus, $b_2(g)=\qty(f^{-1}_{\al,2}\qty(g_\be))_2$ holds. Finally, we get
	\begin{align*}
		h_{\al,\be}\qty(g)=f_{\al,2}\qty(f_{\al,1}(g,b_1(g)),b_2(g))=f_{\al,2}\qty(\qty(f^{-1}_{\al,2}\qty(g_\be))_1,\qty(f^{-1}_{\al,2}\qty(g_\be))_2)=g_\be.
	\end{align*}
	This shows that $h_{\al,\be}$ is a surjection from $\mG_\al$ onto $\mG_\be$ with the desired property.
	
	\medskip
	\underline{Case (b3)}. First consider the case of $s_1\wedge s_2<s_3\wedge s_4$. Let $g\in\mG_\al$. Let $b_1\in[N-3]$ be the rank of $s_1\vee s_2$ among the $N-3$ half-edges other than $s_1\wedge s_2,s_3,s_4$. Then define $b_2=b_2(g)\in[N-1]$ by
	\begin{align*}
		b_2(g)\coloneqq\begin{cases}
			\text{the rank of $t_4$ among the $N-1$ half-edges other than $s_3$}&(s_3<s_4)\\
		\begin{aligned}&	\text{the rank of the partner of $t_4$ in $f_{\al,1}(g,b_1)$}\\&\quad\text{among the $N-1$ half-edges other than $s_4$}\end{aligned} &(s_4<s_3)
		\end{cases}.
	\end{align*}
	Then define a map $h_{\al,\be}:\mG_\al\rightarrow\mG$ by $h_{\al,\be}(g)\coloneqq f_\al(g,b_1,b_2(g))=f_{\al,2}\qty(f_{\al,1}(g,b_1),b_2(g))$. Namely, we restrict $f_{\al,1}$ such that it surely keeps the edge $s_1s_2$, and restrict $f_{\al,2}$ such that it surely breaks up the bond between $s_3$ and $s_4$ and creates the edge $s_3t_4$. Let $g_\be\in\mG_\be$. Let $g=\qty(f^{-1}_\al(g_\be))_1\in\mG_\al$. Let $\qty(f^{-1}_{\al,2}(g_\be))_1\in\mG_{\al,1}=\mG_{s_3s_4}$ be the first coordinate of $f^{-1}_{\al,2}(g_\be)$. The second coordinate of $f^{-1}_{\al,1}\qty(\qty(f^{-1}_{\al,2}\qty(g_\be))_1)$,
	\begin{align*}
		\qty(f^{-1}_{\al,1}\qty(\qty(f^{-1}_{\al,2}\qty(g_\be))_1))_2\in[N-3],
	\end{align*}
	is the rank of the partner of $s_1\wedge s_2$ in $\qty(f^{-1}_{\al,2}(g_\be))_1$ among the $N-3$ half-edges other than $s_1\wedge s_2,s_3,s_4$. Since the edge $s_1s_2$ still exists in $\qty(f^{-1}_{\al,2}(g_\be))_1$,
	\begin{align*}
		b_1=\qty(f^{-1}_{\al,1}\qty(\qty(f^{-1}_{\al,2}\qty(g_\be))_1))_2
	\end{align*}
	holds. Hence $f_{\al,1}\qty(g,b_1)=\qty(f^{-1}_{\al,2}\qty(g_\be))_1$ holds. The second coordinate of $f^{-1}_{\al,2}\qty(g_\be)$, $\qty(f^{-1}_{\al,2}\qty(g_\be))_2\in[N-1]$, is the rank of the partner of $s_3\wedge s_4$ in $g_\be$ among the $N-1$ half-edges other than $s_3\wedge s_4$. $s_3$ is paired to $t_4$ in $g_\be$, and the partner of $s_4$ in $g_\be$ is paired to $t_4$ in $\qty(f^{-1}_{\al,2}\qty(g_\be))_1=f_{\al,1}\qty(g,b_1(g))$. Thus, $b_2(g)=\qty(f^{-1}_{\al,2}\qty(g_\be))_2$ holds. Finally, we get
	\begin{align*}
		h_{\al,\be}\qty(g)=f_{\al,2}\qty(f_{\al,1}(g,b_1),b_2(g))=f_{\al,2}\qty(\qty(f^{-1}_{\al,2}\qty(g_\be))_1,\qty(f^{-1}_{\al,2}\qty(g_\be))_2)=g_\be.
	\end{align*}
	This shows that $h_{\al,\be}$ is a surjection from $\mG_\al$ onto $\mG_\be$ with the desired property.

	\medskip
	Next consider the case of $s_3\wedge s_4<s_1\wedge s_2$. Define $b_1=b_1(g)\in[N-3]$ by
	\begin{align*}
		b_1(g)\coloneqq\begin{cases}
			\text{the rank of $t_4$ among the $N-3$ half-edges other than $s_3,s_1,s_2$}&(s_3<s_4)\\
		\begin{aligned}&	\text{the rank of the partner of $t_4$ in $g$ }\\&\quad\text{among the $N-3$ half-edges other than $s_4,s_1,s_2$} \end{aligned}&(s_4<s_3)
		\end{cases}.
	\end{align*}
	Let $b_2\in[N-1]$ be the rank of $s_1\vee s_2$ among the $N-1$ half-edges other than $s_1\wedge s_2$. Then define a map $h_{\al,\be}:\mG_\al\rightarrow\mG$ by $h_{\al,\be}(g)\coloneqq f_\al(g,b_1(g),b_2)=f_{\al,2}\qty(f_{\al,1}(g,b_1(g)),b_2)$. Namely, we restrict $f_{\al,1}$ such that it surely breaks up the bond between $s_3$ and $s_4$ and creates the edge $s_3t_4$, and restrict $f_{\al,2}$ such that it surely keeps the edge $s_1s_2$.  Let $g_\be\in\mG_\be$. Let $g=\qty(f^{-1}_\al(g_\be))_1\in\mG_\al$. Let $\qty(f^{-1}_{\al,2}(g_\be))_1\in\mG_{\al,1}=\mG_{s_1s_2}$ be the first coordinate of $f^{-1}_{\al,2}(g_\be)$. The second coordinate of $f^{-1}_{\al,1}\qty(\qty(f^{-1}_{\al,2}\qty(g_\be))_1)$,
	\begin{align*}
		\qty(f^{-1}_{\al,1}\qty(\qty(f^{-1}_{\al,2}\qty(g_\be))_1))_2\in[N-3],
	\end{align*}
	is the rank of the partner of $s_3\wedge s_4$ in $\qty(f^{-1}_{\al,2}(g_\be))_1$ among the $N-3$ half-edges other than $s_3\wedge s_4,s_1,s_2$. The edge $s_3t_4$ still exists in $\qty(f^{-1}_{\al,2}(g_\be))_1$, i.e., the partner of $s_3$ in $\qty(f^{-1}_{\al,2}(g_\be))_1$ is $t_4$. The partner of $s_4$ in $\qty(f^{-1}_{\al,2}(g_\be))_1$ is paired to $t_4$ in $g$. Thus,
	\begin{align*}
		b_1(g)=\qty(f^{-1}_{\al,1}\qty(\qty(f^{-1}_{\al,2}\qty(g_\be))_1))_2
	\end{align*}
	holds. Hence $f_{\al,1}\qty(g,b_1(g))=\qty(f^{-1}_{\al,2}\qty(g_\be))_1$ holds. The second coordinate of $f^{-1}_{\al,2}\qty(g_\be)$, $\qty(f^{-1}_{\al,2}\qty(g_\be))_2\in[N-1]$, is the rank of the partner of $s_1\wedge s_2$ in $g_\be$ among the $N-1$ half-edges other than $s_1\wedge s_2$. Since $s_1$ and $s_2$ are paired to each other in $g_\be$, $b_2=\qty(f^{-1}_{\al,2}\qty(g_\be))_2$ holds. Finally, we get
	\begin{align*}
		h_{\al,\be}\qty(g)=f_{\al,2}\qty(f_{\al,1}(g,b_1(g)),b_2)=f_{\al,2}\qty(\qty(f^{-1}_{\al,2}\qty(g_\be))_1,\qty(f^{-1}_{\al,2}\qty(g_\be))_2)=g_\be.
	\end{align*}
	This shows that $h_{\al,\be}$ is a surjection from $\mG_\al$ onto $\mG_\be$ with the desired property.
	
	\medskip
	\underline{Case (b4)}. First consider the case of $s_1\wedge s_2<s_3\wedge s_4$. Let $g\in\mG_\al$. Define $b_1=b_1(g)\in[N-3]$ by
	\begin{align*}
		b_1(g)\coloneqq\begin{cases}\begin{aligned}&
			\text{the rank of the partner of $t_4$ in $g$}\\&\quad\text{among the $N-3$ half-edges other than $s_1,s_3,s_4$}\end{aligned}&(s_1<s_2)\\
			\text{the rank of $t_4$ among the $N-3$ half-edges other than $s_2,s_3,s_4$} &(s_2<s_1)
		\end{cases}.
	\end{align*}
	Then define $b_2=b_2(g)\in[N-1]$ by
	\begin{align*}
		b_2(g)\coloneqq\begin{cases}
			\text{the rank of $s_1$ among the $N-1$ half-edges other than $s_3$}&(s_3<s_4)\\
		\begin{aligned}	&\text{the rank of the partner of $s_1$ in $f_{\al,1}(g,b_1(g))$ }\\&\quad\text{among the $N-1$ half-edges other than $s_4$}\end{aligned} &(s_4<s_3)
		\end{cases}.
	\end{align*}
	Then define a map $h_{\al,\be}:\mG_\al\rightarrow\mG$ by $h_{\al,\be}(g)\coloneqq f_\al(g,b_1(g),b_2(g))=f_{\al,2}\qty(f_{\al,1}(g,b_1(g)),b_2(g))$. Namely, we restrict $f_{\al,1}$ such that it surely breaks up the bond between $s_1$ and $s_2$ and creates the edge $s_2t_4$, and restrict $f_{\al,2}$ such that it surely breaks up the bond between $s_3$ and $s_4$ and creates the edge $s_1s_3$. Let $g_\be\in\mG_\be$. Let $g=\qty(f^{-1}_\al(g_\be))_1\in\mG_\al$. Let $\qty(f^{-1}_{\al,2}(g_\be))_1\in\mG_{\al,1}=\mG_{s_3s_4}$ be the first coordinate of $f^{-1}_{\al,2}(g_\be)$. The second coordinate of $f^{-1}_{\al,1}\qty(\qty(f^{-1}_{\al,2}\qty(g_\be))_1)$,
	\begin{align*}
		\qty(f^{-1}_{\al,1}\qty(\qty(f^{-1}_{\al,2}\qty(g_\be))_1))_2\in[N-3],
	\end{align*}
	is the rank of the partner of $s_1\wedge s_2$ in $\qty(f^{-1}_{\al,2}(g_\be))_1$ among the $N-3$ half-edges other than $s_1\wedge s_2,s_3,s_4$. The edge $s_2t_4$ still exists in $\qty(f^{-1}_{\al,2}(g_\be))_1$, i.e., the partner of $s_2$ in $\qty(f^{-1}_{\al,2}(g_\be))_1$ is $t_4$. The partner of $s_1$ in $\qty(f^{-1}_{\al,2}(g_\be))_1$ is paired to $t_4$ in $g$. Thus,
	\begin{align*}
		b_1(g)=\qty(f^{-1}_{\al,1}\qty(\qty(f^{-1}_{\al,2}\qty(g_\be))_1))_2
	\end{align*}
	holds. Hence $f_{\al,1}\qty(g,b_1(g))=\qty(f^{-1}_{\al,2}\qty(g_\be))_1$ holds. The second coordinate of $f^{-1}_{\al,2}\qty(g_\be)$, $\qty(f^{-1}_{\al,2}\qty(g_\be))_2\in[N-1]$, is the rank of the partner of $s_3\wedge s_4$ in $g_\be$ among the $N-1$ half-edges other than $s_3\wedge s_4$. $s_3$ is paired to $s_1$ in $g_\be$, and the partner of $s_4$ in $g_\be$ is paired to $s_1$ in $\qty(f^{-1}_{\al,2}\qty(g_\be))_1=f_{\al,1}\qty(g,b_1(g))$. Thus, $b_2(g)=\qty(f^{-1}_{\al,2}\qty(g_\be))_2$ holds. Finally, we get
	\begin{align*}
		h_{\al,\be}\qty(g)=f_{\al,2}\qty(f_{\al,1}(g,b_1(g)),b_2(g))=f_{\al,2}\qty(\qty(f^{-1}_{\al,2}\qty(g_\be))_1,\qty(f^{-1}_{\al,2}\qty(g_\be))_2)=g_\be.
	\end{align*}
	This shows that $h_{\al,\be}$ is a surjection from $\mG_\al$ onto $\mG_\be$ with the desired property.
	
	\medskip
	Next consider the case of $s_3\wedge s_4<s_1\wedge s_2$. Let $g\in\mG_\al$. Define $b_1=b_1(g)\in[N-3]$ by
	\begin{align*}
		b_1(g)\coloneqq\begin{cases}
			\text{the rank of $t_4$ among the $N-3$ half-edges other than $s_3,s_1,s_2$}&(s_3<s_4)\\
		\begin{aligned}&	\text{the rank of the partner of $t_4$ in $g$ }\\&\quad\text{among the $N-3$ half-edges other than $s_4,s_1,s_2$} \end{aligned}&(s_4<s_3)
		\end{cases}.
	\end{align*}
	Define $b_2\in[N-1]$ as
	\begin{align*}
		b_2\coloneqq\begin{cases}
			\text{the rank of $s_3$ among the $N-1$ half-edges other than $s_1$}&(s_1<s_2)\\
			\text{the rank of $t_4$ among the $N-1$ half-edges other than $s_2$} &(s_2<s_1)
		\end{cases}.
	\end{align*}
	Then define a map $h_{\al,\be}:\mG_\al\rightarrow\mG$ by $h_{\al,\be}(g)\coloneqq f_\al(g,b_1(g),b_2)=f_{\al,2}\qty(f_{\al,1}(g,b_1(g)),b_2)$. Namely, we restrict $f_{\al,1}$ such that it surely breaks up the bond between $s_3$ and $s_4$ and creates the edge $s_3t_4$, and restrict $f_{\al,2}$ such that it surely breaks up the bond between $s_1$ and $s_2$ and creates the edge $s_1s_3$ if $s_1<s_2$, the edge $s_2t_4$ if $s_2<s_1$. Note that the edge $s_3t_4$ exists in $f_{\al,1}(g,b_1(g))$, thus both of the edges $s_1s_3$ and $s_2t_4$ will be created simultaneously by $h_{\al,\be}$, whether $s_1<s_2$ or $s_2<s_1$. Let $g_\be\in\mG_\be$. Let $g=\qty(f^{-1}_\al(g_\be))_1\in\mG_\al$. Let $\qty(f^{-1}_{\al,2}(g_\be))_1\in\mG_{\al,1}=\mG_{s_1s_2}$ be the first coordinate of $f^{-1}_{\al,2}(g_\be)$. The second coordinate of $f^{-1}_{\al,1}\qty(\qty(f^{-1}_{\al,2}\qty(g_\be))_1)$,
	\begin{align*}
		\qty(f^{-1}_{\al,1}\qty(\qty(f^{-1}_{\al,2}\qty(g_\be))_1))_2\in[N-3],
	\end{align*}
	is the rank of the partner of $s_3\wedge s_4$ in $\qty(f^{-1}_{\al,2}(g_\be))_1$ among the $N-3$ half-edges other than $s_3\wedge s_4,s_1,s_2$. The edge $s_3t_4$ exists in $\qty(f^{-1}_{\al,2}(g_\be))_1$, i.e., the partner of $s_3$ in $\qty(f^{-1}_{\al,2}(g_\be))_1$ is $t_4$. The partner of $s_4$ in $\qty(f^{-1}_{\al,2}(g_\be))_1$ is paired to $t_4$ in $g$. Thus,
	\begin{align*}
		b_1(g)=\qty(f^{-1}_{\al,1}\qty(\qty(f^{-1}_{\al,2}\qty(g_\be))_1))_2
	\end{align*}
	holds. Hence $f_{\al,1}\qty(g,b_1(g))=\qty(f^{-1}_{\al,2}\qty(g_\be))_1$ holds. The second coordinate of $f^{-1}_{\al,2}\qty(g_\be)$, $\qty(f^{-1}_{\al,2}\qty(g_\be))_2\in[N-1]$, is the rank of the partner of $s_1\wedge s_2$ in $g_\be$ among the $N-1$ half-edges other than $s_1\wedge s_2$. $s_1$ is paired to $s_3$ in $g_\be$, and $s_2$ is paired to $t_4$ in $g_\be$. Thus, $b_2=\qty(f^{-1}_{\al,2}\qty(g_\be))_2$ holds. Finally, we get
	\begin{align*}
		h_{\al,\be}\qty(g)=f_{\al,2}\qty(f_{\al,1}(g,b_1(g)),b_2)=f_{\al,2}\qty(\qty(f^{-1}_{\al,2}\qty(g_\be))_1,\qty(f^{-1}_{\al,2}\qty(g_\be))_2)=g_\be.
	\end{align*}
	This shows that $h_{\al,\be}$ is a surjection from $\mG_\al$ onto $\mG_\be$ with the desired property.
	
	\medskip
	\underline{Case (b5)}. We may assume that $s_1\wedge s_2<s_3\wedge s_4$ without loss of generality in this case. Let $g\in\mG_\al$. Let $b_1\in[N-3]$ be the rank of $s_1\vee s_2$ among the $N-3$ half-edges other than $s_1\wedge s_2,s_3,s_4$. Define $b_2\in[N-1]$ as
	\begin{align*}
		b_2\coloneqq\begin{cases}
			\text{the rank of $s_1$ among the $N-1$ half-edges other than $s_3$}&(s_3<s_4)\\
			\text{the rank of $s_2$ among the $N-1$ half-edges other than $s_4$} &(s_4<s_3)
		\end{cases}.
	\end{align*}
	Then define a map $h_{\al,\be}:\mG_\al\rightarrow\mG$ by $h_{\al,\be}(g)\coloneqq f_\al(g,b_1,b_2)=f_{\al,2}\qty(f_{\al,1}(g,b_1),b_2)$. Namely, we restrict $f_{\al,1}$ such that it surely keeps the edge $s_1s_2$, and restrict $f_{\al,2}$ such that it surely breaks up the bond between $s_3$ and $s_4$ and creates the edge $s_1s_3$ if $s_3<s_4$, the edge $s_2s_4$ if $s_4<s_3$. Note that the edge $s_1s_2$ is kept in $f_{\al,1}(g,b_1)$, thus both of the edges $s_1s_3$ and $s_2s_4$ will be created simultaneously by $h_{\al,\be}$, whether $s_3<s_4$ or $s_4<s_3$. Let $g_\be\in\mG_\be$. Let $g=\qty(f^{-1}_\al(g_\be))_1\in\mG_\al$. Let $\qty(f^{-1}_{\al,2}(g_\be))_1\in\mG_{\al,1}=\mG_{s_3s_4}$ be the first coordinate of $f^{-1}_{\al,2}(g_\be)$. The second coordinate of $f^{-1}_{\al,1}\qty(\qty(f^{-1}_{\al,2}\qty(g_\be))_1)$,
	\begin{align*}
		\qty(f^{-1}_{\al,1}\qty(\qty(f^{-1}_{\al,2}\qty(g_\be))_1))_2\in[N-3],
	\end{align*}
	is the rank of the partner of $s_1\wedge s_2$ in $\qty(f^{-1}_{\al,2}(g_\be))_1$ among the $N-3$ half-edges other than $s_1\wedge s_2,s_3,s_4$. Since the edge $s_1s_2$ exists in $\qty(f^{-1}_{\al,2}(g_\be))_1$,
	\begin{align*}
		b_1=\qty(f^{-1}_{\al,1}\qty(\qty(f^{-1}_{\al,2}\qty(g_\be))_1))_2
	\end{align*}
	holds. Hence $f_{\al,1}\qty(g,b_1)=\qty(f^{-1}_{\al,2}\qty(g_\be))_1$ holds. The second coordinate of $f^{-1}_{\al,2}\qty(g_\be)$, $\qty(f^{-1}_{\al,2}\qty(g_\be))_2\in[N-1]$, is the rank of the partner of $s_3\wedge s_4$ in $g_\be$ among the $N-1$ half-edges other than $s_3\wedge s_4$. $s_3$ is paired to $s_1$ in $g_\be$, and $s_4$ is paired to $s_2$ in $g_\be$. Thus, $b_2=\qty(f^{-1}_{\al,2}\qty(g_\be))_2$ holds. Finally, we get
	\begin{align*}
		h_{\al,\be}\qty(g)=f_{\al,2}\qty(f_{\al,1}(g,b_1),b_2)=f_{\al,2}\qty(\qty(f^{-1}_{\al,2}\qty(g_\be))_1,\qty(f^{-1}_{\al,2}\qty(g_\be))_2)=g_\be.
	\end{align*}
	This shows that $h_{\al,\be}$ is a surjection from $\mG_\al$ onto $\mG_\be$ with the desired property.

	\medskip
	Therefore, there exists a desired surjection $h_{\al,\be}$ for each of Cases (b1)--(b5), completing the proof of (b) and thus the lemma. \qed

	\subsection{Proofs for Appendix \ref{stein coupling bound second term}}
			\label{proof 2}
			
	\subsubsection{Proof of Lemma \ref{intersection limited creation}}
	Fix an arbitrary event $\omega$. Let $v^\al_i$, $i=1,\dots,v(H)$ denote the labels of the vertices of $\al$. For each $1\le i\le v(H)$, define
	\begin{align*}
		\La^i\coloneqq\qty{\be\in\Ga_1:\text{$\be\neq\al$ and $\be$ shares vertex $v^\al_i$ with $\al$}}.
	\end{align*}
	Then obviously $\qty{\be\in\Ga_1\setminus\{\al\}:\be\cap\al\neq\varnothing}=\bigcup_{i=1}^{v(H)}\La^i=\biguplus_{i=1}^{v(H)}\qty(\La^i\setminus\bigcup_{j=1}^{i-1}\La^j)$ and
	\begin{align}\label{intersection limited creation union bound}
		\sum_{\substack{\be\in\Ga_1\setminus\{\al\}\\\be\cap\al\neq\varnothing}}J_{\be\al}=\sum_{i=1}^{v(H)}\sum_{\be\in\La^i\setminus\bigcup_{j=1}^{i-1}\La^j}J_{\be\al}\le\sum_{i=1}^{v(H)}\sum_{\be\in\La^i}J_{\be\al}.
	\end{align}
	Now fix arbitrary $1\le i\le v(H)$, and suppose that $J_{\be\al}=1$ for some $\be\in\La^i$. Then $\be$ is realized in $\mathfrak{g}_\al$ (which is equal to $\mathfrak{g}$ if $I_\al=0$, or the perfect matching of $N$ half-edges after the switching if $I_\al=1$). Any other member $\be'$ of $\La^i$ shares vertex $v^\al_i$ with $\be$, so that $\be'\neq\be$ and $\be'\cap\be\neq\varnothing$. Thus by Lemma \ref{intersecting different isolated trees}, any other member $\be'$ of $\La^i$ is incompatible with $\be$ as a matching of half-edges, and cannot be realized in $\mathfrak{g}_\al$. Therefore, if $J_{\be\al}=1$ for some $\be\in\La^i$, then $J_{\be'\al}=0$ for all $\be'\in\La^i\setminus\{\be\}$. This shows that for any $1\le i\le v(H)$, the sum $\sum_{\be\in\La^i}J_{\be\al}$ is at most $1$. Now the claim follows from \eqref{intersection limited creation union bound}. \qed

		\subsubsection{Proof of Lemma \ref{disjoint limited destruction}}
	Fix an arbitrary event $\omega$. If $I_\al=0$, then $J_{\be\al}=I_\be$ and $I_\be(1-J_{\be\al})=I_\be(1-I_\be)=0$. Thus we may assume that $I_\al=1$, i.e., $\mathfrak{g}\in\mG_\al$. Suppose for contradiction that there are distinct $\be_1,\dots,\be_{M}\in\Ga_1$ with $M\ge e(H)+1$ such that $\be_{m}\cap\al=\varnothing$ and $I_{\be_m}(1-J_{\be_m\al})=1$, $m=1,\dots, M$. Since $I_{\be_1}=\cdots=I_{\be_M}=1$, by Lemma \ref{intersecting different isolated trees}, $\be_1,\dots,\be_M$ are all disjoint. For $\ell=1,\dots, e(H)$, let $R_{\al,\be_m,\ell}$ be as in Lemma \ref{destroyed al be disjoint} (i.e., the ranks of the half-edges of $\be_m$ in the $\ell$-th stage of the switching procedure with respect to $\al$). Then for each $1\le \ell\le e(H)$, $R_{\al,\be_m,\ell}$, $1\le m\le M$ are disjoint. For each $1\le m\le M$, since $\mathfrak{g}\in\mG_\al\cap\mG_{\be_m}$ and $J_{\be_m\al}=0$, there exists $1\le\ell_m\le e(H)$ such that $B_{\al,\ell_m}\in R_{\al,\be_m,\ell_m}$ by Lemma \ref{destroyed al be disjoint}. Since $M>e(H)$, there are $1\le m_1\neq m_2\le M$ such that $\ell_{m_1}=\ell_{m_2}$ by the pigeonhole principle. Writing $\ell_*\coloneqq\ell_{m_1}=\ell_{m_2}$, we have $B_{\al,\ell_*}\in R_{\al,\be_{m_1},\ell_*}$ and $B_{\al,\ell_*}\in R_{\al,\be_{m_2},\ell_*}$, which contradicts $R_{\al,\be_{m_1},\ell_*}\cap R_{\al,\be_{m_2},\ell_*}=\varnothing$. Therefore, the number of $\be\in\Ga_1$ such that $\be\cap\al=\varnothing$ and $I_\be(1-J_{\be\al})=1$ is at most $e(H)$, proving the lemma. \qed
	
	\subsection{Proof of Lemma \ref{intersecting isolated tree and selfloop or double edge} in Appendix \ref{stein coupling bound third term}}
	\label{proof 3}
	
	As we have seen in Lemma \ref{cardinality estimates}(e), if $\al$ is a possible isolated edge, then there is no $\be\in\Ga_2$ intersecting $\al$. Thus we may assume that $\al$ is a possible isolated 2-star.
	
	\medskip
	 If $\be$ is a possible self-loop and $\al\cap\be\neq\varnothing$, then $\be$'s only vertex must be the degree 2 vertex of $\al$. In $\al$, each of the two half-edges incident to that degree 2 vertex must be paired to a half-edge incident to a degree 1 vertex, whereas in $\be$, the two half-edges incident the said degree 2 vertex must be paired to each other. Thus $\al$ and $\be$ cannot coexist.
	
	\medskip
	If $\be$ is a possible double edge and $\al\cap\be\neq\varnothing$, then one of the two vertices of $\be$ must be the degree 2 vertex of $\al$. In $\al$, each of the two half-edges incident to that degree 2 vertex must be paired to a half-edge incident to a degree 1 vertex. On the other hand, in $\be$, the two half-edges incident to the said degree 2 vertex must be paired to two half-edges incident to the other vertex of $\be$, which is of degree higher than or equal to 2, thus does not belong to $\al$. Therefore $\al$ and $\be$ cannot coexist in this case also. \qed

	\subsection{Proof of Lemma \ref{two double edges share a common edge} in Appendix \ref{stein coupling bound fifth term}}
	\label{proof 5}
	
	We may describe $\al$ and $\be$ sharing a common edge as follows: Let $s_1,s_2,s_3$ be three distinct half-edges incident to a vertex, and let $t_1,t_2,t_3$ be three distinct half-edges incident to another vertex. We assume that $\al$ is composed of two possible edges $s_1t_1$ and $s_2t_2$, $\be$ is composed of two possible edges $s_2t_2$ and $s_3t_3$, and $s_2t_2$ is the common edge.
	
	\medskip
We first show that if $I_\be=0$, then $J_{\be\al}=0$. Suppose that $I_\be=0$. Then $s_2t_2$ or $s_3t_3$ is not realized in $\mathfrak{g}$. If $I_\al=0$, then $J_{\be\al}=I_\be=0$. Thus it suffices to consider the case where $I_\al=1$ (i.e., $s_1t_1$ and $s_2t_2$ are realized in $\mathfrak{g}$) and $s_3t_3$ is not realized in $\mathfrak{g}$. In this case, we can follow the proof of Lemma \ref{disjoint multigraphs non creation} (Appendix \ref{proof of disjoint multigraphs non creation}) with all references to
$\be$ replaced by the possible edge $s_3t_3$ to conclude that $\mathfrak{g}_\al\not\supset s_3t_3$, whence $J_{\be\al}=0$. To sum up, $J_{\be\al}=0$ if $I_\be=0$, and thus $\e\qty[I_\al|I_\beta-J_{\beta\al}|]=\e\qty[I_\al I_\be\qty(1-J_{\be\al})]=\p\qty(I_\al=I_\be=1,J_{\be\al}=0)$.
	
	\medskip
	In order to compute the probability $\p\qty(I_\al=I_\be=1,J_{\be\al}=0)$, first consider the case of $s_1\wedge t_1<s_2\wedge t_2$. Let $R_1$ denote the set of the ranks of $s_3$ and $t_3$ among the $N-3$ half-edges other than $s_1\wedge t_1,s_2,t_2$, and let $r_2$ denote the rank of $s_2\vee t_2$ among the $N-1$ half-edges other than $s_2\wedge t_2$. Let $g\in\mG_\al\cap\mG_\be$ and $b_{1}\in[N-3]$, and let $g_1\coloneqq f_{\al,1}(g,b_1)$. Let $u_{1,1}$ denote the $b_1$-th smallest half-edge among the $N-3$ half-edges other than $s_1\wedge t_1,s_2,t_2$ and let $u_{1,2}$ be $u_{1,1}$'s partner in $g$. If $b_1\notin R_1$, then by the definition of $R_1$, $\{u_{1,1},u_{1,2} \}\cap\{s_2,t_2,s_3,t_3 \}=\varnothing$, thus
	\begin{align*}
		g_1=f_{\al,1}(g,b_1)=\qty(g\setminus\qty{ \{s_{1},t_{1}\},\{u_{1,1},u_{1,2}\} })\uplus\qty{ \{s_{1}\wedge t_{1},u_{1,1} \},\{s_{1}\vee t_{1},u_{1,2}\} }\supset\be.
	\end{align*}
	Hence by the definition of $r_2$, $f_\al(g,b_1,r_2)=f_{\al,2}(g_1,r_2)=g_1\supset\be$. Conversely, if $b_1\in R_1$, then $u_{1,1}u_{1,2}=s_3t_3$, $g_1=f_{\al,1}(g,b_1)\not\supset\be$, and
	\begin{align*}
		f_\al(g,b_1,b_2)=f_{\al,2}(g_1,b_2)=\qty(g_1\setminus\qty{ \{s_{2},t_{2}\},\{u_{2,1},u_{2,2}\} })\uplus\qty{ \{s_{2}\wedge t_{2},u_{2,1} \},\{s_{2}\vee t_{2},u_{2,2}\} }\not\supset\be
	\end{align*}
	for any $b_2\in[N-1]$, where $u_{2,1}$ denotes the $b_2$-th smallest half-edge among the $N-1$ half-edges other than $s_2\wedge t_2$ and $u_{2,2}$ denotes $u_{2,1}$'s partner in $g_1$. Similarly, if $b_2\in[N-1]\setminus\{r_2\}$, then for any $b_1\in[N-3]$, $f_\al(g,b_1,b_2)=f_{\al,2}(g_1,b_2)\not\supset s_2t_2$, hence $f_\al(g,b_1,b_2)\not\supset\be$. Therefore, for $g\in\mG_\al\cap\mG_\be$, $f_\al(g,b_1,b_2)\supset\be$ if and only if $b_1\in[N-3]\setminus R_1$ and $b_2=r_2$. From this observation, and using the independence, we have
	\begin{align*}
		\p\qty(I_\al=I_\be=1,J_{\be\al}=0)=&\p\qty(\{I_\al=I_\be=1\}\cap\qty(\{B_{\al,1}\in R_1\}\cup \{B_{\al,2}\neq r_2\}))\\
		=&\p\qty(I_\al=I_\be=1)\qty(1-\p\qty(\{B_{\al,1}\notin R_1\}\cap\{B_{\al,2}=r_2\}))\\
		=&\frac{1}{((N-1))_3}\qty(1-\frac{N-5}{N-3}\frac{1}{N-1}).
	\end{align*}
	
	\medskip
	Next consider the case of $s_2\wedge t_2<s_1\wedge t_1$. Let $r_1$ denote the rank of $s_2\vee t_2$ among the $N-3$ half-edges other than $s_2\wedge t_2,s_1,t_1$, and let $R_2$ denote the set of the ranks of $s_2,t_2,s_3,t_3$ among the $N-1$ half-edges other than $s_1\wedge t_1$. Let $g\in\mG_\al\cap\mG_\be$, $b_1\in[N-3]$ and $b_{2}\in[N-1]$, and let $g_1\coloneqq f_{\al,1}(g,b_1)$. Let $u_{2,1}$ denote the $b_2$-th smallest half-edge among the $N-1$ half-edges other than $s_1\wedge t_1$ and let $u_{2,2}$ be $u_{2,1}$'s partner in $g_1$. If $b_1=r_1$ and $b_2\notin R_2$, then by the definition of $r_1$, $g_1=f_{\al,1}(g,r_1)=g\supset\be$, and by the definition of $R_2$, $\{u_{2,1},u_{2,2} \}\cap\{s_2,t_2,s_3,t_3 \}=\varnothing$. Thus
	\begin{align*}
		f_\al(g,r_1,b_2)=f_{\al,2}(g_1,b_2)=\qty(g_1\setminus\qty{ \{s_{1},t_{1}\},\{u_{2,1},u_{2,2}\} })\uplus\qty{ \{s_{1}\wedge t_{1},u_{2,1} \},\{s_{1}\vee t_{1},u_{2,2}\} }\supset\be.
	\end{align*}
	Conversely, if $b_1\in[N-3]\setminus\{r_1\}$, then $g_1=f_{\al,1}(g,b_1)\not\supset s_2t_2$, hence $f_\al(g,b_1,b_2)=f_{\al,2}(g_1,b_2)\not\supset s_2t_2$ and $f_\al(g,b_1,b_2)\not\supset\be$ for any $b_2\in[N-1]$. Similarly, if $b_2\in R_2$, then $u_{2,1}$ will be one of $s_2,t_2,s_3,t_3$ and paired to $s_1\wedge t_1$ in $f_\al(g,b_1,b_2)=f_{\al,2}(g_1,b_2)$. It means that $s_2t_2$ or $s_3t_3$ will not be formed in $f_\al(g,b_1,b_2)$, implying that $f_\al(g,b_1,b_2)\not\supset \be$ for any $b_1\in[N-3]$. Therefore, for $g\in\mG_\al\cap\mG_\be$, $f_\al(g,b_1,b_2)\supset\be$ if and only if $b_1=r_1$ and $b_2\in[N-1]\setminus R_2$. From this observation, and using the independence, we have
	\begin{align*}
		\p\qty(I_\al=I_\be=1,J_{\be\al}=0)=&\p\qty(\{I_\al=I_\be=1\}\cap\qty(\{B_{\al,1}\neq r_1\}\cup \{B_{\al,2}\in R_2\}))\\
		=&\p\qty(I_\al=I_\be=1)\qty(1-\p\qty(\{B_{\al,1}= r_1\}\cap\{B_{\al,2}\notin R_2\}))\\
		=&\frac{1}{((N-1))_3}\qty(1-\frac{1}{N-3}\frac{N-5}{N-1}).
	\end{align*}
	
	\medskip
	In either case,
	\begin{align*}
		\p\qty(I_\al=I_\be=1,J_{\be\al}=0)=\frac{1}{((N-1))_3}\qty(1-\frac{N-5}{((N-1))_2}),
	\end{align*}
	and this completes the proof. \qed

\end{document}